\pdfoutput=1

\documentclass[a4paper,10pt, dvipsnames]{book}

  \usepackage{CJKutf8}
\usepackage[minimal]{leadsheets}
\useleadsheetslibraries{musicsymbols}
 \usepackage{setspace}
\usepackage{amsmath, amsthm,  amsfonts, etoolbox, changepage,lipsum, bigints, relsize, mathtools, color, tikz-cd, amssymb, float, wrapfig,dsfont}
\usepackage{aliascnt}
\usepackage{accents}
\usepackage{amsthm,  amsfonts, etoolbox, changepage,lipsum, bigints, relsize, mathtools, color, tikz-cd, amssymb, tikz, multirow ,adjustbox, tabularx, rotating, booktabs, lscape,xcolor,colortbl}
\usepackage{hyperref}
\usepackage{cleveref}
\usepackage{url}
\usepackage[english]{babel}
\usepackage[utf8]{inputenc}
\usepackage[T1]{fontenc}
\usepackage{textcomp}
\input xypic
\xyoption{all}
\usepackage[all,arc]{xy}
\usepackage{enumerate}
\usepackage{mathrsfs}
\usepackage{mathtools} 
\usepackage[style=alphabetic]{biblatex} 
\usetikzlibrary{decorations.markings}
\usetikzlibrary{snakes}
\renewcommand\qedsymbol{Q.E.D.}

\usepackage{comment}
\usetikzlibrary{calc, decorations.pathreplacing,angles,quotes}

\usepackage{caption}
\usepackage{subcaption}   
   
\pgfkeys{tikz/mymatrixenv/.style={decoration={brace},every left delimiter/.style={xshift=8pt},every right delimiter/.style={xshift=-8pt}}}
\pgfkeys{tikz/mymatrix/.style={matrix of math nodes,nodes in empty cells,left delimiter={[},right delimiter={]},inner sep=1pt,outer sep=1.5pt,column sep=2pt,row sep=2pt,nodes={minimum width=20pt,minimum height=10pt,anchor=center,inner sep=0pt,outer sep=0pt}}}
\pgfkeys{tikz/mymatrixbrace/.style={decorate,thick}}

\hypersetup{
    colorlinks=true,
    linkcolor=blue,
    filecolor=magenta,      
    urlcolor=cyan,
    citecolor = blue
}

  \theoremstyle{plain}

  \newtheorem{definition}{Definition}[section]
\newtheorem{theorem}[definition]{Theorem}
\newtheorem{lemma}[definition]{Lemma}
  
  \newtheorem{question}[definition]{Question}
 \newtheorem{example}[definition]{Example}  
 \newtheorem{corollary}[definition]{Corollary}
 
 \newtheorem{conjecture}[definition]{Conjecture}
 \newtheorem{assumption}[definition]{Assumption}
 \newtheorem{proposition}[definition]{Proposition}
   \newtheorem{remark}[definition]{Remark}
  \newtheorem{note}{Note}[section]

\crefname{lemma}{Lemma}{Lemma}
  \crefname{corollary}{Corollary}{Corollary}
  \crefname{theorem}{Theorem}{Theorem}
  \crefname{definition}{Definition}{Definition}
   \crefname{proposition}{Proposition}{Proposition}
 \crefname{section}{Section}{Section} 
   \crefname{construction}{Construction}{Construction}
   \crefname{generalization}{Generalization}{Generalization}
  \crefname{construction}{Construction}{Construction}
  \crefname{notation}{Notation}{Notation}
   \crefname{example}{Example}{Example}
  \crefname{remark}{Remark}{Remark}
  \crefname{fact}{Fact}{Fact}
  \crefname{conjecture}{Conjecture}{Conjecture}
\crefname{chapter}{Chapter}{Chapter}  
  
  \crefname{motivation}{Motivation}{Motivation}  
  \crefname{figure}{Figure}{Figure}  
  \crefname{table}{Table}{Table}

  \numberwithin{equation}{section}
 \numberwithin{figure}{section}
  \numberwithin{figure}{subsection}

  \renewcommand{\cH}{{\mathcal H}}
  
  \newcommand{\cA}{{\mathcal A}}

  \newcommand{\cF}{{\mathcal F}}
  \renewcommand{\cL}{{\mathcal L}}
  \renewcommand{\cD}{{\mathcal D}}
  
  \newcommand{\cC}{{\mathcal C}}
  \newcommand{\cG}{{\mathcal G }}
  \newcommand{\cO}{{\mathcal O }}
  \newcommand{\cT}{{\mathcal T }}
  \newcommand{\cK}{{\mathcal K }}
  \newcommand{\cS}{{\mathcal S }}
    \newcommand{\cI}{{\mathcal I }}

  \newcommand{\cB}{{\mathcal B }}
    
        \newcommand{\cX}{{\mathcal X }}
         \newcommand{\calR}{{\mathcal R }}
       
       \newcommand{\cU}{{\mathcal U }}
  \renewcommand{\cR}{\mathcal{R}}
  \newcommand{\sH}{\mathscr{H}}

  \newcommand{\Hom}{\text{Hom}}
  \newcommand{\HOM}{\text{HOM}}
      \newcommand{\End}{\text{End}}
  \newcommand{\Com}{\text{Com}}
  \newcommand{\Kom}{\text{Kom}}
  \newcommand{\cone}{\text{cone}}
  \newcommand{\id}{\text{id}}
  \newcommand{\Ba}{\mathscr{B}}
  \newcommand{\Aa}{\mathscr{A}}
  \newcommand{\bigrdim}{\text{bigrdim}}
  
  \newcommand{\sgn}{\text{sgn}}

   \newcommand{\ba}{\begin{eqnarray}}
   \newcommand{\na}{\end{eqnarray}}
   \newcommand{\ban}{\begin{eqnarray*}}
   \newcommand{\nan}{\end{eqnarray*}}

 \newcommand{\fs}{{\mathfrak s}}
 \newcommand{\fl}{{\mathfrak l}}
 \newcommand{\h}{{\mathfrak h}}
 
 \newcommand{\m}{\mathfrak{m}}
  \newcommand{\fc}{\mathfrak{c}}
  \newcommand{\fC}{\mathfrak{C}}
    \newcommand{\fD}{\mathfrak{D}}

     \newcommand{\fq}{\mathfrak{q}}
     \newcommand{\fp}{\mathfrak{p}}
         \newcommand{\ff}{\mathfrak{f}}
          
            \newcommand{\fB}{\mathfrak{B}}
 
  \newcommand{\fT}{\mathfrak{T}}

  \newcommand{\C}{{\mathbb C}}
  \newcommand{\R}{{\mathbb R}}
  \newcommand{\Z}{{\mathbb Z}}
   \newcommand{\Q}{{\mathbb Q}}
   \newcommand{\N}{\mathbb N}
    
       \newcommand{\bS}{\mathbb S}
        
         \newcommand{\D}{\mathbb D}
          \newcommand{\F}{\mathbb F}
  \newcommand{\A}{{\mathbb A}}
     \newcommand{\B}{{\mathbb B}}

  \renewcommand{\d}{\delta}
  
  \newcommand{\sB}{\mathscr{B}}
  
\newcommand{\sD}{\mathscr{D}}

  \newcommand{\sF}{\mathscr{F}}

       \newcommand{\sG}{\mathscr{G}}
  
\newcommand{\ra}{\rightarrow}

\newcommand{\xra}{\xrightarrow}

\newcommand{\hra}{\hookrightarrow}

\newcommand{\wt}{\widetilde}
\newcommand{\ol}{\overline}
\newcommand{\Ra}{\Rightarrow}

\newcommand\rb[1]{\textcolor{red}{\textbf{#1}}}

\newcommand{\DA}{\D^A_{2n}}
\newcommand{\DAz }{\D^A_{2n} \setminus \Delta_0}
\newcommand{\DB}{\D^B_{n+1}}
\newcommand{\bs}{\backslash}
\newcommand{\bsz}{\backslash \{ 0 \} }
\newcommand{\MCG}{\text{MCG}}
\newcommand{\Diff}{\text{Diff}}
\newcommand{\RP}{\mathbb{R} \text{P}^1}
\newcommand{\udfC}{\undertilde{\widetilde{\mathfrak{C}}}}
\newcommand{\RDA}{\fD^A_{\Delta}}
\newcommand{\RDB}{\fD^B_{\Lambda}}
\newcommand{\RDAz}{\fD^A_{\Delta_{0}}  }
\newcommand{\wRDA}{\wt{\fD}^A_{\Delta}}
\newcommand{\wRDB}{\check{\fD}^B_{\Lambda}}
\newcommand{\wRDAz}{\wt{\fD}^A_{\Delta_{0}}}
\newcommand{\MCGB}{\text{MCG} \left( \D^B_{n+1}, \{ 0 \} \right)}
\newcommand{\MCGA}{\text{MCG} \left( \D^A_{2n} \right)}

\newcommand{\DiffB}{\text{Diff} \left( \D^B_{n+1}, \{0\} \right)}
\newcommand{\ut}{\undertilde}
\newcommand{\wuc}{\undertilde{\widetilde{c}}}

     \def\uv{\underline{v}}
       \def\uw{\underline{w}}

\newcommand{\bit}[1]{\textit{\textbf{#1}}}

  \newcommand{\<}{\langle}
  \renewcommand{\>}{\rangle}

\usepackage{scalerel,stackengine}
\stackMath
\newcommand\reallywidehat[1]{%
\savestack{\tmpbox}{\stretchto{%
  \scaleto{%
    \scalerel*[\widthof{\ensuremath{#1}}]{\kern-.6pt\bigwedge\kern-.6pt}%
    {\rule[-\textheight/2]{1ex}{\textheight}}%WIDTH-LIMITED BIG WEDGE
  }{\textheight}% 
}{0.5ex}}%
\stackon[1pt]{#1}{\tmpbox}%
}
\newtheorem*{theorem*}{Theorem}

\newtheoremstyle{named}{}{}{\itshape}{}{\bfseries}{.}{.5em}{\thmnote{#3's }#1}
\theoremstyle{named}

\newtheoremstyle{name}{}{}{\itshape}{}{\bfseries}{.}{.5em}{\thmnote{#3}#1}
\theoremstyle{name}

\makeatletter
\newcommand\xrightleftarrows[2][]{\ext@arrow 0099{\longrightleftarrowsfill@}{#1}{#2}}
\def\longrightleftarrowsfill@{\arrowfill@\leftarrow\relbar\rightarrow}
\makeatother

  \newcommand{\blanknonumber}{\newpage\thispagestyle{empty}}

  \usepackage{graphics}

\begin{document}

  % set page numbers to roman and suppress chapter numbers
  \frontmatter

  % remove or switch the order of these as you see fit
           \begin{titlepage}
\begin{center}

\vspace*{\fill} \LARGE 
                    \textbf{{\Huge From Homological Algebra to Topology }via \textit{Type B Zigzag Algebra} and Heisenberg Algebra.}
\\
\vfill\vfill\Large
                          Kie Seng, Nge
                          \\
                          \vfill
                          \begin{align*}
                          \text{Primary Supervisor and Chair}&: \text{Assoc.  Prof.  Anthony M. Licata} \\
                          \text{Associate Supervisor } 1&: \text{Assoc. Prof. Uri Onn}		\\
                          \text{Associate Supervisor } 2&: \text{Dr. Asilata Bapat}
                           \end{align*}
\\
\vfill\vfill
                          March,  2022
\\
\vfill\vfill \large
         A thesis submitted for the degree of Doctor of Philosophy\\
         of the Australian National University
\vfill
         \includegraphics{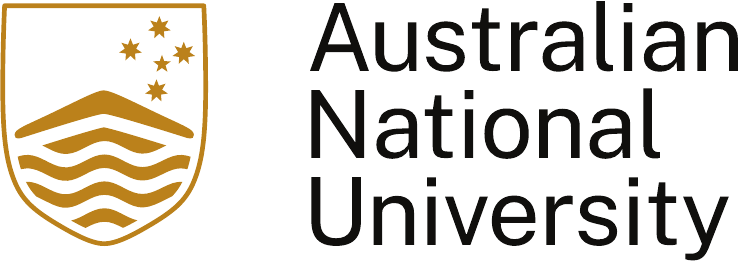}

\end{center}

\end{titlepage}
\blanknonumber

\blanknonumber\ \blanknonumber

\vspace*{\fill}

\begin{center}\emph{
{\large For my father and mother.}
}
\end{center}

\vfill\vfill\vfill

\newpage

\vspace*{\fill}

\begin{singlespace} 
\begin{center}

\begin{CJK*}{UTF8}{gbsn}

\begin{LARGE}
廿八数至寒窗毕， \\
数算真相美妙哉， \\
同调代数论拓扑， \\
世界万千皆弦数。 \\
\end{LARGE}

\end{CJK*}

\end{center}

\end{singlespace}

\vspace{15mm}

\begin{center}
\begin{LARGE}
\textit{Twenty-eight years worth of maths,} \\
\textit{Truth and beauty lie within, }\\
\textit{Topos carol algebras echo,}  \\
\textit{Tick-tocking string whirring gravity.} \\
\end{LARGE}
\end{center}

\vspace{15mm}

\begin{center}

\begin{LARGE}

Berpuluh tahun mengejar kejayaan, \\

Ilmu hisab betapa benar dan cantik, \\

Huruf tanda mengukur garisan, \\

Daya bersatu dawai berdetik. \\
\end{LARGE}

\end{center}

\vfill\vfill\vfill\vfill\vfill\vfill
\blanknonumber
          
\chapter*{Declaration}\label{declaration}
\thispagestyle{empty}
The work in this thesis is my own except where otherwise stated.

\vspace{1in}

\hfill\hfill\hfill
Kie Seng, Nge
\hspace*{\fill}
\blanknonumber 
          
\chapter*{Acknowledgements} \label{ack}
\addcontentsline{toc}{chapter}{Acknowledgements}
\begin{quote}
``We acknowledge and celebrate the First Australians on whose lands the Australian National University (ANU) operates, and pay our respect to the elders of the Ngambri and Ngunnawal people, past, present, and emergent.''
\end{quote}

\begin{CJK*}{UTF8}{gbsn}
First and foremost, I would like to express my greatest gratitude to my beloved supervisor, Anthony Michael Licata or simply Tony, a more familiar name known by his collaborators, his colleagues, his students, and his family.
  I thank Tony for generously accepting me as his first, if not, the second graduate student.
  I am grateful that he brought me into the field of representation theory which has vast connections to other fields. 
  Moreover, he shaped me into an independent and thoughtful mathematician through fortnightly critical dialogues with him.
  As a caring supervisor,  he often cares of his students' welfare.
  Whenever I reached a bottleneck, he could try his best to give the most useful advices, which I appreciated very much.
  Thank you for the conversations we had. 
  Thank you, Tony, as my teacher\footnote{Chinese say, ``A day as a teacher, a lifetime as a father (一日为师，终身为父)''; respect your teacher as your father. }.
  
\end{CJK*}

    Second, I would like to thank my PhD\footnote{Doctor of Philosophy. }-mate, Edmund Xian Chen Heng. 
       PhD\footnote{
       However, permanent head damage is the phrase that netizen refer to whenever they find PhD's are hard to get along with.} is usually completed singly as the topic might get specific, but I'm lucky to have met him in my PhD journey.
    We started our first bachelor degree together at the ANU.
    We were not as close back then, not until he chose to start his PhD with my supervisor.
In the first two years of PhD, we shared the same learning and conference experience before we broke off pursuing our own final project topics.
  In particular, two of the conferences we participated in have resulted in a joint paper on type $B$ zigzag algebra and its sequel to $2$-category that constitute the first two chapters in this thesis.
  I find him particularly easy to communicate with, especially since we both come from the same country, Malaysia.
  Whenever there are two clear paths we can proceed in doing maths, without a second thought, we would pick the distinct path with each other.
  However, it has always been pleasant\footnote{I found it amusing on one occasion that we brought our maths conversation to a local Chinese restaurant with a mix of English and Chinese terms while having lunch together; 
  the customers around us kept turning their heads to us, wondering why we were talking about the languages they knew but didn't quite understand.} and fruitful discussions with him; there are countless days where we discussed maths at the office of Mathematical Science Institute (MSI) until the sun set and walked home in the dark.
  I have learnt a lot from him perceiving maths from his perspective; he possessed good qualities that make him suitable to be a mathematics researcher.
 Congratulation to Heng that he got a postdoctorate offer from Institut des Hautes Études Scientifiques. 
 Good luck to Heng in his future endeavour and undertakings!

      Not forgetting, Heng and I would like to thank Peter McNamara for suggesting the construction of the type $B$ zigzag algebra during the Kiola Conference $2019$.
      We would also like to thank Hoel Queffelec and Daniel Tubbenhauer for their helpful comments on the draft(s) of our earlier version of the paper \cite{hengseng2019curves}.
      We thank Thorge Jensen, who was willing to share with us his Master thesis work on proving the faithfulness of the type $A$ $2$-braid group when we met at the Mooloolaba Conference $2020$.
     
      I personally want to thank Bryan Bailing Wang, who allowed me to learn symplectic geometry in his reading classes and gave me opportunities to share my research interest with his students.
      He even encouraged his previous PhD student, Shuaige Qiao, to work with me on my ``ill-posed'' question on type $B$ Floer cohomology. 
      When Qiao returned to China after graduation, we endeavoured to communicate over Zoom, hoping to produce some new results despite the VPN problem and unsynchronised internet speed. 
      This, in fact, is a fun collaboration between two different mathematics backgrounds where I learn how to communicate with a person who is not so familiar with representation theory.
      Although by the thesis submission date, Qiao and I haven't quite reached our dream goal due to my inadequate knowledge in geometry, we have made decent progress in understanding the subject of Floer cohomology\footnote{It is a folklore that the study of Floer (co)homology is the study of every mathematical fields.}  of Type $A$ Milnor fibre; some different insights from \cite{hengseng2019curves} have been presented in this thesis.
      Qiao is now taking up a postdoctoral position in China. 
      I wish her all the best in her future career!
      
      Before jumping to the next section, I would like to take this opportunity to thank the other members of the MSI who impacted me with their knowledge of maths.
      Thank you,  Uri Onn, Asilata Bapat, Anand Deopurkar, Vigleik Angeltveit, Amnon Neeman, James Borger, and Joan Licata.
       I also thank mathematics researcher Joshua Jeishing Wen, who is willing to share his research on wreath Macdonald polynomial with me through emails even though we had never met each other in person before.
       Thanks to Ivo De Los Santos Vekemans for organising the homological algebra graduate reading seminar, which forms the basis of my research afterwards.
       Thank you, Bradley Windelborn, for many heart-to-heart talks opposite my MSI desk, though you have left for your new PhD life in New Zealand.
 
    Inevitably, homesickness is a part of the PhD journey when it comes to studying in the land, not where the home is based. 
    Certain nostalgic elements make me feel at home;
    badminton is the first on the list. 
    I am grateful to gather a band of badminton fanatics, ``Pattern More Than Badminton,'' who love badminton as crazy as me and can play badminton with me anytime\footnote{Literally, somebody in the group spontaneously asked me for a singles game while writing this sentence.} at any place\footnote{I recall that we played on the street behind the building of College of Business and Economics when we were only allowed a certain period of outdoor exercises in the COVID-$19$ lockdown period.
   This reminded me of playing badminton with my brother and cousins at the street outside the house in my childhood. }. 
    A big shout-out to Guang Hui Wong, Shao Qi Lim, Anderson Li Yang Kho, Sharon Kai Yun Ngu, Qishang Chua, Shuwen Hua,  Hai Zhu, Song Peng, Bin Guo, Jiazhen Xu, Yifan Zhu, Hohin Mo, Joshua Zhen Ting, Wanqi Yao, Fuyao Liu, Weiwei Fan as well as the coach, Shuen Joe and many others. 
      Second, I am thankful to have weekly bondings as well as sharings on zillions of issues and ideas with fellow Malaysian postgraduate friends. 
      I value my friendship with Danny Ka Po Tee, Bryan Pi-Ern Tee, Joshua Zheyan Soo, Wei Wen Wong, and Anson Wee Ooi.
      Third, I appreciate the love and care I received from the Canberra Chinese Methodist Church.
      Special thanks to my sisters-in-christ, Yan Li and Mei Shy Lui, as well as my brothers-in-christ, Andrew Chu King Yao, Dongrui Gao, and Siang Bing Ting, who always provided me with the transport to church and encouraged me when I was down.
      Thanks to all other church members who always keep me in their prayers during my most stressful time.
    Fourth, I thank the Foochow aunt and uncle who cordially invite me to their house to enjoy their home-cooked food on several occasions.
   As a patron, I value the commitment of Madam Lu Malaysian restaurant offering the best service and food to fulfil my craving for the taste of home.

      Last but not least, a huge thank you to my family, especially to my father and my mother, who fully support me in pursuing my PhD path without a single complaint. 
      They had sacrificed their time without my accompany for ten years ever since I left high school.
      Thank you, Pa and Ma.
      At the same time, I congratulate my brother on submitting his Honours thesis in Germany just before I did.

      Before ending, I would like to acknowledge the Australian government for providing me with the four-year PhD scholarship.
      Thanks to all my thesis proofreaders, including Dominic Lik Ren Liew, Jean-Paul Guang Tian Hii, and Qiyu Zhou.
      Thank you to all my teachers, mentors, and lecturers who shower me with the fountain of knowledge from kindergarten to ivory tower.
     Thank you to every single person who had appeared in my life, be it a short period or an extended period, a good memory or an unpleasant memory; you have shaped me into who I am today. 
     Throughout my PhD journey in Canberra, I encountered the worst bushfire\footnote{\href{https://www.wwf.org.au/what-we-do/bushfires}{The most catastrophic bushfire season ever experienced in the history of Australia.}} and then the outbreak of a global pandemic that is still ongoing now.
      COVID-19 lockdown is a hardship\footnote{In the midst of thesis writing and two years of Chinese New Year not able to have a reunion with family.} for me, but thanks to Ken Behrens\footnote{\href{https://www.abc.net.au/news/2021-08-20/meet-the-creator-of-canberra-covid-19-hero-ken-behrens/100393472}{Canberra's COVID-19 lockdown legend `Ken Behrens'.}}.
      It's a wrap.
     Thanks, God.     
     Amen.

\vspace{3mm}

 \begin{center}
  {\Large \textit{Soli Deo gloria.}}
 \end{center}

\begin{flushright}
\trebleclef *Gloria Patri* \stopbar
\end{flushright}

\begin{flushleft}
Sincerely, \\
Kie Seng, Nge \\
\begin{CJK*}{UTF8}{gbsn}
继胜, 倪
\end{CJK*}
\end{flushleft}

\blanknonumber  
           \chapter*{Abstract}\label{abstract}

\addcontentsline{toc}{chapter}{Abstract}

  We construct a faithful categorical action of the type $B$ braid group on the bounded homotopy category of finitely generated projective modules over a finite dimensional algebra which we call the type $B$ zigzag algebra.
This categorical action is closely related to the action of the type  $B$ braid group on curves on the disc.
Thus, our exposition can be seen as a type $B$ analogue of the work of Khovanov-Seidel \cite{KhoSei}.
Moreover, we relate our topological (respectively categorical) action of the type $B$ Artin braid group to their topological (respectively categorical) action of the type $A$ Artin braid group.

 Then, we prove Rouquier's conjecture \cite[Conjecture 9.8]{Rouq} on the faithfulness of Type $B$ $2$-braid group on Soergel category following the strategy used by \cite{jensen_master} with the diagrammatic tools from \cite{EliGeoSoe}.
 
 In the final part of the thesis, we produce a graded Fock vector in the Laurent ring $\Z[t,t^{-1}]$ for a crossingless matching using Heisenberg algebra.
 We conjecture that the span of such vectors forms a Temperley-Lieb representation, and hence, a new presentation of Jones polynomial can be obtained.\blanknonumber
         \chapter*{Preface}\label{preface}

\begin{quote}
\begin{center}
``Categorification is the modern approach to natural science."
\end{center}
\end{quote}

    In search of the truth of life, humans have been unceasingly employing mathematics to explain the mechanics of the universe \footnote{Interested readers can get a flavour of the modern views on the mathematical universe from \cite{Tegmark,Joel, Roger}. }. 
    Galileo, in the 17th century, famously stated that the universe is written in this grand book,  which stands continually open before our eyes,  communicating in  mathematical language.
    Dating back to the days of ancient Greece, a clear example is written in the dialogue \textit{Timaues},  where Plato made a direct metaphysical one-to-one correspondence between five Platonic solids (tetrahedron, cube, octahedron, dodecahedron, and icosahedron) and five basic elements (earth, air, fire, water and quintessence\footnote{Plato remarked, ``...the God used for arranging the constellations on the whole heaven.''}) to classify matters using geometry.
    Those assignments seem to be wrong in many ways, but could there be a merit to the idea? (What about the symmetry of the solids? The McKay correspondence\footnote{Please see related references \cite{Mckay,SteMc, Naka} on classification by symmetry of Platonic solids.}.)
    Another philosopher, Pythagoras, is perhaps, the earliest human being who believes that everything in the world is made up of numbers  (or in the modern view, the world is digital), according to Aristotle. 
    No one then could  expect that the every-high-school-students-must-know formula --- $a^2 + b^2 = c^2$ --- formed the cornerstone of modern differential geometry when it was first established as Pythagoras' theorem on right-angled triangles\footnote{Pythagoras' theorem was proved twice in \textit{Euclid's Elements} \cite[(Book I, Proposition 47);(Book VI, Proposition 31)]{Euclid}.}.
     It determines what form of geometry a space could possibly exhibit; we now know it is either Euclidean, hyperbolic, or spherical.

       For a long time, we thought that we were creatures living in a flat Euclidean setting.
     It was not until 1915 that Einstein found out that the spacetime where we inhabit in is actually a stitch of various geometries, using his widely celebrated theory of general relativity, and astoundingly, gravity can be described as the physical manifestation of a mathematical object known as the curvature! 
      This is one of the greatest triumphs of modern science in the twentieth century; it replaces Newton's law of universal gravitation that fails to explain minute anomalies in the orbit of Mercury.
        Since then, people have started the journey of searching the grand quest - a theory of everything - in order to unify the four fundamental natural forces, namely the gravity, the strong force, the weak force, and the electromagnetic force.
      The last three mentioned forces are quantum mechanics in nature and are successfully unified under the main attribution to Yang-Mills theory through the representation theory of $SU(3) \times SU(2) \times U(1),$ that is, the Standard Model.
      On the other hand, general relativity deals with objects of gigantic scale, for example, Mercury, Sun, and a black hole.
      Hence, it is unclear how an atomic (negligibly small relative to that of celestial bodies) mass of an electron or photon can even interact with the theory.
      
      Driven by this mission, Crane \cite{Crane}  reinterpreted quantum mechanics in a general relativistic context by probing into the algebraic structure governing the quantum theory of gravity and conjectured that quantum gravity is a 2-category which emerges as a ``categorification'' of the 1-category (modular tensor category)  \cite{CraneCSW} formulating Chern-Simons-Witten theory \cite{WitJones}.
      To make sense of categorification mathematically, Crane\footnote{His earlier ideas \cite{CraYetEx,Tftqg,relspinqg,Crane:2006tu,CraneCausal,Crane:2007md,Crane1997APF} related to (combinatorial/diagrammatic) categorification rooted from approaches to quantum gravity.} worked together with Frenkel\footnote{Frenkel's students wrote about his fundamental work on categorification \cite{WorkIgor} which later blossomed into an active field of research in higher representation theory through him.} to outline a combinatorial method to enhance an algebraic structure with categorically higher morphisms and, what is more, the original algebraic structure can be retained by the process of decategerification, forgetting the newly-introduced morphisms in the categorified algebraic structure \cite{CraneFren}.
      It is in this author's birth year 1994 when the idea of categorification \cite{CraneFren} was officially conveyed in precise mathematical terms and illustrated in the example of categorification of Lusztig's idempotented version of quantum algebra $U_q(\fs \fl_2).$
      At a generic value of the quantum parameter, this was done by Lauda \cite{Lausl2, Lausi22, LauEx, LauIntro}, although the case at the root of unity still remains open today.
      Some recent work has been done along this line; Khovanov (a math descendant of Frenkel) laid the first path \cite{KhoHopf}, suggesting ways to categorify at a prime root of unity, which was then picked up by Qi (a math descendant of Khovanov) et al. \cite{QiThesis,QiHopf,KhoQi,EliQi,EliQisl2,JoshQi}.    
   %  When the author started to write this introduction, he found out that the new notice of American Mathematical Society \cite{LaudaJosh} which was just released in January 2022 presented a similar view on categorification on what will follow in this introduction.  
  %   However, proper adjustment and customized content will be adopted to better suit the author's interest and view on the whole subject pertaining to the result of this thesis.
     For more discussion on the subject of categorification, the readers are welcome to refer to \cite{LaudaJosh, AlgCat, savage2015introduction, savage2019string, Lu2014AlgebraicSO, Baez0, Baez2, Baez1, BaezLau}.

    To date, categorification has proven fruitful in various research fields such as representation theory, algebraic topology, link homology, mathematical physics, et cetera.
    For instance, it provides a framework for Kontsevich's Homological Mirror Symmetry \cite{Kont}, which was announced at the   International Congress of Mathematicians, Zurich, in 1994.
    Motivated by the Mirror Symmetry in string theory, Kontsevich conjectured that the derived Fukaya category of a Calabi-Yau manifold (deemed a symplectic manifold) is equivalent to the derived category of coherent sheaves on its dual manifold (deemed a complex algebraic variety).
    A gentle introduction to Homological Mirror Symmetry and Mirror symmetry can be found in \cite{Raf} and \cite{Mir,Miral,MirrorPhy}, respectively.
    Around the same time, as noted in \cite{TriArn}, Arnold \cite{SomeArn} suggested the realisation of symplectomorphism by the monodromy of Milnor fibres. 
    This was then picked up by Kontsevich and Donaldson, where they categorified it to the monodromy action on the Fukaya category of Milnor fibres. 
    
    Consequently, powered by Homological Mirror symmetry, Khovanov, together with Seidel and Thomas (both math descendants of Donaldson), initiated the study of categorical braid group action on the derived category of Fukaya category of type $A$ Milnor fibres.
    This was done in relation to the categorical braid group action on the bounded derived category of coherent sheaves on the minimal resolution of higher dimensional type $A$ Kleinian singularity \cite{SeiTho} and the categorical braid group action on the bounded derived category $D^b(\Aa$-$g_r$mod$)$ of graded modules over the type $A$ zigzag algebra $\Aa$ \cite{KhoSei}.
    Here, we refer to the complex $d$-dimensional type $A_{n}$ singularities \cite{SlowFour,SlowSim} as the variety defined by the zero set of the functions 
    \begin{equation}
     g(x_0, x_1, \hdots, x_d) = x_0^2 + x_1^2 + \cdots + x_{d-1}^2 + x_d^{n+1} 
\end{equation}    
while type $A_n$ Milnor fibre, by a theorem of Milnor \cite{Milnor}, is defined to be the intersection of closed ball $\overline{B}^{2d+2}(\epsilon)$ of sufficiently small radius $\epsilon > 0$ around the origin in $\C^{d+1}$  and the symplectic manifold defined by the zero set of the deformation of the type $A_n$ singularities
\begin{equation}
     g(x_0, x_1, \hdots, x_d) = x_0^2 + x_1^2 + \cdots + x_{d-1}^2 + h_w(x_d) 
\end{equation}    
where $h_w(z) = z^{n+1} + w_n z^n + \cdots + w_1z + w_0 \in \C[z]$ for $w = (w_0, w_1, \hdots, w_n)$ in the interior $B^{2d+2}(\delta)$ of $\overline{B}^{2d+2}(\delta)$  for sufficiently small $\delta > 0$ such that the intersection of the symplectic manifold and the boundary $S^{2d+1}(\epsilon)$ of $\overline{B}^{2d+2}(\epsilon)$ remains transverse.
     Moreover, the zero set of the function $h_w$ is referred to as the bifurcation diagram by Arnold \cite{ArnADE,Arn11,Arn1}.
      Since $\Aa$  is of finite homological dimension \cite[Proposition 2.1]{KhoSei}, the bounded derived category $D^b(\Aa$-$g_r$mod$)$ of graded modules over $\Aa$ is isomorphic to the homotopy category of the  graded projective modules over $\Aa.$

    This project leads to \cite{KhoSei} the faithfulness of braid group representations in low dimensional topology due to categorification.
    Seidel \cite{SeiSym}  first took the initiative to define a type $A_n$ chain of Lagrangian spheres from a set of generating geodesics in bifurcation diagrams associated with the type $A_n$ singularities. 
    Moreover, the symplectomorphisms defined by such Lagrangian spheres satisfy the braid relation in Artin's $(n+1)$-strand braid group, or the type $A_n$ braid group.
    He has also shown that those bifurcation diagrams are topologically punctured discs $\D^A$ \cite[Lemma 6.10]{KhoSei}, and subsequently, the Milnor fibres are Lefschetz fibration over the punctured discs $\D^A$.
    Recall that the Dehn twists, which generate the mapping class group of punctured discs \cite{PrimerMCG}, also form  the type $A_n$ braid group.
   This way, Seidel and Khovanov employed certain bimodule functors in the bounded homotopy category  $\Kom^b(\Aa_{n}$-$\text{p$_{r}$g$_{r}$mod})$ of graded projective modules over the Fukaya algebra of type $A_n$ chain of Lagrangian spheres, which is the geometric interpretation of type $A_n$ zigzag algebra.
   Following this, they formed a correspondence between geodesics connecting the punctures and minimal complexes in  $\Kom^b(\Aa_{n}$-$\text{p$_{r}$g$_{r}$mod})$.
    In particular, they proved that the topological braid group action intertwines with the categorical bimodule braid group action, so that the dimension of graded module homomorphism space agreed with the graded curve intersection number (inspired by graded Lagrangian manifold \cite{SeiGrad}).  
    The faithfulness of the graded topological braid group action on a certain covering space of the punctured disc and following that, the faithfulness of graded symplectic braid group action on the Fukaya category of the Milnor fibre follows from the unique bigrading in  $\Kom^b(\Aa_{n}$-$\text{p$_{r}$g$_{r}$mod}).$

\begin{center}
\begin{table} [H]

  \begin{tabular}{ c | c | c }
    \hline
    Symplectic geometry & Low-dimensonal topology & Homological Algebra \\ \hline
    bifurcation diagram & $(n+1)$-punctured disc & - \\
    - &  curve isotopy & complex homotopy \\ 
    Fukaya algebra & - & zigzag algebra \\
    Lagrangian sphere & geodesics & projective modules($p_r$mod) \\
     symplectomorphism  & Dehn twist & bimodule functor \\
     Lagrangian intersection & curve intersection &  $p_r$mod morphism  \\
    vanishing cycle intersection form  & curves intersection numbers & Cartan matrix \\
     Floer cohomology dimension  &  bigraded intersection number &  $p_r$mod morphism  dimension\\
     graded Lagrangian Grassmannian & graded projective space & graded projective modules \\
 
  \end{tabular}
      \caption{List of Comparison of Terminology in type $A$ setting.} \label{Terminology}
\end{table} 
\end{center}    
    
    See \cref{Terminology} for a comparison of the terminology appearing in \cite{KhoSei}.
      A few categorification results follow: the categorical braid group representation on $\Kom^b(\Aa_{n}$-$\text{p$_{r}$g$_{r}$mod})$ categorifies the classical reduced Burau representation \cite{BraidGroups, LinkHom}, and the dimension of graded endomorphism algebra of  projective modules categorifies the intersection form of vanishing cycles in type $A$ singularities, which is also the Cartan matrix of type $A$ Dynkin diagram.
      It is known that Burau representation of the type $A$ braid group is equivalent to the homological representation\footnote{As an aside, homological representation refers to the study of the representation on the first fundamental group of the punctured disc. On the contrary, homological algebra refers to the study of homology on chain complexes of modules. } of the mapping class group of the punctured disc \cite{BraidGroups} and is not faithful for $n \geq 5$ \cite{Moody,LongPaton,Bigelow}.
      By lifting the punctured disc to a bigraded covering space inspired by graded Lagrangian Grassmannian, the graded mapping class group of punctured disc acts faithfully. 
      Similarly, the type $A$ braid group injects into the graded symplectic mapping class group of Milnor fibre as the graded dimension of Floer cohomology is related to the dimension of graded module homomorphisms, which is also linked to the graded curve intersection numbers.

    In \cref{Chap1} of this thesis, we generalise the results of the type $A$ braid group in \cite{KhoSei} to that of the non-simply laced type $B$ braid group. 
    We focus on the result in \cite{hengseng2019curves} where we utilise the strength of degree $2$ field extension to build a new zigzag algebra, as mentioned in the title, the type $B$ zigzag algebra to categorify the type $B$ non-symmetric integral Cartan matrix (see \cref{catCar}) and the intersection form of vanishing (hemi)cycles in type $B$ singularities under the interpretation of dimension of graded modules over different fields (see \cref{catint}).
      Inspired by the folding of Dynkin diagrams\footnote{With a hint from Peter McNamara, University of Queensland.}, at Kiola Conference 2019,   Heng\footnote{Heng is being acknowledged in the  Acknowledgements  section.} and the author defined a type $B$  zigzag algebra $\Ba_n$ which,  after a degree two scalar extension,  is isomorphic to $\Aa_n \cong   \C  \otimes_{\R}  \Ba_n$ (see \cref{isomorphic algebras}).
      Subsequently,  we defined  the bounded homotopy category $\Kom^b(\Ba_n $-$p_r g_r mod)$ and the type $B$ braid bimodule functors acting on the category.   
      The type $B$ braid bimodule functors indeed defined the  type $B$ generalised braid group action on  $\Kom^b(\Ba_n$-$p_r g_r mod) $ (see \cref{standard braid relations,type B relation}). 
      As in \cite{KhoSei},  the bimodules we used  (see \cref{TBTemp}) provided a functor realisation to the type $B$ Temperley-Lieb algebra \cite{Green}.
      
      Echoing the philosophy in the thesis title, the author is inclined to extend the story to the topology analogous to what Khovanov and Seidel did. 
     In the 1980s,  Slodowy  \cite{SlowSim,  SlowFour} introduced  simple singularities associated to inhomogeneous Dynkin diagrams using the Kleinian singularities together with certain automorphisms on them.   
      In particular,  a type $B$ simple singularity is defined as the type A singularity associated with a $\Z/ 2 \Z$-automorphism on it.
      While studying critical points of differential functions on manifolds with boundary, Arnold also gave a definition for type $B$ simple singularity on surfaces \cite{ArnB,Arn2,Arn11}. 
      On the complex surface $x^2 + y^{2n} = 0$,  it is the involution $(x,y) \mapsto (x,-y).$
      After a generic deformation of the complex surface,  we were bound to consider the quotient space of the $2n$-punctured disc $\DA$ by a $\pi$-rotation. 
      Unsurprisingly,  after some trials and errors, we worked out a similar map $L_B$ (ref.  the left map in \cref{equi}), which to curves in the quotient disc $\DB$ associates complexes in $\Kom^b(\Ba_n $-$p_r g_r mod)$. 
      In addition, the type $B$ braid group $\mathcal{A}(B_n)$ appears in at least two ways as a subgroup of the mapping class groups of punctured disc; one is a classical injection $\mathcal{A}(B_n) \rightarrowtail \mathcal{A}(A_{n})$  into type $A$ braid group $\mathcal{A}(A_{n})$ and the other one $\cA(B_n) \hookrightarrow \mathcal{A}(A_{2n-1})$ with $\mathcal{A}(A_{2n-1})$ is given by Crisp \cite{crisp_1997} from the Dynkin diagram viewpoint or Birman and Hilden \cite{BirHil} from the topological viewpoint.
     The latter helps in connecting the categorical algebraic construction and the topological setting between type $A$ and type $B$, providing us with the following big intertwining picture:
     \begin{theorem} [= \cref{diagram commutes}] \label{equi}
The following diagram is commutative, and the four maps on the square are $\cA(B_n)$-equivariant:
%\
\begin{figure}[H]
\centering
\begin{tikzpicture} [scale=0.9]
\node (tbB) at (-3,1.5) 
	{$\mathcal{A}(B_n)$};
\node (cbB) at (-3,-3.5) 
	{$\mathcal{A}(B_n)$};
\node (tbA) at (10,1.5) 
	{$\cA(B_n) \hookrightarrow \mathcal{A}(A_{2n-1})$}; 
\node (cbA) at (10.5,-3.5) 
	{$\cA(B_n) \hookrightarrow \mathcal{A}(A_{2n-1})$};

\node[align=center] (cB) at (0,0) 
	{Isotopy classes of graded \\ admissible curves $\check{\fC}^{adm}$  in $\DB$};
\node[align=center] (cA) at (7,0) 
	{Isotopy classes of graded \\ admissible multicurves $\ddot{\undertilde{\wt{\fC}}}^{adm}$  in $\DA$};
\node (KB) at (0,-2)
	{$\Kom^b(\Ba_n$-$\text{p$_{r}$g$_{r}$mod})$};
\node (KA) at (7,-2) 
	{$\Kom^b(\Aa_{2n-1}$-$\text{p$_{r}$g$_{r}$mod})$};

\coordinate (tbB') at ($(tbB.east) + (0,-1)$);
\coordinate (cbB') at ($(cbB.east) + (0,1)$);
\coordinate (tbA') at ($(tbA.west) + (1,-1)$);
\coordinate (cbA') at ($(cbA.west) + (0,1)$);

\draw [->,shorten >=-1.5pt, dashed] (tbB') arc (245:-70:2.5ex);
\draw [->,shorten >=-1.5pt, dashed] (cbB') arc (-245:70:2.5ex);
\draw [->, shorten >=-1.5pt, dashed] (tbA') arc (-65:250:2.5ex);
\draw [->,shorten >=-1.5pt, dashed] (cbA') arc (65:-250:2.5ex);

\draw[->] (cB) -- (KB) node[midway, left]{$L_B$};
\draw[->] (cB) -- (cA) node[midway,above]{$\mathfrak{m}$}; 
\draw[->] (cA) -- (KA) node[midway,right]{$L_A$};
\draw[->] (KB) -- (KA) node[midway,above]{$\Aa_{2n-1} \otimes_{\Ba_n} -$};
\end{tikzpicture}
\end{figure}
\end{theorem}
\noindent      It follows from the commutative diagram and the faithful type $A$ categorical action that we indeed construct a faithful type $B$ categorical action. 
   
      Actually, we could have done the same as  Khovanov and Seidel where the faithfulness result relies on the intersection numbers, and we did!
      The author introduced the trigraded covering space for $\DB$ (in \cref{tribundle}) and graded intersection number for curves in $\DB$ motivated by the homomorphism space of corresponding module complexes (in \cref{triint}).     
      Not included in the earlier version of \cite{hengseng2019curves}, the author just established that  the type $B$ zigzag algebra $\Ba_n$, in fact, sit as an invariant algebra inside the type $A$ zigzag algebra $\Aa_{2n-1}$ under the $\Z / 2 \Z$-diagonal action of the product group of  $\Z / 2 \Z$-complex conjugation and the $\Z / 2 \Z$-automorphism of the type $A$ Dynkin diagram.
        
\begin{theorem} [=\cref{invariant algebra}]
The type $B$ zigzag algebra $\Ba_n$ is the $\R$-algebra of the invariant algebra  
$$   \Ba_n  \cong_{\R}  \Aa^{inv}_{2n-1}$$ under the $\Z / 2\Z$-diagonal action of the product group of $\Z / 2\Z \times \Z / 2\Z$ where the first group factor is coming from the symplectic consideration and the second group factor is coming from the algebraic consideration.
\end{theorem}

  \noindent     This serves as another algebraic interpretation of definition of the type $B$ zigzag algebra which seems rather ad hoc at first sight. 
       This result came out from the author's effort in finding the geometric meaning of  the type $B$ zigzag algebra in collaboration with Qiao\footnote{Qiao is being acknowledged in the Acknowledgements section.}.

       Practising the chant of categorification,  the bimodule category $(\Aa_{n},\Aa_{n})$-bimod and similarly, $(\Ba_n, \Ba_n)$-bimod can be interpreted as a $2$-category where the object is the single index $n$, the morphisms are bimodules, and the $2$-morphisms are the bimodule homomorphisms.
       Since the Temperley-Lieb algebra is a quotient algebra of Hecke algebra and the bimodule functors in $\Kom^b(\Aa_{n}$-$\text{p$_{r}$g$_{r}$mod})$ and $\Kom^b(\Ba_n $- $p_r g_r mod)$ categorify a (quotient of) Temperley-Lieb algebra of an appropriate type, it is natural to look into their relations with the category of Soergel bimodules.
        This is because Soergel \cite{SoeKaz} first introduced the category of Soergel bimodules through the combinatorics of Harish-Chandra modules \cite{SoeCom} to categorify the Hecke algebra.
       This fact here was first conjectured by Soergel (he verified it for finite Weyl group over field of certain characteristic) and fully verified through a collective work  \cite{EliKho,EliDiGr,EliTwo,EliThick,EliGeoSoe} mainly from Elias (a math descendant of Khovanov) and Williamson (a math descendent of Soergel) using diagrammatic category description which finally led them to positivity conjecture of the Kazhdan-Lusztig polynomial\footnote{Kazhdan-Lusztig polynomial is the coefficient of Kazhdan-Lusztig basis element expressed in terms of the basis elements indexed by Coxeter group generators \cite{KazLus}. } \cite{EliWilHodge}.
        A nice introduction to Soergel bimodules can be found in the textbook 
       \cite{EliasSoe} and Libedinsky's note
      \cite{LibSoe}. 
      
      In \cref{Chap2} of this thesis, we upgrade the faithfulness result in \cref{Chap1} to a faithful $2$-categorical braid group action on the Soergel category.
        Building on his work \cite{ChuangRou} with Chuang, Rouquier constructed Rouquier complexes \cite{Rouq} which he called  the $2$-braid group.
        He conjectured \cite[Conjecture 9.8]{Rouq} that the $2$-braid group categorifies the generalised braid group where he outlined a proof for the type $A$ case in \cite[Remark 9.9]{Rouq}.
        This proof for the type $A$ braid group was carried out by  Jensen (a math descendant of Williamson) in his Master thesis \cite{jensen_master} and subsequently extended in   \cite{jensen_2016} using combinatorial methods.
        Since we have type $B$ zigzag algebra at our disposal, the content of \cref{Chap2} illustrates the use of the same technique in \cite{jensen_master} to construct a monoidal functor from the diagrammatic Soergel category \cite{EliGeoSoe} to our bimodule category $(\Ba_n, \Ba_n)$-bimod and hence, the faithfulness result is carried across to the $2$-category, thereby proving the Rouquier conjecture for type $B$:
        
        \begin{theorem} [=\cref{Faith2}]
        The type $B$ $2$-braid group decategorifies to the type $B$ braid group.
        \end{theorem}
        
        In \cref{Chap3} of this thesis, we finally consider the wreath product zigzag algebra analogue of \cite{KhoSei}, which is inspired by another $2$-category, the Heisenberg category.
      In 2012, Licata (a math descendant of Frenkel, my supervisor) and Cautis (a math descendant of Harris) defined a $2$-category called Heisenberg category  \cite{CauLiHeiHil}, which categorifies Heisenberg algebra and, together with Sussan (a math descendant of Frenkel), constructed a braid group action on that category  \cite{CaLiSuBraid}. 
        The braid group action there was regarded as a wreath analogue of the bimodule functor considered in \cite{KhoSei}.
        Please refer \cite{KhoHei,AliLi,LiSaHec,LiAliErr} for related definitions of Heisenberg category, \cite{ShiftSym,Elliptic,Walgebra} for algebras arising from Heisenberg category and paper series \cite{AliHei,RoSaWPA,BruSaWeb,BruSaWebFound,SaFrob,BruSaWebQuan,BruSaWebQFH,MacSa,QueSaYa,RoSaTTH,RoSaTH,AliOded,LiAliDe}
      by Savage (a math descendant of Frenkel) et al. on recent developments of Heisenberg category in more general settings.        
   
       However, it is the conjectural braid group action emerging from the categorified vertex operator  \cite{cautis2014vertex} with a similar line of argument generating braid group action from exponential of quantum algebra  Chevalley generators by Lusztig \cite{Lusztig} that is related to the link homology.
        We are currently working on a thick/idempotented calculus for the type $A$ quantum lattice Heisenberg category in order to fully develop the homological algebra part of the category \cite{Penrose, Cvi, Gonz, HerYouOp, AlComp, AlcSimp,AlcTrans}. 
        Even on the chain group level, it heavily uses combinatorial gadgets like symmetric functions \cite{YoungRuth, Mac, MacSymOrth, Stan2, EggeSym, Processi}, tableaux identities \cite{Fulton, Sagan,FultonHarris, GoWa, Tung}, $q$-calculus \cite{Stan1, GeorgeInt, HardyWri}, basic hypergeometric series \cite{Ernst}, and many more that the author is still searching...

        The final question we will answer in this chapter is whether a chain complex can be associated to a multicurve analogue of \cite{KhoSei} via Heisenberg categorification. 
        In the last chapter, we give a partial solution to the problem posed by my supervisor, Licata.
        Instead of a chain complex, we associate a set of chain groups to a crossingless matching without differential (in \cref{Fockinv}).
        In other words, we construct a graded Fock vector living in Fock space representation of Heisenberg algebra over the Laurent ring $\Z[t,t^{-1}]$ given a crossingless matching.
        This is done by using combinatorial tools like Gaussian polynomial  \cite{GeorgeInt}   and partition statistics from Kostka polynomial \cite{Haglund}.

      We conjecture that the span of the graded Fock vector coming from crossingless matchings forms a Temperley-Lieb representation  (see \cref{graTemp}). 
      Since every link can be represented by a braid plat closure \cite{BirBraid}, we may potentially obtain a new presentation of the Jones polynomial  \cite{Jones,JonesSub} from the above Temperley-Lieb representation lying in the Fock space representation of Heisenberg algebra (see \cref{preJones}). 
        This is done through Kauffman's diagrammatic braid group representation into Temperley-Lieb algebra \cite{Kauff,KauffKnot} inspired by state models \cite{KauffFor}.
        
        \begin{conjecture} [=\cref{graTemp,preJones}]
        The span of the graded Fock vector coming from crossingless matchings forms a Temperley-Lieb representation and hence, produces a new presentation of the Jones polynomial in Fock space representation.
        \end{conjecture}
        
   \noindent     The entire construction above admits a $2$-category description where the endomorphism of certain modules can be related to Khovanov arc algebra \cite{Khovfunc}, which in fact, is the original motivation of the project.

\blanknonumber
 
  \tableofcontents\blanknonumber

%  \include{notation}\blanknonumber

  % set page numbers to arabic, reset to 1
  \mainmatter

  % assuming there are files chapter1.tex etc...
          \chapter{The Type B Zigzag Algebra.} \label{Chap1}

\section{Background and Outline of the Chapter}

 Eager readers can skip this section for the subsequent result sections. 
 This section consists of part of the introduction in \cite{hengseng2019curves}.
  The background here provides a detailed comparison between the existing literature for the recent development of the subject after the publication of \cite{KhoSei} and a more concise summary of the chapter, which was not covered in the preface. 
    We also inform some upcoming works inspired by the work of this chapter.
  
In the seminal work \cite{KhoSei}, Khovanov-Seidel introduced a  categorical action of the type $A_m$ Artin braid group $\cA(A_m)$, where the group acts faithfully by exact autoequivalences on  $\Kom^b(\Aa_m$-$\text{p$_{r}$g$_{r}$mod})$, the homotopy category of projective (graded) modules over the type $A_m$ zigzag algebra $\Aa_m$.
Moreover, they showed that this braid group action on $\Kom^b(\Aa_m$-$\text{p$_{r}$g$_{r}$mod})$ is fundamentally related to the mapping class group action on curves on the punctured disc. 
More precisely, they constructed a map $L_A$ that associates  complexes of projective $\Aa_m$-modules to isotopy classes of curves in the disc, and showed that $L_A$ intertwines the categorical action of $\cA(A_m)$ on complexes with the mapping class action of $\cA(A_m)$ on curves.
Furthermore, the geometric intersection number between two curves $c_1$ and $c_2$ can be computed from the dimension of their corresponding total Hom space $\HOM^*(L_A(c_1), L_A(c_2))$.
This may be seen as a bridge connecting two fundamental appearances of the same group: the former as the Artin group associated to the type $A$ Coxeter group, and the latter as the mapping class group of the punctured disc.

Another family of Artin groups which also appear as mapping class groups, are the type $B$ Artin groups  $\cA(B_n)$.
To this end, Gadbled-Thiel-Wagner developed a type $B$ analogue of the Khovanov-Seidel construction in \cite{GTW}.
However, their story is more accurately stated using extended affine type $A$  Artin groups $\widehat{\cA}(\hat{A}_{n-1})$.
Although the groups $\cA(B_n)$ and $\widehat{\cA}(\hat{A}_{n-1})$ are indeed isomorphic, they have rather different origins; both algebraically and topologically.
When interpreted as Artin groups, they are given by different Artin presentations corresponding to different Coxeter diagrams.
Although they are both mapping class groups of an $(n+1)$-punctured disc (fixing one of the punctures), their corresponding natural affine configurations of the disc are different (see \cref{fig: diff affine configuration}).

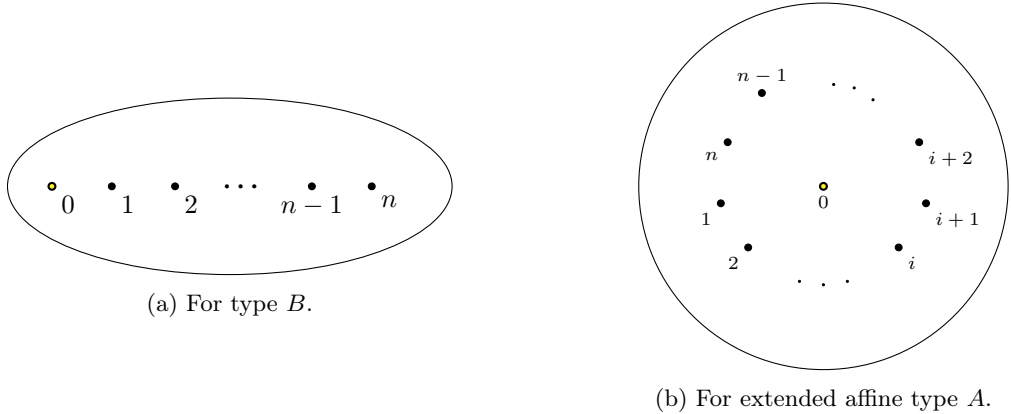
\begin{figure}[h]  
\begin{subfigure}{0.4 \textwidth}
\centering
\begin{tikzpicture}[scale = 0.45]

\draw (5.2,2.5) ellipse (6.5cm and 2.6cm); %%Remove this for orbit space
%\draw[thick] plot[smooth, tension=2.25]coordinates {(0,-.5) (12,2.5) (0,5.5)};
%\draw[thick,green, dashed]  (0,5.5)--(0,3.875);
%\draw[thick, green, dashed]   (0,-.5)--(0,1.125);
%\draw[thick, green, dashed, ->]   (0,2.5)--(0,3.875);
%\draw[thick,  green, dashed, ->]   (0,2.5)--(0,1.125);
\draw[fill] (1.75,2.5) circle [radius=0.1] ;
\draw[fill] (3.6,2.5) circle [radius=0.1]  ;
\draw[fill] (7.6,2.5) circle [radius=0.1]  ;
\draw[fill] (9.35,2.5) circle [radius=0.1]  ;
\filldraw[color=black!, fill=yellow!, thick]  (0,2.5) circle [radius=0.1];

\node  at (5.6,2.5) {$ \boldsymbol{\cdots}$};
\node[below right] at  (0,2.5) {0};
\node[below right] at  (1.75,2.5) {1};
\node[below right] at  (3.6,2.5) {2};
\node[below] at  (7.6,2.5) {$n-1$};
\node[below right] at  (9.35,2.5) {$n$};

\end{tikzpicture}
\caption{For type $B$.} 
\end{subfigure}
\quad
\begin{subfigure}{0.49 \textwidth}
\centering
\begin{tikzpicture} [scale = 0.45]

\draw (0,0) circle (5.4cm);

\filldraw[color=black!, fill=yellow!, thick] (0,0) circle [radius=0.1] ;

%Draw two by two 

\draw[fill] (3,-.5) circle [radius=0.1] ;
\draw[fill] (-3,-.5) circle [radius=0.1] ;

\draw[fill] (2.2,-1.8) circle [radius=0.1] ;
\draw[fill] (-2.2,-1.8) circle [radius=0.1] ;

\draw[fill] (2.8,1.3) circle [radius=0.1] ;
\draw[fill] (-2.8,1.3) circle [radius=0.1] ;

\draw[fill] (-1.8,2.75) circle [radius=0.1] ;

\draw[fill] (0,-2.9) circle [radius=0.03] ;
\draw[fill] (-0.7,-2.8) circle [radius=0.03] ;
\draw[fill] (0.7,-2.8) circle [radius=0.03] ;

\draw[fill] (0.3,3) circle [radius=0.03] ;
\draw[fill] (0.9,2.9) circle [radius=0.03] ;
\draw[fill] (1.45,2.55) circle [radius=0.03] ;

\node[below] at (0,0) {\scriptsize{$0$}} ;
\node[below right] at (2.2,-1.8) {\scriptsize{$ i$}} ;
\node[below left] at (-2.2,-1.8) {\scriptsize{$2$}} ;
\node[below right] at (3,-.5) {\scriptsize{$ i + 1$}} ;
\node[below left] at (-3,-.5) {\scriptsize{$1$}} ;

\node[below right] at (2.8,1.3) {\scriptsize{$ i + 2$}} ;
\node[below left] at (-2.8,1.3) {\scriptsize{$n$}} ;

\node[above] at (-1.8,2.75) {\scriptsize{$n-1$}} ;

\end{tikzpicture}
\caption{For extended affine type $A$.}
\end{subfigure}
\caption{Two different affine configurations of the $(n+1)$-punctured disc.}
\label{fig: diff affine configuration}
\end{figure}

The goal of the present chapter is to develop a proper (non-simply-laced) type $B$ analogue of the constructions given by Khovanov-Seidel and Gadbled-Thiel-Wagner.
To this end, we introduce a (finite dimensional) quotient of a quiver algebra  $\Ba_n$ over $\R$, which we call the type $B_n$ zigzag algebra.
Unlike in type $A$ and extended affine type $A$, the root system of type $B$ is no longer simply-laced.
As a result, the definition of the type $B$ zigzag algebra is somewhat subtle -- the indecomposable projective $\Ba_n$-module whose class in the Grothendieck group is a short simple root will only have the structure of a $\mathbb{R}$-vector space, while all the other indecomposable projective $\Ba_n$-modules whose classes are long simple roots will actually be $\mathbb{C}$-vector spaces.  
Note that this is somewhat reminiscent of other non-simply-laced constructions in the quiver theory, see \cite{DR}.
The relevance of $\Ba_n$ to Coxeter theory is provided by the following theorem:
\begin{theorem}[= \cref{Cat B action} and \cref{faithful action}] \label{faithfulintro}
The homotopy category of projective graded modules $\Kom^b(\Ba_n$-$\text{p$_{r}$g$_{r}$mod})$ carries a faithful (weak) action of the type $B_n$ Artin group $\cA(B_n).$
\end{theorem}
\noindent
Similar to the works of Khovanov-Seidel and Gadbled-Thiel-Wagner,  we establish the following results that relate the categorical notions to low-dimensional topology:
\begin{theorem}[= \cref{L_B equivariant} and \cref{poin poly equals tri int}] \label{L_Bintro} The exists a map $L_B$ that associates complexes in $\Kom^b(\Ba_n$-$\text{p$_{r}$g$_{r}$mod})$ to curves in the $(n+1)$-punctured disc. 
This map $L_B$ is $\cA(B_{n})$-equivariant, intertwining the $\cA(B_n)$-action on curves and the $\cA(B_n)$-action on complexes in  $\Kom^b(\Ba_n$-$\text{p$_{r}$g$_{r}$mod}).$
Moreover, the (trigraded) intersection number between two curves $c_1$ and $c_1$ is given by the Poincar\'e polynomial of the total Hom space between $L_B(c_1)$ and $L_B(c_2)$.
\end{theorem}
% 
%\noindent We also show that the map $L_B$ allows one to compute the intersection number of curves from the dimension of the corresponding total Hom space in $\Kom^b(\Ba_n$-$\text{p$_{r}$g$_{r}$mod})$ (\cref{int num and hom}).
% We also show that the categorical action on $\Kom^b(\Ba_n$-$\text{p$_{r}$g$_{r}$mod})$ categorifies a homological representation of an explicit covering space of the disc (\cref{categorification hom rep}).

One main feature of our work contrasting that of Gadbled-Thiel-Wagner is an explicit realisation of a well-known connection between type $B$ and type $A$.
To explain this, recall that the type $B_n$ Artin group is known to be a (proper) subgroup of the type $A_{n-1}$ Artin group.
Once again, this has two interpretations: algebraically, it embeds through an LCM homomorphism induced from the folding of Coxeter diagrams \cite{crisp_1997};
topologically, the embedding is induced from a double-branched cover\footnote{Such a covering was also considered in \cite{Ito} to construct curve diagrams for $\mathcal{A}(B_n)$.} of the $(n+1)$-punctured disc $\DB$ by a $2n$-punctured disc $\DA$.
The topological interpretation induces an $\mathcal{A}(B_n)$-equivariant map $\mathfrak{m}$ that takes curves in the $(n+1)$-punctured disc to \emph{multicurves} in the $2n$-punctured disc, defined by taking the pre-image of the covering map.
Our work includes a categorical interpretation of this map $\mathfrak{m}$, given by a scalar extension functor:
\begin{theorem}[= \cref{isomorphic algebras} and \cref{fullmaintheorem}]\label{intromaintheorem}
The type $B$ zigzag algebra $\Ba_n$ (over $\R$) algebra is isomorphic to Khovanov-Seidel type $A$ zigzag algebra $\Aa_{2n-1}$ after extending scalars to $\C$, namely 
\[
\C\otimes_\R \Ba_n \cong \Aa_{2n-1} \quad \text{as $\C$-algebras}.
\]
This induces a scalar extension functor $\Aa_{2n-1} \otimes_{\Ba_n} -$, which renders the diagram in \cref{fig: full picture} commutative, with all four maps on the square $\cA(B_n)$-equivariant.
\end{theorem}

\begin{figure}[H]
\centering
\begin{tikzpicture} [scale=0.85]
\node (tbB) at (-3,1.5) 
	{$\mathcal{A}(B_n)$};
\node (cbB) at (-3,-3.5) 
	{$\mathcal{A}(B_n)$};
\node (tbA) at (10,1.5) 
	{$\cA(B_n) \hookrightarrow \mathcal{A}(A_{2n-1})$}; 
\node (cbA) at (10.5,-3.5) 
	{$\cA(B_n) \hookrightarrow \mathcal{A}(A_{2n-1})$};

\node[align=center] (cB) at (0,0) 
	{Isotopy classes of trigraded \\ admissible curves $\check{\fC}^{adm}$  in $\DB$};
\node[align=center] (cA) at (7,0) 
	{Isotopy classes of bigraded \\ admissible multicurves $\ddot{\undertilde{\wt{\fC}}}^{adm}$  in $\DA$};
\node (KB) at (0,-2)
	{$\Kom^b(\Ba_n$-$\text{p$_{r}$g$_{r}$mod})$};
\node (KA) at (7,-2) 
	{$\Kom^b(\Aa_{2n-1}$-$\text{p$_{r}$g$_{r}$mod})$};

\coordinate (tbB') at ($(tbB.east) + (0,-1)$);
\coordinate (cbB') at ($(cbB.east) + (0,1)$);
\coordinate (tbA') at ($(tbA.west) + (1,-1)$);
\coordinate (cbA') at ($(cbA.west) + (0,1)$);

\draw [->,shorten >=-1.5pt, dashed] (tbB') arc (245:-70:2.5ex);
\draw [->,shorten >=-1.5pt, dashed] (cbB') arc (-245:70:2.5ex);
\draw [->, shorten >=-1.5pt, dashed] (tbA') arc (-65:250:2.5ex);
\draw [->,shorten >=-1.5pt, dashed] (cbA') arc (65:-250:2.5ex);

\draw[->] (cB) -- (KB) node[midway, left]{$L_B$};
\draw[->] (cB) -- (cA) node[midway,above]{$\mathfrak{m}$}; 
\draw[->] (cA) -- (KA) node[midway,right]{$L_A$};
\draw[->] (KB) -- (KA) node[midway,above]{$\Aa_{2n-1} \otimes_{\Ba_n} -$};
\end{tikzpicture}
\caption{The commutative diagram relating our type $B$ story (left column) and the type $A$ story by Khovanov-Seidel (right column).}
\label{fig: full picture}
\end{figure}
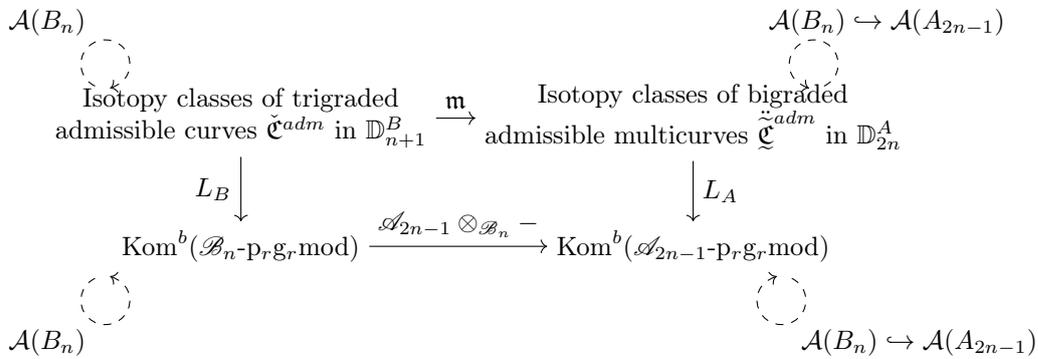

We would like to mention here that the construction of the type $B$ zigzag algebra in this paper can be easily modified to allow for other \emph{Lie-type Dynkin diagrams}, particularly for types $C,F_4,$ and $G_2$ (where for label 6 one uses a field extension of degree 3 instead).
Together with the simply-laced constructions \cite{HueKho, LicQuef}, this (only) covers all of the Lie-type Dynkin diagrams.
To the best of our knowledge, there is no easy generalisation of this construction via (finite-dimensional) algebras that encapsulates all \emph{Coxeter diagrams}; not even for the finite types $H$ and $I_2(k)$.
A different construction via algebra objects in fusion categories that allows for arbitrary Coxeter diagrams will be provided in Heng's PhD thesis.

  Finally, recall that the type $A_n$ zigzag algebra has a geometric origin: it is quasi-isomorphic as a differential graded algebra (dga) with zero differential to the dga associated to an $A_n$-chain of spherical objects \cite{SeiTho}.
   In particular, it is quasi-isomorphic to the Fukaya $A_\infty$-algebra of the Donaldson-Fukaya category corresponding to the Milnor fibre of type $A$ singularities \cite{Seibook}.
   Type $B$ singularities have also been studied; from Arnold's symplectic point of view \cite{Arn1,  Arn2} as boundary singularities and from the algebraic geometry point of view by Slodowy \cite{SlowFour, SlowSim} as simple singularities associated with a $\Z /2 \Z$-group action.
  We expect our type $B$ zigzag algebra to have a similar geometric origin as in the type $A$ case.
This is an ongoing work of the author with Qiao.

\end{comment}

A more precise outline of each section in the chapter is as follows.
\subsection*{Outline of the chapter}
\cref{topology} contains all the topological side of this paper.
We describe the double-branched cover of $\D^B_{n+1}$ by $\D^A_{2n}$, which induces an injection of groups $\Psi: \mathcal{A}(B_n) \hookrightarrow \mathcal{A}(A_{2n-1})$.
Within this section, we also make explicit the definition of curves and admissible curves, and also introduce the notion of trigraded curves -- the type $B$ analogue of the bigraded curves for type $A$.
The construction of the $\cA(B_n)$-equivariant map $\mathfrak{m}$, which lifts trigraded curves to bigraded multicurves can be found in \cref{lift section}.
We also introduce the notion of a trigraded intersection number of trigraded curves -- the type $B$ analogue of bigraded intersection number of bigraded curves, and relate the two through the map $\mathfrak{m}$.

\cref{define zigzag} is where the zigzag algebras are defined.
We start by recalling the type $A_{2n-1}$ zigzag $\C$-algebra $\Aa_{2n-1}$ as defined in \cite{KhoSei} (with a minor change in grading) and describe the $\mathcal{A}(A_{2n-1})$-action on $\Kom^b(\Aa_{2n-1}$-$\text{p$_{r}$g$_{r}$mod})$.
We then define our type $B_n$ zigzag $\R$-algebra $\Ba_n$ and describe the corresponding $\mathcal{A}(B_n)$-action on $\Kom^b(\Ba_n$-$\text{p$_{r}$g$_{r}$mod})$.
We end this section by describing the functor realisation of the type $B_n$ Temperley-Lieb algebra given by certain $(\Ba_n,\Ba_n)$-bimodules.

In \cref{relating categorical b a action}, we develop the algebraic analogue of $\mathfrak{m}$, given by an extension of scalar functor $\Aa_{2n-1}\otimes_{\Ba_n} -$.
In particular, we show that there is an injection $\Ba_n \hookrightarrow \Aa_{2n-1}$ of $\R$-algebras, which induces an extension of scalar functor $\Aa_{2n-1}\otimes_{\Ba_n} - : \Ba_n$-$\text{mod} \ra \Aa_{2n-1}$-$\text{mod}$ and thus on their homotopy categories.
We will then show that the functor $\Aa_{2n-1}\otimes_{\Ba_n} -$ is $\mathcal{A}(B_n)$-equivariant, with the $\mathcal{A}(B_n)$-action on $\Kom^b(\Aa_{2n-1}$-$\text{p$_{r}$g$_{r}$mod})$ induced by the injection $\Psi$ as defined in \cref{topology}.
This in turn, allows us to deduce the faithfulness of the $\cA(B_n)$ categorical action by using the faithfulness of the $\cA(A_{2n-1})$ categorical action proven in \cite{KhoSei}.

\cref{main theorem} is where we prove the remaining pieces of \cref{intromaintheorem}.
We recall the $\cA(A_m)$-equivariant map $L_A$ that associates admissible complexes to curves as in \cite{KhoSei} and construct the type $B$ analogue map $L_B$ in a similar fashion.
To show that $L_B$ is $\cA(B_n)$-equivariant and that the diagram in \cref{intromaintheorem} commutes, we will begin with the latter, which will allow us to derive the former.

\cref{categorification hom rep} contains the decategorified version of the main theorem (see \cref{decat main theorem} for the corresponding diagram).
Namely, just as the $\cA(A_m)$ action on $\Kom^b(\Aa_m$-$\text{p$_{r}$g$_{r}$mod})$ categorifies the Burau representation (which can be described as a representation on the first homology of an explicit covering space of $\DA$), we show that the categorical action of $\cA(B_n)$ on $\Kom^b(\Ba_n$-$\text{p$_{r}$g$_{r}$mod})$ categorifies a representation on (a submodule of) the first homology of an explicit covering space of $\DB$.

\cref{int num and hom} is where we relate the trigraded intersection numbers of (admissible) curves and the Poincar\'e polynomial of the homomorphism spaces of their corresponding complexes.

\cref{InvCat} is where we prove the type $B$ zigzag algebra as an invariant  algebra over $\R$ of the type $A$ zigzag algebra (see \cref{isomorphic algebras}). 
Moreover, we show that the endomorphism space of projective modules over the type $B$ zigzag algebra categorifies the type $B$ intersection form (see \cref{catint}) and the type $B$ Cartan matrix (see \cref{catCar}).

\newpage

\section{Generalised Artin Groups of Type $B_n$ and Type $A_{2n-1}$ and Mapping Class Groups}\label{topology}
	
In this section, we will first describe type $A$ and type $B$ braid groups using generators and relations. 
  After that, we associate these two braid groups to mapping class groups of a surface.
  We then introduce trigraded curves and trigraded intersection number as a trigraded analogue of bigraded curve and bigraded intersection number in \cite{KhoSei} .
  Finally, we construct a lift of the isotopy classes of trigraded curves to the isotopy classes bigraded multicurves, and more importantly, we show that this lift is $\cA(B_n)$-equivariant.

\subsection{Generalised Artin braid groups by generators and relations} \label{genbraid} 
Given a Coxeter graph $\sG$ whose node set is $S$ and whose weighted edges are the unordered pairs $\{s, s'\}$ with their weight $w(s,s') \in \{3,4, \cdots, \infty \}$ (by convention, the edges with $w(s,s') = 3$ is left unlabelled), and $w(s,s')= w(s',s)$, we  define the \emph{generalised braid group} or \emph{Artin-Tits group }$\cA(\sG)$ associated to the graph $\sG$ to be the group generated by  the elements of $S$ and and relations  given by: (i) $\underbrace{ss's\cdots}_{w(s,s')} =\underbrace{s'ss'\cdots}_{w(s,s')}$ where $s,s'$ run over pairs of elements in $S$ with $w(s,s') \neq \infty,$ and (ii)  $ss' = s's$ if there is no edges connecting them.
 By definition, there is no relation between $s$ and $s'$ if $w(s,s') = \infty$.
  See \cite{BraidGroups} and \cite{ComCox} for a more extensive theory on generalised Artin braid groups.

For $n \geq 2,$ the type $A_n$ braid group $\cA({A_n})$ associated to the type $A_n$ Coxeter graph

\begin{figure}[H]
\centering
\begin{tikzpicture}
\draw[thick] (0,0) -- (1,0) ;
\draw[thick] (1,0) -- (2,0) ;
\draw[thick] (2,0) -- (3.2,0) ;
\draw[thick,dashed] (3.2,0) -- (4,0) ;
\draw[thick,dashed] (4,0) -- (4.8,0) ;
\draw[thick] (4.8,0) -- (6,0) ;
\draw[thick] (6,0) -- (7,0) ;
\filldraw[color=black!, fill=white!]  (0,0) circle [radius=0.1];
\filldraw[color=black!, fill=white!]  (1,0) circle [radius=0.1];
\filldraw[color=black!, fill=white!]  (2,0) circle [radius=0.1];
\filldraw[color=black!, fill=white!]  (3,0) circle [radius=0.1];
\filldraw[color=black!, fill=white!]  (5,0) circle [radius=0.1];
\filldraw[color=black!, fill=white!]  (6,0) circle [radius=0.1];
\filldraw[color=black!, fill=white!]  (7,0) circle [radius=0.1];

\node[below] at (0,-0.1) {1};
\node[below] at (1,-0.1) {2};
\node[below] at (2,-0.1) {3};
\node[below] at (3,-0.1) {4};
\node[below] at (5,-0.1) {$n$-2};
\node[below] at (6,-0.1) {$n$-1};
\node[below] at (7,-0.1) {$n$};
\end{tikzpicture}
\end{figure}

\noindent is generated by
$$ \sigma_1^A,\sigma_2^A,\ldots, \sigma_{n}^A $$
\noindent subject to the relations

\begin{align}
      \sigma_j^A \sigma_k^A &= \sigma_k^A \sigma_j^A  & \text{for} \ |j-k|> 1;\\    
      \sigma_j^A \sigma_{j+1}^A \sigma_j^A &=  \sigma_{j+1}^A \sigma_{j}^A \sigma_{j+1}^A & \text{for} \ j= 1, 2, \ldots, n.
\end{align}

\noindent Note that $\cA(A_n)$ is the usual braid group $\cB r_{n+1}$ of $(n+1)$-strands.

For $n \geq 2,$ the type $B_n$ braid group $\cA({B_n})$ associated to the type $B_n$ Coxeter graph 

\begin{figure}[H]
\centering
\begin{tikzpicture}
\draw[thick] (0,0) -- (1,0) ;
\draw[thick] (1,0) -- (2,0) ;
\draw[thick] (2,0) -- (3.2,0) ;
\draw[thick,dashed] (3.2,0) -- (4,0) ;
\draw[thick,dashed] (4,0) -- (4.8,0) ;
\draw[thick] (4.8,0) -- (6,0) ;
\draw[thick] (6,0) -- (7,0) ;
\filldraw[color=black!, fill=white!]  (0,0) circle [radius=0.1];
\filldraw[color=black!, fill=white!]  (1,0) circle [radius=0.1];
\filldraw[color=black!, fill=white!]  (2,0) circle [radius=0.1];
\filldraw[color=black!, fill=white!]  (3,0) circle [radius=0.1];
\filldraw[color=black!, fill=white!]  (5,0) circle [radius=0.1];
\filldraw[color=black!, fill=white!]  (6,0) circle [radius=0.1];
\filldraw[color=black!, fill=white!]  (7,0) circle [radius=0.1];

\node[below] at (0,-0.1) {1};
\node[below] at (1,-0.1) {2};
\node[below] at (2,-0.1) {3};
\node[below] at (3,-0.1) {4};
\node[below] at (5,-0.1) {$n$-2};
\node[below] at (6,-0.1) {$n$-1};
\node[below] at (7,-0.1) {$n$};
\node[above] at (.5,0.1) {4};
\end{tikzpicture}
\end{figure}
\noindent is generated by
$$ \sigma_1^B,\sigma_2^B,\ldots, \sigma_{n}^B $$
\noindent subject to the relations

\begin{align}
\sigma_1^B \sigma_2^B \sigma_1^B \sigma_2^B & = \sigma_2^B \sigma_1^B \sigma_2^B \sigma_1^B;  \\
      \sigma_j^B  \sigma_k^B  &= \sigma_k^B  \sigma_j^B   & \text{for} \ |j-k|> 1;\\    
      \sigma_j^B  \sigma_{j+1}^B  \sigma_j^B  &=  \sigma_{j+1}^B  \sigma_{j}^B  \sigma_{j+1}^B & \text{for} \ j= 2,3, \ldots, n.
\end{align}

\subsection{Discs with marked points and their mapping class groups}  \label{gendefmap}

\subsubsection{Branched covering of $\D^B_{n+1}$ by $\D^A_{2n}$} \label{brcover}
Consider the following closed disk embedded in $\C$ with the set $\Delta$  of $2n$ marked points, that is,
$\D^A_{2n} := \{ z \in \C : \|z \| \leq n+1 \}  \text{ \ and \ }  \Delta = \{-n, \hdots,-1,1, \hdots,n\},$
\noindent as drawn below:
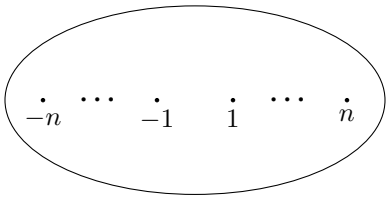
\begin{figure}[H]  
\centering
\begin{tikzpicture} [scale= 0.5]
\draw (0,0) ellipse (5cm and 2.5cm);
\draw[fill] (1,0) circle [radius=0.045];
\draw[fill] (-1,0) circle [radius=0.045];
\draw[fill] (-4,0) circle [radius=0.045];
\draw[fill] (4,0) circle [radius=0.045];

\node  at (2.5,0) {$ \boldsymbol{\cdots}$};9
\node  at (-2.5,0) {$ \boldsymbol{\cdots}$};
\node  [below] at (1,0) {$1$} ;
\node [below] at (-1,0) {$-1$};
\node [below] at (-4,0) {$-n$};
\node [below] at (4,0) {$n$};
\end{tikzpicture}

\caption{{\small The affine configuration of the ${2n}$-punctured $2$-disk $\D^A_{2n}.$}} \label{Disk}
\end{figure}
Let $r:\D_{2n}^A \ra \D_{2n}^A$ be the half rotation of the disc $\D^A_{2n}$ defined by $r(x) = -x$ for $x \in \D^A_{2n}.$ 
Consider the group $\calR$ generated by $r$, $ \calR = \<r \mid r^2=1 \> \cong \Z/2\Z$ and its action on $\D^A_{2n}$.
It is clear that each $x \in \D_{2n} \backslash \{0 \}$ has a neighbourhood $U_x$ such that $r(U_x)\cap r(U_y) = \emptyset$ for all $y \neq x$.
  In this way, the quotient map $q_{br}: \DA \ra \DA / (\Z/2\Z)$ to its orbit space is a normal branched covering with branched point $\{0\}$ \cite{Piek}.
%	Note that the orbit space of of the group action of $\calR$ on  $\D^A_{2n}$ is  $\left(\D^A_{2n} \right)/\left( \Z/ 2\Z \right).$
	From now on, we will denote $\DB$ as the orbit space $\left(\D^A_{2n} \right)/ \left(\Z/2\Z \right)$, and $\Lambda = \{[0], [1], [2], \cdots, [n]\}$ as the set of $n+1$ marked points in $\DB$.
	To simplify notation and to help us picture the orbit space $\DB$, for each equivalence class in $\DB$, we will always pick the element with positive real part as the representative of the equivalence class whenever possible (i.e. as long as the equivalence class does not contain points on the imaginary line).
This way, we can denote the set of marked points $\Lambda$ as $\{0, 1, 2, \cdots, n\}$.
The following figure illustrates how we will be picturing $\DB$ for the case when $n=6$, where we identify the two green lines:
\begin{figure}[H]  
\centering
\begin{tikzpicture} [scale=0.6]

\draw[thick,green, dashed]  (2.525,7.5)--(2.525,5.5);
\draw[thick] plot[smooth, tension=2.25]coordinates {(2.525,3.5) (6.475,5.5) (2.525,7.5)};
\draw[thick, green, dashed]   (2.525,5.5)--(2.525,3.5);
\draw[thick, green, dashed, ->]   (2.525,5.5)--(2.525,4.5);
\draw[thick,  green, dashed, ->]   (2.525,5.5)--(2.525,6.5);
\draw[fill] (3.25,5.5) circle [radius=0.1] ;
\draw[fill] (4.25,5.5) circle [radius=0.1]  ;
\draw[fill] (5.25,5.5) circle [radius=0.1]  ;
\filldraw[color=black!, fill=yellow!, thick]  (2.525,5.5) circle [radius=0.1];

\node[below] at  (2.7,5.5) {0};
\node[below] at  (3.4,5.5) {1};
\node[below] at  (4.4,5.5) {2};
\node[below] at  (5.4,5.5) {3};

\end{tikzpicture}

\caption{{\small The orbit space $\left(\D^A_{6}  \right)/ \left(\Z/2\Z \right)$ with the set of marked points $\{ 0, 1, 2, 3 \}$.}} \label{orbitspace}
\end{figure}

\subsubsection{Artin braid groups as mapping class groups}
Suppose $\cS$ is a compact, connected, oriented surface, possibly with boundary $\partial \cS$, and $\Delta \subset \cS \bs \partial \cS $ a finite set of marked points. 
 We denote such a surface as $(\cS,\Delta)$ and we will just write $\cS$ if it is clear from the context that what $\Delta$ is associated to $\cS$. 
 Let $\Delta^{id} \subset \Delta$ be a subset.
 Denote by Diff$(\cS, \partial \cS ; \Delta^{id} )$ as the group of orientation-preserving diffeomorphisms $f:\cS \ra \cS$ with $f|_{\partial \cS \cup \Delta^{id}} = id$ and $f(\Delta) = \Delta$.
 % that is $f$ fixes the boundary and a subset $\Delta^{id} \subseteq \Delta,$ pointwise as well as leaves the marked points invariant.
%For abuse of notation, we sometimes write the cardinality $|\Delta|$ in place with $\Delta$ .
 If $\Delta^{id} = \emptyset,$ then we write Diff$(\cS, \partial \cS )$ : = Diff$(\cS, \partial \cS; \emptyset)$ for simplicity.
 We then define the mapping class group MCG$(\cS, \Delta^{id})$ of the surface $\cS$ with a set $\Delta$ of marked points fixing elements in $\Delta^{id}$ pointwise by 
 $$ \text{MCG} (\cS, \Delta^{id}) := \pi_0 \big( \text{Diff}(\cS, \partial \cS;\Delta^{id}) \big). $$
%\noindent that is, the group of isotopy classes of elements $\Diff(\cS_n, \partial \cS)$ where the isotopies are smooth and restrict to the identity on the boundary and the set $\Delta^{id}$.
\noindent In a similar fashion, if $\Delta^{id} = \emptyset,$ we denote mapping class group of $\cS$ by $ \MCG(\cS) := \text{MCG} (\cS, \emptyset)$.
	We will just write $\MCG(\cS)$ if those conditions are clear from the context.
	The elements of $\MCG(\cS)$ are called \textit{mapping classes}.
%	Ocassionally, we will view those marked points as punctures and so $\cS$ is a $n$-punctured surface denoted by $\cS_n$. \cite{KhoSei}{[A Primer]}
%
   We will see that both generalised braid groups in \cref{genbraid} appears as mapping class groups, where we refer the reader to \cite{PrimerMCG} for a more detailed exposition on this.

By construction, the marked points on $\DA$ and $\DB$ are subsets of $\Z.$  
  Therefore, we enumerate the marked points on the disc by increasing sequences of points.
  Let $\varrho_j$ (resp. $b_j$) be the horizontal curve connecting the $j$-th marked point and $(j+1)$-th marked point in $\DA$ (resp. $\DB$) for $1 \leq j\leq 2n$ (resp. $1 \leq j\leq n+1$).

 The group $\cA(A_{2n-1})$ is isomorphic to the mapping class group MCG($\DA$) of a closed disk $\DA$ with $2n$ marked points.  
The generator $\sigma^A_j$ corresponds to the half twist $[t^A_{\varrho_j}]$ along the arc $\varrho_j.$
 Here, $t^A_{\varrho_j}$ is a diffeomorphism in $\DA$ rotating a small open disk enclosing the $j$-th and $(j+1)$-th marked points anticlockwise by an angle of $\pi$, permuting the two enclosed marked points, whilst leaving all other marked points fixed as shown in \cref{halftwist}.

\begin{comment}
\begin{proposition}
The group $\cA(A_{2n-1}) $ is isomorphic to $\MCGA$.
\end{proposition}
\begin{proof}
Refer to \cite[Chapter 9.]{PrimerMCG}.
\end{proof}
\end{comment}

Similarly, the group $\cA({B_n})$ is isomorphic to the mapping class group MCG($\D^B_{n+1}, \{0\}$) of a closed disk $\DB$ with $n+1$ marked points fixing the point $\{0\}$ pointwise.  
The generator $\sigma^B_1$ corresponds to the full twist $[(t^B_{b_1})^2]$ along the arc $b_1$ and, 
 for $2 \leq j \leq n,$ each generators $\sigma^B_j$ correspond to the half twist $[t^B_{b_j}]$ along the arc $b_j.$
Here, $t^B_{b_j}$ is a diffeomorphism in $\DB$ rotating a small open disk enclosing the $j$-th and $(j+1)$-th marked points by an angle of $\pi$ anticlockwise as illustrated in \cref{halftwist}.  
 As a result, it interchanges the marked points $j-1$ and $j$ and leaves the other points fixing pointwise.
 	Consequently, $(t^B_{b_1})^2$  is a diffeomorphism rotating a small open disk enclosing the marked points 0 and 1 anticlockwise by an angle of $2 \pi$  leaving all the marked points fixed pointwise as shown in \cref{fulltwist}.

\begin{comment}
\begin{proposition}
The group $\cA(B_n) $ is isomorphic to $\MCG \left( \D^B_{n+1}, \{ 0\} \right)$.
\end{proposition}
\begin{proof}
\rb{Don't know where to reference}
\end{proof}
\end{comment}

\begin{figure}[H]
\centering
\begin{tikzpicture} [scale= 1] 
\draw (-3,0) circle (1cm);
\draw (3,0) circle (1cm);
\draw [|->] (-0.5,0) -- (0.5,0);
\draw[->,thick] (-3.5,0)--(-2.5,0);
\draw[fill] (-3.5,0) circle [radius=0.045];
\draw[fill] (-2.5,0) circle [radius=0.045];
\draw[densely dotted, blue] (-4.5,0) -- (-3.5,0);
\draw[densely dotted,red] (-2.5,0) -- (-1.5,0);

\draw [densely dotted, blue] (2.05,0) arc (180:360:0.725);
\draw[densely dotted, blue] (2.05,0) -- (1.51,0);

\draw [densely dotted, red] (3.95,0) arc (0:180:0.725);
\draw[densely dotted,red] (3.95,0) -- (4.49,0);

\draw [<-|](2.5,0)--(3.5,0);
\draw[fill] (3.5,0) circle [radius=0.045];
\draw[fill] (2.5,0) circle [radius=0.045];

\node [above] at (-3,0) {};
\node [below] at (-3.5,0) { {\scriptsize $j$}};
\node [below] at (-2.5,0) {{\scriptsize $j+1$}};

\end{tikzpicture}
\caption{A half twist $t^A_j$ (similarly $t^B_j$).} \label{halftwist}
\end{figure}
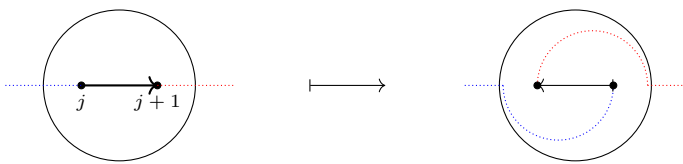

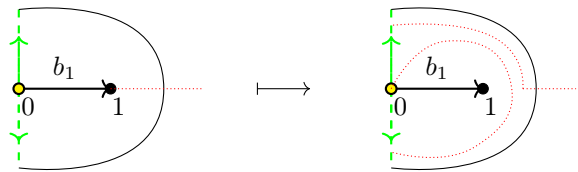
\begin{figure}[H]  
\centering
\begin{tikzpicture} [scale=0.7]
\draw[thick,green, dashed]  (2.525,7)--(2.525,5.5);
\draw plot[smooth, tension=2.25]coordinates {(2.525,4) (5.25,5.5) (2.525,7)};
\draw[thick, green, dashed]   (2.525,5.5)--(2.525,4);
\draw[thick, green, dashed, ->]   (2.525,5.5)--(2.525,4.5);
\draw[thick,  green, dashed, ->]   (2.525,5.5)--(2.525,6.5);
\draw[->,thick] (2.525,5.5)--(4.25,5.5);
\draw[fill] (4.25,5.5) circle [radius=0.1]  ;
\filldraw[color=black!, fill=yellow!, thick]  (2.525,5.5) circle [radius=0.1];
\draw[densely dotted,red] (4.25,5.5) -- (6,5.5);

\node [above] at (3.3875,5.5) {$b_1$};
\node[below] at  (2.7,5.5) {0};
\node[below] at  (4.4,5.5) {1};

\draw [|->] (7,5.5) -- (8,5.5);

\draw[thick,green, dashed]  (9.525,7)--(9.525,5.5);
\draw plot[smooth, tension=2.25]coordinates {(9.525,4) (12.25,5.5) (9.525,7)};
\draw[thick, green, dashed]   (9.525,5.5)--(9.525,4);
\draw[thick, green, dashed, ->]   (9.525,5.5)--(9.525,4.5);
\draw[thick,  green, dashed, ->]   (9.525,5.5)--(9.525,6.5);
\draw[->,thick] (9.525,5.5)--(11.25,5.5);
\draw[fill] (11.25,5.5) circle [radius=0.1]  ;
\filldraw[color=black!, fill=yellow!, thick]  (9.525,5.5) circle [radius=0.1];
\draw[densely dotted,red] (12,5.5) -- (13,5.5);

\draw[densely dotted,red] plot[smooth, tension=1]coordinates {(12,5.5) (11.4, 6.5) (9.525,6.7)};

\draw[densely dotted,red] plot[smooth, tension=1]coordinates {(9.525,5.5) ( 10.55 ,6.4) (11.75,5.7) (11.1,4.4) (9.525,4.3)};

\node [above] at (10.3875,5.5) {$b_1$};
\node[below] at  (9.7,5.5) {0};
\node[below] at  (11.4,5.5) {1};

\end{tikzpicture}
\caption{{\small A full twist $(t^B_{b_1})^2$.}}\label{fulltwist}
\end{figure}

\subsubsection{Injection of $\MCGB$ into $\MCGA$ } \label{InjSec}

A diffeormorphism $f^B$ in $\Diff(\DB,\{0\})$ can be lifted to a unique fiber-preserving diffeomorphism $f^A$ in $\Diff(\DA)$ via the branched covering map $q_{br}.$
Similarly,  an isotopy in $\DB $ can be lifted to an isotopy in $\DA \bsz.$ 
   As such, we have a well-defined map $\Psi$ on the mapping class groups from MCG$\left(\DB, \{0\} \right)\ra$ MCG$ \left(\DA\right)$ defined by lifting the mapping class of $f^B$ to the mapping class of $f^A$. 
   More concretely, using the standard presentation of the groups, $\Psi$ is given by $\sigma^B_1$ mapping to $\sigma^A_n$ and $\sigma^B_{j}$ mapping to $\sigma^A_{n+j-1} \sigma^A_{n-(j-1)}$ for $j \geq 2$.
   In fact, the image of the map $\Psi$ is generated by  fiber-preserving mapping classes in $\MCG_p \left({\DA} \right).$ 
   By \cite[Theorem 1]{BH}, we know that any fibre-preserving diffeomorphism $f^A$ which is isotopic to the identity possesses a fiber-preserving isotopy to the identity, which can then be projected to $\DB$ to get the isotopy $f^B \simeq id$. 
   Therefore, we have the following:

\begin{proposition} \label{injB}
The homomorphism $\Psi: \MCG \left(\D^B_{n+1}, \{0\} \right)\ra$ MCG$\left(\D^A_{2n} \right)$ defined by
$$ \Psi([t^B_{b_i}]) = \begin{cases} 
      [t^A_{b_n}] & \text{for }i = 1; \\
      \left[t^A_{b_{n+i-1}} t^A_{b_{n-(i-1)}} \right] & \text{for } i \geq 2, \\
   \end{cases}
\text{ \quad is injective.}$$
\end{proposition}

\subsection{Curves and geometric intersection numbers} \label{diskcurves}

Let $(\D,\Delta)$ be a surface as in \cref{gendefmap}.
	A \textit{curve} $c$ in $(\D,\Delta)$ is a subset of $\D$ that is either a simple closed curve in the interior $\D^o = \D \bs (\partial \D \cup \Delta)$ of $\D$ and essential (non-nullhomotopic in $\D^o$), or the image of an embedding $\gamma:[0,1] \ra \cS$ which is transverse to the boundary $\partial \D $ of $\D   $ 
with its endpoint lying in $\partial \D \cup \Delta$, that is, $\gamma^{-1}( \partial \D \cup \Delta) = \{0,1\}$.
	In this way, our defined curves are smooth and unoriented.
	A \textit{multicurve} in $(\D,\Delta)$ is the union of a finite collection of disjoint curves in $(\D,\Delta)$.
	We say two curves $c_0$ and $c_1$ are \emph{isotopic} if there exists an isotopy in $\Diff(\D, \partial \D; \Delta)$ deforming one into the other, denoted by $c_0 \simeq c_1.$
	In this way, the points on $\partial \D \cup \Delta$ may not move during an isotopy.
	Therefore, we can partition all curves in $(\D,\Delta)$ into isotopy classes of curves. 
	Two multicurves $\fc_0, \fc_1$ are isotopic if they have the same number of disjoint curves, and each curve in $\fc_0$ is isotopic to one and only one curve in $\fc_1$.   
	Two curves $c_0,c_1$ are said to have \emph{minimal intersection} if given two intersection points $z_- \neq z_+$ in $c_0 \cap c_1,$ the two arcs $\alpha_0 \subset c_0$,  $\alpha_1 \subset c_1$  with endpoints $z_1 \neq z_+$ such that $\alpha_0 \cap \alpha_1 = \{z_-, z_+\}$ do not form an empty bigon (that is, it contains no marked points) unless $z_-,z_+$ are marked points.
	Two multicurves $\fc_0, \fc_1$ are said to have \emph{minimal intersection} if any two curves $c_0 \subseteq \fc_0$ and  $c_1 \subseteq \fc_1$ have minimal intersection.

\begin{comment}	Recall the definition of minimal intersection between curves on $S$ from \cite{KhoSei}; two curves $c_0$ and $c_1$ in   $(S, \Delta)$ have minimal intersection if two conditions are satisfied;
	\begin{enumerate}[(i)]
	\item they intersect transversally satisfying $c_0 \cap c_1 \cap \partial S = \emptyset $, and
	\item consider two points $z_- \neq z_+ \in c_o \cap c_1$ that do not both lie in $\Delta.$ 
	Let $\alpha_0 \subset c_0$ and $\alpha_1 \subset c_1$ be two arcs with $\alpha_0  \cap \alpha_1 = \{z_-,z_+\}.$
	Look at the connected component $\cO$ of $S \bs (c_0 \cup c_1)$ bounded by $\alpha_0  \cup \alpha_1.$
	If $\cO$ is topologically an open disc, then it must contain at least one point in $\Delta.$
	\end{enumerate}
\begin{lemma}
Assume that $\Delta = \emptyset.$ Let 
$c_0,c_1$ be two curves in $(S, \emptyset)$ which intersect transversally and satisfy $c_0 \cap c_1 \cap \partial S = \emptyset.$
They have minimal intersection if and only if the following property holds
\end{lemma}
\end{comment}

Observe that given two arbitrary curves $c_0,c_1$ in $(\D,\Delta),$ we can always find a curve $c'_1 \simeq c_1$ such that $c_0$ and $c'_1$ have minimal intersection.
Given two curves $c_0$ and $c_1$ in $(\D,\Delta),$ we define the \emph{geometric intersection numbers} $I(c_0,c_1) \in \frac{1}{2} \Z$ as follows:
\begin{equation}\label{geoint}
 I(c_0,c_1) =  \left\{
\begin{array}{ll}
      2, & \text{if $c_0,c_1$ are closed and isotopic,} \\
      |(c_0 \cap c_1')\backslash \Delta | + \frac{1}{2}|(c_0 \cap c_1')\cap \Delta |, & \text{if $c_0 \cap c_1'\cap \partial \D = \emptyset $.}  
      % \\ I(c_0^+,c_1) & \text{if $c_0 \cap c_1'\cap \partial \D \neq \emptyset $. }
\end{array} 
\right. 
\end{equation} 

Together with \cite[Lemma 3.2]{KhoSei} and \cite[Lemma 3.3]{KhoSei}, the definition is indeed independent of the choice of $c_1'$.
Moreover, note that  the definition above doesn't depend on the orientation of $\D$ and is symmetric.
    We can extend the definition of geometric intersection numbers for multicurves by just adding up the geometric intersection numbers of each pair of curves $c_0\subseteq \fc_0$ and $c_1 \subseteq \fc_1.$

%\noindent \rb{**}where $c_0^+$ is taken by pushing $c_0$ infinitesimally along the flow of a smooth vector field $\cZ$ that is extended from a nonvanishing postitively-oriented vector field $\fX$ on $\partial \cS$ and vanishes on $\Delta$.
      
     %if we were only considering intersection points in the interior $\cS^0$ of $\cS,$ then 

\subsection{Trigraded curves in $\DB$} \label{tribundle}

   Let us remind the reader that the disk $\DA$ has the set of marked points $\Delta = \{-n, \hdots,-1,1, \hdots,n\}$ and the disk $\DB$ has the set of marked points $\Lambda = \{0,1, \hdots,n\}.$
   We introduce another set of marked points $ \Delta_0 = \Delta \cup \{0\}$ in the same disk $\DA$.
   Fix the notation as follows: $ \RDB := PT \left(\DB \setminus \Lambda \right),$ and $\RDAz := PT \left(\DA \setminus \Delta_0 \right) $ where $PT(\cdot)$ is the real projectivisation of the tangent bundle of the respective disks by taking an oriented trivilisation of its tangent bundle.
   We can then identify $\RDAz \cong \RP \times \left( \DA\setminus \Delta_0 \right).$
    In $\DA,$ pick a small loop $\lambda_i$ winding positively around the puncture $\{i\}$ for every marked points $i $ in $ \Delta.$
  In this way, the classes $[point \times \lambda_i]$ and $[\RP \times point]$ form a basis of $H_1 (\RDAz; \Z).$
	We can then build the covering space $\wRDAz$ of $\RDAz$ classified by the cohomology class $C_0 \in H^1(\RDAz; \Z \times \Z)$ defined as follows:
	\begin{align}
\label{l0} C_0([point \times \lambda_0]) &= (0,0); \\ 
\label{20} C_0([point \times \lambda_i]) &= (-2,1) \ \text{for } i= -n, \cdots, -1, 1, \cdots, n; \\
\label{30} C_0([\R \text{P}^1 \times point]) &= (1,0).
\end{align}

In fact, $\wRDA$ is a covering for $\RDB$ such that with group of deck transformation $\Z \times \Z \times \Z / 2\Z$ as explained in the following lemma.

\begin{lemma} \label{finalcovering} ~ 
\begin{enumerate}
\item  Under the action of the rotation group $\calR$ generated by the half rotation $r,$ the quotient map 
$q : \DAz   \ra   \DB \setminus \Lambda  $
 is a normal covering space with deck transformation group $\calR \cong \Z/2 \Z$.
\item The composite $\wRDAz \xra{\fp} \RDAz \xra{\fq} \RDB $
 is a normal covering where $\fq$ is the normal covering map induced by  the quotient map $q$ on the disk component and identity map on the $\RP$ component.

\item The group of deck transformations for the covering space $\wRDAz \xra{\fp \circ \fq}  \RDB $ is 
$\Z \times \Z \times \Z / 2\Z.$
\end{enumerate}
\end{lemma}
\begin{proof} The proofs of $(1)$ and $(2)$ are straightforward and we leave them to the reader. 

We will now prove $(3).$
 Since the covering space is normal, it is enough to show that 
$$\frac{\pi_1 \left(\RDB \right)} {q_* \left(p_* \left(\pi_1 \left( \RDAz \right)\right)\right)} \cong \Z \times \Z \times \Z / 2\Z.$$
 Observe that we have the following short exact sequences:
\begin{center}
 \begin{tikzcd}[column sep = 4mm,  row sep = 2mm] 
 & & & 1 \arrow[rd] & & 1 \\
 & & 1 \arrow[rd] &  & {\Z \times \Z} \arrow[ru] \arrow[rrdd,bend left, "q_*"]\\
  & &   &\pi_1(\RDAz) \arrow[rd, hook, "q_*" ]  \arrow[ru, "C"] \\
 1 \arrow[rr, hook] & & \pi_1 ( \wRDAz ) \arrow [ru, hook, "p_* "] \arrow[rr, hook, "q_* \circ p_* "]  & & \pi_1(\RDB) \arrow[rr] \arrow[rd, "\lambda_0 \mapsto 1"'] & &  Q
 % \overset{?}{\cong} \Z \times \Z \times \Z / 2\Z  
 \arrow[r] \arrow[dl, bend left] & 1\\
1 \arrow[rru, hook] & & & & & \Z / 2 \Z \arrow[rd] \arrow[ld] \\
& & & &1 & & 1 
 \end{tikzcd}
\end{center}
After that, we see that the group $Q$ is abelian since if you take a commutator in $\pi_1 \left(\RDB \right),$ it lifts to a commutator in $\pi_1 \left(\RDAz \right)$ which evaluates to zero by $C_0$ meaning it is in the image of $q_* \circ p_*.$

	Finally, we see that the rightmost short exact sequence of abelian group 
	$$ 1 \ra \Z \times \Z \ra Q \ra \Z/ 2\Z \ra 1$$
splits giving the isomorphism $Q \cong \Z \times \Z \times \Z / 2  \Z$ as desired.
\end{proof}

   Note that every $f \in$ Diff$ \left(\DB, \{0\} \right)$ preserves the class $C$ and therefore can be lifted to a unique equivariant diffeomorphism $\check{f}$ of $\wRDAz$ that acts trivially on the fibre of $\wRDAz$over all points in $T_z\DB$ for $z \in {\partial \DB}$.
   Furthermore, every curve $c$ in $\DB$ admits a canonical section $s_c : c \backslash \Lambda \ra \wRDAz$ defined by $s_c(z) = T_z c$ which is obtaining the class in every fiber of its tangent line. 
     We define a \textit{trigrading} of $c$ to be a lift $\check{c}$ of $s_c$ to $\wRDAz$  and a \textit{trigraded curve} to be a pair $(c, \check{c})$ consisting a curve and a trigrading; we will often just write $\check{c}$ instead of $(c, \check{c})$ when the context is clear.
    We denote the $\Z \times \Z \times \Z / 2\Z$-action on $\wRDAz$ by $\chi^B.$
    % Besides, $\chi^B$ defines a $\Z \times \Z \times \Z / 2\Z-$action on the set of trigraded curves and the lifts induce a $\Diff \left( \DB, \{ 0\} \right)$-action on the same set commuting with $\chi^B.$
     On top of that, we can easily extend the notion of isotopy to the set of trigraded curves where $\chi^B$ and $ \MCG \left( \DB, \{0\} \right) $ have induced actions on the set of isotopy classes of trigraded curves.
    In particular, for $[f] \in $ $\cA(B_n) \cong$ MCG$\left(\DB, \{0\} \right)$ and a trigraded curve $\check{c}$, $ [\check{f}] (\check{c}) :=  \check{f} \circ \check{c} \circ f^{-1} : f(c) \setminus \Lambda \ra \wRDAz. $

\begin{lemma}\label{freeact} ~
\begin{enumerate}
\item A curve $c$ admits a trigrading if and only if it is not a simple closed curve. 

\item  The $\Z \times \Z \times \Z \backslash 2\Z-$action on the set of isotopy classes of trigraded curves is free. 
	Equivalently, a trigraded curve $\check{c}$ is never isotopic to $\chi(r_1,r_2,r_3) \check{c}$ for any $(r_1,r_2,r_3)\neq 0.$
\end{enumerate}

\end{lemma}

 \begin{proof}
    This is essentially the same proof as in \cite[Lemma 3.12 and 3.13]{KhoSei}.
     \end{proof}

\begin{lemma} \label{action} ~
\begin{enumerate}
\item Let $c$ be a curve in $\DB$ which joins two points of $\Lambda \backslash \{0\},$ $t_c \in \MCG \left(\DB, \{0\} \right)$ the half twist along it, and $\check{t}_c$ its preferred lift to $\wRDAz.$
	Then, $\check{t}_c(\check{c}) = \chi^B(-1,1,0) \check{c}$ for any trigrading $\check{c}$ of $c$.
	\item Let $c$ be a curve in $\DB$ which joins two points of $\Lambda$ with one of them being $\{0\},$ $t_c \in \cG^B$ the full twist along it, and $\check{t}_c$ its preferred lift to $\wRDAz.$
	Then, $\check{t}_c(\check{c}) = \chi^B(-1,1,1) \check{c}$ for any trigrading $\check{c}$ of $c$.
	
\end{enumerate}     
\end{lemma}

\begin{proof}
The proof of $(1)$ is as in \cite[Lemma 3.14]{KhoSei}.
 We will now prove $(2).$ 
 Let $\beta: [0,1] \ra \DB \setminus \Lambda$ be an embedded vertical smooth path from a point $\beta(0) \in \partial\DB$ to the fixed point $\beta(1) \in c$ of $t_c.$
  Note that we have $\check{t}_c(\check{c}) = \chi(r_1, r_2, r_3) \check{c}$ as $t_c(c)= c.$
  Consider the closed path $\kappa:[0,2] \ra \RDAz$ given by 
  $ \kappa(t) = \begin{cases} 
      Dt_c(\R\beta'(t)), & \text{ if } t \leq 1, \\
      \R\beta'(2-t), & \text{ if } t \geq 1, \\
   \end{cases}
$
where $\R\beta'(s) \subset T_{\beta(s) }\DA$. The above situation is illustrated in \cref{ft0}.
    Then,  we compute 
$(r_1,r_2,r_3) = C([\kappa]) = C([\RP \times points]) + C([points \times \lambda_0] + C([points \times \lambda_j]  ) = (-1,1,1).$
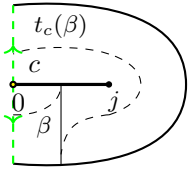
\begin{figure}[H]  
\centering
\begin{tikzpicture} [scale=0.35]

\draw[thick] plot[smooth, tension=2.25]coordinates {(0,-.5) (6.5,2.5) (0,5.5)};
\draw[thick,green, dashed]  (0,5.5)--(0,3.875);
\draw[thick, green, dashed]   (0,-.5)--(0,1.125);
\draw[thick, green, dashed, ->]   (0,2.5)--(0,4);
\draw[thick,  green, dashed, ->]   (0,2.5)--(0,1);
\draw[very thick]   (0,2.5)--(3.6,2.5);
\draw[fill] (3.6,2.5) circle [radius=0.1]  ;
\draw (1.8, -.55) -- (1.8, 2.5);
\draw [dashed] plot[smooth, tension=1.3]coordinates { (3.6,1.35)  (4.8,2.5) (3.6, 3.65)};
\draw [dashed] plot[smooth, tension=1.3]coordinates { (3.6, 3.65) (1.7,3.9) (0, 3.65)};
\draw [dashed] plot[smooth, tension=1.1]coordinates { (0, 1.35) (1.3, 1.65) (1.8, 2.5)};
\draw [dashed] plot[smooth, tension=.85]coordinates { (3.6,1.35)  (2.3,.95) (1.8,-.15)};

\filldraw[color=black!, fill=yellow!, thick]  (0,2.5) circle [radius=0.1];

\node [above] at (0.8, 2.6) {$c$};
\node [left] at (1.8, .9) {{\small $\beta$}};
\node [above] at (1.8,3.9) {{\small $t_c({\beta})$}};
\node [below] at (0.2, 2.5) {$0$};
\node [below] at (3.8, 2.6) {{\small $j$}};
\end{tikzpicture}
\caption{{\small The action of full twist around curve joining $\{0\}$ and another point in $\Lambda.$}} \label{ft0}
\end{figure}
\noindent
\end{proof}

\subsection{Local index and trigraded intersection numbers} \label{triint}
Now, we define the local index (similar to \cite{KhoSei}) of an intersection between two trigraded curves. 
 Suppose $(c_0, \check{c_0})$ and $(c_1, \check{c_1})$ are two trigraded curves, and $z \in \DB \setminus \partial \DB$ is a point where $c_0$ and $c_1$ intersect transversally. Take a small circle $\ell \subset \DB \setminus \Lambda$ around $z$ and an embedded arc $\alpha : [0,1] \ra \ell $ which moves clockwise around $\ell$ such that $\alpha(0) \in c_0$ and $\alpha(1) \in c_1$ and $\alpha(t) \notin c_0 \cup c_1$ for all $t \in (0,1).$ 
 If $z \in \Delta,$ then $\alpha$ is unique up to a change of parametrisation, otherwise, there are two choices which can be told apart by their endpoints.
 	Then, take a smooth path $\kappa:[0,1] \ra \RDB$ with $\kappa(t) \in \left(\RDB \right)_{\alpha(t)}$ for all $t,$ going from $\kappa(0)= T_{\alpha(0)}c_0$ to $\kappa(1)= T_{\alpha(1)}c_1$ such that  $\kappa(t)= T_{\alpha(t)}\ell$ for every $t$. 
 	One can take $\kappa$ as a family of tangent lines along $\alpha$ which are all transverse to $\ell.$
 	After that, lift $\kappa$ to a path $\check{\kappa}:[0,1] \ra \wRDB$ with $\check{\kappa}(0) = \check{c}_0(\alpha(0));$
 	subsequently, there exists some $(\mu_1, \mu_2, \mu_3) \in \Z \times \Z \times \Z / 2\Z$ such that 
\begin{equation} \label{localpath}
\check{c}_1(\alpha(1)) = \chi^B(\mu_1, \mu_2, \mu_3) \check{\kappa}(1) 
\end{equation}
 	 as $\check{c}_1(\alpha(1))$ and $\check{\kappa}(1)$ are the lift of the same point in $\RDB.$ 
 	To this end, we define the local index of $\check{c}_0, \check{c}_1$ at $z$ as 
 	$$ \mu^{trigr}(\check{c}_0, \check{c}_1;z) = (\mu_1, \mu_2, \mu_3) \in \Z \times \Z \times \Z / 2\Z. $$
 	It is easy to see that the definition is independent of all the choices made.
 	
 	The local index has a nice symmetry property similarly in \cite[pg. 21]{KhoSei}, namely:

\begin{lemma} \label{locinsym}
If $(c_0,\check{c}_0)$ and $(c_1,\check{c}_1)$ are two trigraded curves such that $c_0$ and $c_1$ have minimal intersection, then 
\[ \mu^{trigr}(\check{c}_1, \check{c}_0;z) = \left\{
\begin{array}{ll}
     (1,0,0) - \mu^{trigr}(\check{c}_0, \check{c}_1;z), & \text{if } z \notin \Delta; \\
    (0,1,0) - \mu^{trigr}(\check{c}_0, \check{c}_1;z), & \text{if } z \in \Delta \backslash \{0\}; \\
     (1,0,1) - \mu^{trigr}(\check{c}_0, \check{c}_1;z), & \text{if } z \in \{0\}.\\
\end{array} 
\right. \]

\end{lemma}
 
\begin{proof}
The third symmetry is different from \cite{KhoSei} due to the construction of $\DB$ and the definition of $\wRDAz$.
It can be verified using the figure below.
\begin{figure} [H]
\centering
\begin{tikzpicture} [scale=0.65]
\draw[thick, green, dashed] (0,1.1)--(0,2);
\draw[ thick, green, dashed, ->] (0,0)--(0,1.1);
\draw [thick, green, dashed, ->] (0,0)--(0,-1.1);
\draw[thick, green, dashed] (0,-1.1)--(0,-2);

\draw [thick] plot[smooth, tension=.5]coordinates { (0, 0) (1.3, 1.15) (2.5, 1.3)};
\draw [thick] plot[smooth, tension=.5]coordinates { (0, 0) (1.3, -1.15) (2.5, -1.3)};
\filldraw[color=black!, fill=yellow!, thick]  (0,0) circle [radius=0.1];
\draw [thick, blue] (0,-1.3) arc (-90:90:1.3);

\draw[thick, brown] (1,0) -- (1.6,0); 
\draw[thick, brown] (0.7,.55) -- (1.3,1); 
\draw[thick, brown] (0.8,.49) -- (1.4,.8); 
\draw[thick, brown] (0.9,.4) -- (1.5,.6); 
\draw[thick, brown] (.95,.3) -- (1.55,.4); 
\draw[thick, brown] (1,.15) -- (1.6,.2); 

\draw[thick, brown] (0.7,-.55) -- (1.3,-1); 
\draw[thick, brown] (0.8,-.49) -- (1.4,-.8); 
\draw[thick, brown] (0.9,-.4) -- (1.5,-.6); 
\draw[thick, brown] (.95,-.3) -- (1.55,-.4); 
\draw[thick, brown] (1,-.15) -- (1.6,-.2); 

\node[right] at (2.5,1.3) {$c_0$};
\node[right] at (2.5,-1.3) {$c_1$};
\node[right] at (1.6,0) {$\kappa$};
\node[above] at (.3,1.2) {$\ell$};
\node[below]  at (-.15,0) {$0$};
\end{tikzpicture}
\caption{{\small Two curves $c_0, c_1$ intersecting at $\{0\}.$}}
\end{figure}
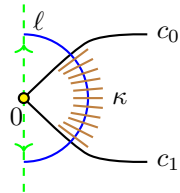  
\end{proof}

Let $\check{c_0}$ and $\check{c_1}$ be two trigraded curves such that they do not intersect at $\partial \DB.$
   Pick a curve $c_1' \simeq c_1$ which intersects minimally with $c_0.$  
   Then, by \cref{freeact}, $c_1'$ has a unique trigrading $\check{c}_1'$ of $c_1'$ so that $\check{c}_1' \simeq \check{c}_1.$ 
   Then, the trigraded intersection number $I^{trigr}(\check{c}_0,\check{c}_1) \in \Z[q_1,q_1^{-1}, q_2, q_2^{-1},q_3]/ \<q_3^2-1\>$  of $\check{c}_0$ and $\check{c}_1$ is defined by the following:
\begin{itemize}
\item if $\check{c}_1 \simeq \chi(r_1, r_2, r_3)\check{c}_0$ with $(r_1,r_2,r_3) \in \Z \times \Z \times \Z \backslash 2\Z $ and $c_0 \cap c_1 \cap \{0 \}$ non-empty, then
\begin{align}
 I^{trigr}(\check{c}_0,\check{c}_1) = q_1^ {r_1} q_2^ {r_2} q_3^ {r_3} (1 + q_2);
\end{align}
\item otherwise
\begin{equation} 
\begin{split}
I^{trigr}(\check{c}_0,\check{c}_1) &=
  (1+q_3)(1+q_1^{-1}q_2) \sum_{z \in (c_0 \cap c_1') \backslash \Delta} q_1^ {\mu_1(z)} q_2^ {\mu_2(z)} q_3^ {\mu_3(z) } \\
& + (1+q_3) \sum_{z \in (c_0 \cap c_1') \cap \Delta \backslash \{0\}} q_1^ {\mu_1(z)} q_2^ {\mu_2(z)} q_3^ {\mu_3(z) } \\ 
& + (1+ q_1^ {-1} q_2 q_3) \sum_{z \in (c_0 \cap c_1') \cap\{0\}} q_1^ {\mu_1(z)} q_2^ {\mu_2(z)} q_3^ {\mu_3(z) } .\\
\end{split}
\end{equation}
\end{itemize}
 The fact that this definition is independent of the choice of $c_1'$ and is an invariant of the isotopy classes of  $(\check{c}_0,\check{c}_1)$ follows similarly as in the case of ordinary geometric intersection numbers.

\begin{lemma}  \label{triintpro}
The trigraded intersection number has the following properties:

\begin{enumerate} [(T1)]
%\item For any $c_0,c_1 $ in $\DB,$ $$I(q_{br}^{-1}(c_0), q_{br}^{-1}(c_1)) = \frac{1}{2} I^{trigr}(\check{c}_0,\check{c}_1) |_{q_1 = q_2= q_3 = 1}. $$ %except for isotopic curves $c_0 \simeq c_1$ has one of their endpoints as $\{0\}$. 
	%In the exceptional case where $\check{c}_0 \simeq\check{c}_1$ with $\{ 0 \}$ as one of their endpoints, we have $I(c_0,c_1) = \frac{1}{2} I^{trigr}(\check{c}_0,\check{c}_1) |_{q_1 = q_2= q_3 = 1} $.
 \item For any $f \in \Diff \left( \DB, \{0 \} \right),$
 $I^{trigr}(\check{f}(\check{c}_0),\check{f}(\check{c}_1)) = I^{trigr}(\check{c}_0,\check{c}_1).$
 \item For any $(r_1,r_2,r_3) \in \Z \times \Z \times \Z \backslash 2\Z,$ \newline 
 $ I^{trigr}(\check{c}_0,\chi(r_1,r_2,r_3)\check{c}_1) = I^{trigr}(\chi(-r_1,-r_2,r_3)\check{c}_0,\check{c}_1) = q_1^ {r_1} q_2^ {r_2} q_3^ {r_3} I^{trigr}(\check{c}_0,\check{c}_1).$
 \item If $c_0, c_1$  are not isotopic curves with $\{0\}$ as one of its endpoints, $c_0 \cap c_1 \cap \partial \DB = \emptyset,$ and $I^{trigr}(\check{c}_0, \check{c}_1) = \sum_{r_1,r_2,r_3 } a_{r_1,r_2,r_3} q_1^{r_1} q_2^{r_2} q_3^{r_3},$ then 
 $I^{trigr}(\check{c}_1, \check{c}_0) = \sum_{r_1,r_2,r_3} a_{r_1,r_2,r_3} q_1^{-r_1} q_2^{1-r_2} q_3^{r_3}.$

\noindent If $c_0, c_1$  are isotopic curves with $\{0\}$ as one of its endpoints, $c_0 \cap c_1 \cap \partial \DB = \emptyset,$ then 
 $I^{trigr}(\check{c}_1, \check{c}_0) = I^{trigr}(\check{c}_0, \check{c}_1).$

\end{enumerate}

\end{lemma}

\begin{proof}
For (T1) and (T2) these can be proven using a simple topological argument which we omit. 
For (T3), this is a consequence of \cref{locinsym}. 
We point out that the term $(1+q_1^{-1}q_2q_3)$ is essential in the definition of trigraded intersection numbers was essential in the verification of the formula for two curves that intersect at the point $\{0\}.$
\end{proof}

\subsection{Admissible curves and normal form in $\DB$} \label{normal DB}
	%Let $(\D, \Delta)$ be a closed disc with a finite set of marked points. 
	%Write $\cG= \Diff(\D, \partial \D; \Delta)$. 
 A curve $c$ is called \bit{admissible} if it is equal to $f(b_j)$ for some diffeomorphisms $f \in \Diff \left(\DB, \{0\} \right)$ and $0 \leq j \leq {n}.$
 Note that the endpoints of $c$ must then lie in $\{0, \hdots, n\}$; conversely, every curves which start and end at $\{0, \hdots, n\}$ are admissible.
 Moreover, the two (distinct) orbits $\mathcal{O}([b_1])$ and $\mathcal{O}([b_2])$ under the action of $\cA(B_n) \cong MCG(\D^B_{n+1}, \{0\})$ partition the set of isotopy classes of admissable curves.
 
 We fix the set of basic curves $b_1, \hdots , b_{n}$ and choose vertical curves $d_1, \hdots, d_{n}$ as in Figure \ref{bidi} which divide $\DB$ into regions $D_0, \hdots, D_{n+1}.$
 If an admissible curve $c$ in its isotopy class has minimal intersection with all the $d_j$'s among its other representatives,  then we say that $c$ is in \bit{normal form}. 
 A normal form of $c$ is always achievable by performing an isotopy.

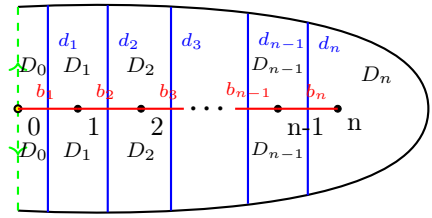
\begin{figure}[H]  
\centering
\begin{tikzpicture} [scale = 0.45]

\draw[thick] plot[smooth, tension=2.25]coordinates {(0,-.5) (12,2.5) (0,5.5)};
\draw[thick,green, dashed]  (0,5.5)--(0,3.875);
\draw[thick, green, dashed]   (0,-.5)--(0,1.125);
\draw[thick, green, dashed, ->]   (0,2.5)--(0,3.875);
\draw[thick,  green, dashed, ->]   (0,2.5)--(0,1.125);
\draw[fill] (1.75,2.5) circle [radius=0.1] ;
\draw[fill] (3.6,2.5) circle [radius=0.1]  ;
\draw[fill] (7.6,2.5) circle [radius=0.1]  ;
\draw[fill] (9.35,2.5) circle [radius=0.1]  ;
\filldraw[color=black!, fill=yellow!, thick]  (0,2.5) circle [radius=0.1];

%vertical line
\draw [thick,blue] (0.875, -.55) -- (0.875, 5.55);
\draw [thick,blue](2.625, -.55) -- (2.625, 5.55);
\draw [thick,blue](4.475, -.5) -- (4.475, 5.5);
\draw [thick, blue] (6.725, -.3) -- (6.725, 5.3);
\draw [thick,blue](8.475, -.05) -- (8.475, 5.05);

\node[right,blue] at (0.9, 4.5) {{\scriptsize $d_1$}};
\node[right,blue] at (2.65, 4.5) {{\scriptsize $d_2$}};
\node[right,blue] at (4.5, 4.5) {{\scriptsize $d_3$}};
\node[right,blue] at (6.75, 4.5) {{\scriptsize $d_{n-1}$}};
\node[right,blue] at (8.5, 4.4) {{\scriptsize $d_n$}};

%redline
\draw [thick,red] (1.75,2.5)--(0,2.5);
\draw [thick,red] (1.75,2.5)--(3.6,2.5);
\draw [thick,red] (3.6,2.5)--(4.85,2.5);
\draw [thick,red] (6.35,2.5)--(7.6,2.5);
\draw [thick,red](7.6,2.5) -- (9.35,2.5);

\node[above right,red] at (0.25,2.5) {{\scriptsize $b_1$}};
\node[above right,red] at (2,2.5) {{\scriptsize $b_2$}};
\node[above right,red] at (3.85,2.5) {{\scriptsize $b_3$}};
\node[above left,red] at (7.7,2.5) {{\scriptsize $b_{n-1}$}};
\node[above left,red] at (9.35,2.5) {{\scriptsize $b_n$}};

\node  at (5.6,2.5) {$ \boldsymbol{\cdots}$};
\node[below right] at  (0,2.5) {0};
\node[below right] at  (1.75,2.5) {1};
\node[below right] at  (3.6,2.5) {2};
\node[below right] at  (7.6,2.5) {n-1};
\node[below right] at  (9.35,2.5) {n};

\node at (.4375,3.75) {{\footnotesize $D_0$}};
\node at (.4375,1.25) {{\footnotesize $D_0$}};
\node at (1.75,3.75) {{\footnotesize $D_1$}};
\node at (1.75,1.25) {{\footnotesize $D_1$}};
\node at (3.6,3.75) {{\footnotesize $D_2$}};
\node at (3.6,1.25) {{\footnotesize $D_2$}};
\node at (7.6,3.75) {{\scriptsize $D_{n-1}$}};
\node at (7.6,1.25) {{\scriptsize  $D_{n-1}$}};
\node at (10.5,3.5) {{\footnotesize  $D_n$}};

\end{tikzpicture}

\caption{{\small The curves $b_i$ and $d_i$ in the aligned configuration with regions $D_i.$} } \label{bidi} 
\end{figure}

% \begin{lemma} \label{isocarry} \cite[Lemma 3.15]{KhoSei}
% Let $c_0$ and $c_1$ be two isotopic admissible curves, both of which are in normal form.
%  Then, there is an isotopy relative to $d_1, \dots, d_{n}$ which carries $c_0$ to $c_1$. 
 %\end{lemma}

Let $c$ be an admissible curve in normal form. 
We use the same classification as in \cite[Section 3e]{KhoSei} to group every connected components of $c \cap D_j$ into finitely many types. 
For $1 \leq j \leq n-1,$ there are six types as depicted in \cref{sixtype} where as for $j=0,n,$ there are two types as shown in \cref{2types0} and \cref{2typesn} .
Moreover, an admissible curve $c$ intersecting all the $d_j$ transversely such that each connected component of $c \cap D_j$ belongs to \cref{sixtype}, \cref{2types0},  and \cref{2typesn} is already in normal form.

\begin{figure}[H] 
\centering
\begin{tikzpicture} [scale = 0.75]
\draw[thick] (0,0)--(14,0);
\draw[thick] (0,0)--(0,12);
\draw[thick] (14,0)--(14,12);
\draw[thick] (0,12)--(14,12);
\draw[thick] (7,0)--(7,12);
\draw[thick] (0,4)--(14,4);
\draw[thick] (0,8)--(14,8);
\draw[thick] (0,0.5)--(14,0.5);
\draw[thick] (0,4.5)--(14,4.5);
\draw[thick] (0,8.5)--(14,8.5);

\draw[thick]  (2.25,1)--(2.25,3.5);
\draw[thick]  (4.75,1)--(4.75,3.5);
\draw[fill] (3.50,2.25) circle [radius=0.1]  ;
\draw[thick,red] (2.25,2.25) -- (3.50,2.25);
\node at (3.50, 0.25) {Type 3};

\draw[thick]  (2.25,5)--(2.25,7.5);
\draw[thick]  (4.75,5)--(4.75,7.5);
\draw[fill] (3.50,6.25) circle [radius=0.1]  ;
\draw[thick,red] plot[smooth,tension=1.75] coordinates {(2.25,5.55) (4.15,6.25) (2.25,6.85) };
\node at (3.50, 4.25) {Type 2};

\draw[thick]  (2.25,9)--(2.25,11.5);
\draw[thick]  (4.75,9)--(4.75,11.5);
\draw[fill] (3.50,10.25) circle [radius=0.1]  ;
\draw[thick,red] (2.25,10.85) -- (4.75,10.85);
\node at (3.50, 8.25) {Type 1};

\draw[thick]  (9.25,1)--(9.25,3.5);
\draw[thick]  (11.75,1)--(11.75,3.5);
\draw[fill] (10.50,2.25) circle [radius=0.1]  ;
\draw[thick,red] (11.75,2.25) -- (10.50,2.25);
\node at (10.50, 0.25) {Type 3'};

\draw[thick]  (9.25,5)--(9.25,7.5);
\draw[thick]  (11.75,5)--(11.75,7.5);
\draw[fill] (10.50,6.25) circle [radius=0.1]  ;
\draw[thick,red] plot[smooth,tension=1.75] coordinates {(11.75,5.55) (9.85,6.25) (11.75,6.85) };
\node at (10.50, 4.25) {Type 2'};

\draw[thick]  (9.25,9)--(9.25,11.5);
\draw[thick]  (11.75,9)--(11.75,11.5);
\draw[fill] (10.50,10.25) circle [radius=0.1]  ;
\draw[thick,red] (9.25,9.55) -- (11.75,9.55);
\node at (10.50, 8.25) {Type 1'};

\node[left] at (2.25,2.25) {{\scriptsize (r$_1$,r$_2$,r$_3$)}};
\node[left] at (9.25,9.55) {{\tiny  (r$_1$+1,r$_2$-1,r$_3$)}};
\node[right] at (11.75,9.55) {{\scriptsize (r$_1$,r$_2$,r$_3$)}};
\node[right] at (11.75,5.55) {{\scriptsize (r$_1$,r$_2$,r$_3$)}};
\node[right] at (11.75,6.85) {{\tiny (r$_1$+1,r$_2$-1,r$_3$)}};
\node[right] at  (11.75,2.25) {{\scriptsize (r$_1$,r$_2$,r$_3$)}};
\node[left] at (2.25,6.85) {{\scriptsize (r$_1$,r$_2$,r$_3$)}};
\node[left] at (2.25,5.55) {{\tiny (r$_1$+1,r$_2$-1,r$_3$)}};
\node[left] at (2.25,10.85)  {{\scriptsize (r$_1$,r$_2$,r$_3$)}};
\node[right] at (4.75,10.85) {{\tiny  (r$_1$+1,r$_2$,r$_3$)}};

\end{tikzpicture}
\caption{{\small The six possible types of connected components $c\cap D_j,$ for $c$ in normal form and $1 \leq j < n,$}} 
\label{sixtype}
\end{figure}
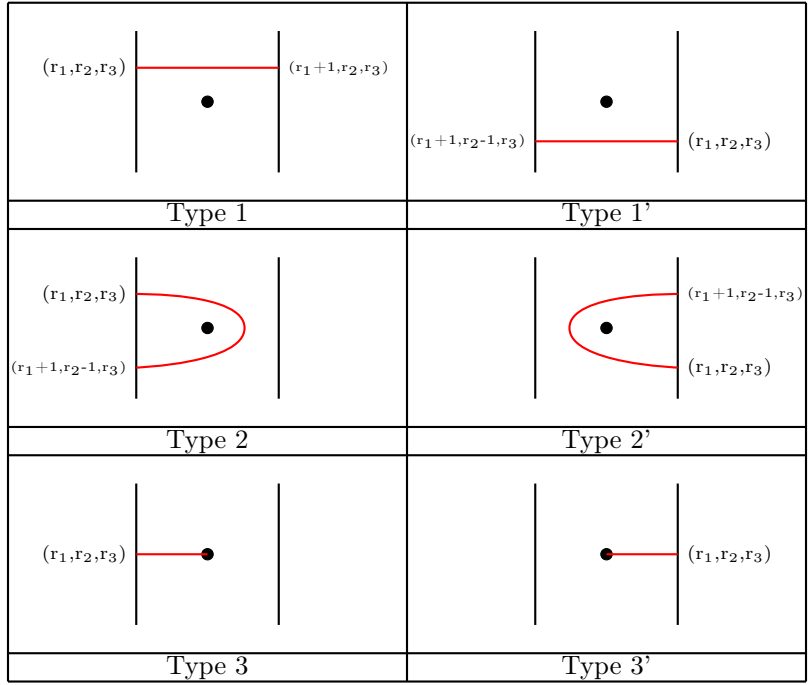

\begin{figure}[H] 
\centering
\begin{tikzpicture} [scale = 0.75]
\draw[thick] (0,0)--(14,0);
\draw[thick] (0,0)--(0,4);
\draw[thick] (14,0)--(14,4);
\draw[thick] (0,4)--(14,4);
\draw[thick]  (7,0)--(7,4);
\draw[thick] (0,0.5)--(14,0.5);

\draw[thick]  (4.75,3.5)--(5,3.5);
\draw[thick]  (4.75,1)--(5,1);
\draw[thick]  (3.5,3.5)--(5,3.5);
\draw[thick]  (3.5,1)--(5,1);
\draw[thick,red]  (3.5,2.85)--(4.75,2.85);
\draw[thick,red]  (3.5,1.55)--(4.75,1.55);
\draw[thick]  (4.75,1)--(4.75,3.5);

\draw[thick,green, dashed]  (3.50,3.5)--(3.50,2.875);
\draw[thick, green, dashed]   (3.50,1)--(3.50,1.825);
\draw[thick, green, dashed, ->]   (3.50,2.25)--(3.50,2.9);
\draw[thick,  green, dashed, ->]   (3.50,2.25)--(3.50,1.6);
\filldraw[color=black!, fill=yellow!, very thick] (3.50,2.25) circle [radius=0.1]  ;
\node at (3.50, 0.25) {Type 2'};

\draw[thick]  (11.75,3.5)--(12,3.5);
\draw[thick]  (11.75,1)--(12,1);
\draw[thick]  (11.75,1)--(11.75,3.5);
\draw[thick]  (10.5,3.5)--(12,3.5);
\draw[thick]  (10.5,1)--(12,1);

\draw[thick,green, dashed]  (10.50,3.5)--(10.50,2.875);
\draw[thick, green, dashed]   (10.50,1)--(10.50,1.825);
\draw[thick, green, dashed, ->]   (10.50,2.25)--(10.50,2.9);
\draw[thick,  green, dashed, ->]   (10.50,2.25)--(10.50,1.6);
\filldraw[color=black!, fill=yellow!, very thick] (10.50,2.25) circle [radius=0.1]  ;

\draw[thick,red] (11.75,2.25) -- (10.50,2.25) ;
\node at (10.50, 0.25) {Type 3'};

%Labelling

\node [right] at (11.75,2.25) {{\scriptsize (r$_1$,r$_2$,r$_3$)}};
\node [right] at (4.75,2.85) {{\scriptsize (r$_1$,r$_2$,r$_3$+1)}};
\node[right] at (4.75,1.55) {{\tiny (r$_1$,r$_2$,r$_3$)}};

\end{tikzpicture}
\caption{{\small The two possible types of connected components $c\cap D_0.$ }} \label{2types0}
\end{figure}

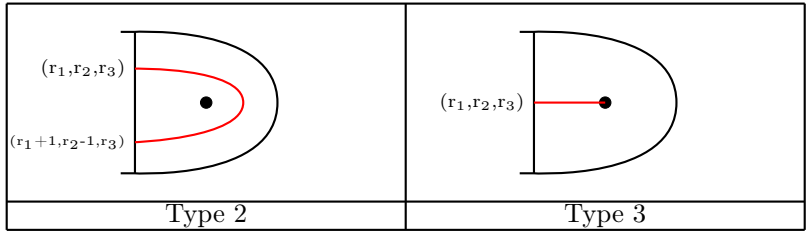
\begin{figure}[H] 
\centering
\begin{tikzpicture} [scale = 0.75]
\draw[thick] (0,0)--(14,0);
\draw[thick] (0,0)--(0,4);
\draw[thick] (14,0)--(14,4);
\draw[thick] (0,4)--(14,4);
\draw[thick]  (7,0)--(7,4);
\draw[thick] (0,0.5)--(14,0.5);

\draw[thick]  (2.25,1)--(2.25,3.5);
\draw[thick]  (2.25,3.5)--(2,3.5);
\draw[thick]  (2.25,1)--(2,1);
\draw[fill] (3.50,2.25) circle [radius=0.1]  ;
\draw[thick] plot[smooth,tension=2] coordinates {(2.25,3.5) (4.75,2.25) (2.25,1)};
\draw[thick,red] plot[smooth,tension=1.75] coordinates {(2.25,1.55) (4.15,2.25) (2.25,2.85) };
\node at (3.50, 0.25) {Type 2};

\draw[thick]  (9.25,1)--(9.25,3.5);
\draw[thick]  (9.25,3.5)--(9,3.5);
\draw[thick]  (9.25,1)--(9,1);
\draw[thick] plot[smooth,tension=2] coordinates {(9.25,3.5) (11.75,2.25) (9.25,1)};
\draw[fill] (10.50,2.25) circle [radius=0.1]  ;
\draw[thick,red] (9.25,2.25) -- (10.50,2.25);
\node at (10.50, 0.25) {Type 3};

\node [left] at (2.25,1.55) {{\tiny (r$_1$+1,r$_2$-1,r$_3$)}};
\node [left] at (2.25,2.85) {{\scriptsize (r$_1$,r$_2$,r$_3$)}};
\node [left] at (9.25,2.25) {{\scriptsize (r$_1$,r$_2$,r$_3$)}};

\end{tikzpicture}
\caption{{\small The two possible types of connected components $c\cap D_n.$}} \label{2typesn}
\end{figure}

\begin{comment}
\begin{remark}
The list of types only classifies each connected component of $c \cap D_i$ up to an isotopy in Diff$(\D, n+1, \{0\}).$ 
	This is because we are working with elements in Diff$(\D, n+1, \{0\})$ which fixes the curves $d_{i-1}$ and $d_{i}$ as well as the region $D_i.$   

\end{remark}

Hence, we can deduce that the number of connected components of each type, their relative position and the way they join each other are an invariant of the isotopy class of $c$ from \cref{isocarry}. 

\end{comment}

For the rest of this section, by a \emph{curve} $c$, we mean an admissible curve in  normal form. 
  We call the intersection of a curve $c$ with the curves $d_i$ \emph{crossings} and denote them 
$cr(c) = c \cap (d_0 \cup d_1 \cup \hdots \cup d_{n-1}).$
  Those intersections $c \cap d_j$ are called \emph{$j$-crossings of $c$}.
   For $0 \leq j \leq n$, the connected components of $c \cap D_j$ are called \emph{segments} of $c$.
 If the endpoints of a segment are both crossings, then it is \emph{essential}.

  Now, we will study the action of half twist $[t_{b_k}^B]$ on normal forms.
 In gerenal, $[t_{b_k}^B](c)$ won't be in normal form even though $c$ is a normal form.
	Nonetheless,  $[t_{b_k}^B](c)$ has minimal intersection with all $d_j$ for $j \neq k.$
 One just need to decrease its intersections with $d_k,$ in order to get  $[t_{b_k}^B](c)$ into normal form.
 The same argument as in \cite{KhoSei} gaves us the analogous result:

\begin{proposition} ~
\begin{enumerate} 
\item The normal form of $[t_{b_k}^B](c)$ coincides with $c$ outside of $D_{k-1} \cup D_{k}.$
	The curve $[t_{b_k}^B](c)$ can be brought into normal gorm by an isotopy inside   $D_{k-1} \cup D_{k}.$
	
\item Suppose that $[t_{b_k}^B](c)$ is in normal form. 
	There is a natural bijection between $j$-crossings of $c$ and the $j$-crossings of $[t_{b_k}^B](c)$ for $j \neq k.$
		There is a natural bijection between connected components of intersections of $c$ and $[t_{b_k}^B](c)$ inside $D_{k-1} \cup D_{k}.$
\end{enumerate}
\end{proposition}

A connected component of $c \cap (D_{j-1} \cup D_{j})$ is called \emph{j-string of $c$}.
Denote by $st(c,j)$ the set of $j$-string of $c.$ 
In addition, we define a \emph{j-string} as curve in  $D_{j-1} \cup D_{j}$ which is a $j$-string of $c$ for some admissible curve $c$ in normal form.

Two $j$-strings are isotopic (equivalently belong to the same isotopy class) if there exists a deformation of one into the other via diffeomorphisms $f$ of $D' = D_{j-1} \cup D_{j}$ which fix $d_{j-1}$ and $d_{j+1}$ as well as preserves the marked points in $D'$ pointwise , that is, $f(d_{j-1})= d_{j-1}, f(d_{j+1})= d_{j+1},$ and $f|_{\Delta \cap D'} = id.$

For $1 < j < n,$ isotopy classes of $j$-strings can be divided into types as follows: there are five infinite families $I_w, II_w, II'_w, III_w, III'_w(w \in \Z)$ and five exceptional types $IV, IV',V,V'$ and $VI$ (see \cref{B j-string}). 
	When $j = n$, there is a similar list, with two infinite families and two exceptional types (see \cref{B n-string}).
    The rule for obtaining the $(w+1)$-th from the $w$-th is by applying $[t^B_{b_j}].$ 
	For $1$-string, there are, instead, four infinite families  the type $II'_w, II'_{w + \frac{1}{2}},  III'_w,  III'_{w+ \frac{1}{2}}(w \in \Z)$ and two exceptional types $V' $ and $VI$ (see \cref{B 1-string}).
	Note that for 1-strings, the rule for obtaining the $(w+1)$-th from the $w$-th is instead by applying $[(t^B_{b_1})^2].$

\begin{comment}
\begin{lemma}
If $j>1,$ the geometric intersection number $I(b_j,c)$ can be computed as follows: every $j-$string of $c$ which is of type $I_w, II_w, II'_w$ or $VI$ contributes $1,$
	those of type $III_w, III'_w$ contribute $1/2, $ and the rest zero. 
	Similarly, for $j=1,$ the types $II'_w$  and $II'_{w+ \frac{1}{2}}$ contribute $1$, the types $III'_w$  and $III'_{w+ \frac{1}{2}}$ contribute $\frac{1}{2}$, and the types $V'$ contribute $0.$
\end{lemma}
\begin{proof}
Follow the proof in \cite{KhoSei}.
\end{proof}	
\end{comment}

\begin{figure}[h] 
\centering
\begin{tikzpicture} [scale=0.7][yscale=1, xscale=0.92]
\draw[thick] (0,0)--(16,0);
\draw[thick] (0,0)--(0,14);
\draw[thick] (16,0)--(16,14);
\draw[thick] (0,14)--(16,14);
\draw[thick] (8,0)--(8,14);
\draw[thick] (0,4)--(16,4);

\draw[thick] (0,0.5)--(16,0.5);
\draw[thick] (0,3.5)--(16,3.5);
\draw[thick] (0,4)--(16,4);
\draw[thick] (0,7)--(16,7);
\draw[thick] (0,7.5)--(16,7.5);
\draw[thick] (0,10.5)--(16,10.5);
\draw[thick] (0,11)--(16,11);
\draw[thick] (0,14)--(16,14);
\draw[thick] (4,14)--(4,17.5);
\draw[thick] (12,14)--(12,17.5);

\draw[thick] (4,-3.5)--(12,-3.5);
\draw[thick] (4,-3)--(12,-3);
\draw[thick] (4,-3.5)--(4,0);
\draw[thick] (12,-3.5)--(12,0);
\draw[thick] (4,14.5)--(12,14.5);
\draw[thick] (4,17.5)--(12,17.5);

\draw[thick]  (2.525,1)--(2.525,3);
\draw[thick]  (5.475,1)--(5.475,3);
\draw[fill] (3.25,2) circle [radius=0.1]  ;
\draw[fill] (4.75,2) circle [radius=0.1]  ;
\draw[thick,red] plot[smooth,tension=2.25] coordinates {(2.525,1.5)  (5.1125,2) (2.525, 2.5)};
\node at (4, .25) {{\small Type V}};

\draw[thick]  (10.525,1)--(10.525,3);
\draw[thick]  (13.475,1)--(13.475,3);
\draw[fill] (11.25,2) circle [radius=0.1]  ;
\draw[fill] (12.75,2) circle [radius=0.1]  ;
\draw[thick,red] plot[smooth,tension=2.25] coordinates {(13.475,2.5)  (10.8675,2) (13.475, 1.5)};
\node at (12, 0.25) {{\small Type V'}};

\draw[thick]  (2.525,4.5)--(2.525,6.5);
\draw[thick]  (5.475,4.5)--(5.475,6.5);
\draw[fill] (3.25,5.5) circle [radius=0.1]  ;
\draw[fill] (4.75,5.5) circle [radius=0.1]  ;
\draw[thick,red]  (2.525,6)--(5.475,6);;
\node at (4, 3.75) {{\small Type IV}};

\draw[thick]  (10.525,4.5)--(10.525,6.5);
\draw[thick]  (13.475,4.5)--(13.475,6.5);
\draw[fill] (11.25,5.5) circle [radius=0.1]  ;
\draw[fill] (12.75,5.5) circle [radius=0.1]  ;
\draw[thick,red]  (10.525,5)--(13.475,5);;
\node at (12, 3.75) {{\small Type IV'}};

\draw[thick]  (2.525,8)--(2.525,10);
\draw[thick]  (5.475,8)--(5.475,10);
\draw[fill] (3.25,9) circle [radius=0.1]  ;
\draw[fill] (4.75,9) circle [radius=0.1]  ;
\draw[thick,red] (2.525,9) -- (3.25,9);
\node at (4, 7.25) {{\small Type III$_0$}};

\draw[thick]  (10.525,8)--(10.525,10);
\draw[thick]  (13.475,8)--(13.475,10);
\draw[fill] (11.25,9) circle [radius=0.1]  ;
\draw[fill] (12.75,9) circle [radius=0.1]  ;
\draw[thick,red] (12.75,9) -- (13.475,9);
\node at (12, 7.25) {{\small Type III'$_0$}};

\draw[thick]  (2.525,11.5)--(2.525,13.5);
\draw[thick]  (5.475,11.5)--(5.475,13.5);
\draw[fill] (3.25,12.5) circle [radius=0.1]  ;
\draw[fill] (4.75,12.5) circle [radius=0.1]  ;
\draw[thick,red] plot[smooth,tension=1.7] coordinates {(2.525,13)  (4.1,12.5) (2.525, 12)};
\node at (4, 10.75) {{\small Type II$_0$}};

\draw[thick]  (10.525,11.5)--(10.525,13.5);
\draw[thick]  (13.475,11.5)--(13.475,13.5);
\draw[fill] (11.25,12.5) circle [radius=0.1]  ;
\draw[fill] (12.75,12.5) circle [radius=0.1]  ;
\draw[thick,red] plot[smooth,tension=1.7] coordinates {(13.475,13)  (12.1,12.5) (13.475, 12)};
\node at (12, 10.75) {{\small Type II'$_0$}};

\draw[thick]  (6.525,15)--(6.525,17);
\draw[thick]  (9.475,15)--(9.475,17);
\draw[fill] (7.25,16) circle [radius=0.1]  ;
\draw[fill] (8.75,16) circle [radius=0.1]  ;
\draw[thick,red] plot[smooth,tension=.7] coordinates {(6.525,16)  (7.4,16.4)  (8.6, 15.6) (9.475,16)};
\node at (8, 14.25) {{\small Type I$_0$}};

\draw[thick]  (6.525,-2.5)--(6.525,-0.5);
\draw[thick]  (9.475,-2.5)--(9.475,-0.5);
\draw[fill] (7.25,-1.5) circle [radius=0.1]  ;
\draw[fill] (8.75,-1.5) circle [radius=0.1]  ;
\draw[thick,red] (7.25,-1.5)--(8.75,-1.5);
\node at (8, -3.25) {{\small Type VI}};

%Labelling
\node [left] at (6.525,16) {{\scriptsize (r$_1$,r$_2$,r$_3$)}};
\node [right] at (9.475,16) {{\tiny (r$_1$,r$_2$+1,r$_3$)}};
\node [left] at (2.525,13) {{\scriptsize (r$_1$,r$_2$,r$_3$)}};
\node [left] at (2.525, 12) {{\tiny (r$_1$+1,r$_2$-1,r$_3$)}};
\node [right] at (13.475,13) {{\tiny (r$_1$+1,r$_2$-1,r$_3$)}};
\node [right] at (13.475, 12) {{\scriptsize (r$_1$,r$_2$,r$_3$)}};
\node [left] at (2.525,9) {{\tiny (r$_1$,r$_2$,r$_3$)}};
\node [right] at (13.475,9) {{\scriptsize (r$_1$,r$_2$,r$_3$)}};
\node [left] at (2.525,6) {{\scriptsize (r$_1$,r$_2$,r$_3$)}};
\node [right] at (5.475,6) {{\tiny (r$_1$+2,r$_2$,r$_3$)}};
\node [right] at (13.475,5) {{\scriptsize (r$_1$,r$_2$,r$_3$)}};
\node [left] at (10.525,5) {{\tiny (r$_1$+2,r$_2$-2,r$_3$)}};
\node [left] at (2.525, 2.5) {{\scriptsize (r$_1$,r$_2$,r$_3$)}};
\node [left] at (2.525,1.5)  {{\tiny (r$_1$+3,r$_2$-2,r$_3$)}};
\node [right] at (13.475, 1.5)  {{\scriptsize (r$_1$,r$_2$,r$_3$)}};
\node [right] at (13.475,2.5)  {{\tiny (r$_1$+3,r$_2$-2,r$_3$)}};

\end{tikzpicture}
\caption{{\small The isotopy classes of $j$-strings, for $1 < j < n,$}}
\label{B j-string}
\end{figure}

\begin{figure}[h] 
\centering
\begin{tikzpicture} [scale=0.7] [yscale=1, xscale=0.92]

%Vertical Line
\draw[thick] (0,0)--(16,0);
\draw[thick] (0,0)--(0,10.5);
\draw[thick] (16,0)--(16,10.5);
\draw[thick] (8,0)--(8,10.5);

%Horizontal line
\draw[thick] (0,0.55)--(16,0.55);
\draw[thick] (0,3.5)--(16,3.5);
\draw[thick] (0,4.05)--(16,4.05);
\draw[thick] (0,7)--(16,7);
\draw[thick] (0,7.55)--(16,7.55);
\draw[thick] (0,10.5)--(16,10.5);

\draw[thick]  (5.475,1)--(5.7,1);
\draw[thick]  (5.475,3)--(5.7,3);  
\draw[thick]  (5.475,1)--(5.475,3);
\draw[thick,red]  (3.25,2.5)--(5.475,2.5);
\draw[thick,red]  (3.25,1.5)--(5.475,1.5);
\draw[thick]  (3.25,3)--(5.7,3);  
\draw[thick]  (3.25,1)--(5.7,1);
\draw[thick,green, dashed]  (3.25,3)--(3.25,2.5);
\draw[thick, green, dashed]   (3.25,1)--(3.25,1.5);
\draw[thick, green, dashed, ->]   (3.25,2)--(3.25,2.6);
\draw[thick,  green, dashed, ->]   (3.25,2)--(3.25,1.4);\filldraw[color=black!, fill=yellow!, very thick] (3.25,2) circle [radius=0.1]  ;
\draw[fill] (4.75,2) circle [radius=0.1]  ;
\node [right] at (5.475, 2.5) {{\tiny (r$_1$+2,r$_2$-1,r$_3$+1)}};
\node [right] at (5.475, 1.5) {{\scriptsize (r$_1$,r$_2$,r$_3$)}};
\node at (4, .25) {{\small Type V'}};

\draw[thick]  (13.475,1)--(13.7,1);
\draw[thick]  (13.475,3)--(13.7,3);  
\draw[thick]  (13.475,1)--(13.475,3);
\draw[thick]  (11.25,3)--(13.7,3);  
\draw[thick]  (11.25,1)--(13.7,1);
\draw[thick,green, dashed]  (11.25,3)--(11.25,2.5);
\draw[thick, green, dashed]   (11.25,1)--(11.25,1.5);
\draw[thick, green, dashed, ->]   (11.25,2)--(11.25,2.6);
\draw[thick,  green, dashed, ->]   (11.25,2)--(11.25,1.4);
\filldraw[color=black!, fill=yellow!, very thick] (11.25,2) circle [radius=0.1]  ;
\draw[fill] (12.75,2) circle [radius=0.1]  ;
\draw[thick,red] (11.25,2) -- (12.75,2);
\node at (12, 0.25) {{\small Type VI}};

\draw[thick]  (5.475,4.5)--(5.7,4.5);
\draw[thick]  (5.475,6.5)--(5.7,6.5);  
\draw[thick]  (5.475,4.5)--(5.475,6.5);

\draw[thick]  (3.25,6.5)--(5.7,6.5);  
\draw[thick]  (3.25,4.5)--(5.7,4.5);
\draw[thick,green, dashed]  (3.25,6.5)--(3.25,6);
\draw[thick, green, dashed]   (3.25,4.5)--(3.25,5);
\draw[thick, green, dashed, ->]   (3.25,5.5)--(3.25,6.1);
\draw[thick,  green, dashed, ->]   (3.25,5.5)--(3.25,4.9);

\filldraw[color=black!, fill=yellow!, very thick]  (3.25,5.5) circle [radius=0.1]  ;
\draw[fill] (4.75,5.5) circle [radius=0.1]  ;
\draw[thick,red] (4.75,5.5) -- (5.475,5.5);
\node [right] at (5.475,5.5) {{\scriptsize (r$_1$,r$_2$,r$_3$)}};
\node at (4, 3.75) {{\small Type III'$_0$}};

\draw[thick]  (13.475,4.5)--(13.7,4.5);
\draw[thick]  (13.475,6.5)--(13.7,6.5);  
\draw[thick]  (13.475,4.5)--(13.475,6.5);
\draw[thick]  (11.25,6.5)--(13.7,6.5);  
\draw[thick]  (11.25,4.5)--(13.7,4.5);
\draw[thick,green, dashed]  (11.25,6.5)--(11.25,6);
\draw[thick, green, dashed]   (11.25,4.5)--(11.25,5);
\draw[thick, green, dashed, ->]   (11.25,5.5)--(11.25,6.1);
\draw[thick,  green, dashed, ->]   (11.25,5.5)--(11.25,4.9);
\filldraw[color=black!, fill=yellow!, very thick]  (11.25,5.5) circle [radius=0.1]  ;
\draw[fill] (12.75,5.5) circle [radius=0.1]  ;
\draw[thick,red] plot[smooth,tension=1] coordinates {(11.25,5.5)  (12.225,5.95) (13.475, 6)};
\node [right] at (13.475, 6) {{\scriptsize (r$_1$,r$_2$,r$_3$)}};
\node at (12, 3.75) {{\small Type III'$_\frac{1}{2}$}};

\draw[thick]  (5.475,8)--(5.7,8);
\draw[thick]  (5.475,10)--(5.7,10);  
\draw[thick]  (5.475,8)--(5.475,10);

\draw[thick]  (3.25,10)--(5.7,10);  
\draw[thick]  (3.25,8)--(5.7,8);
\draw[thick,green, dashed]  (3.25,10)--(3.25,9.5);
\draw[thick, green, dashed]   (3.25,8)--(3.25,8.5);
\draw[thick, green, dashed, ->]   (3.25,9)--(3.25,9.6);
\draw[thick,  green, dashed, ->]   (3.25,9)--(3.25,8.4);

\filldraw[color=black!, fill=yellow!, very thick]  (3.25,9) circle [radius=0.1]  ;
\draw[fill] (4.75,9) circle [radius=0.1]  ;
\draw[thick,red] plot[smooth,tension=1.7] coordinates {(5.475,9.5)  (4.1,9) (5.475, 8.5)};

\node [right] at (5.475, 9.5) {{\tiny (r$_1$+1,r$_2$-1,r$_3$)}};
\node [right] at (5.475, 8.5) {{\scriptsize  (r$_1$,r$_2$,r$_3$)}};
\node at (4, 7.25) {{\small Type II'$_0$}};

\draw[thick]  (13.475,8)--(13.7,8);
\draw[thick]  (13.475,10)--(13.7,10); 
\draw[thick]  (13.475,8)--(13.475,10);
\draw[thick,red]  (11.25,9.65)--(13.475,9.65);
\draw[thick]  (11.25,10)--(13.7,10);  
\draw[thick]  (11.25,8)--(13.7,8);
\draw[thick,green, dashed]  (11.25,10)--(11.25,9.5);
\draw[thick, green, dashed]   (11.25,8)--(11.25,8.5);
\draw[thick, green, dashed, ->]   (11.25,9)--(11.25,9.6);
\draw[thick,  green, dashed, ->]   (11.25,9)--(11.25,8.4);
\filldraw[color=black!, fill=yellow!, very thick]  (11.25,9) circle [radius=0.1]  ;
\draw[fill] (12.75,9) circle [radius=0.1]  ;

\draw[thick,red] plot[smooth,tension=1] coordinates { (11.25,8.35) (11.7,8.5) (12.5,9.2) (13.475,9.35)};

\node [right] at (13.475,9.65) {{\tiny (r$_1$,r$_2$,r$_3$+1)}};
\node [right] at (13.475,9.3) {{\scriptsize (r$_1$,r$_2$,r$_3$)}};
\node at (12, 7.25) {{\small Type II'$_\frac{1}{2}$}};

\end{tikzpicture}
\caption{{\small The isotopy classes of $1$-strings.}}
\label{B 1-string}
\end{figure}

\begin{figure}[h] 
\centering
\begin{tikzpicture} [scale= 0.7][yscale=1, xscale=0.92]
%vertical line
\draw[thick] (0,0)--(16,0);
\draw[thick] (0,0)--(0,7);
\draw[thick] (16,0)--(16,7);
\draw[thick] (8,0)--(8,7);
\draw[thick] (0,4)--(16,4);

%horizontal line
\draw[thick] (0,0.5)--(16,0.5);
\draw[thick] (0,3.5)--(16,3.5);
\draw[thick] (0,4)--(16,4);
\draw[thick] (0,7)--(16,7);

\draw[thick] plot[smooth, tension=2.25]coordinates {(2.525,1) (5.475,2) (2.525,3)};
\draw[thick]  (2.525,1)--(2.3,1);
\draw[thick]  (2.525,3)--(2.3,3);
\draw[thick]  (2.525,1)--(2.525,3);
\draw[fill] (3.25,2) circle [radius=0.1]  ;
\draw[fill] (4.75,2) circle [radius=0.1]  ;
\draw[thick,red] plot[smooth,tension=2.25] coordinates {(2.525,1.5)  (5.1125,2) (2.525, 2.5)};
\node at (4, .25) {{\small Type V}};
\node [left] at (2.525, 2.5) {{\scriptsize (r$_1$,r$_2$,r$_3$)}};
\node [left] at (2.525,1.5)  {{\tiny (r$_1$+3,r$_2$-2,r$_3$)}};

\draw[thick]  (10.525,4.5)--(10.525,6.5);
\draw[thick] plot[smooth, tension=2.25]coordinates {(10.525,4.5) (13.475,5.5) (10.525,6.5)};
\draw[thick]  (10.525,4.5)--(10.3,4.5);
\draw[thick]  (10.525,6.5)--(10.3,6.5);
\draw[fill] (11.25,5.5) circle [radius=0.1]  ;
\draw[fill] (12.75,5.5) circle [radius=0.1]  ;
\draw[thick,red] (10.525,5.5) -- (11.25,5.5);
\node at (12, 3.75) {{\small Type III$_0$}};
\node [left] at (10.525,5.5) {{\scriptsize (r$_1$,r$_2$,r$_3$)}};

\draw[thick]  (2.525,4.5)--(2.525,6.5);
\draw[thick] plot[smooth, tension=2.25]coordinates {(2.525,4.5) (5.475,5.5) (2.525,6.5)};
\draw[thick]  (2.525,4.5)--(2.3,4.5);
\draw[thick]  (2.525,6.5)--(2.3,6.5);
\draw[fill] (3.25,5.5) circle [radius=0.1]  ;
\draw[fill] (4.75,5.5) circle [radius=0.1]  ;
\draw[thick,red] plot[smooth,tension=1.7] coordinates {(2.525,6)  (4.1,5.5) (2.525, 5)};
\node at (4, 3.75) {{\small Type II$_0$}};
\node [left] at (2.525,6) {{\scriptsize (r$_1$,r$_2$,r$_3$)}};
\node [left] at (2.525, 5) {{\tiny (r$_1$+1,r$_2$-1,r$_3$)}};

\draw[thick]  (10.525,1)--(10.525,3);
\draw[thick] plot[smooth, tension=2.25]coordinates {(10.525,1) (13.475,2) (10.525,3)};
\draw[thick]  (10.525,1)--(10.3,1);
\draw[thick]  (10.525,3)--(10.3,3);
\draw[fill] (11.25,2) circle [radius=0.1]  ;
\draw[fill] (12.75,2) circle [radius=0.1]  ;
\draw[thick,red] (11.25,2)--(12.75,2);
\node at (12, 0.25) {{\small Type VI}};

\end{tikzpicture}
\caption{{\small The isotopy classes of $n$-strings.}}  \label{B n-string}
\end{figure}

Based on our definition, $j$-strings are assumed to be in normal form. 
	As before, we can define \emph{crossings} and \emph{essential segments} of $j$-string as in the case for admissible curves in normal form and denote the set of crossings of a $j$-string $g$ by $cr(g).$

Now, let us adapt the discussion to trigraded curves. Choose trigradings $\check{b}_j, \check{d}_j$ of $b_j,d_j$ for $1 \leq j \leq n,$ such that
\begin{equation} \label{basic trigrading condition}
I^{trigr}(\check{d}_j,\check{b}_j) = (1+ q_3)(1+q_1^{-1}q_2), \qquad
 I^{trigr}(\check{b}_j,\check{b}_{j+1}) = 1 + q_3.
\end{equation}
These conditions determine the trigradings uniquely up to an overall shift $\chi(r_1,r_2,r_3).$

Suppose $\check{c}$ is a trigrading of an admisible curve $c$ in normal form. 
	If $a \subset c$ is  a connected component of $c \cap D_j$ for some $j$, and $\check{a}$ is $\check{c}|_{a\setminus \Lambda}$, then $\check{a}$ is evidently determined by $a$ together with the local index $\mu^{trigr}(\check{d}_{j-1}, \check{a};z)$ or $\mu^{trigr}(\check{d}_{j}, \check{a};z)$ at any point $z \in (d_{j-1} \cup d_j) \cap a.$
	Moreover, if there is more than one such point, the local indices determine each other.
	
	In \cref{sixtype}, \cref{2types0}, and \cref{2typesn}, we classify the types of pair $(a,\check{a})$ with the local indices. 
	For instance, consider the type $1(r_1, r_2, r_3)$ with $(k-1)$-crossing $z_0 \in d_{k-1} \cap a$ and $k$-crossing $z_1 \in d_{k} \cap a$ and we have that the local index at $z_0$ and at $z_1$ to be $\mu^{trigr}(\check{d}_{k-1}, \check{a};z_0) = (r_1, r_2, r_3)$ and $\mu^{trigr}(\check{d}_{k}, \check{a};z_1) = (r_1 +1, r_2, r_3)$ respectively.

We recall, in \cref{tribundle}, that there is a unique lift $\check{f}$ that is the identity on the boundary for any diffeomorphism $f \in \DiffB$ to a diffeomorphism of $\wRDAz.$
	Similarly, we can lift mapping classes in $\DB$ to mapping classess in $\wRDAz.$
   Denote by $\check{t}^A_{b_j}$ the canonical lift of the twist $t^A_{b_j}$ along the curve $b_j$ in $\DA.$

\begin{proposition} \label{injBA}
 Understanding $\wRDAz$ as a subset of $\wRDA,$ the map $ \psi: \cA(B_n) \ra \MCG\left(\wRDAz \right) $ defined by
$ \psi(\sigma^B_i) = \begin{cases} 
      [\check{t}^A_{b_n}|_{\wRDAz}] & \text{ for } i = 1; \\
      [\check{t}^A_{b_{n+i-1}}, \check{t}^A_{b_{n-(i-1)}}|_{\wRDAz}] & \text{ for } i \geq 2, \\
   \end{cases}
$
is an injective group homomorphism. 

\end{proposition}

\begin{proof}
 This can be proven similar to  \cref{injB}.
\end{proof}

\begin{comment}
\begin{proposition}
The diffeomorphisms $\check{t}_{b_j}$ for $1 \leq j \leq n$ induce a type $B_n$ braid group action on the set of isotopy classes of admissible trigraded curves; if $\check{c}$ is an admissible trigraded curve, we have the following isotopy relations of curves:

\begin{align}
\check{t}_{b_1} \check{t}_{b_2} \check{t}_{b_1} \check{t}_{b_2}(\check{c}) & = \check{t}_{b_2} \check{t}_{b_1} \check{t}_{b_2} \check{t}_{b_1} (\check{c})  \\
     \check{t}_{b_j} \check{t}_{b_k}(\check{c})  &= \check{t}_{b_k} \check{t}_{b_j} (\check{c})  & \text{for} \ |j-k|> 1,\\    
     \check{t}_{b_j} \check{t}_{b_{j+1}} \check{t}_{b_j} ( \check{c})  &=  \check{t}_{b_{j+1}} \check{t}_{b_j} \check{t}_{b_{j+1}} ( \check{c})  & \text{for} \ j= 2,3, \ldots, n.
\end{align} 

\end{proposition}

\end{comment}

A crossing of $c$ will be also a \emph{crossing of $\check{c}$}, and we denote the set of crossings of $\check{c}$ by $cr(\check{c}).$
 Note that as set, $cr(\check{c}) = cr(c).$ 
 However, a crossing of $\check{c}$ comes with a local index in $\Z \times \Z \times \Z/ 2\Z.$
 
	Moreover, we assign each crossing $y$ of $\check{c}$ to a $4$-tuple $(y_0, y_1, y_2, y_3)$ where $y_0$ denote the vertical curve which contains the crossing $y \in d_{y_0} \cap c$, and $(y_1, y_2, y_3)$ is the local index $(\mu_1, \mu_2,\mu_3)$ of the crossing $y.$  \label{local index function}
	
	We define the \emph{essential segments} of $\check{c}$ as the essential segments of $c$ together with the trigradings which can be obtained from  local indices assigned to the ends of the segments.
	
	We also define a \emph{$j$-string of $\check{c}$} as a connected component of $\check{c} \cap \left( D_{j-1} \cup D_{j} \right)$ together with the trigrading induced from $\check{c}.$ 
	Denote the set of $j$-string of $\check{c}$ by $st(\check{c},j).$
	
	On top of that, we define a trigraded $j$-string as a trigraded curve in $D_{j-1} \cup D_{j}$ that is a connected component of $\check{c} \cap (D_{j-1} \cup D_{j})$ for some trigraded curve $\check{c}.$ 
	
	In \cref{B j-string}, \cref{B 1-string}, and \cref{B n-string}, we depict the isotopy classes of trigraded $j$-strings. 
	Since $j$-string of type $VI$ does not intersect with $d_{j-1} \cup d_{j+1},$ we say that a trigraded $j$-string $\check{g}$ with the underlying $j$-string $g$ of type $VI$ has type $VI(r_1, r_2, r_3)$ if $\check{g}= \chi^B(r_1, r_2, r_3) \check{b}_j$

     The next result shows how one can compute the trigraded intersection number of $\check{b}_j$ with any given trigraded curve. 

\begin{lemma} \label{compute tri int}
Let $(c, \check{c})$ be a trigraded curve. 
Then, $I^{trigr}(\check{b}_j, \check{c})$ can be computed by adding up contributions from each trigraded $j$-string of $\check{c}.$
 For $j >0,$ the contributions are listed in the following table:
\begin{center}
\begin{tabular}{ c|c|c|c} 

 $I_0(0,0,0)$ & $II_0(0,0,0)$ & $II'_0(0,0,0)$ & $III_0(0,0,0)$  \\ 
 \hline 
$q_1 + q_2 + q_2q_3 + q_1q_3 $ & $q_1 + q_2 + q_2q_3 + q_1q_3 $ & $1 + q_1q_2^{-1} + q_3 + q_1q_2^{-1} q_1q_3 $  & $q_2 + q_2q_3$  \\

\end{tabular}
\end{center}
\begin{center}
\begin{tabular}{ c|c|c|c|c|c} 
 $III'_0(0,0,0)$ & $IV$  & $IV'$ & $V$ &  $V'$  & $VI(0,0,0)$\\
 \hline
 $1 + q_3$ & 0 & 0 & 0 & 0 & $1+q_2 +q_3 + q_2q_3$

\end{tabular}
\end{center}
and the remaining ones can be computed as follows: to determine the contribution of a component of type, say, $I_u(r_1,r_2,r_3)$, one takes the contribution of $I_0(0,0,0)$ and multiplies it by $q_1^{r_1}q_2^{r_2}q_3^{r_3}(q_1^{-1}q_{2})^u.$
For $j=0,$ the relevant contributions are
\begin{center}
\begin{tabular}{ c|c|c|c|c|c} 
  $II'_0(0,0,0)$ &  $II'_\frac{1}{2}(0,0,0)$ & $III'_0(0,0,0)$ & $III'_\frac{1}{2}(0,0,0)$ & $V'$  & $VI(0,0,0)$\\ 
 \hline 
$ 1+ q_3 + q_1q_2^{-1} + q_1q_2^{-1}q_3$ & $ 1+ q_3 + q_1^{-1}q_2 + q_1^{-1}q_2q_3$ & $1 + q_3 $  & $q_1^{-1}q_2 + q_3$ & 0 & $1 + q_2$ 
\end{tabular}
\end{center}
and the remaining ones can be computed as follows: to determine the contribution of a component of type, say, $II'_u(r_1,r_2,r_3)$, one takes the contribution of $II'_0(0,0,0)$ and multiplies it by $q_1^{r_1}q_2^{r_2}q_2^{r_2}(q_1^{-1}q_2q_3)^u.$

\end{lemma}

\begin{proof}
Apply \cref{action} a well as $(T2)$ and $(T3)$ of \cref{triintpro}.
\end{proof}

\subsection{Bigraded curves and bigraded multicurves  in $\wRDAz$}
We briefly remind the reader the definition of a bigraded curve in $\wRDAz$.
Refer to \cite[Section 3d]{KhoSei} for a more detailed construction.
	Consider the projectivasation  $\RDA := PT \left(\DA \setminus \Delta \right)$ of the tangent bundle of $\DA \setminus \Delta.$ 
	 The covering $\wRDA$ of $\RDA$ is classified by
	the cohomology class $C \in H^1(\RDA; \Z \times \Z)$ defined as follows:
	\begin{align}
 C([point \times \lambda_i]) &= (-2,1) \ \text{for } i= -n, \cdots, -1, 1, \cdots, n; \\
 C([\R \text{P}^1 \times point]) &= (1,0).
\end{align}
A bigrading of a curve $c \in \DA$ is a lift $\ddot{c}$ of $s^A_c $ to $\wRDA$ where  $s^A_c: c \setminus \Delta \ra \RDA$ is the canonical section given by $s^A_c(z) = T_zc.$ 
  A \emph{bigraded curve} is a pair $(c, \ddot{c}),$ where sometimes we just abbreviate as $\ddot{c}.$  
  A bigraded multicurve $\ddot{\fc}$ consists of a disjoint union of a finite collection of disjoint bigraded curves.
  There is an obvious notion of isotopy for bigraded multicurves.
%  A crossing of $\ddot{\fc}$ is simply the crossings of the  disjoint union of a finite collection of disjoint bigraded curves. \rb{move}

\subsection{Lifting of trigraded curves to bigraded curves}\label{lift section}
Our goal is to define a map  $\m: \check{\fC} \ra \ddot{\undertilde{\wt{\fC}}}$ from the set $\check{\fC}$ of isotopy classes of trigraded curves to the set $\ddot{\undertilde{\wt{\fC}}}$ of isotopy classes of bigraded multicurves.
Let $c $ be a curve in $\D^B_{n+1}$ with trigrading $\check{c}$.
First consider the case when $c \cap \{0\} = \emptyset.$
Recall the map $q_{br}: \DA \ra \DB $ as defined in \cref{brcover}. Then,  $q_{br}^{-1}{(c)}$ has two connected components in $\DA$; denote them as $\widetilde{c}, \undertilde{c}$, such that  $\wt{c} \setminus \Delta = \pi^A \circ \fp \circ \check{c}(c \setminus \Lambda)$ and $\ut{c} \setminus \Delta = \pi^A \circ \fp \circ \chi^B(0,0,1)\check{c}(c \setminus \Lambda).$
Define $\ddot{\wt{c}}: \wt{c}\setminus \Delta \ra \wRDA$ as 
$\ddot{\widetilde{c}} := \wt{\sF} \circ \check{c}  \circ q_{br}|_{\wt{c} \setminus \Delta};$
similarly $\ddot{\ut{c}} : \ut{c}\setminus \Delta \ra \wRDA $ is defined by 
$\ddot{\widetilde{c}} := \wt{\sF} \circ \chi^B(0,0,1) \check{c} \circ q_{br}|_{\wt{c} \setminus \Delta}$
where $\wt{\sF}: \wRDAz \ra \wRDA$ is the unique map induced by the inclusion $\sF: \RDAz \ra \RDA.$
It is easy to check that these are indeed bigradings of their respective curves.
     On the other hand,  if $c$ contains $0$ as one of its endpoints,we define $\widetilde{\undertilde{c}} := \wt{ c \setminus \{0\}} \amalg \{0\} \amalg \ut{c \setminus \{0\}} $, which is just a single connected component. 
     Furthermore, the $\ddot{\wuc}$  is defined to be the unique continuous extension of  $\ddot{\widetilde{c \setminus \{0\}}} \amalg \ddot{\ut{c \setminus \{0\}}}$, which is again an easy verification that it is a bigrading of $\wuc.$

	 In total, we define the map $\m: \check{\fC} \ra \ddot{\undertilde{\wt{\fC}}}$  as follows: for a trigraded curve $\left(c, \check{c} \right)$ in $\wRDAz,$
    \[ \m(\left(c, \check{c} \right))  :=  \begin{cases} 
      (\wuc, \ddot{\tilde{\undertilde{c}}} ),  & \text{ if $c $ has $\{0\}$ as one of its endpoints;} \\
      (\wt{c}, \ddot{\tilde{c}}) \amalg (\undertilde{c},\ddot{\undertilde{c}}), & \text{ otherwise.}\\
       
   \end{cases}
\]
  Due to the isotopy lifting property of the space, $\m$ is well-defined on the isotopy classes of trigraded curves.

Recall the natural induced action of $\cA(B_n) \cong$ MCG$(\DB, \{0\} )$ on $\check{\fC}$ given in the paragraph before \cref{freeact}.
 Since $\mathcal{A}(A_{2n-1}) \cong $ MCG$( \DA )$ acts on $\ddot{\undertilde{\wt{\fC}}} $,  there exists an induced action of $\cA(B_n) \cong$ MCG$_p (\DA )$ on bigraded curves.
\begin{proposition} \label{topological equivariant}

The map 
$\m : \check{\fC} \ra  \ddot{\undertilde{\wt{\fC}}}$ from the isotopy classes of trigraded curves in $\wRDAz$  to  the isotopy classes of bigraded multicurves in $\wRDA$
is $\cA({B_{n}})$-equivariant.
\begin{center}
\begin{tikzpicture} [scale=0.8]
\node (tbB) at (-2.5,1.5) 
	{$\mathcal{A}(B_n)$};

\node (tbA) at (10,1.5) 
	{${\cA}(B_n) \hra \mathcal{A}(A_{2n-1})$};

\node[align=center] (cB) at (0,0) 
	{Isotopy classes $\check{\fC}$ of\\ trigraded curves  in $\DB$};
\node[align=center] (cA) at (6.5,0) 
	{Isotopy classes $\ddot{\undertilde{\wt{\fC}}}$ of\\ bigraded multicurves  in $\DA$};

\coordinate (tbB') at ($(tbB.east) + (-1,-1)$);

\coordinate (tbA') at ($(tbA.west) + (1.5,-.95)$);

\draw [->,shorten >=-1.5pt, dashed] (tbB') arc (245:-70:2.5ex);

\draw [->, shorten >=-1.5pt, dashed] (tbA') arc (-65:250:2.5ex);

\draw[->] (cB) -- (cA) node[midway,above]{$\mathfrak{m}$}; 

\end{tikzpicture}
\end{center}
\end{proposition} 
\begin{proof} This follows directly from the definition $\mathfrak{m}$ and the actions.
\end{proof}

\subsection{Bigraded intersection number and normal form of bigraded curves in $\wRDA$}

The local index for bigraded curve in $\wRDA$ is defined in the same spirit as trigraded curve $\wRDAz.$ 
 For a more explanation, we refer the reader to \cite[Section 3d]{KhoSei}. 

To remind the reader, for two bigraded curves $\ddot{c}_0, \ddot{c}_1$ that do not intersect at $\partial \DA,$ the bigraded intersection number is defined by 
\begin{equation} \label{bigrint}
I^{bigr} (\ddot{c}_0, \ddot{c}_1) = 
  (1+q_1^{-1}q_2) \left( \ \sum_{z \in (c_0 \cap c_1') \backslash \Delta} q_1^ {\mu_1(z)} q_2^ {\mu_2(z)} \ \right) 
 + \left( \ \sum_{z \in (c_0 \cap c_1') \cap \Delta } q_1^ {\mu_1(z)} q_2^ {\mu_2(z)} \ \right)  .
\end{equation}
We extend the definition of bigraded intersection number of bigraded curves to bigraded multicurves by adding up the bigraded intersection numbers of each pair of bigraded curves.

 To talk about the normal form of bigraded  curves on $\DA,$
we need to fix a set of basic bigraded curves.
To do so, first recall the set basic trigraded curves $(b_j, \check{b}_j)$ and $(d_j, \check{d}_j)$ as defined in \cref{normal DB} (see the paragraph after \cref{B n-string}).
Denote this set of basic trigraded curves as $\check{\fB}$.
Consider,  for each $(c,\check{c}) \in \check{\fB}$, the lifts 
$\m(c, \check{c})  :=  \begin{cases} 
      (\wuc, \ddot{\tilde{\undertilde{c}}} )  & \text{ if $(c, \check{c}) = (b_1, \check{b}_1) $;} \\
      (\wt{c}, \ddot{\tilde{c}}) \amalg (\undertilde{c},\ddot{\undertilde{c}}) & \text{ otherwise,}\\
   \end{cases} $
	where $\wt{c}$ is chosen to have all of its points in positive real parts;
 so $\ut{c}$ has all its points in negative real parts. 
%For all basic trigraded curves $(c, \check{c}) \in \check{\fB} \setminus \{(b_1, \check{b}_1)\}$, at least one of $\wt{c}$ or $\undertilde{c}$ will have all its points in with positive real parts.
%Note that by \cref{basic trigrading condition}, if one of the $(c,\check{c}) \in \check{\fB}\setminus \{(b_1, \check{b}_1) \}$ has all its points in $\wt{c}$ (resp. $\undertilde{c}$) with positive real part, then all other $(c',\check{c'}) \in \check{\fB} \setminus \{(b_1, \check{b}_1) \}$ will also have all its points in $\wt{c'}$ (resp. $\undertilde{c'}$) with positive real part.
%Without lost of generality, consider when $\wt{c}$ have positive real part for all $(c,\check{c}) \in \check{\fB} \setminus \{(b_1, \check{b}_1) \}$. 
We shall fix the set of basic trigraded curves $(\theta_j, \ddot{\theta}_j)$, $(\varrho, \ddot{\varrho}_j)$ as follows:
	\begin{itemize}
	\item Choose $(\theta_n, \ddot{\theta}_n) := (\wt{d}_1, \ddot{\wt{d}_1} ) ;$
	\item Choose $(\theta_{n+j-1}, \ddot{\theta}_{n+j-1}) := (\wt{d}_j, \ddot{\wt{d}_j} ) $ and $(\theta_{n-j+1}, \ddot{\theta}_{n-j+1}) := (\ut{d_j}, \ddot{\ut{d_j}} )  $ for $2 \leq j \leq n;$
	\item Choose $( \varrho_{n+j-1}, \ddot{\varrho}_{n+j-1}) := (\wt{b}_j, \ddot{\wt{b}}_j ) $ and $(\varrho_{n-j+1}, \ddot{\varrho}_{n-j+1}) := (\ut{b_j}, \ut{\ddot{b}_j} )  $ for $2 \leq j \leq n;$
	\item Choose $( \varrho_{n}, \ddot{\varrho}_{n}) := (\ut{\wt{b}_1}, \ut{\ddot{\wt{b}}_1}).$
	\end{itemize}
The following figure illustrates the basic curves $\theta_j$ and $\varrho_j$ chosen:
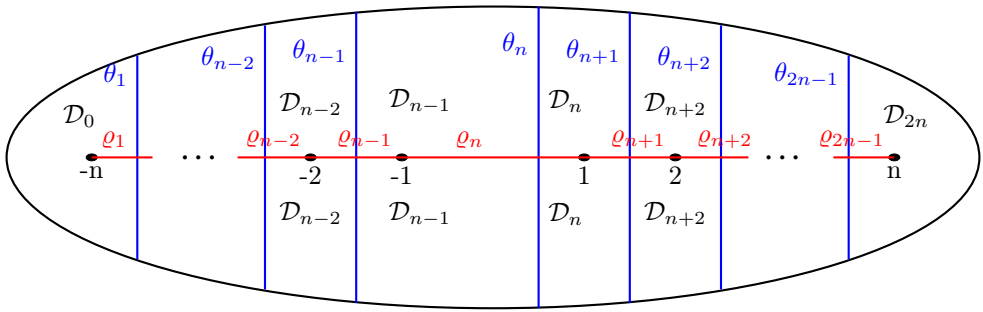
\begin{figure}[H]  
\centering
\begin{tikzpicture}  [xscale=1.6, yscale=1]
\draw[thick] (0,0) ellipse (4cm and 2cm);
\draw[fill] (-3.3,0) circle [radius=0.045];
\node  at (-2.4,0) {$ \boldsymbol{\cdots}$};
\draw[fill] (-1.5,0) circle [radius=0.045];
\draw[fill] (-0.75,0) circle [radius=0.045];
\draw[fill] (.75,0) circle [radius=0.045];
\draw[fill] (1.5,0) circle [radius=0.045];
\node  at (2.4,0) {$ \boldsymbol{\cdots}$};
\draw[fill] (3.3,0) circle [radius=0.045];

\node [below] at (-3.3,0) {-n};
\node  [below] at (-1.5,0) {-2} ;
\node [below] at (-0.75,0) {-1};
\node [below] at (.75,0) {1};
\node  [below] at (1.5,0) {2} ;
\node [below] at (3.3,0) {n};

%vertical line
\draw [thick,blue] (-2.925, 1.35) -- (-2.925, -1.35);
\draw [thick,blue] (-1.875, 1.75) -- (-1.875, -1.75);
\draw [thick,blue](-1.125, 1.93) -- (-1.125, -1.93);
\draw [thick, blue] (0.375, 2) -- (0.375, -2);
\draw [thick,blue](1.125, 1.93) -- (1.125, -1.93);
\draw [thick,blue] (1.875, 1.75) -- (1.875, -1.75);
\draw [thick,blue] (2.925, 1.35) -- (2.925, -1.35);

\node at (-3.4,0.55) {$\cD_0$};
\node at (3.4,0.55) {$\cD_{2n}$};
\node at (-.6,0.75) {$\cD_{n-1}$};
\node at (-.6,-0.75) {$\cD_{n-1}$};
\node at (.6,0.75) {$\cD_{n}$};
\node at (.6,-0.75) {$\cD_{n}$};
\node at (1.5,0.7) {$\cD_{n+2}$};
\node at (1.5,-0.75) {$\cD_{n+2}$};
\node at (-1.5,0.7) {$\cD_{n-2}$};
\node at (-1.5,-0.75) {$\cD_{n-2}$};

\node[left,blue] at (-2.925, 1.1) {$\theta_1$};
\node[left,blue] at (-1.875, 1.3) {$\theta_{n-2}$};
\node[left,blue] at (-1.125, 1.4) {$\theta_{n-1}$};
\node[left,blue] at (0.375, 1.5) {$\theta_n$};
\node[left,blue] at (1.125, 1.4) {$\theta_{n+1}$};
\node[left,blue] at (1.875, 1.3) {$\theta_{n+2}$};
\node[left,blue] at (2.925, 1.1) {$\theta_{2n-1}$};

%Horizontal line
\draw [thick,red] (-3.3,0)--(-2.8,0);
\draw [thick,red] (-1.5,0)--(-2.1,0);
\draw [thick,red] (-1.5,0)--(-0.75,0);
\draw [thick,red](-0.75,0) -- (.75,0);
\draw [thick,red] (3.3,0)--(2.8,0);
\draw [thick,red] (1.5,0)--(2.1,0);
\draw [thick,red] (1.5,0)--(.75,0);

\node[above right,red] at (-3.3,0) {$\varrho_1$};
\node[above left,red] at (-1.5,0) {$\varrho_{n-2}$};
\node[above left,red] at (-0.75,0) {$\varrho_{n-1}$};
\node[above right,red] at (1.6,0) {$\varrho_{n+2}$};
\node[above left,red] at (0,0) {$\varrho_{n}$};
\node[above left,red] at (1.5,0) {$\varrho_{n+1}$};
\node[above left,red] at (3.3,0) {$\varrho_{2n-1}$};

\end{tikzpicture}

\caption{The basic curves $\theta_i$ and $\varrho_i$ in the aligned configuration with regions $\cD_i$ for $\DA.$ } \label{NormalformA}

\end{figure}
\begin{lemma}

The bigradings we choose for the set of basic curves in $\DA$ satisfy the following by properties:
	\begin{align}
	I^{bigr}(\ddot{\theta}_j, \ddot{\varrho}_j) &= 1 + q_1^{-1}q_2 & \text{ for $ 1 \leq j \leq 2n-1;$}\\
	I^{bigr}(\ddot{\varrho}_{j}, \ddot{\varrho}_{j+1}) &= 1   & \text{ for $ n \leq j \leq 2n-2;$} \\
	I^{bigr}(\ddot{\varrho}_{j}, \ddot{\varrho}_{j-1}) &=1 & \text{ for $ 2 \leq j \leq n$.}
	\end{align}

\end{lemma}

\begin{proof}
This follows immediately from the construction.
\end{proof}

\begin{remark} \label{basic curves A}
Note that the set of basic curves chosen are the same as in \cite{KhoSei}, but the bigradings of these curves are different.
In particular, (1.14) in \cite{KhoSei} is replaced with $I^{bigr}(\ddot{\varrho}_{j}, \ddot{\varrho}_{j-1}) =1 \text{ for $ 2 \leq j \leq n$.}$
\end{remark}

We define crossings, essential segments, $j$-strings and bigraded $j$-strings in a similar fashion to the trigraded case (see \cref{local index function}, after \cref{injBA}).
In particular, given a $j$-crossing $x$ of a bigraded curve $\ddot{c}$, we fix $x_0 := j$ and $(x_1, x_2)$ is the local index $(\mu_1, \mu_2)$ of the crossing $y$.
	We can extend these notion that of multicurves and bigraded multicurves.
    
    Suppose $c_0$ and $c_1$ are two curves in $\DB$ intersecting at $z \in \DB$.
    If $z = 0,$ we require $c_0 \not\simeq c_1.$
    Then, their preimage  $q_{br}^{-1}(c_0)$ and $q_{br}^{-1}(c_1)$  in $\DA$ under the map $q_{br}: \DA \ra \DB$  would also intersect minimally. 
    However,  if $c_0 \cap c_1 \cap \{0\} \neq \emptyset,$ and $c_0 \simeq c_1$, they won't intersect minimally as illustrated below:

\begin{figure}[H]
\centering
\begin{tikzpicture}[scale = 0.6]
\draw (-3,0) circle (2cm);
\draw[fill] (-4.5,0) circle [radius=0.07];
\draw[fill] (-1.5,0) circle [radius=0.07];
\draw [thick, blue] (-4.5,0) arc (180:360:.75);
\draw [thick, blue] (-1.5,0) arc (0:180:.75);
\draw [thick, red] (-4.5, 0) -- (-1.5,0);

\draw (3,0) circle (2cm);
\node at (0,0) {$\simeq$};

\draw[fill] (4.5,0) circle [radius=0.07];
\draw[fill] (1.5,0) circle [radius=0.07];
\draw [thick, blue] (4.5,0) arc (0:180:1.5);
\draw [thick, red] (4.5, 0) -- (1.5,0);

\node [above] at (-2.25,.7) {{\scriptsize $q_{br}^{-1}(c_1)$}};
\node [above] at (-3.6,0) {{{\scriptsize $q_{br}^{-1}(c_0)$}}};
\node [below] at (-1.5,0) {{{\scriptsize $j$}}};
\node [below] at (-4.5,0) {{{\scriptsize $-j$}}};
\node [below] at (3,1.4) {{{\scriptsize $(q_{br}^{-1}(c_1))'$}}};
\node [below] at (3,0) {{{\scriptsize $q_{br}^{-1}(c_0)$}}};
\node [below] at (1.5,0) {{{\scriptsize $-j$}}};
\node [below] at (4.5,0) {{{\scriptsize $j$}}};

\end{tikzpicture}
\caption{{\small The preimages $q_{br}^{-1}(c_0), q_{br}^{-1}(c_1)$ in $\DA$.}} 
\end{figure}
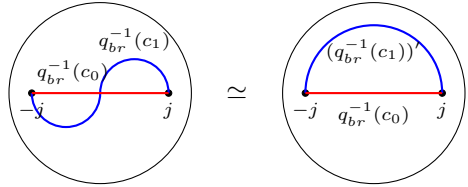
Using this fact, we obtain the following proposition:
\begin{proposition} \label{tribilocin}
Let $\check{c}_0$ and $\check{c}_1$ be two trigraded curves intersecting at $z \in \DB$, with local index $\mu^{trigr}(\check{c}_0, \check{c}_1, z) = (r_1, r_2, r_3)$. If $c_0 \cap c_1 \cap \{0\} \neq \emptyset,$ we require $c_0 \not\simeq c_1$.
	If $z \neq 0$, further suppose that $m({c_0 , \check{c}_0}) = (\wt{c_0}, \check{\wt{c}}_0) \amalg  (\ut{c_0}, \ut{\check{c}_0})$ and  $m({c_1 , \check{c}_1}) = (\wt{c_1}, \check{\wt{c_1}}) \amalg  (\ut{c_1}, \ut{\check{c}_1})$
	such that $\wt{c}_0 \cap \wt{c}_1 = \wt{z}$ and $\ut{c_0} \cap \ut{c_1} = \ut{z}.$ 
Then
$$
\begin{cases}
\mu^{bigr}(\ddot{\wt{c}}_0, \ddot{\wt{c}}_1, \wt{z}) = (r_1, r_2) 
= \mu^{bigr}({\undertilde{\ddot{c}_0}},{\undertilde {\ddot{c}_1}}, \undertilde{z}) &\text{for } z \neq 0; \\
\mu^{bigr}(\ddot{\ut{\wt{c_0}}}, \ddot{\ut{\wt{c_1}}}, 0  ) = (r_1, r_2) &\text{for } z=0.
\end{cases}
$$
\end{proposition}

Furthermore, this proposition allows us to relate trigraded intersection numbers and bigraded intersection number in the following way. 
    \begin{corollary}For any trigraded curves $(c_0, \check{c}_0),(c_1, \check{c}_1)$ in $\wRDAz,$	
	   \[ I^{trigr}(\check{c}_0,\check{c}_1) |_{q_3=1}  =   I^{bigr}\left(\m({\check{c}_0}),\m({\check{c}_1}) \right).\]
In particular,	
	   $ \frac{1}{2} I^{trigr}(\check{c}_0,\check{c}_1) |_{q_1 = q_2 = q_3=1}  = I(\m(c_0), \m(c_1)),
	   $
i.e. $ \frac{1}{2} I^{trigr}(\check{c}_0,\check{c}_1) |_{q_1 = q_2 = q_3=1}$ counts the geometric intersection number of the lift of $c_0$ and $c_1$ in $\DA$ under the map $\m.$
\end{corollary} 
\begin{proof}
The case when $c_0 \not\simeq c_1$ in $\DB$ or when at least one of $c_0$ and $c_1$ does not have its endpoint at $\{0\}$ follows directly from \cref{tribilocin}.
 The other case follows from a direct computation.
 The last statement relating trigraded intersection number and geometric intersection number follows from the property of bigraded intersection number (see \cite[pg.26, property (B1)]{KhoSei}).
\begin{comment}   
    the definition of the trigraded intersection number, and the definition of the bigraded intersection number:
   
   \[ I^{trigr}(\check{c}_0,\check{c}_1) |_{q_3=1}  =  \begin{cases} 
       I^{bigr} \left( \ddot{\ut{\wt{c_0}}},\ddot{\ut{\wt{c_1}}} \right) & \text{ if $c_0 $ and  $ c_1 $ has intersection at $\{0\}$,} \\
    I^{bigr} \left( \ddot{\ut{c_0}},\ddot{\ut{c_1}} \right) +  I^{bigr} \left( \ddot{\wt{c_0}},\ddot{\wt{c_1}} \right) & \text{ otherwise.}\\       
   \end{cases}
\]
\end{comment}
\end{proof}

We shall abuse notation and define $\m$ to also lift crossings of a trigraded curve $(c,\check{c})$ to crossings of the bigraded multicurves $\m((c,\check{c})) = (\wt{c}, \ddot{\wt{c}})\amalg (\ut{c}, \ddot{\ut{c}}) .$
    Suppose $z$ is a $j$-crossing of $c$, for $j >1$. 
    Then, $q_{br}^{-1}(z) = \{ \wt{z}, \ut{z} \}$ where $\wt{z} \in \wt{c}$ and $\ut{z} \in \ut{c}.$
    If $z$ is a $1$-crossing of $c$,  then we also have $q_{br}^{-1} = \{ \wt{z}, \ut{z} \}$; in this case we shall pick $\wt{\ut{z}}$ to be the unique element in $ \{ \wt{z}, \ut{z} \} \cap \theta_n.$
    So, if $z$ is a $j$-crossing of $c$ with $\mu^{trigr}(\check{d}_k, \check{c}, z)=(r_1,r_2,r_3)$,
    We define $\m(z) = \{ \wt{z} , \ut{z} \}$, both $\wt{z}$ and $\ut{z}$ with local index $(r_1, r_2)$ for $j >1$, and similarly $\m(z) = \{\wt{\ut{z}}\}$ for $j =1$, with local index $(r_1, r_2)$ by \cref{tribilocin}. 
	Similarly, if $\check{h}$ is a connected subset of $\check{c}$ together with trigrading given by local indices of crossings of $\check{h}$ induced from crossings of $\check{c}$, we define $\m(\check{h})$ to consist of $q_{br}^{-1}(\check{h})$, with bigradings given by local indices of crossings of $q_{br}^{-1}(\check{h})$ induced from crossings of $\m(\check{c})$.

%\newpage

%SECTION 2

\section{Type $A_n$ and Type $B_n$ Zigzag Algebras}\label{define zigzag}

In this section, we recall the construction of type $A_{2n-1}$ zigzag algebra $\Aa_{2n-1}$ (with slight change in gradings) as given in \cite{KhoSei} and recall the $\cA(A_{2n-1})$ action on $\Kom^b(\Aa_{2n-1}$-$\text{p$_{r}$g$_{r}$mod})$, the homotopy category of complexes of projective graded modules over $\Aa_{2n-1}$. 
We then construct the type $B_n$ zigzag algebra $\Ba_n$, following a similar construction, and show that $\mathcal{A}(B_n)$ acts on $\Kom^b(\Ba_n$-$\text{p$_{r}$g$_{r}$mod})$.

\subsection{Type $A_{2n-1}$ zigzag algebra $\Aa_{2n-1}$} \label{A zigzag} 
Consider the following quiver $\Gamma_{2n-1}$:
\begin{figure}[H]
\centering
\begin{tikzcd}[column sep = 1.5cm]
1				\arrow[r,bend left,"1|2"]  														 &		 
2 			\arrow[l,bend left,"2|1"] \arrow[r,bend left, "2|3"] 
				&
3 			\arrow[l,bend left,"3|2"] \arrow[r,bend left, "3|4"] 
				 &
\cdots \arrow[l,bend left,"4|3"] \arrow[r,bend left, "2n-2|2n-1"] &
2n-1 		.	\arrow[l,bend left,"2n-1|2n-2"]				
\end{tikzcd}
\caption{{\small The quiver $\Gamma_{2n-1}$.}}
\label{B quiver}
\end{figure}
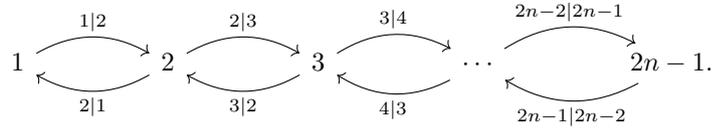
We can take its path algebra $\C \Gamma_{2n-1}$ over $\C$;
$\C \Gamma_{n}$ is the $\C$-vector space spanned by the set of all paths in $\Gamma_{n}$, with multiplication given by concatenation of paths (the multiplication is zero if the endpoints do not agree).
Note that the constant path of each vertex $j$ is denoted by $e_j$. 

The grading we would put on $\Gamma_n$ is slightly different to \cite{KhoSei}; we set 
 \begin{itemize}
 \item the degree of $(j|j-1)$ is $1$  for $j > n$ and $0$ for  $j \leq n$,
 \item the degree of $(j|j+1)$ is $1$ for  $j < n$ and $0 $ for $j \geq n,$ and
 \item the degree of $e_j = 0$ for all $j.$
 \end{itemize}
 In this way, this path algebra is graded with the shift denoted by $\{-\}$ and unital with a family of pairwise orthogonal primitive central idempotent $e_j$ summing up to the unit element.

	Let $\Aa_n$ be the quotient of the path algebra of the quiver $\Gamma_{n}$ by the relations:
	\begin{align*}
	(j|j+1|j) &= (j| j-1|j),  & \text{ for  $j = -n, -n+1, \cdots, n-1, n,$} \\
		(j|j+1|j+2) &= 0 = (j+2| j+1|j),  & \text{ for $ j = -n, -n+1, \cdots, n-1, n.$} 
	\end{align*}
 It is easy to see that these relations are homogenous with respect to the above grading, so that $\Aa_{2n-1}$ is a graded algebra. 
As a $\C$-vector space, it has dimension $8n-6$ with the following basis 
$$\{ e_1 , \ldots, e_{2n-1},  (1|2), \ldots, (2n-2 | 2n-1), (2|1), \ldots (2n-1 | 2n-2), (1|2|1), \ldots, (2n-1|2n-2|2n-1) \}.$$

It follows that the indecomposable projective $\Aa_{2n-1}$-modules are $P^A_j := \Aa_{2n-1}e_j$.
Recall the following results in \cite{KhoSei}:
\begin{theorem}
For each $j$, the following complex of $(\Aa_{2n-1}, \Aa_{2n-1})$-bimodule
\[
\mathcal{R}_j := 0 \ra P^A_j \otimes_\C {}_jP^A \xra{\beta_j} \Aa_{2n-1} \ra 0
\]
with $\Aa_{2n-1}$ in cohomological degree 0 is invertible in $\Kom^b((\Aa_{2n-1}, \Aa_{2n-1})\text{-bimod})$ and satisfies the following relations:
\begin{align*}
\cR_j \otimes \cR_k &\cong \cR_k \otimes \cR_j, \text{ for } |j-k|>1; \\
\cR_j \otimes \cR_{j+1} \otimes \cR_j &\cong \cR_{j+1} \otimes \cR_j \otimes \cR_{j+1}.
\end{align*}
\end{theorem}
\begin{proof}
See \cite[Proposition 2.4 and Theorem 2.5]{KhoSei}.
\end{proof}
\begin{proposition}
There is a (weak) $\cA(A_{2n-1})$-action on $\Kom^b(\Aa_{2n-1}$-$\text{p$_{r}$g$_{r}$mod})$, where each standard generator $\sigma_j^A$ of $\cA(A_{2n-1})$ acts on a complex $M \in \Kom^b(\Aa_{2n-1}$-$\text{p$_{r}$g$_{r}$mod})$ via $\mathcal{R}_j$:
\[
\sigma^A_j(M) := \mathcal{R}_j\otimes_{\Aa_{2n-1}}M.
\]
\end{proposition}
\begin{proof}
See \cite[Proposition 2.7]{KhoSei}.
\end{proof}
We will abuse notation and use $\sigma^A_j$ in place of $\cR_j$ whenever the context is clear.

\subsection{Type $B_n$ zigzag algebra $\Ba_n$} \label{B zigzag}
Consider the following quiver $Q_n$:
\begin{figure}[H]
\centering
\begin{tikzcd}[column sep = 1.5cm]
1				\arrow[r,bend left,"1|2"]  														 &		 
2 			\arrow[l,bend left,"2|1"] \arrow[r,bend left, "2|3"] 
				\arrow[color=blue,out=70,in=110,loop,swap,"ie_2"] &
3 			\arrow[l,bend left,"3|2"] \arrow[r,bend left, "3|4"] 
				\arrow[color=blue,out=70,in=110,loop,swap,"ie_3"] &
\cdots \arrow[l,bend left,"4|3"] \arrow[r,bend left, "n-1|n"] &
n 		.	\arrow[l,bend left,"n|n-1"]
				\arrow[color=blue,out=70,in=110,loop,swap,"ie_n"]
\end{tikzcd}
\caption{{\small The quiver $Q_n$.}}
\label{B quiver}
\end{figure}
%We will use a slightly different path length grading.
%All (oriented) edges $\xi$ in black of \cref{B quiver} have length 1, whereas edges {\color{blue}$\xi'$} in blue have length 0.
%and each $e_j$ also has length 0.
%The length of an arbitrary path $\xi_1...\xi_m$ is then the sum of the length of each $\xi_k$.
%The algebra $\R Q_n$ is then a graded algebra under this path length grading.
%
Take its path algebra $\R Q_n$ over $\R$ and consider the two gradings on $\R Q_n$ given as follows:
	\begin{enumerate} [(i)]
	\item the first grading is defined following the convention in  \cite{KhoSei} where we set \begin{itemize}
	\item 	the degree of $(j+1|j)$ to be $1$ for all $j$, and
	\item   the degree of $e_j$ and of $(j|j+1)$ to be $0$ for all $j$.
\end{itemize}	 
	\item the second grading is a $\Z/2\Z$-grading defined by setting 
	\begin{itemize}
	\item  the degree of $ie_j$ (all blue paths in \cref{B quiver}) as 1 for all $j$, and
	\item the degree of all other paths in \cref{B quiver} and the constant paths as zero. 
	\end{itemize}
	\end{enumerate}

\noindent We denote  a shift in the first grading by  $\{-\}$ and a shift in the second grading by $\< - \>.$

We are now ready to define the zigzag algebra of type $B_n$:
\begin{definition}
The zigzag path algebra of $B_n$, denoted by $\Ba_n$, is the quotient algebra of the path algebra $\R Q_n$ modulo the usual zigzag relations given by
\begin{align}
(j|j-1)(j-1|j) &= (j|j+1)(j+1|j) \qquad (=: X_j);\\
(j-1|j)(j|j+1) = & 0 = (j+1|j)(j|j-1);
\end{align}
 ~ for $2\leq j \leq n-1$, in addition to the relations
\begin{align}
(ie_j)(ie_j) &= -e_j, \qquad \text{for } j \geq 2 \label{imaginary};\\
(ie_{j-1})(j-1|j) &= (j-1|j)(ie_j), \qquad \text{for } j\geq 3; \label{complex symmetry 1}\\
(ie_{j})(j|j-1) &= (j|j-1)(ie_{j-1}), \qquad \text{for } j\geq 3; \label{complex symmetry 2}\\
(1|2)(ie_2)(2|1) &= 0, \\
(ie_2) X_2 &= X_2 (ie_2).
\end{align}
\end{definition}
Since the relations are all homogeneous with respect to the given gradings, $\Ba_n$ is also a bigraded algebra.
As a $\R$-vector space, $\Ba_n$ has dimension $8  n-6$, with basis $\{ e_1 , \ldots, e_{n}, ie_2 , \ldots, ie_{n},  
(1|2), \ldots, (n-1 | n), (2|1), \ldots (n | n-1),  (ie_2)(2|1), (1|2)(ie_2), (ie_2)(2|3), \ldots, (ie_{n-1})(n-1|n), (3|2)(ie_2), \ldots, (n|n-1)(ie_{n-1}), \newline (1|2|1), \ldots, (2n-1|2n-2|2n-1), (ie_2)(2|1|2), \ldots, (ie_n)(n|n-1|n) \}$.

The indecomposable (left) projective $\Ba_n$-modules are given by $P^B_j := \Ba_n e_j$.
For $j=1$, $P^B_j$ is naturally a $(\Ba_n, \R)$-bimodule; there is a natural left $\Ba_n$-action given by multiplication of the algebra and the right $\R$-action induced by the natural left $\R$-action.
But for $j\geq 2$, we shall endow $P^B_j$ with a right $\C$-action.
Note that \cref{imaginary} is the same relation satisfied by the complex imaginary number $i$.
We define a right $\C$-action on $P^B_j$ by $p * (a+ib) = ap + bp(ie_j)$ for $p \in P^B_j, a+ib\in \C$.
Further note that this right action restricted to $\R$ agrees with the natural right (and left) $\R$-action.
This makes $P^B_j$ into a $(\Ba_n,\C)$-bimodule for $j\geq 2$.
Dually, we shall define ${}_jP^B := e_j\Ba_n$, where we similarly consider it as a ($\R,\Ba_n)$-bimodule for $j=1$ and as a $(\C,\Ba_n)$-bimodule for $j\geq 2$.

It is easy to check that we have the following isomorphisms of $\Z$-graded bimodules:
\begin{proposition} \label{Bcat}
\[
  {}_jP^B_k := {}_jP^B\otimes_{\Ba_n}P^B_k \cong 
  \begin{cases}
  		\vspace{1mm}
  		\ _{\C}\C_{\C} & \text{as } (\C,\C)\text{-$g_r$bimod, for } j,k \in \{2,\hdots, n\} \text{ and }  k-j=1;\\
  		\vspace{1mm}
  		\ _{\C}\C_{\C}\{1\} & \text{as } (\C,\C)\text{-$g_r$bimod, for } j,k \in \{2,\hdots, n\} \text{ and }  j-k=1;\\
  		\vspace{1mm}
  		\ _{\C}\C_{\C} \oplus \ _{\C}\C_{\C}\{1\} & \text{as } (\C,\C)\text{-$g_r$bimod, for } j=k=2,3,\hdots,n; \\
  		\vspace{1mm}
        \ _{\R}\C_{\C} & \text{as } (\R,\C)\text{-$g_r$bimod, for } j=1 \text{ and } k=2; \\
        \vspace{1mm}
        \ _{\C}\C_{\R}\{1\} & \text{as } (\C,\R)\text{-$g_r$bimod, for } j=2 \text{ and } k=1; \\ 
        \vspace{1mm}
        \ _{\R}\R_{\R} \oplus \ _{\R}\R_{\R}\{1\} & \text{as } (\R,\R)\text{-$g_r$bimod, for } j=k=1. \\    
  \end{cases}
\]
\label{bimodule isomorphism}
\end{proposition}
\begin{remark}
Note that all the graded bimodules in \cref{bimodule isomorphism} can be restricted to a $(\R, \R)$-bimodule.
By identifying ${}_\R \C_\R \cong \R\oplus \R\<1\>$ as $\Z/2\Z$-graded $(\R, \R)$-bimodules, the bimodules in \cref{bimodule isomorphism} restricted to just the $\R$ actions are also isomorphic as bigraded ($\Z$ and $\Z/2\Z$) $(\R, \R)$-bimodule.
For example, ${}_1 P_2^B$ as a $(\R, \R)$-bimodule is generated by $(1|2)$ and $(1|2)i$, so it is isomorphic to $\R \oplus \R\<1\> \cong {}_\R \C_\R$.
\end{remark}

\begin{lemma}
Denote $k_j := \R$ when $j = 1$ and $k_j := \C$ when $j \geq 2$.
%\begin{cases} 
%\R, & \text{ for } j = 1; \\
%\C, & \text{ for }j \geq 2.
%\end{cases}
The maps \[\beta_j: P^B_j \otimes_{k_j} {}_jP^B  \to \Ba_n \text{ and } \gamma_j: \Ba_n  \to  P^B_j \otimes_{k_j} {}_jP^B  \{-1\}\] defined by:
\begin{align*} 
\beta_j(x\otimes y) &:= xy, \\
\gamma_j(1) &:= 
\begin{cases}
X_j \otimes e_j + e_j \otimes X_j + (j+1|j) \otimes (j|j+1) \\
 \hspace{8mm} + (-ie_{j+1})(j+1|j) \otimes (j|j+1)(ie_{j+1}), &\text{for } j=1;\\
X_j \otimes e_j + e_j \otimes X_j + (j-1|j) \otimes (j|j-1) + (j+1|j) \otimes (j|j+1), &\text{for } 1<j < n; \\
X_j \otimes e_j + e_j \otimes X_j + (j-1|j) \otimes (j|j-1), &\text{for } j = n,
\end{cases}
\end{align*}
are $(\Ba_n,\Ba_n)$-bimodule maps.
\end{lemma}

\begin{proof}
It is obvious that $\beta_j$ are for all $j$. 
The fact that $\gamma_j$ is a $(\Ba_n,\Ba_n)$-bimodule map also follows from a tedious check on each basis elements, which we shall omit and leave it to the reader.
%in particular, the condition that $\gamma_i(xy) = x\gamma_j(y) = \gamma_j(x)y$ for all $x,y \in \Ba_n$ can be checked on a basis.
 \end{proof}

\begin{definition}
Define the following complexes of graded $(\Ba_n,\Ba_n)$-bimodules:
\begin{align*}
R_j &:= (0 \to P^B_j \otimes_{\F_j} {}_jP^B \xra{\beta_j} \Ba_n \to 0), \text{and} \\
R_j' &:= (0 \to \Ba_n  \xra{\gamma_j}  P^B_j \otimes_{\F_j} {}_jP^B\{-1 \} \to 0).
\end{align*}
for each $j \in \{1,2, \cdots, n\},$ with both $\Ba_n$ in cohomological degree 0, $\F_1 = \R$ and $\F_j = \C$ for $j \geq 2$.
\end{definition}

\begin{proposition} \label{standard braid relations}
There are isomorphisms in the homotopy category, $\Kom^b ((\Ba_n, \Ba_n )$-$\text{bimod})$, of complexes of projective graded $(\Ba_n,\Ba_n)$-bimodules:
\begin{align}
R_j \otimes R_j^{'} \cong & \Ba_n \cong R_j^{'} \otimes R_j; \\
R_j \otimes R_k & \cong R_k \otimes R_j, \quad \text{for } |k-j| > 1;\\
R_j \otimes R_{j+1} \otimes R_j &\cong R_{j+1} \otimes R_j \otimes R_{j+1}, \quad \text{for } j \geq 2.
\end{align}
\end{proposition}

\begin{proof} 
These relations can be verified similarly as in \cite[Theorem 2.5]{KhoSei}.
\end{proof}

\subsection{Biadjunction and Dehn twist}
To show the last type $B_n$ relation, we shall introduce a larger family of invertible complexes that will aid us in our calculation. 
The construction mirrors the notion of Dehn Twists in topology and uses the theory on biadjunctions (see \cite{KhoBiadj} for an amazing exposition on biadjunctions in terms of string diagram).
Throughout this section we will denote $\mathbb{K}_1 := \R$ and $\mathbb{K}_j := \C$ for $j \geq 2$.

\begin{definition}
Let $X\in \Kom^{b}((\Ba_n,\mathbb{K}_j)$-$bimod)$ and $X^\vee \in \Kom^{b}((\mathbb{K}_j,\Ba_n)\text{-bimod})$ such that $(X,X^\vee)$ is a biadjoint pair of objects.
We define the complex of $(\Ba_n,\Ba_n)$-bimodule
\[
\sigma_X:= \cone\left( X\otimes_{\mathbb{K}_j} X^\vee \xrightarrow{\varepsilon} \Ba_n \right)
\]
with $\varepsilon$ the counit of the adjunction $X \dashv X^\vee$.
\begin{comment}
Its inverse is given by
\[
\sigma_X^{-1}:= \cone \left( \Ba \xrightarrow{\eta} X\otimes_k X^\vee \right)
\]
with $\eta$ the unit of the adjunction $X \vdash X^\vee$.
\end{comment}
\end{definition}
One can verify from the definition of biadjunctions that $\sigma_X$ is uniquely defined up to isomorphism, i.e. if $X \cong Y$, then $\sigma_X \cong \sigma_{Y}$ .
Moreover, $\Sigma_{X} = \Sigma_{X[r]\{s\}\<t\>}$.
If $(Y, Y^\vee)$ is a biadjoint pair and $X \cong \Sigma \otimes_{\Ba_n} Y$ with $\Sigma$ an invertible object in $\Kom^b((\Ba_n, \Ba_n)\text{-mod})$, then $(X,X^\vee)$ forms a biadjoint pair with $X^\vee := Y^\vee \otimes_{\Ba_n} \Sigma^{-1}$.
Furthermore, $\sigma_X \cong \Sigma \otimes_{\Ba_n} \sigma_Y \otimes_{\Ba_n} \Sigma^{-1}$.

\begin{lemma} \label{P_i biadjoint}
The objects $P_j^B$ and ${}_{j}{P}^B$ are biadjoints to each other.
\end{lemma}
\begin{proof}
To show that the $P_j^B$ is a left adjoint to ${}_{j}{P}^B$, take the counit to be the functor induced by $P_j^B\otimes_{\mathbb{K}_j} {}_{j}{P}^B \xrightarrow{\beta_j} \Ba_n$ and the unit is instead induced by $\mathbb{K}_j \xrightarrow{\varphi} {}_{j}{P}^B\otimes_{\Ba_n} P_j^B$ where $\varphi$ is defined by $\varphi(1) = e_j\otimes_{\Ba_n} e_j$.
\begin{comment}
We get that
\begin{alignat*}{4}
P_i^B\otimes_{\mathbb{K}_j} \mathbb{K}_j 
&\xrightarrow{\id_{P_i^B}\otimes \varphi} 
&& P_i^B \otimes_{\mathbb{K}_j} {}_{i}{P}^B \otimes_\Ba P_i^B 
&& \xrightarrow{\beta_i\otimes \id_{P_i^B}} 
&& \Ba\otimes_\Ba P_i^B \\
a\otimes 1 
&\mapsto 
&& a\otimes e_i \otimes e_i 
&& \mapsto 
&&a\otimes e_i = 1\otimes a
\end{alignat*}
which after identification of $P_i^B \otimes_{\mathbb{K}_j} \mathbb{K}_j \cong P_i^B \cong \Ba\otimes_\Ba P_i^B$ gives us the identity map on $P_i^B$.
The other composition is
\begin{alignat*}{4}
\mathbb{K}_j\otimes_{\mathbb{K}_j} {}_{i}{P}^B 
&\xrightarrow{\varphi\otimes \id_{{}_{i}{P}^B}} 
&& {}_{i}{P}^B \otimes_{\Ba} P_i^B \otimes_{\mathbb{K}_j} {}_{i}{P}^B 
&& \xrightarrow{\id_{{}_{i}{P}^B}\otimes \beta_i} 
&& {}_{i}{P}^B \otimes_\Ba \Ba \\
1\otimes b
&\mapsto 
&& e_i \otimes e_i \otimes b
&& \mapsto 
&& e_i \otimes b = b \otimes 1
\end{alignat*}
which again after the identification gives us the identity map on ${}_{i}{P}^B$.
These two combined shows that the $P_i^B$ is a left adjoint to ${}_{i}{P}^B$.
\end{comment}
To show ${}_{j}{P}^B$ is a left adjoint to $P_j^B$, take the counit to be the functor induced by ${}_{j}{P}^B \otimes_{\Ba_n} P_j^B \xrightarrow{\varphi'} \mathbb{K}_j$ where $\varphi'$ is defined by $\varphi'(e_j\otimes e_j) = 0, \varphi'(X_j\otimes e_j) = 1$ (note that $X_j \otimes e_j = e_j \otimes X_j$), and the unit is instead induced by ${\Ba_n} \xrightarrow{\gamma_j} P_j^B\otimes_{\mathbb{K}_j} {}_{j}{P}^B$.
\begin{comment}
Note that 
\[
b \otimes ((i+1|i)\otimes (i|i+1) -(i+1|i)i\otimes (i|i+1)i + X_i \otimes e_i + e_i\otimes X_i) = X_i\otimes e_i \otimes b
\]
for any $b \in {}_{i}{P}^B$.
Similarly for any $a \in P_i^B$, we get
\[
\gamma_i(1)\otimes a = a \otimes e_i \otimes X_i.
\]
We now get that
\begin{alignat*}{4}
{}_{i}{P}^B \otimes_{\Ba} \Ba
&\xrightarrow{\id_{{}_{i}{P}^B}\otimes \gamma_i} 
&& {}_{i}{P}^B \otimes_\Ba P_i^B \otimes_{\mathbb{K}_j} {}_{i}{P}^B
&& \xrightarrow{\varphi' \otimes \id_{{}_{i}{P}^B}} 
&& \mathbb{K}_j \otimes_{\mathbb{K}_j} {}_{i}{P}^B \\
b \otimes 1 
&\mapsto 
&& b \otimes \gamma_i(1) \quad(= X_i\otimes e_i \otimes b)
&& \mapsto 
&& 1\otimes b
\end{alignat*}
which is the identity map on ${}_{i}{P}^B$.
The other composition is
\begin{alignat*}{4}
\Ba \otimes_{\Ba} P_i^B 
&\xrightarrow{\gamma_i \otimes \id_{P_i^B}} 
&& P_i^B \otimes_{\mathbb{K}_j} {}_{i}{P}^B \otimes_{\Ba} P_i^B 
&& \xrightarrow{\id_{P_i^B}\otimes \varphi'} 
&& P_i^B \otimes_{\mathbb{K}_j} \mathbb{K}_j \\
1\otimes a
&\mapsto 
&& \gamma_i(1) \otimes a
&& \mapsto 
&& a \otimes 1
\end{alignat*}
which again is the identity map on $P_i^B$.
So ${}_{i}{P}^B$ is also a left adjoint to $P_i^B$.
\end{comment}
It is an easy exercise to verify the conditions for both adjunctions.
\end{proof}

Using this, we shall now prove the last type $B_n$ relation required:
\begin{proposition} \label{type B relation}
We have the following isomorphism of $\Kom^b((\Ba_n,\Ba_n)-bimod)$:
\[
R_2 \otimes_{\Ba_n} R_1 \otimes_{\Ba_n} R_2 \otimes_{\Ba_n} R_1 \cong 
R_1 \otimes_{\Ba_n} R_2 \otimes_{\Ba_n} R_1 \otimes_{\Ba_n} R_2 .
\]
\end{proposition}
\begin{proof}
We shall drop the tensor product for the sake of readability:
$
 R_2 R_1 R_2 R_1 \cong R_1 R_2 R_1 R_2. 
 $
Using \cref{standard braid relations}, note that this relation is equivalent to
$
(R_2 R_1 R_2) R_1 (R_2' R_1' R_2') \cong 
R_1.
$
By \cref{P_i biadjoint}, we get $R_1 = \sigma_{P_1^B}$, giving us
$$
(R_2 R_1 R_2) \sigma_{P_1^B} (R_2' R_1' R_2') \cong 
\sigma_{P_1^B}.
$$
Thus, it is equivalent to show that
$$
\sigma_{R_2 R_1 R_2(P_1^B)} \cong \sigma_{P_1^B}.
$$
It is now sufficient to show $R_2 R_1 R_2(P_1^B)$ and $P_1^B$ are isomorphic in $\Kom^b((\Ba_n, \R)$-bimod) up to cohomological or internal gradings shifts.
We shall show this in the computation that follows.
Since the cohomological grading does not matter, we shall omit it during the computation.
\begin{align*}
R_2(P_1^B) &= 0 \ra P_2^B\otimes_\C {}_2P_1^B \ra P_1^B \ra 0 
\cong 0 \ra P_2^B \{1\} \xra{2|1} P_1^B \ra 0 \qquad (\text{by } \cref{bimodule isomorphism})
\end{align*}
\begin{align*}
R_1R_2(P_1^B) &\cong R_1 \left( 0 \ra P_2^B \{1\} \xra{2|1} P_1^B \ra 0 \right)\\
&= \text{cone } \left\{
\begin{tikzcd}[row sep = large, column sep = large, ampersand replacement=\&]
P_1^B\{1\}\otimes_\R {}_1P_2^B \ar{r}{\id\otimes (2|1)} \ar{d}{}\& 
P_1^B \otimes_\R {}_1P_1^B \ar{d}{} \\
P_2^B \{1\} \ar{r}{2|1} \& 
P_1^B
\end{tikzcd} \right\}\\
&\cong \text{cone }\left\{
\begin{tikzcd}[row sep = large, column sep = large, ampersand replacement=\&]
P_1^B\{1\}\oplus P_1^B\{1\}\<1\> 
\ar{r}
	{
	\begin{bmatrix}
	0 & 0 \\
	\id & 0
	\end{bmatrix}
} 
\ar{d}
	{
	\begin{bmatrix}
	1|2 & 1|2i
	\end{bmatrix}
}\& 
P_1^B \oplus P_1^B\{1\} 
\ar{d}{
	\begin{bmatrix}
	\id & X_1
	\end{bmatrix}
} \\
P_2^B \{1\} \ar{r}{2|1} \& 
P_1^B
\end{tikzcd} \right\} \qquad (\text{by } \cref{bimodule isomorphism})\\
&\cong 0 \ra P_1^B\{1\}\<1\> \xra{1|2i} P_2^B\{1\} \ra 0 \\
R_2R_1R_2(P_1^B) &\cong \text{cone }\left\{
\begin{tikzcd} [row sep = large, column sep = large, ampersand replacement=\&]
P_2^B\otimes_\C {}_2P_1^B\{1\} \<1\> 
\ar{r}{} 
\ar{d}{}
\&
P_2^B \otimes_\C {}_2P_2^B\{1\}
\ar{d}{} \\
P_1^B\{1\} \<1\> \ar{r}{1|2i} \&
P_2^B\{1\}
\end{tikzcd}
\right\}\\
&\cong \text{cone }\left\{
\begin{tikzcd} [row sep = large, column sep = large, ampersand replacement=\&]
P_2^B\{2\} \<1\> 
\ar{r}{
	\begin{bmatrix}
	0 \\
	\id
	\end{bmatrix}
} 
\ar{d}{2|1}
\&
P_2^B\{1\} \oplus P_2^B\{2\}\<1\>
\ar{d}{
	\begin{bmatrix}
	\id & X_2i
	\end{bmatrix}
} \\
P_1^B\{1\} \<1\> \ar{r}{1|2i} \&
P_2^B\{1\}.
\end{tikzcd}
\right\} \qquad (\text{by } \cref{bimodule isomorphism})\\
&\cong P_1^B\{1\}\<1\>
\end{align*}
\end{proof}
%
%\end{align*}
%
%
%\begin{align*}

\begin{theorem} \label{Cat B action}
We have a (weak) $\mathcal{A}(B_n)$-action on $\Kom^b(\Ba_n$-$p_r g_r mod)$, where each standard generator $\sigma^B_j$ for $j \geq 2$ of $\mathcal{A}(B_n)$ acts on a complex $M \in \Kom^b(\Ba_n$-$p_rg_rmod)$ via $R_j$, and $\sigma^B_1$ acts via $R_1 \<1\>$:
$$\sigma^B_j(M):= R_j \otimes_{\Ba_n} M, \text{ and } {\sigma^B_j}^{-1}(M):= R_j' \otimes_{\Ba_n} M,$$
\[
\sigma^B_1(M):= R_1 \<1\> \otimes_{\Ba_n} M, \text{ and } {\sigma^B_1}^{-1}(M):= R_1' \<1\> \otimes_{\Ba_n} M.
\]
\end{theorem}
\begin{proof}
This follows directly from \cref{standard braid relations} and \cref{type B relation}, where the required relations still hold with the extra third grading shift $\<1\>$ on $R_1$ and $R_1'$.
\end{proof}
\noindent From now on, we will abuse notation and use $\sigma^B_j$ and ${\sigma^B_j}^{-1}$ in place of $R_j$ and $R_j'$ (with an extra grading shift $\<1\>$ for $j=1$) respectively  whenever it is clear from the context what we mean.
\subsection{Functor realisation of the type $B$ Temperley-Lieb algebra} \label{Functorial TL}
In \cite[section 2b]{KhoSei}, Khovanov-Seidel showed that the $(\Aa_m, \Aa_m)$-bimodules $U_j := P_j^A \otimes_\C {}_j P^A$ provides a functor realisation of the type $A_m$ Temperley-Lieb algebra. 
Though they don't really care about the grading matching in that particular paper, in \cite[section 1.3]{LinkHom}, Khovanov introduced the path length grading, denoted by $(-),$ on $\Aa_m$-modules. In which case, when $U_j := P_j^A \otimes_\C {}_j P^A (-1),$ we really see the multiplication by $q$ in type $A$ Temperley-Lieb algebra over the ring $\Z[q, q^{-1}]$ become the grading shift $(1).$
We will see that we have a similar type $B_n$ analogue of this.

Recall that the type $B_n$ Temperley-Lieb algebra $TL_v(B_n)$ over $\Z[v,v^{-1}]$ can be described explictly as (see \cite[Proposition 1.3]{Green}) the algebra generated by $E_1, ..., E_n$ with relation
\begin{align*}
E_j^2 &= vE_j + v^{-1}E_j; \\
E_j E_k &= E_k E_j, \quad \text{if } |j-k| > 1; \\
E_j E_k E_j &= E_k, \quad \text{if } |j-k| = 1 \text{ and } j,k > 1; \\
E_j E_k E_j E_k &= 2E_j E_k, \quad \text{if } \{j, k\} = \{1,2\}.
\end{align*}

To match with the above, for this subsection only, we shall adopt the path length grading on our algebra $\Ba_n$ instead, where we insist that the blue paths $i e_j$ in \ref{B quiver} have length 0. We shall denote this path length grading shift by $(-)$ to avoid confusion.
Define $\cU_j := P_j^B \otimes_{\mathbb{F}_j} {}_j P^B (-1)$, where $\mathbb{F}_1 = \R$ and $\mathbb{F}_j = \C$ when $j \geq 2$.
It is easy to check that:
\begin{proposition}
The following are isomorphic as (path length graded) $(\Ba_n, \Ba_n)$-bimodules:
\begin{align*}
\cU_j^2 &\cong \cU_j(1) \oplus \cU_j(-1) \\
\cU_j \cU_k &\cong 0 \quad \text{if } |j-k| > 1, \\
\cU_j \cU_k \cU_j &\cong \cU_k \quad \text{if } |j-k| = 1 \text{ and } i,j > 1, \\
\cU_j \cU_k \cU_j \cU_k &\cong \cU_j \cU_k \oplus \cU_j \cU_k \quad \text{if } \{j, k\} = \{1,2\}.
\end{align*}
\end{proposition}
Comparing this with the relations of $TL_v(B_n)$ above, it follows that $\cU_j$ provides a (degenerate) functor realisation of the type $B$ Temperley-Lieb algebra.

\section{Relating Categorical Type $B_n$ and Type $ {A_{2n-1}}$ actions} \label{relating categorical b a action}

\begin{comment}
\subsection*{Notation}
\begin{enumerate}
\item We denote the type $A_{2n-1}$ and type $B_n$ zigzag algebra as $\Aa_{2n-1}$ and $\Ba_n$ respectively.  When it is clear what $n$ is, or the statements made do not depend on $n$, we allow ourselves to drop the subscript $n$.
%\item We denote the Artin group defined by the graph $G$ as $\mathcal{A}(G)$.
%\item The homotopy category of bounded complexes is denoted by $\Kom^b(-)$.
\item When the context is clear, we will abuse notation and use $\sigma_i^B$ in place of the complex of $(\Ba,\Ba)$-mod $R_i$.
\end{enumerate}
\end{comment}

In \cref{topology}, we have defined $\m$ that lifts isotopy classes of trigraded curves in $\D^B_{n+1}$ to isotopy classes of bigraded curves in $\D^A_{2n}$. Furthermore, we showed that the map $\mathfrak{m}$ is equivariant under the $\cA(B_n)$-action.
In this section, we shall develop the algebraic version of this story.
We will first relate our type $B_n$ zigzag algebra $\Ba_n$ to the type $A_{2n-1}$ zigzag algebra $\Aa_{2n-1}$ by showing that $\C\otimes_\R \Ba_n \cong \Aa_{2n-1}$ as graded $\C$-algebras (forgetting the $\<-\>$ grading in $\Ba_n$).
Through this, we have an injection $\Ba_n \hookrightarrow \C\otimes_\R \Ba_n \cong \Aa_{2n-1}$ as graded $\R$-algebras.
Thus, we can relate the two categories $\Kom^b(\Ba_n$-$\text{p$_{r}$g$_{r}$mod})$ and $\Kom^b(\Aa_{2n-1}$-$\text{p$_{r}$g$_{r}$mod})$  through an extension of scalar $\Aa_{2n-1} \otimes_{\Ba_n} -$. 
We end  this section by showing that the functor $\Aa_{2n-1} \otimes_{\Ba_n} -$ is $\cA(B_n)$-equivariant, which also allows us to deduce that the $\cA(B_n)$ action on $\Kom^b(\Ba_n\text{-}p_rg_rmod)$ is faithful.

\begin{proposition} \label{isomorphic algebras}
Consider the graded $\R$-algebra $\ddot{\Ba}_n$, where $\ddot{\Ba}_n$ is just $\Ba_n$ without the $\Z/2\Z$ grading $\<-\>$.
The graded $\C$-algebras $\C \otimes_{\R} \ddot{\Ba}_n $ and $\Aa_{2n-1}$ are isomorphic.
\end{proposition}
\begin{proof}
Note that as $\C$-vector space, we have the following decomposition:
\begin{align*}
\C \otimes_\R \ddot{\Ba}_n 
& \cong \bigoplus_{j=2}^{n} \C \otimes_\R \left(  {}_jP^B_j \ \oplus {}_jP^B_{j-1} \oplus  {}_{j-1}P^B_{j} \right) \oplus (\C \otimes_\R {}_1P^B_1) \\ 
& \cong 
\bigoplus_{j=2}^{n} \big(\C \otimes_\R  {}_jP^B_j \big) \oplus  \bigoplus_{j=2}^{n} \big( \C \otimes_\R {}_jP^B_{j-1} \big) \oplus \bigoplus_{j=2}^{n} \big(\C \otimes_\R {}_{j-1}P^B_{j} \big) \oplus 
(\C \otimes_\R {}_1P^B_1). 
\end{align*}

Firstly, for each $j\geq 2$, note that ${}_jP^B_j$ is itself a $\C$-algebra with unit $e_j\otimes 1$.
Moreover, after tensoring with $\C$ over $\R$, $\C \otimes_\R {}_jP^B_j$ has idempotent $\nu_j := \frac{1}{2}(1 \otimes e_j + i \otimes ie_j)$.
We shall define a $\C$-linear morphism $\Phi: \C\otimes_\R \ddot{\Ba}_n \ra \Aa_{2n-1}$ by specifying the image of the basis elements of $\C\otimes_\R \ddot{\Ba}_n$.
It is easy to see that $\Phi$ is grading preserving and we leave the routine check that $\Phi$ is an algebra morphism to the reader.
\\
For $j=1$,
\begin{align*}
  \C \otimes_\R {}_1P^B_1 &\to {}_n P^A_n 
& \\
 1 \otimes \frac{1}{2} X_1  & \mapsto X_n
\\
  1 \otimes e_1 & \mapsto e_n
\end{align*}
For $2 \leq j \leq n$,
\begin{align*}
\C \otimes_\R {}_jP^B_j &\to {}_{n-j+1}P^A_{n-j+1} \oplus   {}_{n+j-1}P^A_{n+j-1}\\
 (1 \otimes X_j ) \nu_j & \mapsto X_{n-j+1}\\
 (1 \otimes X_j )(1\otimes e_j - \nu_j) & \mapsto X_{n+j-1}\\
  \nu_j & \mapsto e_{n-j+1} \\
  (1\otimes e_j - \nu_j)& \mapsto e_{n+j-1}
\end{align*}
\begin{align*}
 \C \otimes_\R {}_{j-1}P^B_{j} &\to {}_{n-j+2}P^A_{n-j+1} \oplus {}_{n+j-2}P^A_{n+j-1}\\
 \big(1 \otimes (j-1|j)\big) \nu_j & \mapsto ((n-j+2) \mid \big(n-j+1)\big) \\
   \big(1 \otimes (j-1|j)\big)(1 \otimes e_j - \nu_j) & \mapsto \big((n+j-2) \mid (n+j-1) \big)
\end{align*}
\begin{align*}
   \C \otimes_\R {}_jP^B_{j-1} &\to {}_{n-j+1}P^A_{n-j+2} \oplus {}_{n+j-1}P^A_{n+j-2}\\
  v_j( 1 \otimes  (j|j-1) ) & \mapsto \big((n-j+1) \mid (n-j+2) \big) \\
   ( 1 \otimes  e_j -v_j)( 1 \otimes  (j|j-1)) & \mapsto \big((n+j-1) \mid (n+j-2) \big)
\end{align*}
%%
\begin{comment}
\begin{align*}
{}_1P^B_1 \otimes_\R \C &\to \text{Hom}(P^B_2,P^B_2) 
& \\
\frac{1}{2}X_1 \otimes 1  & \mapsto X_2 
\\
e_1 \otimes 1 & \mapsto e_2 \\
{}_2P^B_2 \otimes_\R \C &\to \text{Hom}(P^B_1,P^B_1) \oplus   \text{Hom}(P^B_3,P^B_3)\\
 \frac{1}{2}(X_j \otimes 1 + i X_2 \otimes i) & \mapsto X_1\\
 \frac{1}{2}(X_j \otimes 1 - i X_2 \otimes i) & \mapsto X_3\\
   \frac{1}{2}(e_2 \otimes 1 + ie_2 \otimes i) & \mapsto e_1 \\
   \frac{1}{2}(e_2 \otimes 1 - ie_2 \otimes i) & \mapsto e_3 \\
 {}_1P^B_2 \otimes_\R \C &\to \text{Hom}(P^B_2,P^B_1)  \oplus   \text{Hom}(P^B_2,P^B_3)\\
  \frac{1}{2}((1|2) \otimes 1 + (1|2)i \otimes i) & \mapsto (2 | 1) \\
   \frac{1}{2}((1|2) \otimes 1 - (1|2)i \otimes i) & \mapsto (2 | 3) \\
   {}_2P^B_1 \otimes_\R \C &\to \text{Hom}(P^B_2,P^B_1)  \oplus   \text{Hom}(P^B_2,P^B_3)\\
  \frac{1}{2}((2|1) \otimes 1 + i(2|1) \otimes i) & \mapsto (1| 2) \\
   \frac{1}{2}((2|1) \otimes 1 - i(2|1) \otimes i) & \mapsto (3| 2)
\end{align*}
\end{comment}
It is easy to see that this map is surjective and $\dim_\C \left( \C \otimes_\R \ddot{\Ba}_n \right) = \dim_\C \left( \Aa_{2n-1} \right)$.
\end{proof} 

Let $i:\ddot{\Ba}_n \hookrightarrow \C \otimes_\R \ddot{\Ba}_n$ be the canonical injection of graded $\R$-algebras.
\cref{isomorphic algebras} allows us to view $\ddot{\Ba}_n$ as a graded subalgebra over $\R$ of $\Aa_{2n-1}$ through $\Phi\circ i$.
Thus, every $\Aa_{2n-1}$ module can be restricted to a $\ddot{\Ba}_n$ module.
In particular, $\Aa_{2n-1}$ as a graded $(\Aa_{2n-1}, \Aa_{2n-1})$-bimodule can be restricted to a graded $(\Aa_{2n-1}, \ddot{\Ba}_n)$-bimodule.
This gives us an extension of scalar functor $\Aa_{2n-1} \otimes_{\ddot{\Ba}_n} -$, sending graded $\ddot{\Ba}_n$-modules to graded $\Aa_{2n-1}$-modules. 
Furthermore, $\Aa_{2n-1} \otimes_{\ddot{\Ba}_n} -$ sends projectives to projectives.
Therefore, we can extend $\Aa_{2n-1} \otimes_{\ddot{\Ba}_n} -$ to a functor from $\Kom^b(\ddot{\Ba}_n$-$\text{p$_{r}$g$_{r}$mod})$ to $\Kom^b(\Aa_{2n-1}$-$\text{p$_{r}$g$_{r}$mod})$.
We will denote the functor
$$
\Aa_{2n-1} \otimes_{\Ba_n} - := \Aa_{2n-1} \otimes_{\ddot{\Ba}_n} (\mathfrak{F}(-)) : \Kom^b(\Ba_n\text{-p$_{r}$g$_{r}$mod})\ra \Kom^b(\Aa_{2n-1}\text{-p$_{r}$g$_{r}$mod}),
$$
where $\mathfrak{F}$ is the functor which forgets the $\Z/2\Z$ grading $\<-\>$ of the bigraded $\Ba_n$-modules.
The following proposition identifies the indecomposable projectives under the functor $\Aa_{2n-1} \otimes_{\Ba_n} -$.

\subsection*{Notation}
Let $Q$ be a left $\C$-module. We shall denote ${}_{\bar{\C}} Q$ to be the left $\C$-module with a deformed left action, given by multiplication its with complex conjugate:
\[
a(c) = \bar{a}c.
\]
Similarly for $Q$ a right $\C$-module, we use $Q_{\bar{\C}}$ to denote the right $\C$-module with the deformed action.

\begin{proposition} \label{B tensor to A}
We have the following isomorphisms of graded, left $\Aa_{2n-1}$-modules:
\begin{align*}
\Aa_{2n-1} \otimes_{\Ba_n} P^B_1 &\cong P^A_n, \text{ and } \quad \Aa_{2n-1} \otimes_{\Ba_n} P^B_j \cong P^A_{n-(j-1)} \oplus P^A_{n+(j-1)}, \quad \text{for } j \geq 2. 
\end{align*}
\end{proposition}
\begin{proof}
We will prove a slightly stronger result, which will come in handy later.
We show that
\[
\Aa_{2n-1} \otimes_{\Ba_n} P^B_1 \cong (P^A_n)_\R
\]
as graded $(\Aa_{2n-1}, \R)$-bimodules, and for $j \geq 2$ we will show that
\[
\Aa_{2n-1} \otimes_{\Ba_n} P^B_j \cong 
		\left( P^A_{n-(j-1)} \right)_{\bar{\C}} \oplus 
		P^A_{n+(j-1)}
\]
as graded $(\Aa_{2n-1}, \C)$-bimodules.

Define $\Phi_1: \Aa_{2n-1} \otimes_{\Ba_n} P^B_1 \ra (P^A_n)_\R$ and $\Phi_j: \Aa_{2n-1} \otimes_{\Ba_n} P^B_j \ra 
		\left( P^A_{n-(j-1)} \right)_{\bar{\C}} \oplus 
		P^A_{n+(j-1)}$
as the maps given on the basis elements by $a\otimes b \mapsto a \Phi(1\otimes b)$ and extend linearly.
It easy to check that $\Phi_1$ is a graded $(\Aa_{2n-1},\R)$-bimodule morphism and $\Phi_j$ is a graded $(\Aa_{2n-1},\C)$-bimodule morphism; the only detail that one should beware of is that $\Phi_j$ maps into $\left( P^A_{n-(j-1)} \right)_{\bar{\C}} \oplus 
		P^A_{n+(j-1)}$ instead of $P^A_{n-(j-1)}\oplus P^A_{n+(j-1)}$.
The fact that they are isomorphisms follows easily from looking at the dimensions.
\end{proof}

%%%%
\begin{comment}
\begin{proposition} 
For all $j\geq 2$, we have that
\[
\eta_{\pm j}: {}_jP^B \xra{\cong} {}_{n\pm(j-1)}P^A
\]
is an isomorphism of $(\C,\Ba)$-mod. 
%%
%%with the property that $\Phi^{-1}(\eta_{-j} + \eta_{+j})$ is the injection ${}_jP^B \ra \C\otimes_\R {}_jP^B$.
\end{proposition}
\begin{proof}
Restrict $\Phi$ to $\C\otimes_\R {}_jP^B$, which by definition of $\Phi$ gives 
\[
\C\otimes_\R {}_jP^B \xra{\Phi} {}_{n+(j-1)}P^A\oplus {}_{n-(j-1)}P^A.
\]
Let $v_j = \frac{1}{2}(1\otimes e_j + i \otimes ie_j) \in \C\otimes_\R {}_jP^B$ and let $ \rho_{\pm j}: {}_jP^B \ra \C\otimes_\R {}_jP^B$ be the injection of $(\C,\Ba)$-bimodules uniquely defined by
\[
\rho_{\alpha}(e_j) = 
\begin{cases}
v_j, \quad \text{for } \alpha = +j \\
1\otimes e_j -v_j, \quad \text{for } \alpha = -j.
\end{cases}
\]
Define $\eta_{\pm j}:= \Phi\circ\rho_{\pm j}$.
It follows from the definition of $\Phi$ that $\eta_{\pm j}$ map into the respective codomains and $\eta_{\pm j}$ are morphisms of $(\C,\Ba)$-modules.
In particular, $\eta_{\pm j}$ are $\C$-vector space maps, and since they are clearly surjective, $\eta_{\pm j}$ are isomorphisms.
\end{proof}
\end{comment}
%%%%%

Recall the injection $\Psi: \mathcal{A}(B_n) \ra \mathcal{A}(A_{2n-1}) $  as defined in \cref{injB}; the image of standard generators are explicitly given by:
\begin{align*}
\Psi(\sigma^B_j) &= 
\begin{cases}
\sigma^A_n, & \text{for } j = 1; \\
\sigma^A_{n-(j-1)} \sigma^A_{n+(j-1)}, & \text{otherwise}.
\end{cases}
\end{align*}
We have previously shown that $\mathcal{A}(B_n)$ acts on $\Kom^b(\Ba_n$-$\text{p$_{r}$g$_{r}$mod})$ and similarly $\mathcal{A}(A_{2n-1})$ acts on $\Kom^b(\Aa_{2n-1}$-$\text{p$_{r}$g$_{r}$mod})$.
Through $\Psi$, we have an induced action of $\mathcal{A}(B_n)$ on   $\Kom^b(\Aa_{2n-1}$-$\text{p$_{r}$g$_{r}$mod})$.
We will now prove the algebraic version of \cref{topological equivariant}.

\begin{theorem}\label{tensor equivariant}
For all $1 \leq j \leq n$, we have that $
\Aa_{2n-1} \otimes_{\Ba_n} \sigma^B_j \cong \Psi(\sigma^B_j)_{\Ba_n}$
in $\Kom^b((\Aa_{2n-1},\Ba_n)\text{-bimod})$.
In particular, the functor $\Aa_{2n-1} \otimes_{\Ba_n} -$ is $\mathcal{A}(B_n)$-equivariant:
\begin{center}
\begin{tikzpicture} [scale=0.8]
\node (tbB) at (-2.5,1.5) 
	{$\mathcal{A}(B_n)$};

\node (tbA) at (10,1.5) 
	{$\mathcal{A}(A_{2n-1})\xleftarrow{\Psi} \mathcal{A}(B_n) $};

\node[align=center] (cB) at (0,0) 
	{$\Kom^b(\Ba_n$-$p_rg_rmod)$};
\node[align=center] (cA) at (7,0) 
	{$\Kom^b(\Aa_{2n-1}$-$p_rg_rmod),$};

\coordinate (tbB') at ($(tbB.east) + (0,-1)$);

\coordinate (tbA') at ($(tbA.west) + (0,-1)$);

\draw [->,shorten >=-1.5pt, dashed] (tbB') arc (245:-70:2.5ex);

\draw [->, shorten >=-1.5pt, dashed] (tbA') arc (-65:250:2.5ex);

\draw[->] (cB) -- (cA) node[midway,above]{$\Aa_{2n-1}\otimes_{\Ba_n} -$}; 

\end{tikzpicture}
\end{center}
i.e. for any $\sigma \in \mathcal{A}(B_n)$ and any complex $C \in \Kom^b(\Ba_n$-$p_rg_rmod)$,
$
\Aa_{2n-1} \otimes_{\Ba_n} (\sigma \otimes_{\Ba_n} C) \cong \Psi(\sigma)\otimes_\Aa (\Aa_{2n-1} \otimes_{\Ba_n} C)
$.
\end{theorem}

Before we start with the proof, we will need the following lemma:
\begin{lemma} \label{lemma bimod iso}
\begin{align}
\Aa_{2n-1} \otimes_{\Ba_n} P^B_1 &\cong (P^A_n)_\R  \quad &\text{as graded $(\Aa_{2n-1},\R)$-bimodules}; 
	\label{lemma 1}\\
\C\otimes_\R {}_1P^B &\cong ({}_nP^A)_{\Ba_n} \quad &\text{as graded $(\C, \Ba_n)$-bimodules}; 
	\label{lemma 2}\\
\Aa_{2n-1} \otimes_{\Ba_n} P^B_j &\cong \left( P^A_{n-(j-1)} \right)_{\bar{\C}} \oplus P^A_{n+(j-1)}\quad &\text{as graded $(\Aa_{2n-1},\C)$-bimodules}; 
	\label{lemma 3}\\
{}_jP^B &\cong \left({}_{n+ (j-1)}P^A \right)_{\Ba_n}
	\quad &\text{as graded $(\C, \Ba_n)$-bimodules;}
	\label{lemma 4}\\
{}_jP^B &\cong 
	{}_{\bar{\C}} \left( {}_{n- (j-1)}P^A
		\right)_{\Ba_n}
	\quad &\text{as graded $(\C, \Ba_n)$-bimodules;}
	\label{lemma 5} \\
{}_\C \C \otimes_{\bar{\C}} \C_\C &\cong {}_\C \C_\C
	\quad &\text{as graded $(\C,\C)$-bimodues.}
	\label{lemma 6}
\end{align}
\end{lemma}
\begin{proof}
Both \ref{lemma 1} and \ref{lemma 3} follows from the proof of \cref{B tensor to A}.
For the rest of the lemma we will only define the maps; the proofs that they are isomorphisms respecting the required structures follows from a simple verification.

For \ref{lemma 2} take the morphism $\phi_1 : \C \otimes_\R {}_1P^B \ra ({}_nP^A)_{\Ba_n}$ as the restriction of $\Phi$ given  by $ c\otimes b \mapsto \Phi(c\otimes b)$.
Note that $\phi_1$ does indeed map into $({}_nP^A)_{\Ba_n}$ since
\[
\Phi(c\otimes b) = \Phi(c\otimes e_1b)= \Phi((1\otimes e_1)(c\otimes b)) =  \Phi(1\otimes e_1)\Phi(c\otimes b) = e_n\Phi(c\otimes b).
\]
\begin{comment}
Since $(c\otimes b)\cdot b' = c\otimes bb' = (c\otimes b)(1\otimes b') $ and $\Phi$ is respects multiplication, $\phi_1$ is indeed a left $\Ba_n$-module homomorphism.
Furthermore $\Phi$ is a $\C$-algebra homomorphism, therefore $\phi_1$ is indeed a $(\C, \Ba_n)$-bimodule homomorphism. 
To see that $\phi_1$ is an isomorphism, note that
\[
a = e_na = \Phi(1\otimes e_1)\Phi( \Phi^{-1}(a)) = \phi_1( (1\otimes e_1) \Phi^{-1}(a))
\]
for all $a \in {}_nP^A$, which shows that $\phi_1$ is surjective.
As $\dim_\C (\C\otimes_\R {}_1P^B) = \dim_\C ({}_nP^A) = 4$ and $\phi_1$ is $\C$-linear, $\phi_1$ is indeed an isomorphism.
\end{comment}

For \ref{lemma 4} and \ref{lemma 5}, consider the morphisms
\begin{align*}
\phi_{+j} : {}_jP^B  &\ra 
	\left({}_{n+ (j-1)}P^A \right)_{\Ba_n} &and  &&\phi_{-j} : {}_jP^B  &\ra 
	{}_{\bar{\C}}\left({}_{n- (j-1)}P^A \right)_{\Ba_n} \\
b &\mapsto 
	e_{n+ (j-1)}\Phi(1\otimes b),  &    &&b &\mapsto 
	e_{n-(j-1)}\Phi(1\otimes b). 
\end{align*}
\begin{comment}
$\\
\phi_{-j} : {}_jP^B  &\ra 
	{}_{\bar{\C}}\left({}_{n- (j-1)}P^A \right)_{\Ba_n} \\
b &\mapsto 
	e_{n-(j-1)}\Phi(1\otimes b).$
Note that $\phi_{\pm j}$ is just the projection of the restriction of $\Phi\circ i$ to ${}_jP^B$, as seen below:
\begin{align*}
\Phi(1\otimes b) &= e_{n\pm (j-1)}\Phi(1\otimes e_jb) \\
&= \Phi((1\otimes e_j)(1\otimes b)) \\
&=  \Phi(1\otimes e_j)\Phi(1\otimes b) \\
&= (e_{n + (j-1)} + e_{n - (j-1)})\Phi(1\otimes b).
\end{align*}
The fact that $\phi_{\pm j}$ are projections follows from $e_{n\pm (j-1)}$  both being primitive idempotents.
Given the definition of $\phi_{\pm j}$, it is clear that $\phi_{\pm j}$ preserve the right $\Ba_n$ action.
To see that they preserve the respective left $\C$-actions, note that
\begin{align*}
\phi_{\pm j}((r+is) \cdot b) &= \phi_{\pm j}(rb + s(ie_jb)) \\
&= e_{n\pm (j-1)}\Phi(1\otimes (rb + s(ie_jb))) \\
&= e_{n\pm (j-1)}\Phi(r\otimes b + s\otimes ie_j b) \\
&= e_{n\pm (j-1)}\Phi((r\otimes e_j + s\otimes ie_j)(1\otimes b)) \\
&= e_{n\pm (j-1)}\Phi(r\otimes e_j + s\otimes ie_j)\Phi(1\otimes b) \\
&= e_{n\pm (j-1)} \left(
	r(e_{n + (j-1)} + e_{n - (j-1)}) 
		+ is(e_{n + (j-1)} - e_{n - (j-1)})
	\right) 
	\Phi(1\otimes b).
\end{align*}
\rb{finish the proof that it is an isomorphism.}
\end{comment}

Finally for \ref{lemma 6}, consider the morphism $c: {}_\C \C \otimes_{\bar{\C}} \C_\C \ra \C$ uniquely defined by $
1\otimes 1 \mapsto 1.$
\begin{comment}}
This is because in ${}_\C \C \otimes_{\bar{\C}} \C_\C$ we have
\begin{align*}
a\otimes b &= (1*\bar{a}) \otimes b \\
&= 1\otimes (\bar{a}\cdot b) \\
&= 1\otimes ab \\
&= 1 \otimes (1* ab) \\
&= (1 \otimes 1) * (ab); \\
a\otimes b &= a \otimes (\bar{b} \cdot 1) \\
&= (a*\bar{b}) \otimes 1 \\
&= ab \otimes 1 \\
&= (ab\cdot 1) \otimes 1 \\
&=  (ab) \cdot (1\otimes 1).
\end{align*}
\end{comment}
\end{proof}

%% &= e_{n\pm (j-1)}\Phi(r\otimes b) + e_{n\pm (j-1)} \Phi(s \otimes ie_j b) \\

\begin{proof}(of \cref{tensor equivariant})
Let $j\geq 2$.
Before we proceed, note that
\begin{align*}
P_{n-(j-1)}^A \otimes_\C {}_{n-(j-1)}P^A 
&\cong P_{n-(j-1)}^A \otimes_\C \C \otimes_\C {}_{n-(j-1)}P^A \\
&\cong P_{n-(j-1)}^A \otimes_\C \C \otimes_{\bar{\C}} \C \otimes_\C {}_{n-(j-1)}P^A \qquad (\text{by } \ref{lemma 6} \text{ in } \cref{lemma bimod iso})\\
&\cong P_{n-(j-1)}^A \otimes_{\bar{\C}} {}_{n-(j-1)}P^A.
\end{align*}
Denote this composition of isomorphisms as $h : P_{n-(j-1)}^A \otimes_{\bar{\C}} {}_{n-(j-1)}P^A \ra P_{n-(j-1)}^A \otimes_\C {}_{n-(j-1)}P^A$, which is simply defined by $h(a\otimes b) = a\otimes b$.
We obtain the following chain of isomorphisms
\begin{align*}
\Psi(\sigma_j^B)_\Ba 
&= \left(\sigma_{n-(j-1)}^A\sigma_{n+(j-1)}^A \right)_\Ba\\
&= \left(P_{n-(j-1)}^A \otimes_\C \left({}_{n-(j-1)}P^A\right)_\Ba \right) \oplus \left(P_{n+(j-1)}^A \otimes_\C \left( {}_{n+(j-1)}P^A\right)_\Ba \right) 
	\xra{
		\begin{bmatrix}
		\beta_{n-(j-1)}^A & \beta_{n+(j-1)}^A
		\end{bmatrix}
	} 
	\Aa_\Ba \\
&\cong \left(P_{n-(j-1)}^A \otimes_{\bar{\C}} \left( {}_{n-(j-1)}P^A \right)_\Ba \right) \oplus \left(P_{n+(j-1)}^A \otimes_\C \left( {}_{n+(j-1)}P^A \right)_\Ba \right) 
	\xra{
		\begin{bmatrix}
		\beta_{n-(j-1)}^A h & \beta_{n+(j-1)}^A
		\end{bmatrix}
	} 
	\Aa_\Ba \\
&\cong \left( \left( P_{n-(j-1)}^A \right)_{\bar{\C}} \otimes_\C {}_{j}P^B \right) \oplus \left(P_{n+(j-1)}^A \otimes_\C {}_{j}P^B \right) 
	\xra{
		\begin{bmatrix}
		\beta_{n-(j-1)}^A h (\id \otimes_\C \phi_{-j}) &
		\beta_{n+(j-1)}^A(\id \otimes_\C \phi_{+j})
		\end{bmatrix}
	} 
	\Aa_\Ba \\
&\cong \left( \left( P_{n-(j-1)}^A \right)_{\bar{\C}} \oplus P_{n+(j-1)}^A \right) \otimes_\C {}_{j}P^B
	\xra{
		\beta_{n-(j-1)}^A h (e_{n-(j-1)} \otimes_\C \phi_{-j}) +
		\beta_{n+(j-1)}^A(e_{n+(j-1)} \otimes_\C \phi_{+j})
	} 
	\Aa_\Ba \\ 
&\cong \Aa \otimes_\Ba P_j^B \otimes_\C {}_jP^B 
	\xra{
		\left(
		\beta_{n-(j-1)}^A h (e_{n-(j-1)} \otimes_\C \phi_{-j}) +
		\beta_{n+(j-1)}^A(e_{n+(j-1)} \otimes_\C \phi_{+j})
		\right)
		(\Phi_j \otimes_\C \id)
	} 
\Aa_\Ba \\
&\cong \Aa \otimes_\Ba P_j^B \otimes_\C {}_jP^B 
	\xra{
		g
		\left(
		\beta_{n-(j-1)}^A h (e_{n-(j-1)} \otimes_\C \phi_{-j}) +
		\beta_{n+(j-1)}^A(e_{n+(j-1)} \otimes_\C \phi_{+j})
		\right)
		(\Phi_j\otimes_\C \id)
	} 
\Aa \otimes_\Ba \Ba
\end{align*}
where $g$ is the canonical isomorphism $\Aa_\Ba \xra{\cong} \Aa \otimes_\Ba \Ba$.
We claim that 
\[
g
		\left(
		\beta_{n-(j-1)}^A h (e_{n-(j-1)} \otimes_\C \phi_{-j}) +
		\beta_{n+(j-1)}^A(e_{n+(j-1)} \otimes_\C \phi_{+j})
		\right)
		(\Phi_j\otimes_\C \id)
=
\id\otimes_\Ba \beta_j^B,
\]
which can be done by showing that they have the same image as follows:
\begin{align*}
a \otimes_\Ba p \otimes_\C p'
&\xmapsto{\Phi_j \otimes_\C \id}
	a\Phi(1\otimes p)\otimes_\C p' \\
&\xmapsto{
		\beta_{n-(j-1)}^A h (e_{n-(j-1)} \otimes_\C \phi_{-j}) +
		\beta_{n+(j-1)}^A(e_{n+(j-1)} \otimes_\C \phi_{+j})
	} \\
	&{} \qquad a\Phi(1\otimes p) e_{n-(j-1)}\Phi(1\otimes p') 
	+ a\Phi(1\otimes p) e_{n+(j-1)}\Phi(1\otimes p')
\\
&\quad = a\Phi(1\otimes p)(e_{n-(j-1)}+ e_{n+(j-1)})\Phi(1\otimes p') \\
&\quad = a\Phi(1\otimes p)\Phi(1\otimes e_j)\Phi(1\otimes p') \\
&\quad = a\Phi(1\otimes pp') \\
&\xmapsto{g} a\Phi(1\otimes pp')\otimes_\Ba 1 = a \otimes_\Ba pp', 
\end{align*}
Thus we get
\begin{align*}
\Psi(\sigma_j^B)_\Ba 
\cong \left( \Aa \otimes_\Ba P_j^B \otimes_\C {}_jP^B 
	\xra{\id \otimes_\Ba \beta_j^B
	} 
\Aa \otimes_\Ba \Ba \right)
= \Aa \otimes_\Ba \left( P_j^B \otimes_\C {}_jP^B 
	\xra{\beta_j^B} \Ba \right) 
= \Aa \otimes_\Ba \sigma_j^B.
\end{align*}
The proof for the case where $j = 1$ is a simpler version of the proof above, which we shall omit and leave it to the interested reader.

These cases combined are sufficient to show that the action intertwines, since we have that for any $\sigma \in \mathcal{A}(B_n)$ and any complex $C \in \Kom^b(\Ba_n$-$\text{p$_{r}$g$_{r}$mod})$,
\[
\Aa\otimes_\Ba \sigma \otimes_\Ba C \cong \Psi(\sigma)\otimes_\Ba C \cong \Psi(\Sigma)\otimes_\Aa \Aa \otimes_\Ba C.
\]
\end{proof}
\begin{remark}
Recall the definitions of \ $\mathcal{U}_j$ and \ $\cU_j$ from \cref{Functorial TL}.
In particular, this proposition implies that $\Aa_{2n-1} \otimes_{\Ba_n} U_j \cong
\begin{cases}
\left( \mathcal{U}_1 \right)_{\Ba_n}, &\text{ for } j = 1; \\
\left( \mathcal{U}_{n-(j-1)} \oplus \mathcal{U}_{n+(j-1)} \right)_{\Ba_n}, &\text{ otherwise.}
\end{cases}$

\noindent When $n=2$, this also relate our bimodules construction to the bimodules construction given in \cite{MackTubb} (see Example 2.12 in particular).
\end{remark}
We may now use this relation to deduce that the categorical action of $\cA(B_n)$ is faithful:
\begin{theorem}\label{faithful action}
The (weak) action of $\cA(B_n)$ on the category $\Kom(\Ba_n$-$\text{p$_{r}$g$_{r}$mod})$ given in \ref{Cat B action} is faithful.
\end{theorem}
\begin{comment}
The first one will be a purely algebraic proof independent of the topological relations.
To do so we need the following result from \rb{REF}:
\begin{proposition}\label{check on split gen}
Let $\beta \in \mathcal{A}(A_{2n-1})$ such that
\[
\beta\left(\oplus_{j=1}^{2n-1} P^A_j \right) \cong \oplus_{j=1}^{2n-1} P^A_j
\]
in $\Kom^b(\Aa_{2n-1}\text{-prgrmod})$. Then $\beta = \id$.
\end{proposition}
\begin{proof}
\rb{REF??}
\end{proof}
\end{comment}
\begin{proof}
Assume that we are given $\sigma \in \cA(B_n)$ such that
$ \sigma\left(C\right) \cong C $
for all $C \in \Kom^b(\Ba_n$-$\text{p$_{r}$g$_{r}$mod})$.
We will show that this implies $\sigma$ is the identity.
In particular, take $C = \oplus_{j=1}^n P^B_j$ so that we have

 \[\sigma\left(\oplus_{j=1}^n P^B_j \right) \cong \oplus_{j=1}^n P^B_j.\]

\noindent Applying the functor $\Aa_{2n-1} \otimes_{\Ba_n} -$, we obtain

\begin{equation}\label{faithful 1}
\Aa_{2n-1}\otimes_{\Ba_n} \sigma\left(\oplus_{j=1}^n P^B_j \right) \cong 
	\Aa_{2n-1}\otimes_{\Ba_n} \left(\oplus_{j=1}^n P^B_j \right) \cong 
	\oplus_{j=1}^{2n-1} P^A_j.
\end{equation}

\noindent Applying \cref{tensor equivariant} to the LHS of \ref{faithful 1}, we get 

\begin{equation}\label{faithful 2}
\Aa_{2n-1}\otimes_{\Ba_n} \sigma \left(\oplus_{j=1}^n P^B_j \right)
	\cong \Psi(\sigma) \left( \Aa_{2n-1} \otimes_{\Ba_n} \left(\oplus_{j=1}^n P^B_j \right) \right) \cong \Psi(\sigma) \left( \oplus_{j=1}^{2n-1} P^A_j \right).
\end{equation}

\noindent Combining the two equations \ref{faithful 1} and \ref{faithful 2} above we deduce that

\[
\Psi(\sigma) \left( \oplus_{j=1}^{2n-1} P^A_j \right) \cong \oplus_{j=1}^{2n-1} P^A_j.
\]

Since it was shown in \cite[Corollary 1.2]{KhoSei} that the type $A$ categorical action is faithful,  we conclude that $\Psi(\sigma) = \id$.
But $\Psi$ is injective, so we must have that $\sigma = \id$ as required.
\end{proof}

%\newpage

\section{Main Theorem} \label{main theorem}

This section contains the main result of this paper, which relates the topological action of $\mathcal{A}(B_n)$ on isotopy classes of admissable curves in $\D^B_{n+1}$ defined in \cref{topology} and the categorical action of $\mathcal{A}(B_n)$ on $\Kom^b(\Ba_n$-$\text{p$_{r}$g$_{r}$mod})$ defined in \cref{define zigzag}.
This will be done in a similar fashion as in \cite{KhoSei}:
they define a map that assigns a complex in $\Kom^b(\Aa_{2n-1}$-$\text{p$_{r}$g$_{r}$mod})$ for each isotopy classes of curves in $\DA$.
Moreover, they show that this assignment is $\cA(A_{2n-1})$-equivariant.
Following a similar idea, we will also define a similar map $L_B$ for type $B_n$, which assigns a complex in $\Kom^b(\Ba_n$-$\text{p$_{r}$g$_{r}$mod})$ for each isotopy classes of curves in $\D^B_{n+1}$.
We will also show that this assignment is $\cA(B_n)$-equivariant.
Instead of proving this directly, we will make use of the topological relation between isotopy classes of curves in $\D^B_{n+1}$ and in $\D^A_{2n}$, and the algebraic relation between $\Kom^b(\Ba_n$-$\text{p$_{r}$g$_{r}$mod})$ and $\Kom^b(\Aa_{2n-1}$-$\text{p$_{r}$g$_{r}$mod})$.

\subsection{Complexes associated to admissible multicurves and categorical action (Type $A$)}

We start by recalling the constructions and results shown in \cite{KhoSei} for admissable curves, which can be extended easily to admissable multicurves.
Note that we use $L_A$ in place of $L$ in \cite{KhoSei} to differentiate between the maps for type $A$ and type $B$ later on.
Let $\ddot{c}$ be a bigraded admissable curve.
We associate to $\ddot{c}$ an object $L_A(\ddot{c})$ in the category $\Kom^b(\Aa_{2n-1}$-$\text{p$_{r}$g$_{r}$mod})$.
Start by defining $L_A(\ddot{c})$ as a bigraded $\Aa_{2n-1}$-module:

\begin{align}
L_A(\ddot{c}) = \bigoplus_{x \in cr(\ddot{c})}  P(x)
\end{align}

where $P(x) = P^A_{x_0}[-x_1]\{x_2\}$ (see the paragraph after \cref{basic curves A} for the definition of $(x_0, x_1, x_2)$).
For every $x,y \in cr(\ddot{c})$ define
$
\partial_{yx}: P(x) \ra P(y)
$
by the following rules:
\begin{itemize}
\item If $x$ and $y$ are the endpoints of an essential segment and $y_1 = x_1 + 1$, then
\begin{enumerate}
\item If $x_0 = y_0$ (then it must be that $x_2 = y_2 + 1$) then
\[
\partial_{yx}: P(x) \ra P(y) \cong P(x)[-1]\{1\}
\]
is the multiplication on the right by $X_{x_0} \in \Aa_{2n-1}$.
\item If $x_0 = y_0 \pm 1$, then $\partial_{yx}$ is the right multiplication by $(x_0|y_0) \in \Aa_{2n-1}$;
\end{enumerate}
\item otherwise $\partial_{yx} = 0$.
\end{itemize}
We define the differential $\partial$ as
$
\partial := \sum_{x,y} \partial_{yx}.
$
See \cite[Lemma 4.1]{KhoSei} for a proof that this defines a complex.
Moreover, it follows easily that 

\begin{equation} \label{action shift comm A}
L_A(\chi_A(r_1, r_2)\ddot{c}) \cong L_A(\ddot{c})[-r_1]\{r_2\}.
\end{equation}

For $\check{g}$ a bigraded $j$-string of $\check{c}$, we can also assign a complex $L_A(\check{g})$ to $\check{g}$, where as a bigraded abelian group, $L_A(\check{g}) = \bigoplus_{x \in cr{g}} P^A(x)$ and the differentials are obtained from essential segments of $\check{g}$ the same way as for admissible curves.
We can easily extend this to define $L_A(\check{h})$ for $h \subseteq c$ a connected subset of $c$ such that $h = \cup g_{\alpha, j}$ with each $g_{\alpha, j}$ some bigraded $j$-string of $c$.
The following theorem is proven in \cite[Theorem 4.3]{KhoSei}:
\begin{theorem} \label{L_A equivariant}
For a braid $\sigma \in \cA(A_{m})$ and a bigraded admissible curve $\ddot{c}$ in $\D^A_{m+1}$, we have $\sigma L_A(\ddot{c}) \cong L_A(\sigma(\ddot{c}))$ in the category $\Kom^b(\Aa_m$-$p_rg_rmod)$, i.e. $L_A$ is $\cA(A_m)$-equivariant.
\end{theorem}
We extend $L_A$ to admissable multicurves as follows:
given bigraded multicurves $\coprod_j \ddot{c}_j$,

\[
L_A\left(\coprod_j \ddot{c}_j \right) := \bigoplus_j L_A(\ddot{c}_j).
\]

It follows easily that this defines a complex, and both \cref{action shift comm A} and \cref{L_A equivariant} still hold for admissable multicurves.

\subsection{Complexes associated to admissible curves and categorical action (Type $B$)}
Consider a trigraded admissable  curve $\check{c}$.
We associate to $\check{c}$ an object $L_B(\check{c})$ in the category $\Kom^b(\Ba_n$-$p_rg_rmod)$.
Start by defining $L_B(\check{c})$ as a trigraded $\Ba_n$-module:

\[
L_B(\check{c}) = \bigoplus_{y \in cr(\check{c})}  P(y)
\]

where $P(y) = P^B_{y_0}[-y_1]\{y_2\}\<y_3\>$ (see the second paragraph after \cref{injBA} in \cref{local index function} for definition of $y_0, y_1, y_2$ and $y_3$).

We now define maps $\partial_{zy}: P(y) \ra P(z)$ for each $y,z \in cr(\check{c})$ using the following rules (note that these are \emph{\textbf{not}} the differentials yet):

\begin{itemize}
\item If $y$ and $z$ are the endpoints of an essential segment in $D_j$ for $j \geq 1$ and $z_1 = y_1 + 1$, then
	\begin{enumerate}[1.]
 	\item If $y_0 = z_0 $ (then also $y_2 = z_2 + 1$ and $y_3 = z_3$), then
 	
 	$$ \partial_{zy}: P(y) \ra P(z) \cong P(y)[-1]\{1\} $$

 	is the right multiplication by the element $X_{y_0} \in \Ba_n.$
 	\item If  $y_0 = z_0 \pm 1$ then $\partial_{zy}$
  	is the right multiplication by $(y_0 | z_0) \in \Ba_n$;
\end{enumerate}
\item otherwise $\partial_{yz} = 0.$
\end{itemize}

We will modify some of these maps before using them as differentials.
Define the following equivalence relation on the set of 1-crossings:
\[
y \sim y' \iff y \text{ and } y' \text{ are connected by an essential segment in } D_0.
\]
Consider the partitioning of the set of 1-crossings using the equivalence relation above.
Note that every equivalence classes under this relation consists of either 1 or 2 elements.
For each $[y]$ an equivalence class of 1-crossing, we modify the some of the maps given previously by the following rule:
\begin{itemize}
\item If $[y] = \{y \}$, we modify nothing;
\item otherwise, we have $[y] = \{y, y'\}$ with $y'$ a distinct 1-crossing. 
Note that at least one of $y$ or $y'$ must be an endpoint of some essential segment in $D_1$; without lost of generality let this 1-crossing be $y$, with the other endpoint of the essential segment in $D_1$ be $z$.
Consider the two possible cases for $z$:
	\begin{enumerate}[1.]
	\item $z$ is a 2-crossing:
		\begin{enumerate}[(a)]
		\item if $y_1 = z_1 + 1$, then we have that $\partial_{yz}: P(z) \ra P(y)$ is given by the right multiplication by $(2|1)$. 
		We modify 
		$\partial_{y'z}: P(z) \ra P(y')\cong P(y)\<1\>$ (which was necessarily 0 previously) so that it is now the right multiplication by $-i(2|1)$;
		\item otherwise, we have instead $z_1 = y_1 + 1$.
		In this case, $\partial_{zy}: P(y) \ra P(z)$ is given by the right multiplication by $(1|2)$. 
		We modify $\partial_{zy'}: P(y') \cong P(y)\<1\> \ra P(z)$ (which was necessarily 0 previously) so that it is now the right multiplication by $(1|2)i$.
		\end{enumerate}
	\item $z$ is a 1-crossing: 
		\begin{enumerate}[(a)]
		\item if $y_1 = z_1 + 1$, we modify nothing;
		\item otherwise, we have instead $z_1 = y_1 + 1$.
		In this case $\partial_{zy}: P(y) \ra P(z)$ is given by the right multiplication by $X_1$.
		Once again consider the two possible cases of the equivalence class $[z]$:  
  			\begin{enumerate}[(i)] 
 			\item If $[z] = \{z \}$, we modify nothing;
			\item otherwise $[z] = \{z, z'\}$ with $z'$ a distinct 1-crossing.
 			We then modify $\partial_{z'y'}: P(y)\<1\> \cong P(y') \ra P(z') \cong P(z)\<1\>$ (which was necessarily $0$ previously) so that it is now the right multiplication by $X_1$;
			\end{enumerate}
		\end{enumerate}
	\end{enumerate} 
	Repeat the process above for $y'$ if $y'$ is also an endpoint of some essential segment in $D_1$.
\end{itemize}
Finally, we define the differential as
$
\partial = \sum_{x,y \in cr(\check{c})} \partial_{xy},
$
where $\partial_{xy}$ are the modified version above.

\begin{lemma}\label{check complex}
$(L_B(\check{c}), \partial)$ is a complex of projective  graded $\Ba_n$-modules with a grading-preserving differential.
\end{lemma}
\begin{proof}
For $x, y, z \in cr(\check{c})$ with $x_0,y_0,z_0 \geq 2$, the same argument as in the type $A$ shows that the product of $\partial_{zy} \partial_{yx}:P_x \ra P_z$ is always 0.
%
\begin{comment}
Given any arrow from $x$ to $y$ and from $y$ to $z,$ the product of $\partial_{zy} \partial_{yx}:P_x \ra P_z$ is the multiplication of a certain product of $(i|i \pm 1) \in \Ba_n$.
The main point to note is that the only non-zero products of two such generators are the product $(j|j \pm 1)(j \pm 1|j) = X_j$ and $i(j|j \pm 1)(j \pm 1|j) = (j|j \pm 1)(j \pm 1|j)i = iX_j$.

For $j \geq 3$, $(j|j \pm 1)$ cant be followed with $(j \pm 1|i)$.
This is because the only non-zero differentials between $P^B(x)$ and $P^B(y)$ for $x_0, y_0 \geq 2$ is when $x$ and $y$ are connected by an essential segment, and having $(j|j \pm 1)$ followed with $(j \pm 1|i)$ would mean that we have a closed curve (which is not possible for admissable curves).
For $j \leq 2$, there is a possibility that we have $(j|j \pm 1)$ followed by $(j \pm 1|i)$.
\end{comment}
%
The only occurrence of $\partial_{zy} \partial_{yx} \neq 0$ is when $\partial_{zy} = \partial_{ac}, \partial_{ad}$ and $\partial_{yx} = \partial_{db}, \partial_{cb}$ with $a,b,c,d$ the crossings of the following type of 1-string labelled below:
\begin{figure}[H]
\centering
\begin{tikzpicture}  [scale=1]
\draw[thick]  (5.475,1)--(5.7,1);
\draw[thick]  (5.475,3)--(5.7,3);  
\draw[thick]  (5.475,1)--(5.475,3);
\draw[thick,red]  (3.25,2.5)--(5.475,2.5);
\draw[thick,red]  (3.25,1.5)--(5.475,1.5);
\draw[thick]  (3.25,3)--(5.7,3);  
\draw[thick]  (3.25,1)--(5.7,1);
\draw[thick]  (4,1)--(4,3);
\draw[thick,green, dashed]  (3.25,3)--(3.25,2.5);
\draw[thick, green, dashed]   (3.25,1)--(3.25,1.5);
\draw[thick, green, dashed, ->]   (3.25,2)--(3.25,2.6);
\draw[thick,  green, dashed, ->]   (3.25,2)--(3.25,1.4);\filldraw[color=black!, fill=yellow!, very thick] (3.25,2) circle [radius=0.1]  ;
\draw[fill] (4.75,2) circle [radius=0.1]  ;

\node [above right] at (5.475,2.5) {{\small $a$}};
\node [below right] at (5.475,1.5) {{\small $b$}};
\node [below right] at (4,1.5) {{\small $c$}};
\node [above right] at (4,2.5) {{\small $d$}};
\end{tikzpicture}
\end{figure}
Note that the two non-zero composition $\partial_{ad}\partial_{db}$ and $\partial_{ac}\partial_{cb}$ always occur as a pair.
Moreover, we see that their sum is equal to 0: $\partial_{ad}\partial_{db} + \partial_{ac}\partial_{cb} = X_2i - X_2i = 0$, thus showing that $\partial^2 = 0$ as required.
%
%
\begin{comment}
 \begin{center}
\begin{tikzcd}
&  & P^B(c)  \arrow[r,"(1|2)"] \arrow[d, phantom, "\bigoplus"] & P^B(a) & \nabla \arrow[leftrightarrow, l]\\
\nabla \arrow[leftrightarrow, r] & P^B(b) \arrow[r,"(2|1)"] \arrow[ur,"-i(2|1)"]  &P^B(d) \arrow[ru,"(1|2)i"']
\end{tikzcd} 
 \end{center}
 \end{comment}
%
%
%
\begin{comment}
with the corresponding component $L(g)$ of $L(c)$:
 \begin{center}
\begin{tikzcd}
 \nabla \arrow[leftrightarrow,r] &  P_{2} \arrow[r,"(2|1)"] \arrow[rd,"-i(2|1)"', "\bigoplus" xshift=3.5ex, yshift=-1.5ex ] & P_{1}  \arrow[r,"(1|2)"] & P_{2} &   \arrow[leftrightarrow,l] \nabla\\
&  &P'_{1} \arrow[ru,"(1|2)i"']
\end{tikzcd} 
 \end{center}
where $\nabla$ denotes the complement of $L(g)$ in $L(c)$.
More precisely, denote $g' := g \setminus (a \cup b)$ and let $c_1$ and $c_2$ be the two (distinct) connected components of $c \setminus g'$, with $c_1$ containing $a$ and $c_2$ containing $b$.
In this case, $\nabla$ on the left is the part of the complex $L(c)$ which contains $L(c_2)$ and nabla on the right is the part of the complex $L(c_1)$.
The notation $\nabla \leftrightarrow P_2$ means that $\nabla$ (equivalently $L(c_2)$ relates to $L(g)$ through the module $P_2$.
Namely, when considering the whole complex $L(c)$, there might be an arrow $P(x) \rightarrow P_2$ or $P_2 \ra P(x)$ where $P(x)$ is in $L(c_2)$ and $x \neq b$.
We will continue to use this notation in the upcoming proofs, where $\nabla$ always denotes a similar notion of complement, and the meaning should be clear from the context.
\end{comment}
%
%
%
\end{proof}

\begin{lemma}

 For any triple $(r_1,r_2,r_3)$ of integers and any  trigraded admissible curve $\check{c}$ we have:
 $$ L_B(\chi(r_1,r_2,r_3)\check{c}) \cong L_B(\check{c})[-r_1]\{r_2\}\<r_3\>. $$

\end{lemma}
\begin{proof}
This follows directly from the definition.
\end{proof}

\subsection{Main result}
Let us recall the main theorem that we aim to prove:

\begin{figure}[H] 
\centering
\begin{tikzpicture} [scale = 0.9]
\node (tbB) at (-3,1.5) 
	{$\mathcal{A}(B_n)$};
\node (cbB) at (-3,-3.5) 
	{$\mathcal{A}(B_n)$};
\node (tbA) at (10,1.5) 
	{$\mathcal{A}(A_{2n-1})$}; 
\node (cbA) at (10.5,-3.5) 
	{$\mathcal{A}(A_{2n-1})$};

\node[align=center] (cB) at (0,0) 
	{Isotopy classes of admissible \\ trigraded curves $\check{\fC}^{adm}$  in $\D^A_{n+1}$};
\node[align=center] (cA) at (7,0) 
	{Isotopy classes of admissible \\ bigraded multicurves $\ddot{\udfC}^{adm}$  in $\D^B_{2n}$};
\node (KB) at (0,-2)
	{$\Kom^b(\Ba_n$-$\text{p$_{r}$g$_{r}$mod})$};
\node (KA) at (7,-2) 
	{$\Kom^b(\Aa_{2n-1}$-$\text{p$_{r}$g$_{r}$mod})$};

\coordinate (tbB') at ($(tbB.east) + (0,-1)$);
\coordinate (cbB') at ($(cbB.east) + (0,1)$);
\coordinate (tbA') at ($(tbA.west) + (1,-1)$);
\coordinate (cbA') at ($(cbA.west) + (0,1)$);

\draw [->,shorten >=-1.5pt, dashed] (tbB') arc (245:-70:2.5ex);
\draw [->,shorten >=-1.5pt, dashed] (cbB') arc (-245:70:2.5ex);
\draw [->, shorten >=-1.5pt, dashed] (tbA') arc (-65:250:2.5ex);
\draw [->,shorten >=-1.5pt, dashed] (cbA') arc (65:-250:2.5ex);

\draw[->] (cB) -- (KB) node[midway, left]{$L_B$};
\draw[->] (cB) -- (cA) node[midway,above]{$\mathfrak{m}$}; 
\draw[->] (cA) -- (KA) node[midway,right]{$L_A$};
\draw[->] (KB) -- (KA) node[midway,above]{$\Aa_{2n-1} \otimes_{\Ba_n} -$};
\end{tikzpicture}
\caption{Main theorem}
\label{full picture}
\end{figure}

We aim to show that the diagram is commutative and the four maps on the square are $\mathcal{A}(B_n)$-equivariant; recall that the $\mathcal{A}(B_n)$-actions on curves of $\D^A_{2n}$ and $\Kom^b(A_{2n-1}$-$\text{p$_{r}$g$_{r}$mod})$ are given through the injection $\mathcal{A}(B_n) \xra{\Psi} \mathcal{A}(A_{2n-1})$ introduced in \cref{topology}.

In \cref{topology}, we introduced and showed that $\mathfrak{m}$ is $\cA(B_n)$-equivariant.
In \cref{relating categorical b a action}, we showed that the functor $(\Aa_{2n-1} \otimes_{\Ba_n} -)$ is also $\cA(B_n)$-equivariant.
By \cref{L_A equivariant}, the map $L_A$ is $\cA(A_{2n-1})$-equivariant and therefore $\cA(B_n)$-equivariant.
Thus, we are left to show that the diagram commutes and that $L_B$ is $\cA(B_n)$-equivariant.
The proof of these two statements, first of which is technical, will occupy the rest of this section.

\begin{proposition} \label{diagram commutes}
The diagram in \cref{full picture} commutes, i.e. for each trigraded admissible curve $\check{c}$ in $\D^B_{n+1}$ we have that 
$
\Aa_{2n-1}\otimes_{\Ba_n} L_B(\check{c})\cong L_A (\mathfrak{m}(\check{c}))
$
in $\Kom^b(\Aa_{2n-1}$-$p_rg_rmod)$.
\end{proposition}
\begin{proof}
Let $x$ be any $j$-crossing of $\check{c}$.
If $j \geq 2$, we have that $\mathfrak{m}(x)$ consists of a $(n+(j-1))$-crossing $\wt{x}$ and a $(n-(j-1))$-crossing $\undertilde{x}$ of $\mathfrak{m}(\check{c})$; 
if $j=1$, $\mathfrak{m}(x)$ consists of a single $n$-crossing $\undertilde{\wt{x}}$ of $\mathfrak{m}(\check{c})$.
In either cases, we have isomorphisms $\Phi_j: \Aa_{2n-1} \otimes_{\Ba_n} P^B(x) \ra P^A(\wt{x}) \oplus P^A(\undertilde{x}) = P^A_{n+(j-1)} \oplus P^A_{n-(j-1)}$ or $\Phi_1: \Aa_{2n-1} \otimes_{\Ba_n} P^B(x) \ra P^A(\undertilde{\wt{x}}) = P^A_n$ given in the proof of \cref{B tensor to A}.
Putting together these isomorphisms for each crossing $x$ of $\check{c}$, we obtain a cohomological and internal grading preserving isomorphism of $\Aa_{2n-1}$-modules between the underlying modules of $\Aa_{2n-1} \otimes_{\Ba_n} L_B(\check{c})$ and $L_A (\mathfrak{m}(\check{c}))$; denote this isomorphism as $\eta$.
Denoting the complexes $\Aa_{2n-1} \otimes_{\Ba_n} L_B(\check{c})$ as $(Q, \d)$ and $L_A (\mathfrak{m}(\check{c}))$ as $(Q', \d')$ (so $\eta$ is an isomorphism from $Q$ to $Q'$), it follows that $\eta$ induces an isomorphism of complexes:
\[
(Q, \d) \cong (Q',\d_0),
\]
with $\d_0 = \eta \d \eta^{-1}$.
We now aim to show that $(Q', \d_0) \cong (Q', \d')$ in $\Kom^b(\Aa_{2n-1}$-$\text{p$_{r}$g$_{r}$mod})$.
In fact, we will show that they are isomorphic in the ordinary category $\Com^b(\Aa_{2n-1}$-$\text{p$_{r}$g$_{r}$mod})$ of complexes in $\Aa_{2n-1}$-$\text{p$_{r}$g$_{r}$mod}$.

\paragraph{(\textbf{Slicing $c$})}\label{slicing}
Recall that $c$, being admissable, must have both of its end points at two \emph{distinct} marked points; so at least one of its end points is at a marked point in $\Delta \setminus \{0\}$.
Fix such an end point and call it $m$.
Orient the curve $c$ so that it starts from $m$ and ends at its other end point.
Following this orientation, we can slice $c$ into distinct connected components $c_j \subset c \cap \left( \cup_{j \geq 2} D_j \right)$ and $g_{j'} \subset c \cap \left( D_0 \cup D_1 \right)$ (note that $g_{j'}$ are the 1-strings of $c$), where we start from $c_0$ to $g_0$ to $c_1$ and so on if $m \in \Delta \setminus \{0,1\}$;
for $m = 1$, we instead start from $g_0$ to $c_0$ to $g_1$ and so on. 
Following the same orientation, for $m \in \Delta \setminus \{0,1\}$ we will also enumerate the $2$-crossings of $c$ as $r_0, r_1$ and so on, whereas for $m=1$ we will enumerate the $2$-crossings as $r_1, r_2, r_2$ and so on instead. 
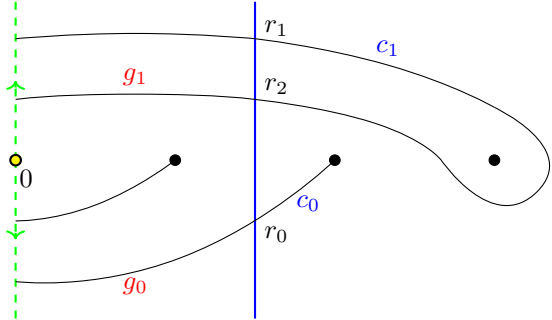
\begin{figure}[H] 
\centering 
\begin{tikzpicture} [scale=.7]

\draw[thick,green, dashed]  (0,5.5)--(0,3.875);
\draw[thick, green, dashed]   (0,-.5)--(0,1.125);
\draw[thick, green, dashed, ->]   (0,2.5)--(0,4);
\draw[thick,  green, dashed, ->]   (0,2.5)--(0,1);
\draw[fill] (3,2.5) circle [radius=0.1]  ;
\draw[fill] (6,2.5) circle [radius=0.1]  ;
\draw[fill] (9,2.5) circle [radius=0.1]  ;
\draw[thick, blue] (4.5,-.5) -- (4.5,5.5);

\draw  plot[smooth, tension=1]coordinates { (0,4.8) (2.25, 4.9) (4.5,4.8)  };
\draw  plot[smooth, tension=1.3]coordinates { (0, 3.65) (2.25, 3.75) (4.5,3.65) };
\draw  plot[smooth, tension=1]coordinates { (0, 1.35) (1.5, 1.65) (3, 2.5)};
\draw  plot[smooth, tension=1]coordinates { (0, .2) (3.1, .65) (6, 2.5)};
\draw   plot[smooth, tension=3]coordinates { (8, 2.5) (9.8, 1.9)  (9, 3.5)};
\draw   plot[smooth, tension=1]coordinates { (9, 3.5) (7, 4.3) (4.5,4.8) };
\draw   plot[smooth, tension=1]coordinates { (8, 2.5) (6.75, 3.2) (4.5,3.65)};

\filldraw[color=black!, fill=yellow!, thick]  (0,2.5) circle [radius=0.1];

\node [above, blue] at (7, 4.3) {$c_1$};
\node [below, red] at (2.25, .45) {$g_0$};
\node [above, red] at (2.25,3.75) {$g_1$};
\node [below] at (0.2, 2.5) {$0$};
\node [below, blue] at (5.5, 2) {$c_0$};
\node [right] at (4.5, 5) {$r_1$};
\node [right] at (4.5, 3.9) {$r_2$};
\node [right] at (4.5, 1.1) {$r_0$};

\end{tikzpicture}
\caption{{\small Example of slicing a curve with one of its endpoints in $\Delta \setminus \{0,1\}$.}} \label{notend01}
\end{figure}
\begin{figure}[H]  
\centering
\begin{tikzpicture} [scale=.7]

\draw[thick,green, dashed]  (12,5.5)--(12,3.875);
\draw[thick, green, dashed]   (12,-.5)--(12,1.125);
\draw[thick, green, dashed, ->]   (12,2.5)--(12,4);
\draw[thick,  green, dashed, ->]   (12,2.5)--(12,1);
\draw[fill] (15,2.5) circle [radius=0.1]  ;
\draw[fill] (18,2.5) circle [radius=0.1]  ;
\draw[fill] (21,2.5) circle [radius=0.1]  ;
\draw[thick, blue] (16.5,-.5) -- (16.5,5.5);

\draw  plot[smooth, tension=1]coordinates { (12,4.5) (14.25, 4.9) (16.5,5)  };
\draw  plot[smooth, tension=1.3]coordinates { (12, 3.65) (14.25, 4.1) (16.5,4.2) };
\draw  plot[smooth, tension=1]coordinates { (12, 1.35) (13.5, 1.65) (15, 2.5)};
\draw  plot[smooth, tension=1]coordinates { (12,.5 ) (13.8, .7) (16.5, 1.2)};
\draw   plot[smooth, tension=3]coordinates { (20.25, 2.5) (21.8, 1.9)  (21, 3.5)};
\draw   plot[smooth, tension=1]coordinates { (21, 3.5) (19, 4.5) (16.5,5) };
\draw   plot[smooth, tension=1]coordinates { (20.25, 2.5) (18.75, 3.7) (16.5,4.2)};
\draw  plot[smooth, tension=1.3]coordinates { (12, 2.5) (13.8, 3.1) (16.5,3.4) };
\draw  plot[smooth, tension=1.5]coordinates { (16.5, 3.4) (19, 2.5) (16.5,1.2) };

\filldraw[color=black!, fill=yellow!, thick]  (12,2.5) circle [radius=0.1];

\node [above, blue] at (19, 3.7) {$c_1$};
\node [above, red] at (13.8, 1.8) {$g_0$};
\node [above, red] at (14.25,4.9) {$g_1$};
\node [above, red] at (14.25,3.2) {$g_2$};
\node [below] at (12.2, 2.5) {$0$};
\node [below, blue] at (18, 1.5) {$c_2$};
\node [right] at (16.5, 5.2) {$r_1$};
\node [right] at (16.5, 4.4) {$r_2$};
\node [right] at (16.5, 1) {$r_0$};
\node [right] at (16.5, 3.6) {$r_3$};

\end{tikzpicture}
\caption{{\small Example of slicing a curve with its endpoints in $ \{0,1\}$.}} \label{end01}
\end{figure}
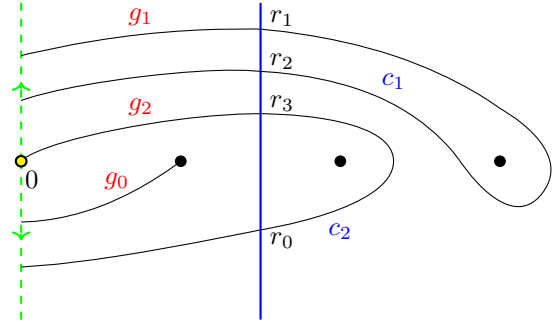
Now consider the following subsets of (graded) crossings of $\mathfrak{m}(\check{c})$:

\begin{equation} \label{subset crossing}
\begin{cases}
C_j := \mathfrak{m}(\check{c}_j) \cap (\bigcup_i \ddot{\theta}_i); \\
\bar{C}_j := \mathfrak{m}(\check{c}_j) \cap (\bigcup_{i\neq n-1, n+1} \ddot{\theta}_i);\\
G_{j'} := \mathfrak{m}(\check{g}_{j'}) \cap (\ddot{\theta}_{n-1} \cup \ddot{\theta}_n \cup \ddot{\theta}_{n+1}); \\
\bar{G}_{j'} := \mathfrak{m}(\check{g}_{j'}) \cap \ddot{\theta}_n; \text{ and} \\
R_j := \mathfrak{m}(\check{r}_j).
\end{cases}
\end{equation}

Note that by definition, the subsets of crossings $\bar{C}_i, R_j$ and $\bar{G}_k$ are pairwise disjoint, and  $\left( \coprod_j \bar{C}_j\right) \amalg \left( \coprod_j R_j \right) \amalg \left( \coprod_j \bar{G}_j \right)$ is the set of all crossings of $\mathfrak{m}(\check{c})$.

When $m \in \Delta \setminus \{0, 1\}$, the slicing of $c$ is of the form $c_0$ to $g_0$ to $c_1$ to $g_1$ and so on, so we get that

\[
\bar{C}_0 \amalg R_0 = C_0, \quad 
R_0 \amalg \bar{G}_0 \amalg R_1 = G_0, \quad
R_1 \amalg \bar{C}_1 \amalg R_2 = C_1, \quad \dots
\]

Following this slicing, we decompose $Q'$ as
\begin{equation}\label{R decomp}
Q' = 
\rlap{$\underbrace{\phantom{Q'_{\bar{C}_0} \oplus Q'_{R_0}}}_{Q'_{C_0}}$} 
	Q'_{\bar{C}_0} \oplus 
\rlap{$\overbrace{\phantom{Q'_{R_0} \oplus Q'_{\bar{G}_0} \oplus Q'_{R_1}}}^{Q'_{G_0}}$}
	Q'_{R_0} \oplus Q'_{\bar{G}_0} \oplus 
\rlap{$\underbrace{\phantom{Q'_{R_1} \oplus Q'_{\bar{C}_1} \oplus Q'_{R_2}}}_{Q'_{C_1}}$}
	Q'_{R_1} \oplus Q'_{\bar{C}_1} \oplus Q'_{R_2} \oplus \dots
\end{equation}
where $Q'_K := \bigoplus_{x \in K} P^A(x) \subseteq Q'$.
If $m =1$, the slicing of $c$ is instead of the form $g_0 $ to $c_1$ to $g_1$ to $c_2$ and so on, so we get that

\[
\bar{G}_0 \amalg R_1 = G_0, \quad 
R_1 \amalg \bar{C}_1 \amalg R_2 = C_1, \quad
R_2 \amalg \bar{G}_2 \amalg R_3 = G_2, \quad \dots
\]

\noindent Following this slicing, we instead decompose $Q'$ as

\begin{equation}\label{R' decomp}
Q' = 
\rlap{$\underbrace{\phantom{Q'_{\bar{G}_0} \oplus Q'_{R_1}}}_{Q'_{G_0}}$} 
	Q'_{\bar{G}_0} \oplus 
\rlap{$\overbrace{\phantom{Q'_{R_1} \oplus Q'_{\bar{C}_1} \oplus Q'_{R_2}}}^{Q'_{C_1}}$}
	Q'_{R_1} \oplus Q'_{\bar{C}_1} \oplus 
\rlap{$\underbrace{\phantom{Q'_{R_2} \oplus Q'_{\bar{G}_2} \oplus Q'_{R_3}}}_{Q'_{G_2}}$}
	Q'_{R_2} \oplus Q'_{\bar{G}_2} \oplus Q'_{R_3} \oplus \dots
\end{equation}

\noindent where $Q'_K := \bigoplus_{x \in K} P^A(x) \subseteq Q'$.

In general, given a decomposition of modules $M = \oplus_{i\in Y} M_{K_i}$ and a complex $(M, \partial)$, we can then write $\partial$ as a block matrix.
We will use the notation $\partial_{K_i K_j}$ to denote the block of $\partial$ that maps from $M_{K_j}$ to $M_{K_i}$, where we use the shorthand notation $\partial_{K_i}$ for the block of $\partial$ that maps from $M_{K_i}$ to itself.
We will also use the notation $\partial_{\oplus_{i \in X \subseteq Y} K_i}$ for the block of $\partial$ that maps from $\oplus_{i \in X} M_{K_i}$ to itself.
To illustrate, consider the decomposition of $Q'$ as in \ref{R decomp} and let $(Q', \partial)$ be a cochain complex with differential $\partial$.
We will then write the differential $\partial$ as the matrix
\begin{center}
\begin{tikzpicture}
\matrix (m) [mymatrix, inner sep = 2pt, column sep = 7pt, row sep = 5pt] {
\partial_{\bar{C}_0} & \partial_{R_0 \bar{C}_0} & \partial_{\bar{G}_0 \bar{C}_0} & \partial_{R_1 \bar{C}_0}  & \partial_{\bar{C}_1 \bar{C}_0} & \dots \\
 \partial_{\bar{C}_0 R_0}& \partial_{R_0} & \partial_{\bar{G}_0 R_0} & \partial_{R_1 R_0} &  \partial_{\bar{C}_1 R_0} & \dots \\
\partial_{\bar{C}_0 \bar{G}_0}  & \partial_{R_0 \bar{G}_0}  & \partial_{\bar{G}_0} & \partial_{R_1 \bar{G}_0} &  \partial_{\bar{C}_1 \bar{G}_0} & \dots \\
\partial_{\bar{C}_0 R_1} & \partial_{R_0 R_1}  & \partial_{\bar{G}_0 R_1}   & \partial_{R_1} & \partial_{\bar{C}_1 R_1} & \dots \\
\partial_{\bar{C}_0 \bar{C}_1}  & \partial_{R_0 \bar{C}_1}&\partial_{\bar{G}_0 \bar{C}_1} &\partial_{R_1 \bar{C}_1}  & \partial_{\bar{C}_1} & \dots \\
\vdots & \vdots & \vdots & \vdots & \vdots & \ddots \\
};

\draw[red] (m-1-1.north-|m-2-1.west) rectangle (m-2-2.east|-m-2-2.south);
\draw[blue] (m-2-2.north-|m-4-2.west) rectangle (m-4-4.east|-m-4-4.south);
\draw[red] (m-4-5.north east)-|(m-5-4.south west);
\end{tikzpicture}.
\end{center}
The blocks $\partial_{C_j}$ corresponding to the components $Q'_{C_j}$ are the blocks in red and similarly blocks $\partial_{G_j}$ corresponding to the components $Q'_{G_j}$ are the blocks in blue.

Let us now analyse the difference between the two differentials $\d_0$ and $\d'$.
First consider the case when $m \in \Delta\setminus \{0,1\}$ and we have a decomposition $Q'$ as in \ref{R decomp}.
By looking at how the components $c_i, r_j$ and $g_k$ are connected, it follows that $\d_0$ and $\d'$ are given by the matrices
\[
\d_0 = 
\begin{bmatrix}
\d_{0_{\bar{C}_0}} & \d_{0_{R_0} Q'_{\bar{C}_0}} 
	& 0 & 0 & 0 & \dots \\
\d_{0_{\bar{C}_0 R_0}} & 0 & \d_{0_{\bar{G}_0} Q'_{R_0}}
	    & 0 & 0 & \dots \\
0 
	& \d_{0_{R_0} {\bar{G}_0}}   & \d_{0_{\bar{G}_0}} & \d_{0_{R_1 \bar{G}_0}}
		    & 0 & \dots \\
0 & 0 
	& \d_{0_{\bar{G}_0 R_1}}  & 0 & \d_{0_{\bar{C}_1 R_1}}  & \dots \\
0 & 0 & 0 
	& \d_{0_{R_1 \bar{C}_1}} & \d_{0_{\bar{C}_1}} & \dots \\
\vdots & \vdots & \vdots & \vdots & \vdots & \ddots
\end{bmatrix}
\]
and
\[
\d' = \begin{bmatrix}
\d'_{\bar{C}_0} & \d'_{R_0 \bar{C}_0}  & 0 & 0 & 0 & \dots \\
\d'_{\bar{C}_0 R_0}& 0 & \d'_{\bar{G}_0 R_0}  & 0 & 0 & \dots \\
0 & \d'_{R_0 \bar{G}_0}  & \d'_{\bar{G}_0}  & \d'_{R_1 \bar{G}_0} & 0 & \dots \\
0 & 0 &  \d'_{\bar{G}_0 R_1}  & 0 &  \d'_{\bar{C}_1 R_1} & \dots \\
0 & 0 & 0 &  \d'_{R_1 \bar{C}_1} & \d'_{\bar{C}_1}  & \dots \\
\vdots & \vdots & \vdots & \vdots & \vdots & \ddots
\end{bmatrix}
\]
respectively.
Furthermore, a direct computation shows that

\begin{equation} \label{C diff component}
\d_{0_{C_j}} = \d'_{C_j}, \quad \text{for all } j.
\end{equation}

So between the differentials $\d_0$ and $\d'$, only the differential components $\d_{0_{G_i}}$ and $\d'_{0_{G_i}}$ may differ.
The case for when $m = 1$ follows with the same argument.
If we were in the case when $c$ has no 1-string, i.e. $c = c_0$, then we are done.
We shall assume otherwise for the rest of the proof.

Let $\check{g}_0, \dots , \check{g}_{s-1}$ be the 1-strings of $\check{c}$.
To show $(Q', \d_0) \cong (Q', \d')$, \cref{C diff component} shows that it is sufficient to construct a cohomological and internal grading preserving isomorphism of modules $\mu_i: Q' \ra Q'$ for each $ 0 \leq i \leq s-1$, such that we have an induced chain of isomorphisms in $\Kom^b(\Aa_{2n-1}\text{p$_{r}$g$_{r}$mod})$(these will actually be isomorphism of complexes without homotopy):
$
(Q', \d_0)  \cong (Q', \d_1) \cong \cdots  \cong (Q', \d_{s-1}) \cong (Q', \d_s) = (Q', \d');
$
with $\d_{i+1} := \mu_i \d_i \mu_i^{-1}$ and where each $\d_{i+1}$ for $0 \leq i \leq s-1$ satisfies the following property:

\[
(*)
\begin{cases}
\d_{{i+1}_{G_i}} = \d'_{G_i}; \\
\d_{{i+1}_{G_j}} = \d_{i_{G_j}}, \qquad \text{for all $j\neq i$; and}\\
\d_{{i+1}_{C_j}} = \d_{i_{C_j}}, \qquad \text{for all $j$.} \\
\end{cases}
\]

In other words, $\mu_i$ will be constructed in a way that conjugation of $\d_i$ by $\mu_i^{-1}$ \emph{only} alters the differential component $\d_{i_{G_i}}$, so that $\d_{{i+1}_{G_i}} = \mu\d_{i_{G_i}}\mu_i^{-1} = \d'_{G_i}$.
In particular, property $(*)$ will guarantee that $\d_s = \d'$.

What remains is to define the required $\mu_i : Q' \ra Q'$. 
We will define $\mu_i$ according to the type of 1-string $\check{g}_i$.
Within each possible types of 1-string $\check{g}_i$, we will show the following:
\begin{enumerate}
\item For $i=0$, we show that can always construct $\mu_0$ to get $(Q'_0, \d_0) \cong (Q'_1, \d_1)$ such that $\d_1 = \mu_0\d_0 \mu_0^{-1}$ satisfies $(*)$.
\item For $i\geq 1$, we show that given $(Q', \d_0) \cong \cdots \cong (Q'_i,\d_i)$ with $\d_1, ..., \d_i$ satisfying $(*)$ for $i \geq 1$, we can construct $ \mu_i: Q' \ra Q' $ such that $\d_{i+1} = \mu_i \d_i \mu_i^{-1}$ satisfies $(*)$.
\end{enumerate}
By an induction on $i$, this shows that we can always construct a chain of isomorphisms $\mu_i$ and hence completes the proof of this theorem.

As from now on we will be focusing on the module component $Q'_{G_i}$, let us simplify both the decompositions in \ref{R decomp} and \ref{R' decomp} as 

\begin{equation} \label{new R decomp}
Q' = Q'_{V_i} \oplus 
	\rlap{$\underbrace{
		\phantom{
			Q'_{R_i} \oplus Q'_{\bar{G_i}} \oplus Q'_{R_{i+1}}
		}
	}_{Q'_{G_i}}$}
	Q'_{R_i} \oplus Q'_{\bar{G_i}} \oplus Q'_{R_{i+1}} \oplus Q'_{W_i},
\end{equation}

\noindent where $Q'_{V_i}$ (resp. $Q'_{W_i}$) consists of all the modules before $Q'_{R_i}$ (resp. after $Q'_{R_{i+1}}$) for both decompositions \ref{R decomp} and \ref{R' decomp}.

The list below shows that given each type of $\check{g}_i$, we can construct $ \mu_i: Q' \ra Q' $  accordingly such that $\d_{i+1} = \mu_i \d_i \mu_i^{-1}$ satisfies $(*)$.
We shall start with the simple cases: types IV, III$'_k$ and II$'_k$, followed by the types V$'$, III$'_{k + \frac{1}{2}}$ and II$'_{k + \frac{1}{2}}$ that requires some further analysis (see \cref{B 1-string} for the list of possible types).
Within the list, we will omit the gradings when writing out the modules in $Q'$ since the maps chosen for $\mu_i$ will always be cohomological and internal grading preserving.
We will also use black arrows for differentials and {\color{blue} blue} arrows for the isomorphism $\mu_i : Q' {\color{blue}\ra} Q'$.
\begin{enumerate} 
\item $\check{g}_i$ is of Type VI: \\
This case is only possible when $i=0$, and $g_i = c$.
But in this case $(Q', \d_0)$ is given by $0 \ra P_n^A \ra 0$, which is already the complex $(Q', \d')$ as required.

\item $\check{g}_i$ is of Type III$'_k$ for $k \in \Z$: \\
The case when $k=0$ is straightforward, where we just pick $\mu_i$ to be the identity.

Now consider the case when $k > 0$.
If $i=0$, we get that $\d_{i_{G_i}} = \d_{0_{G_i}}$.
For $i \geq 1$, assuming that $\d_i$ satisfies $(*)$, we can conclude that $\d_{i_{G_i}} = \d_{0_{G_i}}$.
Thus for all $i$, we can draw the part of $(Q', \d_i)$ that contains $(Q'_{G_i}, \d_{i_{G_i}}) = (Q'_{G_i}, \d_{0_{G_i}})$ as follows:

\begin{center}
\begin{tikzpicture}[>=stealth, baseline]% added "baseline"
\matrix (M) [matrix of math nodes, column sep=7mm]
{
P_n^A  & P_n^A  & \cdots & P_n^A  & P_n^A  & P_{n-1}^A & \nabla\\
\oplus & \oplus &        & \oplus & \oplus & \oplus    &       & = (Q', \d_i)\\
P_n^A  & P_n^A  & \cdots & P_n^A  & P_n^A  & P_{n+1}^A &       \\
};

\draw[-> ,font=\small](M-1-1.east |- M-1-2) -- (M-1-2) 
	node[midway,above] {$2X_n$};% changed coordinates for arrow
\draw[-> ,font=\small](M-1-2.east |- M-1-3) -- (M-1-3)
	;
\draw[-> ,font=\small](M-1-3.east |- M-1-4) -- (M-1-4)
	;
\draw[-> ,font=\small](M-1-4.east |- M-1-5) -- (M-1-5)
	node[midway,above] {$2X_n$}; 
\draw[-> ,font=\small](M-1-5) -- (M-1-6)
	node[midway,above] {$F$}; 
\draw[-> ,font=\small](M-1-5) -- (M-3-6)
	;
\draw[<->](M-1-6.east |- M-1-7) -- (M-1-7)
	;
\draw[<->](M-3-6) -- (M-1-7)
	;

\draw[-> ,font=\small](M-3-1.east |- M-3-2) -- (M-3-2) 
	node[midway,above] {$2X_n$};% changed coordinates for arrow
\draw[-> ,font=\small](M-3-2.east |- M-3-3) -- (M-3-3)
	;
\draw[-> ,font=\small](M-3-3.east |- M-3-4) -- (M-3-4)
	;
\draw[-> ,font=\small](M-3-4.east |- M-3-5) -- (M-3-5)
	node[midway,above] {$2X_n$};
\draw[-> ,font=\small](M-3-5) -- (M-1-6)
	; 
\draw[-> ,font=\small](M-3-5) -- (M-3-6)
	;
\end{tikzpicture}
\end{center}

where $ F =
\begin{bmatrix}
(n|n-1) & -(n|n-1)i \\
(n|n+1) &  (n|n+1)i
\end{bmatrix}
=
\begin{bmatrix}
(n|n-1) & 0 \\
0       & (n|n+1)
\end{bmatrix}
\begin{bmatrix}
1 & -i \\
1 &  i
\end{bmatrix}
$
and where $\nabla$ denotes the rest of the complex $(Q', \d_i)$ containing the modules complement to $Q'_{G_i}$.
In particular, for $i=0$, $\nabla$ is the part of $(Q', \d_0)$ that contains the module $Q'_{W_0}$;
if $i\geq 1$, then this case is only possible when $i=s-1$ and $\nabla$ is the part of $(Q', \d_0)$ that contains the module $Q'_{V_{s-1}}$.
Nevertheless, the construction of $\mu_i$ below depends only on the form above, so the construction will work for all $i$.

Denote $M =
\begin{bmatrix}
1 & -i \\
1 &  i
\end{bmatrix}$ 
and $I$ the $2\times 2$ identity matrix.
We define $\mu_i|_{Q'_{G_i}}$ to be the following map in {\color{blue}blue}, with $\mu_i$ acting as the identity map on $Q'_{V_i}$:

\begin{center}
\begin{tikzpicture}[scale = 0.8][>=stealth, baseline]% added "baseline"
\matrix (M) [matrix of math nodes, column sep=7mm]
{
P_n^A  & P_n^A  & \cdots & P_n^A  & P_n^A  & P_{n-1}^A & \nabla\\
\oplus & \oplus &        & \oplus & \oplus & \oplus    &       & = (Q', \d_i)\\
P_n^A  & P_n^A  & \cdots & P_n^A  & P_n^A  & P_{n+1}^A &       \\
{}     &        &        &        &        &           &       \\
{}     &        &        &        &        &           &       \\
{}     &        &        &        &        &           &       \\
P_n^A  & P_n^A  & \cdots & P_n^A  & P_n^A  & P_{n-1}^A & \nabla\\
\oplus & \oplus &        & \oplus & \oplus & \oplus    &       & =: (Q', \d_{i+1})\\
P_n^A  & P_n^A  & \cdots & P_n^A  & P_n^A  & P_{n+1}^A &       \\
};

%%Arrows for first complex
\draw[-> ,font=\small](M-1-1.east |- M-1-2) -- (M-1-2) 
	node[midway,above] {$2X_n$};% changed coordinates for arrow
\draw[-> ,font=\small](M-1-2.east |- M-1-3) -- (M-1-3)
	;
\draw[-> ,font=\small](M-1-3.east |- M-1-4) -- (M-1-4)
	;
\draw[-> ,font=\small](M-1-4.east |- M-1-5) -- (M-1-5)
	node[midway,above] {$2X_n$}; 
\draw[-> ,font=\small](M-1-5) -- (M-1-6)
	; 
\draw[-> ,font=\small](M-1-5) -- (M-3-6)
	;
\draw[<->](M-1-6.east |- M-1-7) -- (M-1-7)
	;
\draw[<->](M-3-6) -- (M-1-7)
	;

\draw[-> ,font=\small](M-3-1.east |- M-3-2) -- (M-3-2) 
	node[midway,above] {$2X_n$};% changed coordinates for arrow
\draw[-> ,font=\small](M-3-2.east |- M-3-3) -- (M-3-3)
	;
\draw[-> ,font=\small](M-3-3.east |- M-3-4) -- (M-3-4)
	;
\draw[-> ,font=\small](M-3-4.east |- M-3-5) -- (M-3-5)
	node[midway,above] {$2X_n$};
\draw[-> ,font=\small](M-3-5) -- (M-1-6)
	; 
\draw[-> ,font=\small](M-3-5) -- (M-3-6)
	;

%%Arrows for second complex
\draw[-> ,font=\small](M-7-1.east |- M-7-2) -- (M-7-2) 
	node[midway,above] {$X_n$};% changed coordinates for arrow
\draw[-> ,font=\small](M-7-2.east |- M-7-3) -- (M-7-3)
	;
\draw[-> ,font=\small](M-7-3.east |- M-7-4) -- (M-7-4)
	;
\draw[-> ,font=\small](M-7-4.east |- M-7-5) -- (M-7-5)
	node[midway,above] {$X_n$}; 
\draw[-> ,font=\small](M-7-5) -- (M-7-6)
	;
\draw[<->](M-7-6.east |- M-7-7) -- (M-7-7)
	;
\draw[<->](M-9-6) -- (M-7-7)
	;

\draw[-> ,font=\small](M-9-1.east |- M-9-2) -- (M-9-2) 
	node[midway,above] {$X_n$};% changed coordinates for arrow
\draw[-> ,font=\small](M-9-2.east |- M-9-3) -- (M-9-3)
	;
\draw[-> ,font=\small](M-9-3.east |- M-9-4) -- (M-9-4)
	;
\draw[-> ,font=\small](M-9-4.east |- M-9-5) -- (M-9-5)
	node[midway,above] {$X_n$};
\draw[-> ,font=\small](M-9-5) -- (M-9-6)
	; 

%%Arrows between first and second complexes
\draw[-> ,font=\small, color=blue](M-3-1) -- (M-7-1) 
	node[midway,left,scale=0.8] {$2^{k-1}M$};
\draw[-> ,font=\small, color=blue](M-3-2) -- (M-7-2) 
	node[midway,right,scale=0.8] {$2^{k-2}M$};
\draw[-> ,font=\small, color=blue](M-3-4) -- (M-7-4) 
	node[midway,left,scale=0.8] {$2M$};
\draw[-> ,font=\small, color=blue](M-3-5) -- (M-7-5) 
	node[midway,left,scale=0.8] {$M$};
\draw[-> ,font=\small, color=blue](M-3-6) -- (M-7-6) 
	node[midway,left,scale=0.8] {$I$};
\end{tikzpicture}
\end{center}

The black arrows in the last two rows shows the differential component $\d_{{i+1}_{G_i}}$ in $\d_{i+1}$, induced by the conjugation of $\mu_i^{-1}$.
Hence, the required condition $(*)$ follows directly.

Now consider when $k<0$.
As before, we have that $\d_{i_{G_i}} = \d_{0_{G_i}}$ for all $i$, so the analysis of the part of $(Q', \d_i)$ that contains $(Q'_{G_i}, \d_{i_{G_i}})$ will be the same. 
Similarly the construction of $\mu_i$ below will work for both cases.
We draw the part of $(Q', \d_i)$ that contains $(Q'_{G_i}, \d_{i_{G_i}}) = (Q'_{G_i}, \d_{0_{G_i}})$ as follows:

\begin{center}
\begin{tikzpicture} [scale= 0.8][>=stealth, baseline]% added "baseline"
\matrix (M) [matrix of math nodes, column sep=7mm]
{
       & P_{n+1}^A & P_n^A  & P_n^A  & \cdots & P_n^A  & P_n^A  \\
       & \oplus    & \oplus & \oplus &        & \oplus & \oplus & =(Q', \d_i)\\
\nabla & P_{n-1}^A & P_n^A  & P_n^A  & \cdots & P_n^A  & P_n^A  \\
};

\draw[-> ,font=\small](M-1-2.east |- M-1-3) -- (M-1-3) 
	node[midway,above] {$E$};
\draw[-> ,font=\small](M-1-2) -- (M-3-3) 
	;
\draw[-> ,font=\small](M-1-3.east |- M-1-4) -- (M-1-4) 
	node[midway,above] {$2X_n$};% changed coordinates for arrow
\draw[-> ,font=\small](M-1-4.east |- M-1-5) -- (M-1-5)
	;
\draw[-> ,font=\small](M-1-5.east |- M-1-6) -- (M-1-6)
	;
\draw[-> ,font=\small](M-1-6.east |- M-1-7) -- (M-1-7)
	node[midway,above] {$2X_n$};
\draw[<->](M-3-1) -- (M-1-2)
	;

\draw[-> ,font=\small](M-3-2) -- (M-1-3)
	;
\draw[-> ,font=\small](M-3-2) -- (M-3-3)
	;
\draw[-> ,font=\small](M-3-3.east |- M-3-4) -- (M-3-4) 
	node[midway,above] {$2X_n$};% changed coordinates for arrow
\draw[-> ,font=\small](M-3-4.east |- M-3-5) -- (M-3-5)
	;
\draw[-> ,font=\small](M-3-5.east |- M-3-6) -- (M-3-6)
	;
\draw[-> ,font=\small](M-3-6.east |- M-3-7) -- (M-3-7)
	node[midway,above] {$2X_n$};
\draw[<->](M-3-1) -- (M-3-2)
	;

\end{tikzpicture}
\end{center}

with $
E = 
\begin{bmatrix}
-(n+1|n)i &  (n-1|n)i \\
(n+1|n)   &  (n-1|n)
\end{bmatrix}
$ 
and where $\nabla$ denotes the rest of the complex $(Q', \d_i)$ containing the modules complement to $Q'_{G_i}$.
Denote $N =
\begin{bmatrix}
i  & 1 \\
-i & 1
\end{bmatrix}$ 
and $I$ the $2\times 2$ identity matrix.
Note that 

\[
N
\begin{bmatrix}
-(n+1|n)i &  (n-1|n)i \\
(n+1|n)   &  (n-1|n)
\end{bmatrix}
=
2
\begin{bmatrix}
(n+1|n) &   0      \\
0       &  (n-1|n)
\end{bmatrix}.
\]

We define $\mu_i|_{Q'_{G_i}}$ to be the following map in {\color{blue}blue}, with $\mu_i$ act as the identity map on $Q'_{V_i}$:

\begin{center}
\begin{tikzpicture}[scale=0.8][>=stealth, baseline]% added "baseline"
\matrix (M) [matrix of math nodes, column sep=7mm]
{
       & P_{n+1}^A & P_n^A  & P_n^A  & \cdots & P_n^A  & P_n^A  \\
       & \oplus    & \oplus & \oplus &        & \oplus & \oplus & =(Q', \d_i)\\
\nabla & P_{n-1}^A & P_n^A  & P_n^A  & \cdots & P_n^A  & P_n^A  \\
{}     &           &        &        &        &        &        \\
{}     &           &        &        &        &        &        \\
{}     &           &        &        &        &        &        \\
       & P_{n+1}^A & P_n^A  & P_n^A  & \cdots & P_n^A  & P_n^A  \\
       & \oplus    & \oplus & \oplus &        & \oplus & \oplus & =:(Q', \d_{i+1})\\
\nabla & P_{n-1}^A & P_n^A  & P_n^A  & \cdots & P_n^A  & P_n^A .\\
};

%% Arrows for first complex
\draw[-> ,font=\small](M-1-2.east |- M-1-3) -- (M-1-3) 
	;
\draw[-> ,font=\small](M-1-2) -- (M-3-3) 
	;
\draw[-> ,font=\small](M-1-3.east |- M-1-4) -- (M-1-4) 
	node[midway,above] {$2X_n$};% changed coordinates for arrow
\draw[-> ,font=\small](M-1-4.east |- M-1-5) -- (M-1-5)
	;
\draw[-> ,font=\small](M-1-5.east |- M-1-6) -- (M-1-6)
	;
\draw[-> ,font=\small](M-1-6.east |- M-1-7) -- (M-1-7)
	node[midway,above] {$2X_n$};

\draw[-> ,font=\small](M-3-2) -- (M-1-3)
	;
\draw[-> ,font=\small](M-3-2) -- (M-3-3)
	;
\draw[-> ,font=\small](M-3-3.east |- M-3-4) -- (M-3-4) 
	node[midway,above] {$2X_n$};% changed coordinates for arrow
\draw[-> ,font=\small](M-3-4.east |- M-3-5) -- (M-3-5)
	;
\draw[-> ,font=\small](M-3-5.east |- M-3-6) -- (M-3-6)
	;
\draw[-> ,font=\small](M-3-6.east |- M-3-7) -- (M-3-7)
	node[midway,above] {$2X_n$};
\draw[<->](M-3-1.east |- M-3-2) -- (M-3-2);
\draw[<->](M-3-1) -- (M-1-2); 

%% Arrows for second complex
\draw[-> ,font=\small](M-7-2.east |- M-7-3) -- (M-7-3) 
	node[midway,above,scale=0.7] {$(n+1|n)$};
\draw[-> ,font=\small](M-7-3.east |- M-7-4) -- (M-7-4) 
	node[midway,above] {$X_n$};% changed coordinates for arrow
\draw[-> ,font=\small](M-7-4.east |- M-7-5) -- (M-7-5)
	;
\draw[-> ,font=\small](M-7-5.east |- M-7-6) -- (M-7-6)
	;
\draw[-> ,font=\small](M-7-6.east |- M-7-7) -- (M-7-7)
	node[midway,above] {$X_n$};

\draw[-> ,font=\small](M-9-2) -- (M-9-3)
	node[midway,above,scale=0.7] {$(n-1|n)$};
\draw[-> ,font=\small](M-9-3.east |- M-9-4) -- (M-9-4) 
	node[midway,above] {$X_n$};% changed coordinates for arrow
\draw[-> ,font=\small](M-9-4.east |- M-9-5) -- (M-9-5)
	;
\draw[-> ,font=\small](M-9-5.east |- M-9-6) -- (M-9-6)
	;
\draw[-> ,font=\small](M-9-6.east |- M-9-7) -- (M-9-7)
	node[midway,above] {$X_n$};
\draw[<->](M-9-1.east |- M-9-2) -- (M-9-2);
\draw[<->](M-9-1) -- (M-7-2); 

%% Arrows between two complexes
\draw[-> ,font=\small, color=blue](M-3-2) -- (M-7-2) 
	node[midway,left,scale=0.8] {$I$};
\draw[-> ,font=\small, color=blue](M-3-3) -- (M-7-3) 
	node[midway,right,scale=0.8] {$2^{-1}N$};
\draw[-> ,font=\small, color=blue](M-3-4) -- (M-7-4) 
	node[midway,right,scale=0.8] {$2^{-2}N$};
\draw[-> ,font=\small, color=blue](M-3-6) -- (M-7-6) 
	node[midway,left,scale=0.8] {$2^{k+1}N$};
\draw[-> ,font=\small, color=blue](M-3-7) -- (M-7-7) 
	node[midway,left,scale=0.8] {$2^{k}N$};

\end{tikzpicture}
\end{center}

The black arrows in the last two rows shows the differential component $\d_{{i+1}_{G_i}}$ in $\d_{i+1}$, induced by the conjugation of $\mu_i^{-1}$.
It follows directly that the required condition $(*)$ is satisfied.

\item $\check{g}_i$ is of Type II$'_k$ for $k \in \Z$: \\
As in Type III$'_k$, we have that $\d_{i_{G_i}} = \d_{0_{G_i}}$ for all $i$, so the part of $(Q', \d_i)$ that contains $(Q'_{G_i}, \d_{i_{G_i}})$ will be of the same form. 
It follows similarly that the construction of $\mu_i$ below will work for all $i$.

We shall start with $k=0$.
We draw the part of $(Q', \d_i)$ that contains $(Q'_{G_i}, \d_{i_{G_i}}) = (Q'_{G_i}, \d_{0_{G_i}})$ as follows:

\begin{center}
\begin{tikzpicture}[>=stealth, baseline]% added "baseline"
\matrix (M) [matrix of math nodes, column sep=7mm]
{
\nabla & P_{n+1}^A & P_{n+1}^A & \nabla' \\
       & \oplus    & \oplus    &       & = (Q', \d_i) \\
       & P_{n-1}^A & P_{n-1}^A &        \\
};

\draw[<->](M-1-1) -- (M-1-2); 
\draw[-> ,font=\small](M-1-2.east |- M-1-3) -- (M-1-3)
	node[midway,above] {$X_{n+1}$};
\draw[<->](M-1-3) -- (M-1-4); 

\draw[<->](M-1-1) -- (M-3-2); 
\draw[-> ,font=\small](M-1-2.east |- M-3-3) -- (M-3-3)
	node[midway,above] {$X_{n-1}$};
\draw[<->](M-3-3) -- (M-1-4); 
\end{tikzpicture}
\end{center}

where either $\nabla$ or $\nabla'$ is the part of $(Q', \d_i)$ that contains the module $Q'_{V_i}$ and the other contains $Q'_{W_i}$.
But in this case we already have that $\d_{i_{G_i}} = \d'_{G_i}$, thus we just choose $\mu_i$ to be the identity map.

For $k > 0$ ,we draw the part of $(Q', \d_i)$ containing $(Q'_{G_i}, \d_{i_{G_i}}) = (Q'_{G_i}, \d_{0_{G_i}})$ as follows:

\begin{center}
\begin{tikzpicture}[>=stealth, baseline]% added "baseline"
\matrix (M) [matrix of math nodes, column sep=7mm]
{
P_n^A  & P_n^A  & P_n^A  & \cdots & P_n^A  & P_n^A  & P_{n-1}^A & \nabla \\
\oplus & \oplus & \oplus &        & \oplus & \oplus & \oplus    & \\
P_n^A  & P_n^A  & P_n^A  & \cdots & P_n^A  & P_n^A  & P_{n+1}^A &        \\
       & \oplus & \oplus &        & \oplus & \oplus &       &  &=(Q', \d_i)\\ 
       & P_n^A  & P_n^A  & \cdots & P_n^A  & P_{n-1}^A & \nabla' \\
       & \oplus & \oplus &        & \oplus & \oplus    &        \\
       & P_n^A  & P_n^A  & \cdots & P_n^A  & P_{n+1}^A &        \\
};

\draw[-> ,font=\small](M-1-1) -- (M-1-2) 
	node[midway,above] {$2X_n$};
\draw[-> ,font=\small](M-1-1) -- (M-5-2) 
	node[midway,below] {$2X_n$};
\draw[-> ,font=\small](M-1-2) -- (M-1-3) 
	node[midway,above] {$2X_n$};
\draw[-> ,font=\small](M-1-3) -- (M-1-4) 
	;
\draw[-> ,font=\small](M-1-4) -- (M-1-5)
	;
\draw[-> ,font=\small](M-1-5) -- (M-1-6)
	node[midway,above] {$2X_n$};
\draw[-> ,font=\small](M-1-6) -- (M-1-7)
	node[midway, above] {$F$};
\draw[-> ,font=\small](M-1-6) -- (M-3-7)
	;
\draw[<-> ,font=\small](M-1-7) -- (M-1-8)
	;
\draw[<-> ,font=\small](M-3-7) -- (M-1-8)
	;

\draw[-> ,font=\small](M-3-1) -- (M-3-2)
	node[midway,above] {$2X_n$};
\draw[-> ,font=\small](M-3-2) -- (M-3-3)
	node[midway,above] {$2X_n$};
\draw[-> ,font=\small](M-3-1) -- (M-7-2)
	node[midway,above] {$2X_n$};
\draw[-> ,font=\small](M-3-3) -- (M-3-4) 
	;
\draw[-> ,font=\small](M-3-4) -- (M-3-5)
	;
\draw[-> ,font=\small](M-3-5) -- (M-3-6)
	node[midway,above] {$2X_n$};
\draw[-> ,font=\small](M-3-6) -- (M-1-7)
	;
\draw[-> ,font=\small](M-3-6) -- (M-3-7)
	;

\draw[-> ,font=\small](M-5-2) -- (M-5-3)
	node[midway,above] {$2X_n$};
\draw[-> ,font=\small](M-5-3) -- (M-5-4) 
	;
\draw[-> ,font=\small](M-5-4) -- (M-5-5)
	;
\draw[-> ,font=\small](M-5-5) -- (M-5-6)
	node[midway, above] {$F$};
\draw[-> ,font=\small](M-5-5) -- (M-7-6)
	;
\draw[<-> ,font=\small](M-5-6) -- (M-5-7)
	;
\draw[<-> ,font=\small](M-7-6) -- (M-5-7)
	;

\draw[-> ,font=\small](M-7-2) -- (M-7-3)
	node[midway,above] {$2X_n$};
\draw[-> ,font=\small](M-7-3) -- (M-7-4) 
	;
\draw[-> ,font=\small](M-7-4) -- (M-7-5)
	;
\draw[-> ,font=\small](M-7-5) -- (M-5-6)
	;
\draw[-> ,font=\small](M-7-5) -- (M-7-6)
	;
\end{tikzpicture}
\end{center}

where either $\nabla$ or $\nabla'$ is the part of $(Q', \d_i)$ that contains the module $Q'_{V_i}$ and the other contains $Q'_{W_i}$.
We define $\mu_i|_{Q'_{G_i}}$ to be the following map in {\color{blue}blue}:

\begin{center}
\begin{tikzpicture}[>=stealth, baseline]% added "baseline"
\matrix (M) [matrix of math nodes, column sep=7mm]
{
P_n^A  & P_n^A  & P_n^A  & \cdots & P_n^A  & P_n^A  & P_{n-1}^A & \nabla \\
\oplus & \oplus & \oplus &        & \oplus & \oplus & \oplus    \\
P_n^A  & P_n^A  & P_n^A  & \cdots & P_n^A  & P_n^A  & P_{n+1}^A &        \\
       & \oplus & \oplus &        & \oplus & \oplus &       & & =(Q', \d_i)\\ 
       & P_n^A  & P_n^A  & \cdots & P_n^A  & P_{n-1}^A & \nabla'\\
       & \oplus & \oplus &        & \oplus & \oplus    &       \\
       & P_n^A  & P_n^A  & \cdots & P_n^A  & P_{n+1}^A &       \\
{}     &        &        &        &        &        &          \\
{}     &        &        &        &        &        &          \\ {}     &        &        &        &        &        &          \\
P_n^A  & P_n^A  & P_n^A  & \cdots & P_n^A  & P_n^A  & P_{n-1}^A & \nabla\\
\oplus & \oplus & \oplus &        & \oplus & \oplus & \oplus    \\
P_n^A  & P_n^A  & P_n^A  & \cdots & P_n^A  & P_n^A  & P_{n+1}^A &       \\
       & \oplus & \oplus &        & \oplus & \oplus &       & & =: (Q', \d_{i+1})\\ 
       & P_n^A  & P_n^A  & \cdots & P_n^A  & P_{n-1}^A & \nabla'\\
       & \oplus & \oplus &        & \oplus & \oplus    &       \\
       & P_n^A  & P_n^A  & \cdots & P_n^A  & P_{n+1}^A &       \\
};

%%Arrows for first complex
\draw[-> ,font=\small](M-1-1) -- (M-1-2) 
	node[midway,above] {$2X_n$};
\draw[-> ,font=\small](M-1-1) -- (M-5-2) 
	node[midway,below] {$2X_n$};
\draw[-> ,font=\small](M-1-2) -- (M-1-3) 
	node[midway,above] {$2X_n$};
\draw[-> ,font=\small](M-1-3) -- (M-1-4) 
	;
\draw[-> ,font=\small](M-1-4) -- (M-1-5)
	;
\draw[-> ,font=\small](M-1-5) -- (M-1-6)
	node[midway,above] {$2X_n$};
\draw[-> ,font=\small](M-1-6) -- (M-1-7)
	node[midway, above] {$F$};
\draw[-> ,font=\small](M-1-6) -- (M-3-7)
	;
\draw[<-> ,font=\small](M-1-7) -- (M-1-8)
	;
\draw[<-> ,font=\small](M-3-7) -- (M-1-8)
	;

\draw[-> ,font=\small](M-3-1) -- (M-3-2)
	node[midway,above] {$2X_n$};
\draw[-> ,font=\small](M-3-2) -- (M-3-3)
	node[midway,above] {$2X_n$};
\draw[-> ,font=\small](M-3-1) -- (M-7-2)
	node[midway,above] {$2X_n$};
\draw[-> ,font=\small](M-3-3) -- (M-3-4) 
	;
\draw[-> ,font=\small](M-3-4) -- (M-3-5)
	;
\draw[-> ,font=\small](M-3-5) -- (M-3-6)
	node[midway,above] {$2X_n$};
\draw[-> ,font=\small](M-3-6) -- (M-1-7)
	;
\draw[-> ,font=\small](M-3-6) -- (M-3-7)
	;

\draw[-> ,font=\small](M-5-2) -- (M-5-3)
	node[midway,above] {$2X_n$};
\draw[-> ,font=\small](M-5-3) -- (M-5-4) 
	;
\draw[-> ,font=\small](M-5-4) -- (M-5-5)
	;
\draw[-> ,font=\small](M-5-5) -- (M-5-6)
	node[midway, above] {$F$};
\draw[-> ,font=\small](M-5-5) -- (M-7-6)
	;
\draw[<-> ,font=\small](M-5-6) -- (M-5-7)
	;
\draw[<-> ,font=\small](M-7-6) -- (M-5-7)
	;

\draw[-> ,font=\small](M-7-2) -- (M-7-3)
	node[midway,above] {$2X_n$};
\draw[-> ,font=\small](M-7-3) -- (M-7-4) 
	;
\draw[-> ,font=\small](M-7-4) -- (M-7-5)
	;
\draw[-> ,font=\small](M-7-5) -- (M-5-6)
	;
\draw[-> ,font=\small](M-7-5) -- (M-7-6)
	;

%%Arrows for second complex
\draw[-> ,font=\small](M-11-1) -- (M-11-2) 
	node[midway,above] {$X_n$};
\draw[-> ,font=\small](M-11-1) -- (M-15-2) 
	node[midway,below] {$X_n$};
\draw[-> ,font=\small](M-11-2) -- (M-11-3) 
	node[midway,above] {$X_n$};
\draw[-> ,font=\small](M-11-3) -- (M-11-4) 
	;
\draw[-> ,font=\small](M-11-4) -- (M-11-5)
	;
\draw[-> ,font=\small](M-11-5) -- (M-11-6)
	node[midway,above] {$X_n$};
\draw[-> ,font=\small](M-11-6) -- (M-11-7)
	node[midway,above,scale=0.7] {$(n|n-1)$};
\draw[<-> ,font=\small](M-11-7) -- (M-11-8)
	;
\draw[<-> ,font=\small](M-13-7) -- (M-11-8)
	;

\draw[-> ,font=\small](M-13-1) -- (M-13-2)
	node[midway,above] {$X_n$};
\draw[-> ,font=\small](M-13-2) -- (M-13-3)
	node[midway,above] {$X_n$};
\draw[-> ,font=\small](M-13-1) -- (M-17-2)
	node[midway,left] {$X_n$};
\draw[-> ,font=\small](M-13-3) -- (M-13-4) 
	;
\draw[-> ,font=\small](M-13-4) -- (M-13-5)
	;
\draw[-> ,font=\small](M-13-5) -- (M-13-6)
	node[midway,above] {$X_n$};
\draw[-> ,font=\small](M-13-6) -- (M-13-7)
	node[midway,above,scale=0.7] {$(n|n+1)$};

\draw[-> ,font=\small](M-15-2) -- (M-15-3)
	node[midway,above] {$X_n$};
\draw[-> ,font=\small](M-15-3) -- (M-15-4) 
	;
\draw[-> ,font=\small](M-15-4) -- (M-15-5)
	;
\draw[-> ,font=\small](M-15-5) -- (M-15-6)
	node[midway,above,scale=0.7] {$(n|n-1)$};
\draw[<-> ,font=\small](M-15-6) -- (M-15-7)
	;
\draw[<-> ,font=\small](M-17-6) -- (M-15-7)
	;

\draw[-> ,font=\small](M-17-2) -- (M-17-3)
	node[midway,above] {$X_n$};
\draw[-> ,font=\small](M-17-3) -- (M-17-4) 
	;
\draw[-> ,font=\small](M-17-4) -- (M-17-5)
	;
\draw[-> ,font=\small](M-17-5) -- (M-17-6)
	node[midway,above,scale=0.7] {$(n|n+1)$};

%%Arrows between two complexes
\draw[-> ,font=\small, color=blue](M-3-1) -- (M-11-1)
	node[midway,left] {$2M$};
\draw[-> ,font=\small, color=blue](M-7-2) -- (M-11-2)
	node[midway,left,scale=0.6] {$
		\begin{bmatrix}
		M & 0 \\
		0 & M \\
		\end{bmatrix}
		$};
\draw[-> ,font=\small, color=blue](M-7-3) -- (M-11-3)
	node[midway, left,scale=0.6] {$ 2^{-1}
		\begin{bmatrix}
		M & 0 \\
		0 & M \\
		\end{bmatrix}
		$};
\draw[-> ,font=\small, color=blue](M-7-5) -- (M-11-5)
	node[midway, left,scale=0.6] {$2^{-(k-2)}
		\begin{bmatrix}
		M & 0 \\
		0 & M \\
		\end{bmatrix}
		$};
\draw[-> ,font=\small, color=blue](M-7-6) -- (M-11-6)
	node[midway, right,scale=0.6] {$
		\begin{bmatrix}
		2^{-(k-1)}M & 0 \\
		0 & 2^{-(k-2)}I \\
		\end{bmatrix}
		$};
\draw[-> ,font=\small, color=blue](M-3-7) to [out=-30,in=30]node[midway,right] {$2^{-(k-1)}I$}
	(M-11-7);

\end{tikzpicture}
\end{center}

For the rest of the modules in $Q'$, $\mu_i$ sends $v$ to $2^{-(k-1)}v$ (resp. $v$ to $2^{-(k-2)}v$) for any $v$ belonging to the modules in $\nabla$ (resp. $\nabla'$).
The black arrows in the last four rows shows the differential component $\d_{{i+1}_{G_i}}$ in $\d_{i+1}$, induced by the conjugation of $\mu_i^{-1}$.
It is easy to see that the required condition $(*)$ is satisfied.
The construction for $k < 0$ is similar, using the map $N$ in place of $M$.

\item $\check{g}_i$ is of Type V': \\
Recall the definitions of $g_i, c_i$ and $r_i$ from the paragraph \nameref*{slicing} and recall the subsets of crossings of $\mathfrak{m}(\check{c})$ as defined in \ref{subset crossing}.
Let $h_i$ be the connected component of $(c\setminus g_i) \cup (g_i \cap d_2)$ that contains the point $m$, so $g_i$ and $h_i$ intersects at the point $r_i$.
To illustrate, in \cref{notend01}, $h_0 = c_0$ and $h_1 = c_0 \cup g_0 \cup c_1$ whereas in \cref{end01}, $h_0 = \emptyset,$ $h_1 = g_0 \cup c_1,$ and $h_2 = g_0 \cup c_1 \cup g_1 \cup c_2.$ 
Let $\mathfrak{m}(\check{h}_i) =  \wt{h}_i \coprod \undertilde{h}_i$ and $\mathfrak{m}(r_i) = \wt{r}_i \coprod \undertilde{r}_i$ so that the curves of $\mathfrak{m}(\check{g}_i)$ and $ \wt{h_i}$ intersect at the point $\wt{r}_i \in \theta_{n+1}$ and the curves of $\mathfrak{m}(\check{g}_i)$ and $\undertilde{h_i}$ intersect at the point $\undertilde{r_i}\in \theta_{n-1}$.
Now recall the decomposition of $Q'$ given in \ref{new R decomp}.
By definition, $\{ \wt{r}_i, \undertilde{r}_i \}$ is the subset of crossings $R_i \subseteq G_i$.
We get that

\begin{equation*}
P^A(\wt{r_i}) \oplus P^A(\undertilde{r_i}) = Q'_{R_i}
\end{equation*}

Now first consider when $i=0$.
Then $g_0$ is of this type only when $m \in \Delta\setminus \{0, 1\}$ since $g_0 \cap \{1 \} = \emptyset$, so we have $Q'_{V_0 \oplus R_0} = Q'_{C_0}$.
Furthermore, the equation \ref{C diff component} implies that $\d_{0_{C_0}} = \d'_{C_0}$, giving us

\begin{equation*}
(Q'_{V_0}\oplus Q'_{R_0}, \d_{i_{V_0 \oplus R_0}}) 
= (Q'_{C_0}, \d_{i_{C_0}}) 
= (Q'_{C_0}, \d'_{C_0}) 
= L_A(\mathfrak{m}(\check{c}_0))
\end{equation*}

with the last equality following from the definition of $(Q', \d') = L_A(\mathfrak{m}(\check{c}))$.
By definition of $h_i$, it follows that $c_0 = h_0$, so we can conclude that

\begin{equation*}
(Q'_{V_0}\oplus Q'_{R_0}, \d_{i_{V_0 \oplus R_0}}) = L_A(\mathfrak{m}(\check{h}_0))
= L_A(\wt{h_0}) \oplus L_A(\undertilde{h_0})
\end{equation*}

where the last equaltiy follows from the fact that $\mathfrak{m}(\check{h}_i) =  \wt{h_i} \coprod \undertilde{h_i}$.

Now consider when $i \geq 1$.
In this case we are given $\d_i$ with $\d_i$ satisfying property $(*)$, so we can conclude that
\begin{align*}
\d_{i_{V_i}} = \d'_{V_i}, \quad
\d_{i_{V_i R_i}} = \d'_{V_i R_i}, \text{ and }
\d_{i_{R_i V_i}} = \d'_{R_i V_i}.
\end{align*}

Thus we have that $\d_{i_{V_i \oplus R_i}} = \d'_{V_i \oplus R_i}$, giving us

\begin{align*}
(Q'_{V_i}\oplus Q'_{R_i}, \d_{i_{V_i \oplus R_i}}) 
= (Q'_{V_i}\oplus Q'_{R_i}, \d'_{V_i \oplus R_i}) 
= L_A(\mathfrak{m}(\check{h}_i)) 
= L_A(\wt{h_i}) \oplus L_A(\undertilde{h_i}),
\end{align*}

where the second equality follows from the definition of $(Q', \d') = L_A(\mathfrak{m}(\check{c}))$ and the third equality follows from the fact that $\mathfrak{m}(\check{h}_i) =  \wt{h_i} \coprod \undertilde{h_i}$.

Thus for all $i$, we obtain

\begin{equation} \label{VR splits}
(Q'_{V_i}\oplus Q'_{R_i}, \d_{i_{V_i \oplus R_i}}) = L_A(\wt{h_i}) \oplus L_A(\undertilde{h_i}).
\end{equation}

Furthermore, note that $\wt{h_i}$ and $\undertilde{h_i}$ contain the points $\wt{r}_i$ and $\undertilde{r_i}$ respectively, so $L_A(\wt{h}_i)$ contains $P^A(\wt{r_i})$ as a submodule and $L_A(\undertilde{h_i})$ contains $P^A(\undertilde{r_i})$ as a submodule.
Let us now understand the relation between $(Q'_{G_i}, \d_{i_{G_i}})$, $(Q'_{V_i}\oplus Q'_{R_i}, \d_{i_{V_i \oplus R_i}})$ and $(Q', \d_i)$.
As in Type III$'_k$, we have $\d_{i_{G_i}} = \d_{0_{G_i}}$ for all $i$.
So the the part of $(Q', \d_i)$ that contains $(Q'_{G_i}, \d_{i_{G_i}})$ will be the same for all $i$, and it has either of the following two forms:

\begin{center}
\begin{tikzpicture}[>=stealth, baseline]% added "baseline"
%%% Example of bracket in tikzpicture
%\draw[thick] plot[smooth,tension=1.5]  coordinates {(3.2,-1) (3.7,0) (3.2,1)};
\matrix (M) [matrix of math nodes, column sep=7mm]
{
       & P^A(\wt{r_i}) & P_n^A  & P_{n-1}^A & \nabla' \\
       & \oplus    & \oplus & \oplus    &        & =(Q', \d_i)\\
\nabla & P^A(\undertilde{r_i}) & P_n^A  & P_{n+1}^A &        \\
};

\draw[-> ,font=\small](M-1-2) -- (M-1-3) 
	node[midway,above] {$E$};
\draw[-> ,font=\small](M-1-2) -- (M-3-3) 
	;
\draw[-> ,font=\small](M-1-3) -- (M-1-4) 
	node[midway,above] {$F$};
\draw[-> ,font=\small](M-1-3) -- (M-3-4) 
	;
\draw[<->](M-1-4) -- (M-1-5)
	;
\draw[<-> ,font=\small](M-3-4) -- (M-1-5)
	;

\draw[<->](M-3-1) -- (M-1-2)
	;
\draw[<->](M-3-1) -- (M-3-2)
	;
\draw[-> ,font=\small](M-3-2) -- (M-3-3) 
	;
\draw[-> ,font=\small](M-3-2) -- (M-1-3) 
	;
\draw[-> ,font=\small](M-3-3) -- (M-3-4)
	;
\draw[-> ,font=\small](M-3-3) -- (M-1-4)
	;

\draw[BurntOrange] (M-3-1.south west) rectangle (M-1-2.north east)
	node[midway, xshift = -3cm] 
  	{$(Q'_{V_i}\oplus Q'_{R_i}, \d_{i_{V_i \oplus R_i}})$};
  	
\draw[OliveGreen] (M-1-2.north west)++(0,5pt) rectangle 
	(M-3-4.south east)
		node[midway, xshift=3.2cm , yshift = -22pt] 
  		{$(Q'_{G_i}, \d_{i_{G_i}})$};
\end{tikzpicture}
\end{center}

{\LARGE or}

\begin{center}
\begin{tikzpicture}[>=stealth, baseline]% added "baseline"
%%% Example of bracket in tikzpicture
%\draw[thick] plot[smooth,tension=1.5]  coordinates {(3.2,-1) (3.7,0) (3.2,1)};
\matrix (M) [matrix of math nodes, column sep=7mm]
{
       & P^A_{n+1} & P_n^A  & P^A(\undertilde{r_i}) & \nabla \\
       & \oplus    & \oplus & \oplus    &   &  & & &   & =(Q', \d_i)\\
\nabla' & P^A_{n-1} & P_n^A  & P^A(\wt{r_i}) &        \\
};

\draw[-> ,font=\small](M-1-2) -- (M-1-3) 
	node[midway,above] {$E$};
\draw[-> ,font=\small](M-1-2) -- (M-3-3) 
	;
\draw[-> ,font=\small](M-1-3) -- (M-1-4) 
	node[midway,above] {$F$};
\draw[-> ,font=\small](M-1-3) -- (M-3-4) 
	;
\draw[<->](M-1-4) -- (M-1-5)
	;
\draw[<-> ,font=\small](M-3-4) -- (M-1-5)
	;

\draw[<->](M-3-1) -- (M-1-2)
	;
\draw[<->](M-3-1) -- (M-3-2)
	;
\draw[-> ,font=\small](M-3-2) -- (M-3-3) 
	;
\draw[-> ,font=\small](M-3-2) -- (M-1-3) 
	;
\draw[-> ,font=\small](M-3-3) -- (M-3-4)
	;
\draw[-> ,font=\small](M-3-3) -- (M-1-4)
	;

\draw[BurntOrange] (M-3-4.south west) rectangle (M-1-5.north east)
	node[midway, xshift = 3cm, yshift = 26pt] 
  	{$(Q'_{V_i}\oplus Q'_{R_i}, \d_{i_{V_i \oplus R_i}})$};
  	
\draw[OliveGreen] (M-1-2.north west)++(0,5pt) rectangle 
	(M-3-4.south east)
		node[midway, xshift=-3.2cm, yshift = 22pt] 
  		{$(Q'_{G_i}, \d_{i_{G_i}})$};
\end{tikzpicture}
\end{center}

where $\nabla$ denotes the rest of the complex $(Q', \d_i)$ that contains the module $Q'_{V_i}$ and $\nabla'$ denotes the rest of the complex $(Q', \d_i)$ that contains the module $Q'_{W_i}$.
%
%
\begin{comment}
By construction, $\nabla$ together with $P^A(r_{n+1}) \oplus P^A(r_{n-1})$ is $(Q'_{H_i}, \d_i(Q'_{H_i}))$.
By property ($*$), we may conclude that
\[
(Q'_{H_i}, \d_i(Q'_{H_i})) =
L_A(\wt{h_i}) \oplus L_A(\undertilde{h_i}) =
L_A(\mathfrak{m}(h_i)),
\]
with $L_A(\wt{h_i})$ containing $P^A(r_{n+1})$ and $L_A(\undertilde{h_i})$ containing $P^A(r_{n-1})$, as depicted below: 
\begin{center}
\begin{tikzpicture}
\matrix (M) [matrix of math nodes, column sep=7mm]
{
       & P^A(r_{n+1}) & & \nabla & P^A(r_{n+1}) & &\\
       & \oplus       &=&        & \oplus       &=& L_A(\mathfrak{m}(h_i))\\
\nabla & P^A(r_{n-1}) & & \nabla & P^A(r_{n-1}) & &\\
};

\draw[<->](M-3-1) -- (M-1-2)
	;
\draw[<->](M-3-1) -- (M-3-2)
	;
\draw[<->](M-1-4) -- (M-1-5)
	;
\draw[<->](M-3-4) -- (M-3-5)
	;
\draw[decoration={brace,raise=8pt},decorate]
  (M-1-4)++(-5pt,0) -- 
  	node[above, yshift = 10pt] 
  	{$L_A\left(\wt{h_i}\right)$} ++(2.7,0);
\draw[decoration={brace,mirror,raise=8pt},decorate]
  (M-3-4)++(-5pt,0) -- 
  	node[below, yshift = -10pt] 
  	{$L_A\left(\undertilde{h_i}\right)$} ++(2.7,0);
\end{tikzpicture}
\end{center}
\end{comment}
%
%
%

Using \cref{VR splits}, which holds for all $i$ as well, the part of $(Q', \d_i)$ that contains $(Q'_{G_i}, \d_{i_{G_i}})$, $L_A(\wt{h_i})$ and $L_A(\undertilde{h_i})$ has either of the following two forms
\begin{center}
\begin{tikzpicture}[>=stealth, baseline]% added "baseline"
%%% Example of bracket in tikzpicture
%\draw[thick] plot[smooth,tension=1.5]  coordinates {(3.2,-1) (3.7,0) (3.2,1)};
\matrix (M) [matrix of math nodes, column sep=7mm]
{
\wt{\nabla} & P^A(\wt{r_i}) & P_n^A  & P_{n-1}^A & \nabla' \\
       & \oplus    & \oplus & \oplus    &        & =(Q', \d_i)\\
\undertilde{\nabla} & P^A(\undertilde{r_i}) & P_n^A  & P_{n+1}^A &        \\
};

\draw[<->](M-1-1) -- (M-1-2)
	;
\draw[-> ,font=\small](M-1-2) -- (M-1-3) 
	node[midway,above] {$E$};
\draw[-> ,font=\small](M-1-2) -- (M-3-3) 
	;
\draw[-> ,font=\small](M-1-3) -- (M-1-4) 
	node[midway,above] {$F$};
\draw[-> ,font=\small](M-1-3) -- (M-3-4) 
	;
\draw[<->](M-1-4) -- (M-1-5)
	;
\draw[<-> ,font=\small](M-3-4) -- (M-1-5)
	;

\draw[<->](M-3-1) -- (M-3-2)
	;
\draw[-> ,font=\small](M-3-2) -- (M-3-3) 
	;
\draw[-> ,font=\small](M-3-2) -- (M-1-3) 
	;
\draw[-> ,font=\small](M-3-3) -- (M-3-4)
	;
\draw[-> ,font=\small](M-3-3) -- (M-1-4)
	;
	
\draw[red] (M-1-1.north west) rectangle (M-1-2.south east)
	node[midway, xshift = -2cm] 
  	{$L_A\left(\wt{h_i}\right)$};
\draw[Brown] (M-3-1.north west) rectangle (M-3-2.south east)
	node[midway, xshift = -2cm] 
  	{$L_A\left(\undertilde{h_i}\right)$};
  	
\draw[OliveGreen] (M-1-2.north west) rectangle 
	(M-3-4.south east)
		node[midway, xshift= 3.2cm, yshift = -23pt] 
  		{$(Q'_{G_i}, \d_{i_{G_i}})$};
\end{tikzpicture}
\end{center}

{\LARGE or}

\begin{center}
\begin{tikzpicture}[>=stealth, baseline]% added "baseline"
%%% Example of bracket in tikzpicture
%\draw[thick] plot[smooth,tension=1.5]  coordinates {(3.2,-1) (3.7,0) (3.2,1)};
\matrix (M) [matrix of math nodes, column sep=7mm]
{
       & P^A_{n+1} & P_n^A  & P^A(\undertilde{r_i}) & \undertilde{\nabla} \\
       & \oplus    & \oplus & \oplus    &    &    & =(Q', \d_i).\\
\nabla' & P^A_{n-1} & P_n^A  & P^A(\wt{r_i}) & \wt{\nabla} \\
};

\draw[<->](M-3-1) -- (M-1-2)
	;
\draw[-> ,font=\small](M-1-2) -- (M-1-3) 
	node[midway,above] {$E$};
\draw[-> ,font=\small](M-1-2) -- (M-3-3) 
	;
\draw[-> ,font=\small](M-1-3) -- (M-1-4) 
	node[midway,above] {$F$};
\draw[-> ,font=\small](M-1-3) -- (M-3-4) 
	;
\draw[<->](M-1-4) -- (M-1-5)
	;
\draw[<-> ,font=\small](M-3-4) -- (M-3-5)
	;

\draw[<->](M-3-1) -- (M-3-2)
	;
\draw[-> ,font=\small](M-3-2) -- (M-3-3) 
	;
\draw[-> ,font=\small](M-3-2) -- (M-1-3) 
	;
\draw[-> ,font=\small](M-3-3) -- (M-3-4)
	;
\draw[-> ,font=\small](M-3-3) -- (M-1-4)
	;

\draw[Brown] (M-1-4.south west) rectangle (M-1-5.north east)
	node[midway, xshift = 2cm] 
  	{$L_A\left(\undertilde{h_i}\right)$};
\draw[red] (M-3-4.north west) rectangle (M-3-5.south east)
	node[midway, xshift = 2cm] 
  	{$L_A\left(\wt{h_i}\right)$};
  	
\draw[OliveGreen] (M-1-2.north west) rectangle 
	(M-3-4.south east)
		node[midway, xshift = -3.3cm, yshift= 0.8cm] 
  		{$(Q'_{G_i}, \d_{i_{G_i}})$};
\end{tikzpicture}
\end{center}
Thus for all $i$, we conclude that $(Q', \d_i)$ must be of the above two possible forms.
It is now sufficient to give a construction of $\mu_i$ for each form.

We begin with the construction of $\mu_i$ for the first possible form of $(Q', \d_i)$.
Firstly, note that 

\begin{align*}
ME
=
	\begin{bmatrix}
	1 & -i \\
	1 &  i
	\end{bmatrix}
	\begin{bmatrix}
	-(n+1|n)i &  (n-1|n)i \\
	(n+1|n)   &  (n-1|n)
	\end{bmatrix}
&= 2i
	\begin{bmatrix}
	-(n+1|n) & 0      \\
	0        & (n-1|n)
	\end{bmatrix}
\\
 &=
	2i
	\begin{bmatrix}
	(n+1|n) & 0 \\
	0       & (n-1|n)
	\end{bmatrix}
	\begin{bmatrix}
	-1 & 0 \\
	0  & 1
	\end{bmatrix}.
\end{align*}

We define $\mu_i|_{Q'_{G_i}}$ to be the following map in {\color{blue}blue}:

\begin{center}
\begin{tikzpicture}[>=stealth, baseline]% added "baseline"
\matrix (M) [matrix of math nodes, column sep=7mm]
{
\wt{\nabla} & P^A(\wt{r_i}) & P_n^A  & P_{n-1}^A & \nabla' \\
       & \oplus       & \oplus & \oplus    &        & =(Q', \d_i) \\
\undertilde{\nabla} & P^A(\undertilde{r_i}) & P_n^A  & P_{n+1}^A &        \\
{}     &           &        &           &        \\
{}     &           &        &           &        \\
{}     &           &        &           &        \\
\wt{\nabla} & P^A(\wt{r_i}) & P_n^A  & P_{n-1}^A & \nabla' \\
       & \oplus       & \oplus & \oplus    &        & =:(Q', \d_{i+1})\\
\undertilde{\nabla} & P^A(\undertilde{r_i}) & P_n^A  & P_{n+1}^A &        \\
};

%%Arrows for First complex
\draw[<->](M-1-1) -- (M-1-2)
	;
\draw[-> ,font=\small](M-1-2) -- (M-1-3) 
	;
\draw[-> ,font=\small](M-1-2) -- (M-3-3) 
	;
\draw[-> ,font=\small](M-1-3) -- (M-1-4) 
	;
\draw[-> ,font=\small](M-1-3) -- (M-3-4) 
	;
\draw[<->](M-1-4) -- (M-1-5)
	;
\draw[<-> ,font=\small](M-3-4) -- (M-1-5)
	;

\draw[<->](M-3-1) -- (M-3-2)
	;
\draw[-> ,font=\small](M-3-2) -- (M-3-3) 
	;
\draw[-> ,font=\small](M-3-2) -- (M-1-3) 
	;
\draw[-> ,font=\small](M-3-3) -- (M-3-4)
	;
\draw[-> ,font=\small](M-3-3) -- (M-1-4)
	;
	
%%Arrows for second complex
\draw[<->](M-7-1) -- (M-7-2)
	;
\draw[-> ,font=\small](M-7-2) -- (M-7-3) 
	node[midway,above,scale=0.7] {$(n+1|n)$};
\draw[-> ,font=\small](M-7-3) -- (M-7-4) 
	node[midway,above,scale=0.7] {$(n|n-1)$};
\draw[<->](M-7-4) -- (M-7-5)
	;
\draw[<-> ,font=\small](M-9-4) -- (M-7-5)
	;

\draw[<->](M-9-1) -- (M-9-2)
	;
\draw[-> ,font=\small](M-9-2) -- (M-9-3) 
	node[midway,above,scale=0.7] {$(n-1|n)$};
\draw[-> ,font=\small](M-9-3) -- (M-9-4)
	node[midway,above,scale=0.7] {$(n|n+1)$};
	
%%Arrows between two complexes
\draw[-> ,font=\small, color=blue](M-3-2) -- (M-7-2)
	node[midway,left,scale=0.7] {$
		2i
		\begin{bmatrix}
		-1 & 0 \\
		0  & 1
		\end{bmatrix}
		$};
\draw[-> ,font=\small, color=blue](M-3-3) -- (M-7-3)
	node[midway,left,scale=0.7] {$M$};
\draw[-> ,font=\small, color=blue](M-3-4) -- (M-7-4)
	node[midway,left,scale=0.7] {$I$};
\end{tikzpicture}
\end{center}

For the rest of the modules in $Q'$, we define $\mu_i$ as the identity for modules contained in $\nabla'$, i.e. $\mu_i|_{Q'_{W_i}} = \id$, and $\mu_i$ sends $v$ to $-2iv$ (resp. $v$ to $2iv$) for any $v$ belonging to the modules in $\wt{\nabla}$ (resp. $\undertilde{\nabla}$).
The black arrows in the last two rows shows the differential component $\d_{{i+1}_{G_i}}$ in $\d_{i+1}$, induced by the conjugation of $\mu_i^{-1}$.
It is easy to see that $\d_{i+1}$ does indeed satisfy the required property $(*)$.

The construction of $\mu_i$ for the second form is similar, changing $\mu_i|_{P^A_n \oplus P^A_n}$ to $N$ instead of $M$.

\item $\check{g}_i$ is of Type III$'_{k+\frac{1}{2}}$ for $k\in Z$:\\
Note that the case for $k=0$ is straightforward; $\mu_i$ is just the identity in this case.
We will provide the analysis and construction of $\mu_i$ for $k>0$ and $k <0$.

Using the same argument in Type V$'_k$, the equations $\d_{i_{G_i}} = \d_{0_{G_i}}$ and \cref{VR splits} hold for all $i$.
So the part of $(Q', \d_i)$ that contains $(Q'_{G_i}, \d_{i_{G_i}})$, $L_A(\wt{h_i})$ and $L_A(\undertilde{h_i})$ will be of the same form for all $i$.
 So, the construction of $\mu_i$ below will also work for all $i$.

Let us start with $k > 0$.
Using the same notation as in the analysis of Type V$'$ we can draw the part of $(Q', \d_i)$ that contains $(Q'_{G_i}, \d_{i_{G_i}}) = (Q'_{G_i}, \d_{0_{G_i}})$, $L_A(\wt{h}_i)$ and $L_A(\undertilde{h_i})$ as either of the two forms:
\begin{figure}[H]
\begin{tikzpicture}[>=stealth, baseline]% added "baseline"
\matrix (M) [matrix of math nodes, column sep=7mm]
{
       & P_n^A  & \cdots & P_n^A  & P_n^A  & P^A(\undertilde{r_i}) & \undertilde{\nabla}\\
       & \oplus &        & \oplus & \oplus & \oplus  &   & & =(Q', \d_i)\\
P_n^A  & P_n^A  & \cdots & P_n^A  & P_n^A  & P^A(\wt{r_i}) & \wt{\nabla}\\
};

\draw[-> ,font=\small](M-1-2.east |- M-1-3) -- (M-1-3)
	;
\draw[-> ,font=\small](M-1-3.east |- M-1-4) -- (M-1-4)
	;
\draw[-> ,font=\small](M-1-4.east |- M-1-5) -- (M-1-5)
	node[midway,above] {$2X_n$}; 
\draw[-> ,font=\small](M-1-5) -- (M-1-6)
	node[midway,above] {$F$}; 
\draw[-> ,font=\small](M-1-5) -- (M-3-6)
	;
\draw[<->](M-1-6.east |- M-1-7) -- (M-1-7)
	;

\draw[-> ,font=\small](M-3-1.east |- M-3-2) -- (M-3-2) 
	node[midway,above] {$2X_n$};% changed coordinates for arrow
\draw[-> ,font=\small](M-3-2.east |- M-3-3) -- (M-3-3)
	;
\draw[-> ,font=\small](M-3-3.east |- M-3-4) -- (M-3-4)
	;
\draw[-> ,font=\small](M-3-4.east |- M-3-5) -- (M-3-5)
	node[midway,above] {$2X_n$};
\draw[-> ,font=\small](M-3-5) -- (M-1-6)
	; 
\draw[-> ,font=\small](M-3-5) -- (M-3-6)
	;
\draw[<->](M-3-6.east |- M-3-7) -- (M-3-7)
	;

\draw[Brown] (M-1-6.north west) rectangle (M-1-7.south east)
	node[midway, xshift = 2cm] 
  	{$L_A\left(\undertilde{h_i}\right)$};
\draw[red] (M-3-6.north west) rectangle (M-3-7.south east)
	node[midway, xshift = 2cm] 
  	{$L_A\left(\wt{h_i}\right)$};
  	
\draw[OliveGreen] (M-3-1.south west) rectangle (M-1-6.north east)
		node[midway,  xshift = -5.2cm] 
  		{$(Q'_{G_i}, \; \d_{i_{G_i}})$};
\end{tikzpicture}

\vspace*{2mm}{\LARGE or}

\begin{tikzpicture}[>=stealth, baseline]% added "baseline"
\matrix (M) [matrix of math nodes, column sep=7mm]
{
P_n^A  & P_n^A  & \cdots & P_n^A  & P_n^A  & P^A(\undertilde{r_i}) & \undertilde{\nabla}\\
       & \oplus &        & \oplus & \oplus & \oplus   & & & =(Q', \d_i)\\
       & P_n^A  & \cdots & P_n^A  & P_n^A  & P^A(\wt{r_i}) & \wt{\nabla}\\
};

\draw[-> ,font=\small](M-1-1.east |- M-1-2) -- (M-1-2) 
	node[midway,above] {$2X_n$};% changed coordinates for arrow
\draw[-> ,font=\small](M-1-2.east |- M-1-3) -- (M-1-3)
	;
\draw[-> ,font=\small](M-1-3.east |- M-1-4) -- (M-1-4)
	;
\draw[-> ,font=\small](M-1-4.east |- M-1-5) -- (M-1-5)
	node[midway,above] {$2X_n$}; 
\draw[-> ,font=\small](M-1-5) -- (M-1-6)
	node[midway,above] {$F$}; 
\draw[-> ,font=\small](M-1-5) -- (M-3-6)
	;
\draw[<->](M-1-6.east |- M-1-7) -- (M-1-7)
	;

\draw[-> ,font=\small](M-3-2.east |- M-3-3) -- (M-3-3)
	;
\draw[-> ,font=\small](M-3-3.east |- M-3-4) -- (M-3-4)
	;
\draw[-> ,font=\small](M-3-4.east |- M-3-5) -- (M-3-5)
	node[midway,above] {$2X_n$};
\draw[-> ,font=\small](M-3-5) -- (M-1-6)
	; 
\draw[-> ,font=\small](M-3-5) -- (M-3-6)
	;
\draw[<->](M-3-6.east |- M-3-7) -- (M-3-7)
	;
	
\draw[Brown] (M-1-6.north west) rectangle (M-1-7.south east)
	node[midway, xshift = 2cm] 
  	{$L_A\left(\undertilde{h_i}\right)$};
\draw[red] (M-3-6.north west) rectangle (M-3-7.south east)
	node[midway, xshift = 2cm] 
  	{$L_A\left(\wt{h_i}\right)$};
  	
\draw[OliveGreen] (M-1-1.north west) rectangle (M-3-6.south east)
		node[midway, xshift = -5.2cm] 
  		{$(Q'_{G_i}, \; \d_{i_{G_i}})$};
\end{tikzpicture}
\end{figure}
We will construct $\mu_i$ for the first form; the second form is just a mirrored construction.
Define $\mu_i|_{Q'_{G_i}}$ to be the following map in {\color{blue}blue}:

\begin{center}
\begin{tikzpicture}[>=stealth, baseline]% added "baseline"
\matrix (M) [matrix of math nodes, column sep=7mm]
{
       & P_n^A  & \cdots & P_n^A  & P_n^A  & P^A(\undertilde{r_i}) & \undertilde{\nabla}\\
       & \oplus &        & \oplus & \oplus & \oplus    & & =(Q', \d_i)\\
P_n^A  & P_n^A  & \cdots & P_n^A  & P_n^A  & P^A(\wt{r_i}) & \wt{\nabla}\\
{}     &        &        &        &        &           &       \\
{}     &        &        &        &        &           &       \\
{}     &        &        &        &        &           &       \\
       & P_n^A  & \cdots & P_n^A  & P_n^A  & P^A(\undertilde{r_i}) & \undertilde{\nabla}\\
       & \oplus &        & \oplus & \oplus & \oplus    & & =:(Q', \d_{i+1})\\
P_n^A  & P_n^A  & \cdots & P_n^A  & P_n^A  & P^A(\wt{r_i}) & \wt{\nabla}\\
};

%%Arrows for first complex
\draw[-> ,font=\small](M-1-2.east |- M-1-3) -- (M-1-3)
	;
\draw[-> ,font=\small](M-1-3.east |- M-1-4) -- (M-1-4)
	;
\draw[-> ,font=\small](M-1-4.east |- M-1-5) -- (M-1-5)
	node[midway,above] {$2X_n$}; 
\draw[-> ,font=\small](M-1-5) -- (M-1-6)
	; 
\draw[-> ,font=\small](M-1-5) -- (M-3-6)
	;
\draw[<->](M-1-6.east |- M-1-7) -- (M-1-7);

\draw[-> ,font=\small](M-3-1.east |- M-3-2) -- (M-3-2) 
	node[midway,above] {$2X_n$};% changed coordinates for arrow
\draw[-> ,font=\small](M-3-2.east |- M-3-3) -- (M-3-3)
	;
\draw[-> ,font=\small](M-3-3.east |- M-3-4) -- (M-3-4)
	;
\draw[-> ,font=\small](M-3-4.east |- M-3-5) -- (M-3-5)
	node[midway,above] {$2X_n$};
\draw[-> ,font=\small](M-3-5) -- (M-1-6)
	; 
\draw[-> ,font=\small](M-3-5) -- (M-3-6)
	;
\draw[<->](M-3-6.east |- M-3-7) -- (M-3-7)
	;

%%Arrows for second complex
\draw[-> ,font=\small](M-7-2.east |- M-7-3) -- (M-7-3)
	node[midway,above] {$X_n$};
\draw[-> ,font=\small](M-7-3.east |- M-7-4) -- (M-7-4)
	;
\draw[-> ,font=\small](M-7-4.east |- M-7-5) -- (M-7-5)
	node[midway,above] {$X_n$}; 
\draw[-> ,font=\small](M-7-5) -- (M-7-6)
	;
\draw[<->](M-7-6.east |- M-7-7) -- (M-7-7)
	;

\draw[-> ,font=\small](M-9-1) -- (M-7-2) 
	node[midway,above,transform canvas={xshift=-2mm}] {$X_n$};
\draw[-> ,font=\small](M-9-1.east |- M-9-2) -- (M-9-2) 
	node[midway,below] {$X_n$};% changed coordinates for arrow
\draw[-> ,font=\small](M-9-2.east |- M-9-3) -- (M-9-3)
	node[midway,above] {$X_n$};
\draw[-> ,font=\small](M-9-3.east |- M-9-4) -- (M-9-4)
	;
\draw[-> ,font=\small](M-9-4.east |- M-9-5) -- (M-9-5)
	node[midway,above] {$X_n$};
\draw[-> ,font=\small](M-9-5) -- (M-9-6)
	; 
\draw[<->](M-9-6.east |- M-9-7) -- (M-9-7)
	;

%%Arrows between first and second complexes
\draw[-> ,font=\small, color=blue](M-3-1) -- (M-9-1) 
	node[midway,left,scale=0.8, color=blue] {$2^{k-1}i$};
\draw[-> ,font=\small, color=blue](M-3-2) -- (M-7-2) 
	node[midway,left,scale=0.8, color=blue] 
		{$2^{k-2}JM$};
\draw[-> ,font=\small, color=blue](M-3-4) -- (M-7-4) 
	node[midway,left,scale=0.8, color=blue] 
		{$2JM$};
\draw[-> ,font=\small, color=blue](M-3-5) -- (M-7-5) 
	node[midway,left,scale=0.8, color=blue] {$JM$};
\draw[-> ,font=\small, color=blue](M-3-6) -- (M-7-6) 
	node[midway,left,scale=0.8, color=blue] {$J$};
\end{tikzpicture}
\end{center}

with $J = \begin{bmatrix}
			-1 & 0 \\
			0  & 1
			\end{bmatrix}$.
For the rest of the modules in $Q'$, $\mu_i$ sends $v$ to $-v$ (resp. $v$ to $v$) for any $v$ belonging to the modules in $\undertilde{\nabla}$ (resp. $\wt{\nabla}$).
The black arrows in the last two rows shows the differential component $\d_{{i+1}_{G_i}}$ in $\d_{i+1}$, induced by the conjugation of $\mu_i^{-1}$.
It is easy to see that $\mu_i$ does indeed satisfy the required property $(*)$.

Similarly for $k < 0$, we draw the part of $(Q', \d_i)$ that contains $(Q'_{G_i}, \d_{i_{G_i}}) = (Q'_{G_i}, \d_{0_{G_i}})$, $L_A(\wt{h}_i)$ and $L_A(\undertilde{h_i})$ as either of the two forms:

\begin{figure}[H]
\begin{tikzpicture}[>=stealth, baseline]% added "baseline"
\matrix (M) [matrix of math nodes, column sep=7mm]
{
\wt{\nabla} & P^A(\wt{r_i}) & P_n^A  & P_n^A  & \cdots & P_n^A  &        \\
       & \oplus    & \oplus & \oplus &        & \oplus &  & &  &=(Q', \d_i)\\
\undertilde{\nabla} & P^A(\undertilde{r_i}) & P_n^A  & P_n^A  & \cdots & P_n^A  & P_n^A  \\
};

\draw[-> ,font=\small](M-1-2.east |- M-1-3) -- (M-1-3) 
	node[midway,above] {$E$};
\draw[-> ,font=\small](M-1-2) -- (M-3-3) 
	;
\draw[-> ,font=\small](M-1-3.east |- M-1-4) -- (M-1-4) 
	node[midway,above] {$2X_n$};% changed coordinates for arrow
\draw[-> ,font=\small](M-1-4.east |- M-1-5) -- (M-1-5)
	;
\draw[-> ,font=\small](M-1-5.east |- M-1-6) -- (M-1-6)
	;

\draw[-> ,font=\small](M-3-2) -- (M-1-3)
	;
\draw[-> ,font=\small](M-3-2) -- (M-3-3)
	;
\draw[-> ,font=\small](M-3-3.east |- M-3-4) -- (M-3-4) 
	node[midway,above] {$2X_n$};% changed coordinates for arrow
\draw[-> ,font=\small](M-3-4.east |- M-3-5) -- (M-3-5)
	;
\draw[-> ,font=\small](M-3-5.east |- M-3-6) -- (M-3-6)
	;
\draw[-> ,font=\small](M-3-6.east |- M-3-7) -- (M-3-7)
	node[midway,above] {$2X_n$};
\draw[<->](M-3-1.east |- M-3-2) -- (M-3-2);
\draw[<->](M-1-1) -- (M-1-2); 

\draw[red] (M-1-1.north west) rectangle (M-1-2.south east)
	node[midway, xshift = -2cm] 
  	{$L_A\left(\wt{h_i}\right)$};
\draw[Brown] (M-3-1.north west) rectangle (M-3-2.south east)
	node[midway, xshift = -2cm] 
  	{$L_A\left(\undertilde{h_i}\right)$};
  	
\draw[OliveGreen] (M-1-2.north west) rectangle 
	(M-3-7.south east)
		node[midway, xshift=5.2cm, yshift = 22pt] 
  		{$(Q'_{G_i}, \d_{i_{G_i}})$};
\end{tikzpicture}

\vspace*{2mm}{\LARGE or}

\begin{tikzpicture}[>=stealth, baseline]% added "baseline"
\matrix (M) [matrix of math nodes, column sep=7mm]
{
\wt{\nabla} & P^A(\wt{r_i}) & P_n^A  & P_n^A  & \cdots & P_n^A  & P_n^A \\
       & \oplus    & \oplus & \oplus &        & \oplus & & & & =:(Q', \d_{i+1})\\
\undertilde{\nabla} & P^A(\undertilde{r_i}) & P_n^A  & P_n^A  & \cdots & P_n^A  &       \\
};

\draw[-> ,font=\small](M-1-2.east |- M-1-3) -- (M-1-3) 
	node[midway,above] {$E$};
\draw[-> ,font=\small](M-1-2) -- (M-3-3) 
	;
\draw[-> ,font=\small](M-1-3.east |- M-1-4) -- (M-1-4) 
	node[midway,above] {$2X_n$};% changed coordinates for arrow
\draw[-> ,font=\small](M-1-4.east |- M-1-5) -- (M-1-5)
	;
\draw[-> ,font=\small](M-1-5.east |- M-1-6) -- (M-1-6)
	;
\draw[-> ,font=\small](M-1-6.east |- M-1-7) -- (M-1-7)
	node[midway,above] {$2X_n$};

\draw[-> ,font=\small](M-3-2) -- (M-1-3)
	;
\draw[-> ,font=\small](M-3-2) -- (M-3-3)
	;
\draw[-> ,font=\small](M-3-3.east |- M-3-4) -- (M-3-4) 
	node[midway,above] {$2X_n$};% changed coordinates for arrow
\draw[-> ,font=\small](M-3-4.east |- M-3-5) -- (M-3-5)
	;
\draw[-> ,font=\small](M-3-5.east |- M-3-6) -- (M-3-6)
	;
\draw[<->](M-3-1.east |- M-3-2) -- (M-3-2);
\draw[<->](M-1-1) -- (M-1-2); 

\draw[red] (M-1-1.north west) rectangle (M-1-2.south east)
	node[midway, xshift = -2cm] 
  	{$L_A\left(\wt{h_i}\right)$};
\draw[Brown] (M-3-1.north west) rectangle (M-3-2.south east)
	node[midway, xshift = -2cm] 
  	{$L_A\left(\undertilde{h_i}\right)$};
  	
\draw[OliveGreen] (M-1-2.north west) rectangle 
	(M-3-7.south east)
		node[midway, xshift=5.2cm, yshift = -.8cm] 
  		{$(Q'_{G_i}, \d_{i_{G_i}})$};
\end{tikzpicture}
\end{figure}

Once again we construct $\mu_i$ for the first form; the second form is just a mirror construction.
Note that 

\[
N
\begin{bmatrix}
-(n+1|n)i &  (n-1|n)i \\
(n+1|n)   &  (n-1|n)
\end{bmatrix}
=
2
\begin{bmatrix}
(n+1|n) &   0      \\
0       &  (n-1|n)
\end{bmatrix}.
\]

We define $\mu_i|_{Q'_{G_i}}$ to be the following map in {\color{blue}blue}, with $\mu_i$ acting as the identity for the rest of the modules in both $\wt{\nabla}$ and $\undertilde{\nabla}$:
\begin{center}
\begin{tikzpicture}[>=stealth, baseline]% added "baseline"
\matrix (M) [matrix of math nodes, column sep=7mm]
{
\wt{\nabla} & P^A(\wt{r_i}) & P_n^A  & P_n^A  & \cdots & P_n^A  &        \\
       & \oplus    & \oplus & \oplus &        & \oplus & & =(Q', \d_i) \\
\undertilde{\nabla} & P^A(\undertilde{r_i}) & P_n^A  & P_n^A  & \cdots & P_n^A  & P_n^A  \\
{}     &           &        &        &        &        &        \\
{}     &           &        &        &        &        &        \\
{}     &           &        &        &        &        &        \\
\wt{\nabla} & P^A(\wt{r_i}) & P_n^A  & P_n^A  & \cdots & P_n^A  &        \\
       & \oplus    & \oplus & \oplus &        & \oplus & & =:(Q', \d_{i+1}) \\
\undertilde{\nabla} & P^A(\undertilde{r_i}) & P_n^A  & P_n^A  & \cdots & P_n^A  & P_n^A. \\
};

%% Arrows for first complex
\draw[-> ,font=\small](M-1-2.east |- M-1-3) -- (M-1-3) 
	;
\draw[-> ,font=\small](M-1-2) -- (M-3-3) 
	;
\draw[-> ,font=\small](M-1-3.east |- M-1-4) -- (M-1-4) 
	node[midway,above] {$2X_n$};% changed coordinates for arrow
\draw[-> ,font=\small](M-1-4.east |- M-1-5) -- (M-1-5)
	;
\draw[-> ,font=\small](M-1-5.east |- M-1-6) -- (M-1-6)
	;

\draw[-> ,font=\small](M-3-2) -- (M-1-3)
	;
\draw[-> ,font=\small](M-3-2) -- (M-3-3)
	;
\draw[-> ,font=\small](M-3-3.east |- M-3-4) -- (M-3-4) 
	node[midway,above] {$2X_n$};% changed coordinates for arrow
\draw[-> ,font=\small](M-3-4.east |- M-3-5) -- (M-3-5)
	;
\draw[-> ,font=\small](M-3-5.east |- M-3-6) -- (M-3-6)
	;
\draw[-> ,font=\small](M-3-6.east |- M-3-7) -- (M-3-7)
	node[midway,above] {$2X_n$};
\draw[<->](M-3-1.east |- M-3-2) -- (M-3-2);
\draw[<->](M-1-1) -- (M-1-2); 

%% Arrows for second complex
\draw[-> ,font=\small](M-7-2.east |- M-7-3) -- (M-7-3) 
	node[midway,above,scale=0.7] {$(n+1|n)$};
\draw[-> ,font=\small](M-7-3.east |- M-7-4) -- (M-7-4) 
	node[midway,above] {$X_n$};% changed coordinates for arrow
\draw[-> ,font=\small](M-7-4.east |- M-7-5) -- (M-7-5)
	;
\draw[-> ,font=\small](M-7-5.east |- M-7-6) -- (M-7-6)
	;
\draw[-> ,font=\small](M-7-6) -- (M-9-7)
	node[midway, above, transform canvas={xshift=2mm}] {$X_n$};

\draw[-> ,font=\small](M-9-2) -- (M-9-3)
	node[midway,above,scale=0.7] {$(n-1|n)$};
\draw[-> ,font=\small](M-9-3.east |- M-9-4) -- (M-9-4) 
	node[midway,above] {$X_n$};% changed coordinates for arrow
\draw[-> ,font=\small](M-9-4.east |- M-9-5) -- (M-9-5)
	;
\draw[-> ,font=\small](M-9-5.east |- M-9-6) -- (M-9-6)
	;
\draw[-> ,font=\small](M-9-6.east |- M-9-7) -- (M-9-7)
	node[midway,above] {$X_n$};
\draw[<->](M-9-1.east |- M-9-2) -- (M-9-2);
\draw[<->](M-7-1) -- (M-7-2); 

%% Arrows between two complexes
\draw[-> ,font=\small, color=blue](M-3-2) -- (M-7-2) 
	node[midway,left,scale=0.8] {$I$};
\draw[-> ,font=\small, color=blue](M-3-3) -- (M-7-3) 
	node[midway,right,scale=0.8] {$2^{-1}N$};
\draw[-> ,font=\small, color=blue](M-3-4) -- (M-7-4) 
	node[midway,right,scale=0.8] {$2^{-2}N$};
\draw[-> ,font=\small, color=blue](M-3-6) -- (M-7-6) 
	node[midway,left,scale=0.8] {$2^{k+1}N$};
\draw[-> ,font=\small, color=blue](M-3-7) -- (M-9-7) 
	node[midway,right,scale=0.8] {$2^{k+1}$};
\end{tikzpicture}
\end{center}
The black arrows in the last two rows shows the differential component $\d_{{i+1}_{G_i}}$ in $\d_{i+1}$, induced by the conjugation of $\mu_i^{-1}$.
Thus, the required condition $(*)$ follows directly.

\item $\check{g}_i$ is of Type II$'_{k+\frac{1}{2}}$ for $k \in \Z$: \\
By the same argument in Type V$'_k$, we have that the equations $\d_{i_{G_i}} = \d_{0_{G_i}}$ and \cref{VR splits} hold for all $i$.
So the part of $(Q', \d_i)$ that contains $(Q'_{G_i}, \d_{i_{G_i}})$, $L_A(\wt{h}_i)$ and $L_A(\undertilde{h_i})$ will be of the same form for all $i$.
It then follows similarly that the construction of $\mu_i$ will also work for all $i$.

Let us start with $k=0$.
With the same notation as in the analysis of Type V$'$, we can draw the part of $(Q', \d_i)$ that contains $(Q'_{G_i}, \d_{i_{G_i}}) = (Q'_{G_i}, \d_{0_{G_i}})$, $L_A(\wt{h}_i)$ and $L_A(\undertilde{h_i})$ as
\begin{center}
\begin{tikzpicture}[>=stealth, baseline]% added "baseline"
\matrix (M) [matrix of math nodes, column sep=7mm]
{
{}       & P_{n-1}^A & \nabla' \\
         & \oplus    &        \\
 P_{n}^A & P_{n+1}^A &        \\
 \oplus  & \oplus    & & & &  =(Q', \d_i) \\
 P_{n}^A & P^A(\wt{r_i}) & \wt{\nabla} \\
         & \oplus    &        \\
         & P^A(\undertilde{r_i}) & \undertilde{\nabla} \\
};

\draw[<->](M-1-2) -- (M-1-3)
	;
\draw[<->](M-3-2) -- (M-1-3)
	;
\draw[<->](M-5-2) -- (M-5-3)
	; 
\draw[<->](M-7-2) -- (M-7-3)
	; 

\draw[-> ,font=\small](M-3-1) -- (M-1-2.west)
	;
\draw[-> ,font=\small](M-3-1) -- (M-3-2)
	;
\draw[-> ,font=\small](M-5-1) -- (M-1-2)
	;
\draw[-> ,font=\small](M-5-1) -- (M-3-2)
	;
	
\draw[-> ,font=\small](M-3-1) -- (M-5-2)
	;
\draw[-> ,font=\small](M-3-1) -- (M-7-2)
	;
\draw[-> ,font=\small](M-5-1) -- (M-5-2)
	;
\draw[-> ,font=\small](M-5-1) -- (M-7-2.west)
	;

\draw[red] (M-5-2.north west) rectangle (M-5-3.south east)
	node[midway, xshift = 2.3cm] 
  	{$L_A\left(\wt{h_i}\right)$};
\draw[Brown] (M-7-2.north west) rectangle (M-7-3.south east)
	node[midway, xshift = 2.3cm] 
  	{$L_A\left(\undertilde{h_i}\right)$};
  	
\draw[OliveGreen] (M-1-1.north west) ++(-10pt,10pt) rectangle 
	(M-7-2.south east)
		node[midway, xshift =  -2.3cm] 
  		{$(Q'_{G_i}, \d_{i_{G_i}})$};
\end{tikzpicture}
\end{center}
with the first map given by $
	\begin{bmatrix}
	F \\
	F'
	\end{bmatrix}
	$,
where $\begin{bmatrix}
	-i & 0 \\
	 0 & i 
	\end{bmatrix}
	F' =
	\begin{bmatrix}
	(n|n+1) & 0 \\
	 0      & (n|n-1) 
	\end{bmatrix}
	M
	$.
Denote $P = 
	\begin{bmatrix}
	-i & 0 \\
	 0 & i 
	\end{bmatrix}
	$.
We define $\mu_i|_{Q'_{G_i}}$ to be the following map in {\color{blue}blue}:
\begin{center}
\begin{tikzpicture}[>=stealth, baseline]% added "baseline"
\matrix (M) [matrix of math nodes, column sep=7mm]
{
         & P_{n-1}^A & \nabla' \\
         & \oplus    &        \\
 P_{n}^A & P_{n+1}^A &        \\
 \oplus  & \oplus    & & =(Q', \d_i)  \\
 P_{n}^A & P^A(\wt{r_i}) & \wt{\nabla} \\
         & \oplus    &        \\
         & P^A(\undertilde{r_i}) & \undertilde{\nabla} \\
{}       &           &        \\
{}       &           &        \\
{}       &           &        \\         
         & P_{n-1}^A & \nabla' \\
         & \oplus    &        \\
 P_{n}^A & P_{n+1}^A &        \\
 \oplus  & \oplus    & & =(Q', \d_{i+1}) \\
 P_{n}^A & P^A(\wt{r_i}) & \wt{\nabla} \\
         & \oplus    &        \\
         & P^A(\undertilde{r_i}) & \undertilde{\nabla} \\
};

%%Arrows for first complex
\draw[<->](M-1-2) -- (M-1-3)
	;
\draw[<->](M-3-2) -- (M-1-3)
	;
\draw[<->](M-5-2) -- (M-5-3)
	; 
\draw[<->](M-7-2) -- (M-7-3)
	; 

\draw[-> ,font=\small](M-3-1) -- (M-1-2.west)
	;
\draw[-> ,font=\small](M-3-1) -- (M-3-2)
	;
\draw[-> ,font=\small](M-5-1) -- (M-1-2)
	;
\draw[-> ,font=\small](M-5-1) -- (M-3-2)
	;
	
\draw[-> ,font=\small](M-3-1) -- (M-5-2)
	;
\draw[-> ,font=\small](M-3-1) -- (M-7-2)
	;
\draw[-> ,font=\small](M-5-1) -- (M-5-2)
	;
\draw[-> ,font=\small](M-5-1) -- (M-7-2.west)
	;

%%Arrows for second complex
\draw[<->](M-11-2) -- (M-11-3)
	;
\draw[<->](M-13-2) -- (M-11-3)
	;
\draw[<->](M-15-2) -- (M-15-3)
	; 
\draw[<->](M-17-2) -- (M-17-3)
	; 

\draw[-> ,font=\small](M-13-1) -- (M-11-2)
	;
\draw[-> ,font=\small](M-15-1) -- (M-13-2)
	;
\draw[-> ,font=\small](M-13-1) -- (M-15-2)
	;
\draw[-> ,font=\small](M-15-1) -- (M-17-2)
	;
	
%%Arrows between two complexes
\draw[-> ,font=\small, color=blue](M-5-1) -- (M-13-1)
	node[midway, left,scale=0.6] {$M$};
\draw[-> ,font=\small, color=blue](M-7-2) -- (M-11-2)
	node[midway, left,scale=0.6] {$
	\begin{bmatrix}
	I & 0 \\
	0 & P
	\end{bmatrix}		
	$};	
\end{tikzpicture}
\end{center}
and $\mu_i$ sends $v$ to $-iv$ (resp. $iv$) for any $v$ belonging to the modules in $\wt{\nabla}$ (resp. $\undertilde{\nabla}$).
The black arrows in the last four rows shows the differential component $\d_{{i+1}_{G_i}}$ in $\d_{i+1}$, induced by the conjugation of $\mu_i^{-1}$.
Thus, the required condition $(*)$ is easily satisfied.

For $k > 0$ , we can draw the part of $(Q', \d_i)$ that contains $(Q'_{G_i}, \d_{i_{G_i}}) = (Q'_{G_i}, \d_{0_{G_i}})$, $L_A(\wt{h}_i)$ and $L_A(\undertilde{h_i})$ as
\begin{center}
\begin{tikzpicture}[>=stealth, baseline]% added "baseline"
\matrix (M) [matrix of math nodes, column sep=7mm]
{
P_n^A  & P_n^A  & P_n^A  & \cdots & P_n^A  & P_n^A  & P_{n-1}^A & \nabla' \\
\oplus & \oplus & \oplus &        & \oplus & \oplus & \oplus    & \\
P_n^A  & P_n^A  & P_n^A  & \cdots & P_n^A  & P_n^A  & P_{n+1}^A &        \\
       & \oplus & \oplus &        & \oplus & \oplus & \oplus & & & =(Q', \d_i)\\ 
       & P_n^A  & P_n^A  & \cdots & P_n^A  & P_n^A  & P^A(\wt{r_i})& \wt{\nabla} \\
       & \oplus & \oplus &        & \oplus & \oplus & \oplus    & \\
       & P_n^A  & P_n^A  & \cdots & P_n^A  & P_n^A  & P^A(\undertilde{r_i})& \undertilde{\nabla} \\
};

\draw[-> ,font=\small](M-1-1) -- (M-1-2) 
	node[midway,above] {$2X_n$};
\draw[-> ,font=\small](M-1-1) -- (M-5-2) 
	node[midway,above] {$2X_n$};
\draw[-> ,font=\small](M-1-2) -- (M-1-3) 
	node[midway,above] {$2X_n$};
\draw[-> ,font=\small](M-1-3) -- (M-1-4) 
	;
\draw[-> ,font=\small](M-1-4) -- (M-1-5)
	;
\draw[-> ,font=\small](M-1-5) -- (M-1-6)
	node[midway,above] {$2X_n$};
\draw[-> ,font=\small](M-1-6) -- (M-1-7)
	node[midway,above] {$F$};
\draw[-> ,font=\small](M-1-6) -- (M-3-7)
	;
\draw[<-> ,font=\small](M-1-7) -- (M-1-8)
	;
\draw[<-> ,font=\small](M-3-7) -- (M-1-8)
	;

\draw[-> ,font=\small](M-3-1) -- (M-3-2)
	node[midway,below] {$2X_n$};
\draw[-> ,font=\small](M-3-2) -- (M-3-3)
	node[midway,above] {$2X_n$};
\draw[-> ,font=\small](M-3-1) -- (M-7-2)
	node[midway,above] {$2X_n$};
\draw[-> ,font=\small](M-3-3) -- (M-3-4) 
	;
\draw[-> ,font=\small](M-3-4) -- (M-3-5)
	;
\draw[-> ,font=\small](M-3-5) -- (M-3-6)
	node[midway,above] {$2X_n$};
\draw[-> ,font=\small](M-3-6) -- (M-1-7)
	;
\draw[-> ,font=\small](M-3-6) -- (M-3-7)
	;

\draw[-> ,font=\small](M-5-2) -- (M-5-3)
	node[midway,above] {$2X_n$};
\draw[-> ,font=\small](M-5-3) -- (M-5-4) 
	;
\draw[-> ,font=\small](M-5-4) -- (M-5-5) 
	;
\draw[-> ,font=\small](M-5-5) -- (M-5-6)
	node[midway,above] {$2X_n$};
\draw[-> ,font=\small](M-5-6) -- (M-5-7)
	node[midway,above] {$F'$};
\draw[-> ,font=\small](M-5-6) -- (M-7-7)
	;
\draw[<-> ,font=\small](M-5-7) -- (M-5-8)
	;

\draw[-> ,font=\small](M-7-2) -- (M-7-3)
	node[midway,above] {$2X_n$};
\draw[-> ,font=\small](M-7-3) -- (M-7-4) 
	;
\draw[-> ,font=\small](M-7-4) -- (M-7-5) 
	;
\draw[-> ,font=\small](M-7-5) -- (M-7-6)
	node[midway,above] {$2X_n$};
\draw[-> ,font=\small](M-7-6) -- (M-5-7)
	;
\draw[-> ,font=\small](M-7-6) -- (M-7-7)
	;
\draw[<-> ,font=\small](M-7-7) -- (M-7-8)
	;
	
\draw[red] (M-5-7.north west) rectangle (M-5-8.south east)
	node[midway, above, xshift = 1cm, yshift = 5pt] 
  	{$L_A\left(\wt{h_i}\right)$};
\draw[Brown] (M-7-7.north west) rectangle (M-7-8.south east)
	node[midway, xshift = 2cm ] 
  	{$L_A\left(\undertilde{h_i}\right)$};
  	
\draw[OliveGreen] (M-1-1.north west)++(0,5pt) rectangle 
	(M-7-7.south east)
		node[midway, above, yshift = 2.1cm] 
  		{$(Q'_{G_i}, \d_{i_{G_i}})$};
\end{tikzpicture}
\end{center}
Similarly, we define $\mu_i|_{Q'_{G_i}}$ to be the following map in {\color{blue}blue}:
\begin{center}
\begin{tikzpicture}[>=stealth, baseline]% added "baseline"
\matrix (M) [matrix of math nodes, column sep=7mm]
{
P_n^A  & P_n^A  & P_n^A  & \cdots & P_n^A  & P_n^A  & P_{n-1}^A & \nabla' \\
\oplus & \oplus & \oplus &        & \oplus & \oplus & \oplus    & \\
P_n^A  & P_n^A  & P_n^A  & \cdots & P_n^A  & P_n^A  & P_{n+1}^A &        \\
       & \oplus & \oplus &        & \oplus & \oplus & \oplus & & =(Q', \d_i)\\ 
       & P_n^A  & P_n^A  & \cdots & P_n^A  & P_n^A  & P^A(\wt{r_i})& \wt{\nabla} \\
       & \oplus & \oplus &        & \oplus & \oplus & \oplus    & \\
       & P_n^A  & P_n^A  & \cdots & P_n^A  & P_n^A  & P^A(\undertilde{r_i})&\undertilde{\nabla} \\
{}     &        &        &        &        &        &          \\
{}     &        &        &        &        &        &          \\ 
{}     &        &        &        &        &        &          \\
P_n^A  & P_n^A  & P_n^A  & \cdots & P_n^A  & P_n^A  & P_{n-1}^A & \nabla' \\
\oplus & \oplus & \oplus &        & \oplus & \oplus & \oplus    & \\
P_n^A  & P_n^A  & P_n^A  & \cdots & P_n^A  & P_n^A  & P_{n+1}^A &        \\
       & \oplus & \oplus &        & \oplus & \oplus & \oplus & & =:(Q', \d_{i+1})\\ 
       & P_n^A  & P_n^A  & \cdots & P_n^A  & P_n^A  & P^A(\wt{r_i})& \wt{\nabla} \\
       & \oplus & \oplus &        & \oplus & \oplus & \oplus    & \\
       & P_n^A  & P_n^A  & \cdots & P_n^A  & P_n^A  & P^A(\undertilde{r_i})&\undertilde{\nabla} \\
};

%%Arrows for first complex
\draw[-> ,font=\small](M-1-1) -- (M-1-2) 
	node[midway,above] {$2X_n$};
\draw[-> ,font=\small](M-1-1) -- (M-5-2) 
	node[midway,above] {$2X_n$};
\draw[-> ,font=\small](M-1-2) -- (M-1-3) 
	node[midway,above] {$2X_n$};
\draw[-> ,font=\small](M-1-3) -- (M-1-4) 
	;
\draw[-> ,font=\small](M-1-4) -- (M-1-5)
	;
\draw[-> ,font=\small](M-1-5) -- (M-1-6)
	node[midway,above] {$2X_n$};
\draw[-> ,font=\small](M-1-6) -- (M-1-7)
	node[midway,above] {$F$};
\draw[-> ,font=\small](M-1-6) -- (M-3-7)
	;
\draw[<-> ,font=\small](M-1-7) -- (M-1-8)
	;
\draw[<-> ,font=\small](M-3-7) -- (M-1-8)
	;

\draw[-> ,font=\small](M-3-1) -- (M-3-2)
	node[midway,below] {$2X_n$};
\draw[-> ,font=\small](M-3-2) -- (M-3-3)
	node[midway,above] {$2X_n$};
\draw[-> ,font=\small](M-3-1) -- (M-7-2)
	node[midway,above] {$2X_n$};
\draw[-> ,font=\small](M-3-3) -- (M-3-4) 
	;
\draw[-> ,font=\small](M-3-4) -- (M-3-5)
	;
\draw[-> ,font=\small](M-3-5) -- (M-3-6)
	node[midway,above] {$2X_n$};
\draw[-> ,font=\small](M-3-6) -- (M-1-7)
	;
\draw[-> ,font=\small](M-3-6) -- (M-3-7)
	;

\draw[-> ,font=\small](M-5-2) -- (M-5-3)
	node[midway,above] {$2X_n$};
\draw[-> ,font=\small](M-5-3) -- (M-5-4) 
	;
\draw[-> ,font=\small](M-5-4) -- (M-5-5) 
	;
\draw[-> ,font=\small](M-5-5) -- (M-5-6)
	node[midway,above] {$2X_n$};
\draw[-> ,font=\small](M-5-6) -- (M-5-7)
	node[midway,above] {$F'$};
\draw[-> ,font=\small](M-5-6) -- (M-7-7)
	;
\draw[<-> ,font=\small](M-5-7) -- (M-5-8)
	;

\draw[-> ,font=\small](M-7-2) -- (M-7-3)
	node[midway,above] {$2X_n$};
\draw[-> ,font=\small](M-7-3) -- (M-7-4) 
	;
\draw[-> ,font=\small](M-7-4) -- (M-7-5) 
	;
\draw[-> ,font=\small](M-7-5) -- (M-7-6)
	node[midway,above] {$2X_n$};
\draw[-> ,font=\small](M-7-6) -- (M-5-7)
	;
\draw[-> ,font=\small](M-7-6) -- (M-7-7)
	;
\draw[<-> ,font=\small](M-7-7) -- (M-7-8)
	;
%%Arrows for second complex
\draw[-> ,font=\small](M-11-1) -- (M-11-2) 
	node[midway,above] {$X_n$};
\draw[-> ,font=\small](M-11-1) -- (M-15-2) 
	node[midway,below] {$X_n$};
\draw[-> ,font=\small](M-11-2) -- (M-11-3) 
	node[midway,above] {$X_n$};
\draw[-> ,font=\small](M-11-3) -- (M-11-4) 
	;
\draw[-> ,font=\small](M-11-4) -- (M-11-5)
	;
\draw[-> ,font=\small](M-11-5) -- (M-11-6)
	node[midway,above] {$X_n$};
\draw[-> ,font=\small](M-11-6) -- (M-11-7)
	node[midway,above,scale=0.7] {$(n|n-1)$};
\draw[<-> ,font=\small](M-11-7) -- (M-11-8)
	;
\draw[<-> ,font=\small](M-13-7) -- (M-11-8)
	;

\draw[-> ,font=\small](M-13-1) -- (M-13-2)
	node[midway,above] {$X_n$};
\draw[-> ,font=\small](M-13-2) -- (M-13-3)
	node[midway,above] {$X_n$};
\draw[-> ,font=\small](M-13-1) -- (M-17-2)
	node[midway,left] {$X_n$};
\draw[-> ,font=\small](M-13-3) -- (M-13-4) 
	;
\draw[-> ,font=\small](M-13-4) -- (M-13-5)
	;
\draw[-> ,font=\small](M-13-5) -- (M-13-6)
	node[midway,above] {$X_n$};
\draw[-> ,font=\small](M-13-6) -- (M-13-7)
	node[midway,above,scale=0.7] {$(n|n+1)$};

\draw[-> ,font=\small](M-15-2) -- (M-15-3)
	node[midway,above] {$X_n$};
\draw[-> ,font=\small](M-15-3) -- (M-15-4) 
	;
\draw[-> ,font=\small](M-15-4) -- (M-15-5)
	;
\draw[-> ,font=\small](M-15-5) -- (M-15-6)
	node[midway,above] {$X_n$};
\draw[-> ,font=\small](M-15-6) -- (M-15-7)
	node[midway,above,scale=0.7] {$(n|n+1)$};
\draw[<-> ,font=\small](M-15-7) -- (M-15-8)
	;

\draw[-> ,font=\small](M-17-2) -- (M-17-3)
	node[midway,above] {$X_n$};
\draw[-> ,font=\small](M-17-3) -- (M-17-4) 
	;
\draw[-> ,font=\small](M-17-4) -- (M-17-5)
	;
\draw[-> ,font=\small](M-17-5) -- (M-17-6)
	node[midway,above] {$X_n$};
\draw[-> ,font=\small](M-17-6) -- (M-17-7)
	node[midway,above,scale=0.7] {$(n|n-1)$};
\draw[<-> ,font=\small](M-17-7) -- (M-17-8)
	;

%%Arrows between two complexes
\draw[-> ,font=\small, color=blue](M-3-1) -- (M-11-1)
	node[midway,left] {$2^{k}M$};
\draw[-> ,font=\small](M-7-2) -- (M-11-2)
	node[midway,left,scale=0.6, color=blue] {$2^{k-1}
		\begin{bmatrix}
		M & 0 \\
		0 & M \\
		\end{bmatrix}
		$};
\draw[-> ,font=\small, color=blue](M-7-3) -- (M-11-3)
	node[midway, left,scale=0.6] {$ 2^{k-2}
		\begin{bmatrix}
		M & 0 \\
		0 & M \\
		\end{bmatrix}
		$};
\draw[-> ,font=\small, color=blue](M-7-5) -- (M-11-5)
	node[midway, left,scale=0.6] {$2
		\begin{bmatrix}
		M & 0 \\
		0 & M \\
		\end{bmatrix}
		$};
\draw[-> ,font=\small, color=blue](M-7-6) -- (M-11-6)
	node[midway, left,scale=0.6] {$
		\begin{bmatrix}
		M & 0 \\
		0 & M \\
		\end{bmatrix}
		$};
\draw[-> ,font=\small, color=blue](M-7-7) -- (M-11-7)
	node[midway, right,scale=0.6]{$
	\begin{bmatrix}
	I & 0 \\
	0 & P 
	\end{bmatrix}		
	$};
\end{tikzpicture}
\end{center}
and $\mu_i$ sends $v$ to $-iv$ (resp. $iv$) for any $v$ belonging to the $\nabla$ connected to $P^A(r_{n+1})$ (resp. $P^A(r_{n+1})$).
The black arrows in the last two rows shows the differential component $\d_{i+1}(Q'_{G_i})$ in $T_{i+1}$, given by the conjugation of $\mu_i^{-1}$.
Thus,  the required condition $(*)$ is easily satisfied.
The construction for $k < 0$ is similar, using the map $N$ in place of $M$.
\end{enumerate}
This completes the list of $\mu_i$ for all possible types of 1-string $\check{g}_i$ for all $i$, and thus completing the proof.
\end{proof}
\begin{remark}
\cref{diagram commutes} was stated to be true in the (triangulated) homotopy category $\Kom^b(\Ba_n$-$\text{p$_{r}$g$_{r}$mod})$.
However, the proof shows that it is indeed true in the (additive) category of complexes (not up to homotopy) $\Com^b(\Ba_n$-$\text{p$_{r}$g$_{r}$mod})$.
\end{remark}
Using this result, we can now deduce the type $B_n$ version of \cref{L_A equivariant}:
\begin{theorem} \label{L_B equivariant}
For $\sigma^B \in \cA({B_n})$ and an admissible trigraded curve $\check{c}$ in $\D^B_{n+1}$, we have that $$\sigma^B(L_B(\check{c})) \cong L_B(\sigma^B(\check{c}))$$ in the category of $\Kom^b(\Ba_n$-$p_rg_rmod)$, i.e. the map $L_B$ is $\mathcal{A}(B_n)$-equivariant.
\end{theorem}
\begin{proof}
Let $\check{c}$ be a trigraded curve in $\D^B_{n+1}$ and $\sigma^B$ be an element of $\mathcal{A}(B_n)$.
By \cref{diagram commutes}, the diagram in \cref{full picture} commutes.
Together with the three other maps being equivariant, we can conclude that
\begin{equation} \label{commute iso}
\Aa_{2n-1} \otimes_{\Ba_n} \left( L_B(\sigma^B(\check{c}))\right) \cong \Aa_{2n-1} \otimes_{\Ba_n} (\sigma^B(L_B(\check{c}))).
\end{equation}
Recall that the functor $\Aa_{2n-1} \otimes_{\Ba_n} -$ was defined as $\Aa_{2n-1} \otimes_{\ddot{\Ba}_n} \mathfrak{F}(-)$ (see paragraph after \cref{isomorphic algebras} for the definitions).
Since $\Aa_{2n-1} \cong \C\otimes_\R \ddot{\Ba}_n$ as graded $\C$-algebras, \cref{commute iso} together with the fact that both categories $\Kom^b(\ddot{\Ba}_n\text{-}p_rg_rmod)$ and $\Kom^b(\Aa_{2n-1}\text{-}p_rg_rmod)$ are Krull-Schmidt implies that
\begin{comment}
Since $\Ba$ is a finite dimensional $\R$-algebra and $\C$ is a degree two extension of $\R$, the functor $\Aa_{2n-1} \otimes_{\ddot{\Ba}_n} - \cong \left( \C\otimes_\R \ddot{\Ba}_n\right) \otimes_{\ddot{\Ba}_n} -$ reflects isomorphic objects up to a shift in $\< - \>$.
Thus we conclude that
\end{comment}
\begin{equation}\label{iso shift}
L_B(\sigma^B(\check{c})) \cong \sigma^B (L_B(\check{c}))\<s\>,
\end{equation}
with $s = 0$ or $1$.

We aim to show that $s = 0$ for all cases.
First consider the case when $c \cap \{0\} = \emptyset$.
As $\check{c} \simeq \chi(r_1,r_2,r_3)\beta(\check{b}_2)$ and $P_2^B \cong P_2^B \<1\>$, it follows easily that $(L_B(\check{c}))\<0\> \cong (L_B(\check{c}))\<1\>$, and so we are done.
Now consider when $c$ has one of its end points at $0$.
Note that it is sufficient to prove the statement for $\sigma^B = \sigma^B_j$ for each $j$.
We assign to each complex $C$ in $\Kom^b(\Ba_n$-$\text{p$_{r}$g$_{r}$mod})$ an element of $\Z/2\Z$ denoted by $\sgn(C)$, by taking the sum of the third grading $\<-\>$ over all modules $P_1$ in $C$.
   It is easy to show that $\sgn$ is invariant under isomorphisms in  $\Kom^b(\Ba_n$-$\text{p$_{r}$g$_{r}$mod})$, and so \cref{gradinginvariant} below shows that $s$ must be 0.
   \end{proof}
%For instance, \newline
%$\sgn\left(P_1 \oplus P_1\<1\> \xra{[(1|2) \  (1|2)ie_2]} P_2 \right) = 0 + 1 \equiv 1. $
%
\begin{lemma} \label{gradinginvariant}
For any trigraded curve $\check{c}$ with one of its endpoint at $\{0\}$ and any generating braid $\sigma^B_j,$  we have $\sgn \left( \sigma^B_j L_B(\check{c}) \right) = \sgn (L_B(\sigma^B_j(\check{c}))).$
\end{lemma}

\begin{proof} (of \cref{gradinginvariant})
For $j \geq 2,$ it is clear that
$\sgn \left( \sigma^B_j L_B(\check{c}) \right) = \sgn \left( L_B(\check{c}) \right) = \sgn (L_B(\sigma^B_j(\check{c}))).$
\begin{comment}
 This is because of the number of the modules $P_1$ in $\sigma^B_j L_B(\check{c})$ won't change compared to that in $\sgn \left( L_B(\check{c}) \right) $, so does their third grading.
 Thus, $\sgn \left( \sigma^B_j L_B(\check{c}) \right) = \sgn \left( L_B(\check{c}) \right) .$
  Moreover,  the number of $1$-crossings in $\sigma^B_j(\check{c})$ also doesn't change that in $\check{c}$, so does their local index.
  Therefore, $\sgn (L_B(\sigma^B_j(\check{c})))= \sgn \left( L_B(\check{c}) \right).$
  As a result, for $j \geq 2,$
  $$\sgn \left( \sigma^B_j L_B(\check{c}) \right) = \sgn \left( L_B(\check{c}) \right) = \sgn (L_B(\sigma^B_j(\check{c}))).$$
\end{comment}
  
  Now fix $j=1.$ 
  First consider the case when $\check{c}$ is of type VI, i.e.  $\check{c} = \chi^B(r_1, r_2, r_3)\check{b}_1$.
  Then $\sigma^B_1(\check{c}) = \chi^B(r_1-1, r_2+1, r_3+1)\check{b}_1 $ by \cref{action}(2). 
  So, $L_B(\sigma^B_1(\check{c})) = P_1 [-r_1+1] \{r_2 +1\} \< r_3+1\>$.
  On the other hand, 
\begin{align*}
\sigma^B_1 L_B((\check{c})) 
&\cong  P_1 [-r_1+1] \{r_2 +1 \} \< r_3+1\> \oplus P_1 [-r_1 +1] \{r_2 \} \< r_3 + 1\> \xra{[X_1 \ id]} P_1 [-r_1] \{r_2 \} \< r_3 + 1\> \\
&\cong  P_1 [-r_1+1] \{r_2 +1 \} \< r_3+1\>. 
\end{align*}  
\noindent Thus, $\sgn \left( \sigma^B_j L_B(\check{c}) \right) = r_3 +1 = \sgn (L_B(\sigma^B_j(\check{c}))).$
  
  Otherwise, we analyse $\sgn$ based on the number of $2$-crossing in $\check{c}.$
  Note that,  for $1$-strings $\check{g},$
 \[\sgn \left( L_B(\sigma^B_1 (\check{g})) \right)  =  \begin{cases} 
     \sgn \left( L_B(\check{g}) \right),  & \text{when $\check{g}$ is of type $II'_w$, $II'_{w+ \frac{1}{2}}$, $III'_{w+ \frac{1}{2}}$ and $V'$;} \\
     \sgn \left( L_B(\check{g}) \right) + 1, & \text{when $\check{g}$ is of type $III'_w$.}\\       
   \end{cases}
    \]  
   Note that $\sigma^B_1$ won't change the number of $2$-crossings of $\check{c}$ and as $\check{c}\cap \{0\} = \{0\}$, $\check{c}$ contains $1$-string of type  $III'_w$ if and only if $\check{c}$ has even number of 2-crossings.
   Since $\sgn \left(  L_B(\sigma^B_1 (\check{c})) \right)$ can be computed by summing over all 1-strings of $\check{c}$, we conclude that
    \[\sgn \left(  L_B(\sigma^B_1 (\check{c})) \right)  =  \begin{cases} 
     \sgn \left( L_B(\check{c}) \right),  & \text{if $\check{c}$ has even number of $2$-crossings;} \\
     \sgn \left( L_B(\check{c}) \right) + 1, & \text{if $\check{c}$ has odd number of $2$-crossings,}\\       
   \end{cases}
    \]
 \noindent 
 On the other hand,  it follows that 
  \[\sgn \left( \sigma^B_1  L_B(\check{c}) \right)  =  \begin{cases} 
     \sgn \left( L_B(\check{c}) \right),  & \text{if $L_B(\check{c}$) has even number of modules $P^B_2$;} \\
     \sgn \left( L_B(\check{c}) \right) + 1, & \text{if $L_B(\check{c}$) has odd number of modules $P^B_2$,}\\       
   \end{cases}
    \]
    since ${}_1P^B\otimes_\R P^B_1 \cong \R \oplus \R$, ${}_1P^B\otimes_\R P^B_2 \cong \R \oplus \R \<1\>$ and ${}_1P^B \otimes_\R P^B_j = 0$ for all $j \geq 3$.
 Thus, we get $\sgn \left( \sigma^B_1 ( L_B(\check{c}) )\right) = \sgn \left(   L_B(\sigma^B_1(\check{c})) \right). $                                                                                                                                                                                                                                                                                                                                                                                                                                                                                                                                                                                                                                                                                                                                                                                                                                                                                                                                                           
  \end{proof}

Collecting all the results, we have the following main theorem of this paper:
\begin{theorem} \label{fullmaintheorem}
The diagram in \cref{full picture} is commutative, where the maps $\mathfrak{m}, L_B, L_A$ and $\Aa_{2n-1} \otimes_{\Ba_n} -$ are all $\mathcal{A}(B_n)$-equivariant.
\end{theorem}

\section{Categorification of Homological Representations}\label{categorification hom rep}
In this section, we shall relate the categorical representations of type $A_{2n-1}$ and type $B_n$ Artin groups to their representations on the first homologies of surfaces.

Throughout this section we will denote $\cK_A := \Kom^b(\Aa_{2n-1}$-$\text{p$_{r}$g$_{r}$mod})$ and $\cK_B := \Kom^b(\Ba_n$-$\text{p$_{r}$g$_{r}$mod})$.
We denote the Grothendieck group of a triangulated category $\mathcal{C}$ as $K_0(\mathcal{C})$; recall that it is the abelian group freely generated by objects in $\mathcal{C}$, quotient by the relation
$
[B] = [A] + [C]
$
for each exact triangle $A \ra B \ra C \ra$.

\subsection{Categorical representations}

First, consider the Grothendieck group $K_0(\cK_A)$.
The functor $\{1\}$ make $K_0(\cK_A)$ into a $\Z[\Z] \cong \Z[q,q^{-1}] =: \mathcal{Z}_A$-module.
It is easy to see that $K_0(\cK_A)$ is a free module over $\mathcal{Z}_A$ of rank $2n-1$, generated by $\{[P^A_j] \mid 1 \leq j \leq 2n-1\}$.
Since $\cA(B_n)$ acts on $\cK_B$, we have an induced representation $\rho_{KA}$ of $\cA(A_{2n-1})$ on $K_0(\cK_A)$, given by the following on each of the generators of $\cA(A_{2n-1})$:

\begin{align*}
\rho_{KA}(\sigma^A_1) &= 
\begin{bmatrix}
-q & -q & 0 \\
0  &  1 & 0 \\
0  &  0 & I_{n-2}
\end{bmatrix}, \\
\rho_{KA}(\sigma^A_j) &=
\begin{cases}
\vspace{3mm}
\begin{bmatrix}
I_{j-2} & 0  & 0  & 0  & 0 \\
0       & 1  & 0  & 0  & 0 \\
0       & -1 & -q & -q & 0 \\
0       & 0  & 0  & 1  & 0 \\
0       & 0  & 0  & 0  & I_{n-j-1}
\end{bmatrix}, \qquad \text{for } 2 \leq j \leq n-1; \\
\vspace{3mm}
\begin{bmatrix}
I_{j-2} & 0  & 0  & 0  & 0 \\
0       & 1  & 0  & 0  & 0 \\
0       & -1 & -q & -1 & 0 \\
0       & 0  & 0  & 1  & 0 \\
0       & 0  & 0  & 0  & I_{n-j-1}
\end{bmatrix}, \qquad \text{for }  j = n; \\
\begin{bmatrix}
I_{j-2} & 0  & 0  & 0  & 0 \\
0       & 1  & 0  & 0  & 0 \\
0       & -q& -q & -1 & 0 \\
0       & 0  & 0  & 1  & 0 \\
0       & 0  & 0  & 0  & I_{n-j-1}
\end{bmatrix}, \qquad \text{for } n+1 \leq j \leq 2n-2;
\end{cases} \\
\rho_{KA}(\sigma^A_{2n-1}) &= 
\begin{bmatrix}
I_{n-2} & 0  & 0 \\
0       & 1  & 0 \\
0       & -q & -q
\end{bmatrix}
\end{align*}

Now consider the Grothendieck group $K_0(\cK_B)$.
The functors $\{1\}$ and $\<1\>$ make $K_0(\cK_B)$ into a $\Z[\Z\times \Z/2\Z] \cong \Z[q,q^{-1},s]/\<s^2-1\> =: \mathcal{Z}_B$-module.
As $P^B_j\<1\> \cong P^B_j$ for all $j \geq 2$, we have that $s[P_j] = [P_j]$ for all $j\geq 2$.
It is easy to see now that

$$
K_0(\cK_B) \cong \mathcal{Z}_B \oplus \left( \bigoplus_{j=2}^n \mathcal{Z}_B/\<s-1\> \right)
$$

\noindent as $\mathcal{Z}_B$-modules, generated by $\{[P^B_j] | 1 \leq j \leq n\}$.
Since $\cA(B_n)$ acts on $\cK_B$, we have an induced representation $\rho_{BK}$ of $\cA(B_n)$ on $K_0(\cK_B)$, given by the following on each of the generators of $\cA(B_n)$ 

\begin{align*}
\rho_{KB}(\sigma^B_1) &= 
\begin{bmatrix}
-sq & -(1+s) & 0 \\
0  &  1     & 0 \\
0  &  0     & I_{n-2}
\end{bmatrix}; \\
\rho_{KB}(\sigma^B_j) &=
\begin{bmatrix}
I_{j-2} & 0  & 0  & 0  & 0 \\
0       & 1  & 0  & 0  & 0 \\
0       & -q & -q & -1 & 0 \\
0       & 0  & 0  & 1  & 0 \\
0       & 0  & 0  & 0  & I_{n-j-1}
\end{bmatrix}, \qquad \text{for } j \neq 1, n; \\
\rho_{KB}(\sigma^B_n) &= 
\begin{bmatrix}
I_{n-2} & 0  & 0 \\
0       & 1  & 0 \\
0       & -q & -q
\end{bmatrix}.
\end{align*}

\begin{comment}
It follows that (see \rb{REF?}) $\rho_{BK}|_{s=1}$ is isomorphic to the type $B_n$ reduced Burau representation.
Hence we obtain the following:
\begin{proposition}
The $\cA(B_n)$-action on $\cK_B$ categorifies the reduced Burau representation of $\cA(B_n)$.
\end{proposition}
\end{comment}

\subsection{Homological representations}
Let us first recall the homological representation of $\cA(A_{2n-1})$ that is isomorphic to the reduced Burau representation.
Note that it has been shown in \cite{KhoSei} that the categorical representation categorifies the reduced Burau representation (and hence this homological representation), but we shall spell out the construction here as it will shed some light on the construction of the homological representation for type $B_n$ and also make clear the relationship between them.

Consider the covering space $\cD_{2n}$ classified by the cohomology class $C_{\cD} \in H^1(\DA \setminus \Delta, \Z)$ defined by
$
[\lambda_j] \mapsto 1, \text{for all } j \in \Delta = \{-n, ..., -1, 1, ..., n\},
$
where each $\lambda_j$ is a closed loop around the puncture $j$.
It is easy to see that $H_1(\cD_{2n}, \Z)$ is a free module over $\Z[\Z] \cong \mathcal{Z}_A$ of rank $2n-1$, with basis $\{[\gamma_1], ..., [\gamma_{2n-1}] \}$ defined by:
\[
[\gamma_{j}] := 
\begin{cases}
[\lambda_{-1}] - [\lambda_{1}], &\text{ for } j = n; \\
(-1)^{n-j}([\lambda_{-n+j-1}] - [\lambda_{-n+j}]), &\text{ for } j \leq n-1; \\
(-q)^{n-j}([\lambda_{-n+j}] - [\lambda_{-n+j+1}], &\text{ for } j \geq n+1,
\end{cases}
\]

\noindent and the space $\cD_{2n}$ is homotopy equivalent to the infinite graph $\Gamma_{\Z}$ given below:
\begin{figure}[H]  
\centering
\begin{tikzpicture} [scale=0.7]

\draw[fill] (3,2.5) circle [radius=0.08] ;
\draw[fill] (6,2.5) circle [radius=0.08]  ;
\draw[fill] (9,2.5) circle [radius=0.08]  ;
\draw[fill] (12,2.5) circle [radius=0.08]  ;
\draw[fill] (0,2.5) circle [radius=0.08];

\draw  [cyan, dashed] plot[smooth, tension=1]coordinates { (0,2.5) (-.55, 3) (-1.25, 3.2)  };
\draw  [cyan, dashed] plot[smooth, tension=1]coordinates { (0,2.5) (-.55, 2.85) (-1.25, 3)  };
\draw  [cyan, dashed] plot[smooth, tension=1]coordinates { (0,2.5) (-.55, 2.75) (-1.25, 2.8)  };
\draw  [cyan, dashed] plot[smooth, tension=1]coordinates { (0,2.5) (-.55, 2) (-1.25, 1.8)  };
\draw  [cyan, dashed] plot[smooth, tension=1]coordinates { (0,2.5) (-.55, 2.15) (-1.25, 2)  };
\draw  [cyan, dashed] plot[smooth, tension=1]coordinates { (0,2.5) (-.55, 2.25) (-1.25, 2.2)  };

\draw  [cyan, dashed] plot[smooth, tension=1]coordinates { (12,2.5) (12.55, 3) (13.25, 3.2)  };
\draw  [cyan, dashed] plot[smooth, tension=1]coordinates { (12,2.5) (12.55, 2.85) (13.25, 3)  };
\draw  [cyan, dashed] plot[smooth, tension=1]coordinates { (12,2.5) (12.55, 2.75) (13.25, 2.8)  };
\draw  [cyan, dashed] plot[smooth, tension=1]coordinates { (12,2.5) (12.55, 2) (13.25, 1.8)  };
\draw  [cyan, dashed] plot[smooth, tension=1]coordinates { (12,2.5) (12.55, 2.15) (13.25, 2)  };
\draw  [cyan, dashed] plot[smooth, tension=1]coordinates { (12,2.5) (12.55, 2.25) (13.25, 2.2)  };

\draw  [cyan] plot[smooth, tension=1]coordinates { (0,2.5) (1.5, 3.2) (3,2.5)};
\draw  [cyan] plot[smooth, tension=1]coordinates { (0,2.5) (1.5, 3) (3,2.5)};
\draw  [cyan] plot[smooth, tension=1]coordinates { (0,2.5) (1.5, 2.8) (3,2.5)};
\draw  [cyan] plot[smooth, tension=1]coordinates { (0,2.5) (1.5, 2) (3,2.5)};
\draw  [cyan] plot[smooth, tension=1]coordinates { (0,2.5) (1.5, 2.2) (3,2.5)};
\draw  [cyan] plot[smooth, tension=1]coordinates { (0,2.5) (1.5, 1.8) (3,2.5)};

\draw  [cyan] plot[smooth, tension=1]coordinates { (3,2.5) (4.5, 3.2) (6,2.5)};
\draw  [cyan] plot[smooth, tension=1]coordinates { (3,2.5) (4.5, 3) (6,2.5)};
\draw  [cyan] plot[smooth, tension=1]coordinates { (3,2.5) (4.5, 2.8) (6,2.5)};
\draw  [cyan] plot[smooth, tension=1]coordinates { (3,2.5) (4.5, 2) (6,2.5)};
\draw  [cyan] plot[smooth, tension=1]coordinates { (3,2.5) (4.5, 2.2) (6,2.5)};
\draw  [cyan] plot[smooth, tension=1]coordinates { (3,2.5) (4.5, 1.8) (6,2.5)};

\draw  [cyan] plot[smooth, tension=1]coordinates { (6,2.5) (7.5, 3.2) (9,2.5)};
\draw  [cyan] plot[smooth, tension=1]coordinates { (6,2.5) (7.5, 3) (9,2.5)};
\draw  [cyan] plot[smooth, tension=1]coordinates { (6,2.5) (7.5, 2.8) (9,2.5)};
\draw  [cyan] plot[smooth, tension=1]coordinates { (6,2.5) (7.5, 2) (9,2.5)};
\draw  [cyan] plot[smooth, tension=1]coordinates { (6,2.5) (7.5, 2.2) (9,2.5)};
\draw  [cyan] plot[smooth, tension=1]coordinates { (6,2.5) (7.5, 1.8) (9,2.5)};

\draw  [cyan] plot[smooth, tension=1]coordinates { (9,2.5) (10.5, 3.2) (12,2.5)};
\draw  [cyan] plot[smooth, tension=1]coordinates { (9,2.5) (10.5, 3) (12,2.5)};
\draw  [cyan] plot[smooth, tension=1]coordinates { (9,2.5) (10.5, 2.8) (12,2.5)};
\draw  [cyan] plot[smooth, tension=1]coordinates { (9,2.5) (10.5, 2) (12,2.5)};
\draw  [cyan] plot[smooth, tension=1]coordinates { (9,2.5) (10.5, 2.2) (12,2.5)};
\draw  [cyan] plot[smooth, tension=1]coordinates { (9,2.5) (10.5, 1.8) (12,2.5)};

\node [cyan] at (1.5,3.2) {$>$};
\node [cyan] at (1.5,3) {$>$};
\node [cyan] at (1.5,2.8) {$>$};
\node [cyan] at (1.5,2.2) {$>$};
\node [cyan] at (1.5,2) {$>$};
\node [cyan] at (1.5,1.8) {$>$};

\node [cyan] at (4.5,3.2) {$>$};
\node [cyan] at (4.5,3) {$>$};
\node [cyan] at (4.5,2.8) {$>$};
\node [cyan] at (4.5,2.2) {$>$};
\node [cyan] at (4.5,2) {$>$};
\node [cyan] at (4.5,1.8) {$>$};

\node [cyan] at (7.5,3.2) {$>$};
\node [cyan] at (7.5,3) {$>$};
\node [cyan] at (7.5,2.8) {$>$};
\node [cyan] at (7.5,2.2) {$>$};
\node [cyan] at (7.5,2) {$>$};
\node [cyan] at (7.5,1.8) {$>$};

\node [cyan] at (10.5,3.2) {$>$};
\node [cyan] at (10.5,3) {$>$};
\node [cyan] at (10.5,2.8) {$>$};
\node [cyan] at (10.5,2.2) {$>$};
\node [cyan] at (10.5,2) {$>$};
\node [cyan] at (10.5,1.8) {$>$};

\node  at (1.5,2.6) {$ \boldsymbol{\vdots}$};
\node  at (4.5,2.6) {$ \boldsymbol{\vdots}$};
\node  at (10.5,2.6) {$ \boldsymbol{\vdots}$};
\node  at (7.5,2.6) {$ \boldsymbol{\vdots}$};

\node[above,cyan] at (1.5,3.2) {\scriptsize {$q^{-1}\lambda_n$}};
\node[above,cyan] at (4.5,3.2) {\scriptsize {$\lambda_n$}};
\node[above,cyan] at (7.5,3.2) {\scriptsize {$q\lambda_n$}};
\node[above,cyan] at (10.5,3.2) {\scriptsize {$q^{2}\lambda_n$}};
\node[below,cyan] at (10.5,1.8) {\scriptsize {$q^{2}\lambda_{-n}$}};
\node[below,cyan] at (1.5,1.8) {\scriptsize {$q^{-1}\lambda_{-n}$}};
\node[below,cyan] at (4.5,1.8) {\scriptsize {$\lambda_{-n}$}};
\node[below,cyan] at (7.5,1.8) {\scriptsize {$q\lambda_{-n}$}};

\node[below] at  (0,2.5) {\scriptsize {$q^{-1}p$}};
\node[below] at  (3,2.5) {\scriptsize {$p$}};
\node[below] at  (6,2.5) {\scriptsize {$qp$}};
\node[below] at  (9,2.5) {\scriptsize {$q^2p$}};
\node[below] at  (12,2.5) {\scriptsize {$q^3p$}};

\end{tikzpicture}

\caption{{\small The infinite graph $\Gamma_\Z.$} } 
\end{figure}
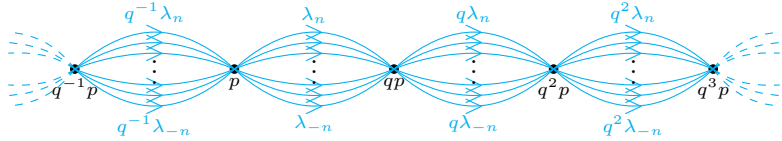
The action of $\cA(A_{2n-1})$ on $\DA \setminus \Delta$ lifts to an action on $\cD_{2n}$ that commutes with the deck transformation $\Z$, which then induces an action on the $\Z[q,q^{-1}]$-module $H_1(\cD_{2n}, \Z)$.
Through this we obtain a linear representation:
\[
\rho_{RHA} : \cA(A_{2n-1}) \ra GL_{2n-1}(\Z[q,q^{-1}]),
\]
where a simple calculation shows that $\rho_{RHA}$ agrees with the linear representation $\rho_{KA}$.
Thus we have the following result:
\begin{proposition}\label{Hom and cat rep A}
The $\mathcal{Z}_A$-linear map $\Theta_A: H_1(\cD_{2n}, \Z) \ra K_0(\cK_A)$ defined by
\[
\Theta_A([\gamma_j]) = [P^A_j]
\]
is an isomorphism of $\cA(A_{2n-1})$-representations.
\end{proposition}

Now consider the covering space $\sD_{2n+1}$ of $\DA \setminus \Delta_0$ classified by the cohomology class $C_\sD \in H^1(\DA \setminus \Delta_0, \Z)$ defined by
\[
[\lambda_j] \mapsto 
\begin{cases}
1, &\text{ for } j \neq 0 \\
0, &\text{ for } j = 0;
\end{cases}
\]
where each $\lambda_j$ is a closed loop around the puncture $j$.
Note that the composition of coverings
\[
\sD_{2n+1} \ra \DA \setminus \Delta_0 \ra \DB \setminus \Lambda
\]
is a normal covering space of $\DB \setminus \Lambda$, with its group of deck transformation isomorphic to $\Z \times \Z/2\Z$.
Note that $H_1(\sD_{2n+1}, \Z)$ is a module over $\Z[\Z\times\Z/2\Z] \cong \Z[q,q^{-1}, r]/\< r^2-1\> =: \mathcal{Z}_{B,r}$, and the space $\sD_{2n+1}$ is homotopy equivalent to the space given below:

\begin{figure}[H] 
\centering
\begin{tikzpicture}  [scale=0.8]

\draw[fill] (3,2.5) circle [radius=0.08] ;
\draw[fill] (6,2.5) circle [radius=0.08]  ;
\draw[fill] (9,2.5) circle [radius=0.08]  ;
\draw[fill] (12,2.5) circle [radius=0.08]  ;
\draw[fill] (0,2.5) circle [radius=0.08];

\draw[fill] (3,5) circle [radius=0.08] ;
\draw[fill] (6,5) circle [radius=0.08]  ;
\draw[fill] (9,5) circle [radius=0.08]  ;
\draw[fill] (12,5) circle [radius=0.08]  ;
\draw[fill] (0,5) circle [radius=0.08];

\draw  [green] plot[smooth, tension=1]coordinates { (0,2.5) (0.25, 3.75) (0,5)};
\draw  [green] plot[smooth, tension=1]coordinates { (0,2.5) (-0.25, 3.75) (0,5)};
\node [green] at  (-0.25, 3.75) {$\wedge$};
\node [green] at  (0.25, 3.75) {$\vee$};

\draw  [green] plot[smooth, tension=1]coordinates { (3,2.5) (2.75, 3.75) (3,5)};
\draw  [green] plot[smooth, tension=1]coordinates { (3,2.5) (3.25, 3.75) (3,5)};
\node [green] at  (2.75, 3.75) {$\wedge$};
\node [green] at  (3.25, 3.75) {$\vee$};

\draw  [green] plot[smooth, tension=1]coordinates { (6,2.5) (5.75, 3.75) (6,5)};
\draw  [green] plot[smooth, tension=1]coordinates { (6,2.5) (6.25, 3.75) (6,5)};
\node [green] at  (5.75, 3.75) {$\wedge$};
\node [green] at  (6.25, 3.75) {$\vee$};

\draw  [green] plot[smooth, tension=1]coordinates { (9,2.5) (8.75, 3.75) (9,5)};
\draw  [green] plot[smooth, tension=1]coordinates { (9,2.5) (9.25, 3.75) (9,5)};
\node [green] at  (8.75, 3.75) {$\wedge$};
\node [green] at  (9.25, 3.75) {$\vee$};

\draw  [green] plot[smooth, tension=1]coordinates { (12,2.5) (11.75, 3.75) (12,5)};
\draw  [green] plot[smooth, tension=1]coordinates { (12,2.5) (12.25, 3.75) (12,5)};
\node [green] at  (11.75, 3.75) {$\wedge$};
\node [green] at  (12.25, 3.75) {$\vee$};

\node [green, left] at  (-0.25, 3.75) {\scriptsize {$q^{-1}l_0$}};
\node [green, right] at  (0.25, 3.75) {\scriptsize {$rq^{-1}l_0$}};
\node [green, left] at  (2.75, 3.75) {\scriptsize {$l_0$}};
\node [green, right] at  (3.25, 3.75) {\scriptsize {$rl_0$}};
\node [green, left] at  (5.75, 3.75) {\scriptsize {$ql_0$}};
\node [green, right] at  (6.25, 3.75) {\scriptsize {$rql_0$}};
\node [green, left] at  (8.75, 3.75) {\scriptsize {$q^2l_0$}};
\node [green, right] at  (9.25, 3.75) {\scriptsize {$rq^2l_0$}};
\node [green, left] at  (11.75, 3.75) {\scriptsize {$q^3l_0$}};
\node [green, right] at  (12.25, 3.75) {\scriptsize {$rq^3l_0$}};

\node [cyan] at (1.5,5.3) {$>$};
\node [cyan] at (1.5,4.7) {$>$};

\node [cyan] at (4.5,5.3) {$>$};
\node [cyan] at (4.5,4.7) {$>$};

\node [cyan] at (7.5,5.3) {$>$};
\node [cyan] at (7.5,4.7) {$>$};

\node [cyan] at (10.5,5.3) {$>$};
\node [cyan] at (10.5,4.7) {$>$};

\draw  [cyan, dashed] plot[smooth, tension=1]coordinates { (0,5) (-.55, 5.2) (-1.25, 5.3)  };
\draw  [cyan, dashed] plot[smooth, tension=1]coordinates { (0,5) (-.55, 4.8) (-1.25, 4.7)  };

\draw  [cyan, dashed] plot[smooth, tension=1]coordinates { (0,2.5) (-.55, 2.7) (-1.25, 2.8)  };
\draw  [cyan, dashed] plot[smooth, tension=1]coordinates { (0,2.5) (-.55, 2.3) (-1.25, 2.2)  };

\draw  [cyan, dashed] plot[smooth, tension=1]coordinates { (12,5) (12.55, 5.2) (13.25, 5.3)  };
\draw  [cyan, dashed] plot[smooth, tension=1]coordinates { (12,5) (12.55, 4.8) (13.25, 4.7)  };

\draw  [cyan, dashed] plot[smooth, tension=1]coordinates { (12,2.5) (12.55, 2.7) (13.25, 2.8)  };
\draw  [cyan, dashed] plot[smooth, tension=1]coordinates { (12,2.5) (12.55, 2.3) (13.25, 2.2)  };

\draw  [cyan] plot[smooth, tension=1]coordinates { (0,5) (1.5, 5.3) (3,5)};
\draw  [cyan] plot[smooth, tension=1]coordinates { (0,5) (1.5, 4.7) (3,5)};

\draw  [cyan] plot[smooth, tension=1]coordinates { (3,5) (4.5, 5.3) (6,5)};
\draw  [cyan] plot[smooth, tension=1]coordinates { (3,5) (4.5, 4.7) (6,5)};

\draw  [cyan] plot[smooth, tension=1]coordinates { (6,5) (7.5, 5.3) (9,5)};
\draw  [cyan] plot[smooth, tension=1]coordinates { (6,5) (7.5, 4.7) (9,5)};

\draw  [cyan] plot[smooth, tension=1]coordinates { (9,5) (10.5, 5.3) (12,5)};
\draw  [cyan] plot[smooth, tension=1]coordinates { (9,5) (10.5, 4.7) (12,5)};

\draw  [cyan] plot[smooth, tension=1]coordinates { (0,2.5) (1.5, 2.8) (3,2.5)};
\draw  [cyan] plot[smooth, tension=1]coordinates { (0,2.5) (1.5, 2.2) (3,2.5)};

\draw  [cyan] plot[smooth, tension=1]coordinates { (3,2.5) (4.5, 2.8) (6,2.5)};
\draw  [cyan] plot[smooth, tension=1]coordinates { (3,2.5) (4.5, 2.2) (6,2.5)};

\draw  [cyan] plot[smooth, tension=1]coordinates { (6,2.5) (7.5, 2.8) (9,2.5)};
\draw  [cyan] plot[smooth, tension=1]coordinates { (6,2.5) (7.5, 2.2) (9,2.5)};

\draw  [cyan] plot[smooth, tension=1]coordinates { (9,2.5) (10.5, 2.8) (12,2.5)};
\draw  [cyan] plot[smooth, tension=1]coordinates { (9,2.5) (10.5, 2.2) (12,2.5)};

\node [cyan] at (1.5,2.8) {$>$};
\node [cyan] at (1.5,2.2) {$>$};

\node [cyan] at (4.5,2.8) {$>$};
\node [cyan] at (4.5,2.2) {$>$};

\node [cyan] at (7.5,2.8) {$>$};
\node [cyan] at (7.5,2.2) {$>$};

\node [cyan] at (10.5,2.8) {$>$};
\node [cyan] at (10.5,2.2) {$>$};

\node  at (1.5,2.6) {$ \boldsymbol{\vdots}$};
\node  at (4.5,2.6) {$ \boldsymbol{\vdots}$};
\node  at (10.5,2.6) {$ \boldsymbol{\vdots}$};
\node  at (7.5,2.6) {$ \boldsymbol{\vdots}$};

\node  at (1.5,5.1) {$ \boldsymbol{\vdots}$};
\node  at (4.5,5.1) {$ \boldsymbol{\vdots}$};
\node  at (10.5,5.1) {$ \boldsymbol{\vdots}$};
\node  at (7.5,5.1) {$ \boldsymbol{\vdots}$};

\node[above,cyan] at (1.5,5.3) {\scriptsize {$rq^{-1}l_1$}};
\node[above,cyan] at (4.5,5.3) {\scriptsize {$rl_1$}};
\node[above,cyan] at (7.5,5.3) {\scriptsize {$rql_1$}};
\node[above,cyan] at (10.5,5.3) {\scriptsize {$rq^{2}l_1$}};
\node[below,cyan] at (10.5,4.7) {\scriptsize {$rq^{2}l_{n}$}};
\node[below,cyan] at (1.5,4.7) {\scriptsize {$rq^{-1}l_{n}$}};
\node[below,cyan] at (4.5,4.7) {\scriptsize {$rl_{n}$}};
\node[below,cyan] at (7.5,4.7) {\scriptsize {$rql_{n}$}};

\node[above,cyan] at (1.5,2.8) {\scriptsize {$q^{-1}l_1$}};
\node[above,cyan] at (4.5,2.8) {\scriptsize {$l_1$}};
\node[above,cyan] at (7.5,2.8) {\scriptsize {$ql_1$}};
\node[above,cyan] at (10.5,2.8) {\scriptsize {$q^{2}l_1$}};
\node[below,cyan] at (10.5,2.2) {\scriptsize {$q^{2}l_{n}$}};
\node[below,cyan] at (1.5,2.2) {\scriptsize {$q^{-1}l_{n}$}};
\node[below,cyan] at (4.5,2.2) {\scriptsize {$l_{n}$}};
\node[below,cyan] at (7.5,2.2) {\scriptsize {$ql_{n}$}};

\node[below] at  (0,2.5) {\scriptsize {$q^{-1}p$}};
\node[below] at  (3,2.5) {\scriptsize {$p$}};
\node[below] at  (6,2.5) {\scriptsize {$qp$}};
\node[below] at  (9,2.5) {\scriptsize {$q^2p$}};
\node[below] at  (12,2.5) {\scriptsize {$q^3p$}};

\node[above] at  (0,5) {\scriptsize {$rq^{-1}p$}};
\node[above] at  (3,5) {\scriptsize {$rp$}};
\node[above] at  (6,5) {\scriptsize {$rqp$}};
\node[above] at  (9,5) {\scriptsize {$rq^2p$}};
\node[above] at  (12,5) {\scriptsize {$rq^3p$}};

\end{tikzpicture}

\caption{{\small The infinite graph $\Gamma_{\Z \times \Z/ 2\Z}$.} }
\end{figure}

Let $l_j$ be a closed loop around each puncture $j \in \DB \setminus \Lambda$.
Note that in $H_1(\sD_{2n+1}, \Z)$, we have that
\begin{equation} \label{loops in B and A}
\begin{cases}
[\lambda_0] = [l_0] + r[l_0]; \\
[\lambda_{-j}] = r[l_j] &\text{for } j > 0; \\
[\lambda_{j}]  =  [l_j] &\text{for } j > 0.
\end{cases}
\end{equation}
Consider the sub $\mathcal{Z}_B$-module $\sH_n \subseteq H_1(\sD_{2n+1}, \Z)$ generated by $\{ [\xi_1], ... , [\xi_n] \}$, where:
\[
[\xi_j] = \begin{cases}
(1-q)[l_0] - (1-r)[l_1] &\text{ for } j =1; \\
(-q)^{1-j}(1-r)([l_{j-1}] - [l_j]) &\text{ for } j \geq 2.
\end{cases}
\]
Note that $[\xi_j] = -r[\xi_j]$ for all $j \geq 2$, and it is easy to see that
\[
\sH_n \cong \mathcal{Z}_{B,r} \oplus \left( \mathcal{Z}_{B,r}/\<r+1\> \right)^{\oplus n-1}
\]
as $\mathcal{Z}_{B,r}$-module, generated by $\{ [\xi_1], ..., [\xi_n]\}$.
The action of $\cA(B_n)$ on $\DB \setminus \Lambda$ lifts to an action on $\sD_{2n+1}$ that commutes with deck transformations, which then induces an action on the $\mathcal{Z}_{B,r}$-module $H_1(\sD_{2n+1}, \Z)$.
It is easy to see that $\sH_n \subseteq H_1(\sD_{2n+1}, \Z)$ is closed under the action of $\cA(B_n)$, so we obtain a subrepresentation:
\[
\rho_{RHB}: \cA(B_n) \ra \text{Aut}_{\mathcal{Z}_{B,r} \text{-mod}} \left(\mathcal{Z}_{B,r} \oplus \left( \mathcal{Z}_{B,r}/\<r+1\> \right)^{\oplus n-1} \right).
\]
A simple computation shows that $\rho_{RHB}|_{r\mapsto -s, q \mapsto q}$ agrees with the representation $\rho_{KB}$.
Thus we have the following result:
\begin{proposition}\label{Hom and cat rep B}
Let $\sH_n$ be a $\mathcal{Z}_B$-module under the identification $\mathcal{Z}_{B,r} \cong \mathcal{Z}_B$ given by $r \mapsto -s$ and $q \mapsto q$.
Then the $\mathcal{Z}_B$-linear map $\Theta_B: \sH_n \ra K_0(\cK_B)$ defined by
\[
\Theta_B([\xi_j]) = [P^B_j]
\]
is an isomorphism of $\cA(B_n)$-representations.
\end{proposition}

\subsection{Relating categorical and homological representations}
We shall now show a ``decategorified'' version of our main result in \cref{full picture}.
Recall that $\sH_n$ and $K_0(\cK_B)$ are $\mathcal{Z}_B = \Z[q,q^{-1}, s]/\<s^2 - 1\>$-modules, whereas $H_1(\cD_{2n}, \Z)$ and $K_0(\cK_A)$ are $\mathcal{Z}_A = \Z[q,q^{-1}]$-modules.
Denote $ev_1, ev_{-1} : \mathcal{Z}_B \ra \mathcal{Z}_A$ as the two evaluation maps defined by $s \mapsto 1, -1$ respectively.
Throughout the rest of this section, we shall view $H_1(\cD_{2n}, \Z)$ as a $\mathcal{Z}_B$ module through $ev_{-1}$ and $K_0(\cK_A)$ as a $\mathcal{Z}_B$-module through $ev_1$.

Using \cref{B tensor to A}, we obtain a $\mathcal{Z}_B$-linear map $K_0(\Aa_{2n-1} \otimes_{\Ba_n} -)$, given by
\[
K_0(\Aa_{2n-1} \otimes_{\Ba_n} -)([P^B_j]) = 
\begin{cases}
	[P^A_n], &\text{for } j = 1; \\
	[P^A_{n-(j-1)}] + [P^A_{n+(j-1)}], &\text{ otherwise}.
\end{cases}
\]
On the other hand, the natural inclusion map $\iota : \DA \setminus \Delta_0 \ra \DA \setminus \Delta$ induces a map on the homology $\iota: H_1(\DA \setminus \Delta_0, \Z) \ra H_1(\DA \setminus \Delta,\Z)$, which sends
\[
[\lambda_j] \mapsto \begin{cases}
0, &\text{ for } j = 0; \\
[\lambda_j], &\text{ otherwise}.
\end{cases}
\]
Thus, $\iota$ lifts uniquely to $\wt{\iota} : \sD_{2n+1} \ra \cD_{2n}$, which induces a map on the homology $\wt{\iota}: H_1(\sD_{2n+1}, \Z) \ra H_1(\cD_{2n}, \Z)$.
One easily sees from \cref{loops in B and A} that the restriction of $\wt{\iota}$ to $\sH_n$ is given by
\[
\wt{\iota}([\xi_j]) = \begin{cases}
[\gamma_n], &\text{for } j=1; \\
[\gamma_{n-(j-1)}] + [\gamma_{n+(j-1)}], &\text{for } j \geq 2. 
\end{cases}
\]
and $\wt{\iota}$ is a $\mathcal{Z}_B$-linear map.

Finally, it follows immediately that $\Theta_A \circ \wt{\iota} = K_0(\Aa_{2n-1} \otimes_{\Ba_n} -) \circ \Theta_B$, where $\Theta_A$ and $\Theta_B$ are as defined in \cref{Hom and cat rep A} and \cref{Hom and cat rep B} respectively.
Thus, we have the following ``decategorified'' version of our main result in \cref{full picture}:
\begin{theorem}\label{decat main theorem}
The following diagram is commutative:
\begin{center}
\begin{tikzpicture} [scale= 0.8]
\node (tbB) at (-3,1.5) 
	{$\mathcal{A}(B_n)$};
\node (cbB) at (-3,-3.5) 
	{$\mathcal{A}(B_n)$};
\node (tbA) at (10.5,1.5) 
	{$\mathcal{A}(A_{2n-1})$}; 
\node (cbA) at (10.5,-3.5) 
	{$\mathcal{A}(A_{2n-1})$};

\node[align=center] (cB) at (0,0) 
	{$\sH_n$};
\node[align=center] (cA) at (7,0) 
	{$H_1(\cD_{2n}, \Z)$};
\node (KB) at (0,-2)
	{$K_0(\cK_B)$};
\node (KA) at (7,-2) 
	{$K_0(\cK_A)$};

\coordinate (tbB') at ($(tbB.east) + (1,-1)$);
\coordinate (cbB') at ($(cbB.east) + (1,1)$);
\coordinate (tbA') at ($(tbA.west) + (-1,-1)$);
\coordinate (cbA') at ($(cbA.west) + (-1,1)$);

\draw [->,shorten >=-1.5pt, dashed] (tbB') arc (245:-70:2.5ex);
\draw [->,shorten >=-1.5pt, dashed] (cbB') arc (-245:70:2.5ex);
\draw [->, shorten >=-1.5pt, dashed] (tbA') arc (-65:250:2.5ex);
\draw [->,shorten >=-1.5pt, dashed] (cbA') arc (65:-250:2.5ex);

\draw[->] (cB) -- (KB) node[midway, left]{$\cong$} node[midway, right]{$\Theta_B$};
\draw[->] (cB) -- (cA) node[midway,above]{$\wt{\iota}$}; 
\draw[->] (cA) -- (KA) node[midway,left]{$\cong$} node[midway, right]{$\Theta_A$};
\draw[->] (KB) -- (KA) node[midway,above]{$K_0(\Aa_{2n-1} \otimes_{\Ba_n} -)$};
\end{tikzpicture}
\end{center}
with all four maps $\mathcal{Z}_B$-linear and $\cA(B_n)$-equivariant.
\end{theorem}

\begin{comment}
Similarly, consider the covering space $X_{n+1}$ of $\DB \setminus \Lambda$ classified by the cohomology class $C_X \in H^1(\DB \setminus \Lambda, \Z\times \Z/2\Z)$ defined by
\[
[\lambda_j] \mapsto 
\begin{cases}
(1,0), &\text{ for all } j \neq 0 \\
(0,1), &\text{ for } j = 0;
\end{cases}
\]
\end{comment}

\section{Trigraded Intersection Numbers, Graded Dimensions of Homomorphism Spaces}\label{int num and hom}
In this section, we shall relate the trigraded intersection number and the Hom spaces between the corresponding complexes.
Through out this section, we will fix the following notations: $\cK_B := \Kom^b(\sB_n$-$\text{p$_{r}$g$_{r}$mod})$, $\cK_A := \Kom^b(\Aa_{2n-1}$-$\text{p$_{r}$g$_{r}$mod})$, $\Ba m := \Ba_n$-$\text{mod}$ and $\Aa m := \Aa_{2n-1}$-$\text{mod}$.

For $V = \oplus_{(r,s) \in \Z \times \Z/2\Z} V_{(r,s)}\{r\}\<s\>$ a $(\Z\times\Z/2\Z)$-graded $\R$-vector space, we denote its \emph{bigraded dimension} as
\[
\bigrdim(V) := \sum_{(r,s) \in \Z \times \Z/2\Z} \dim(V_{(r,s)})q_2^r q_3^s.
\]

Recall that for each pair of objects $(C^*, \partial_C), (D^*, \partial_D)$ in  $\cK_B$, one can consider the internal (bigraded) Hom complex $\HOM^*_{\cK_B}(C,D)$ defined as follows:
for each cohomological degree $s_1 \in \Z$,
\[
\HOM^{s_1}_{\cK_B}(C,D) := \bigoplus_{\substack{(s_2,s_3) \in \Z \times \Z /2\Z \\ m + n = s_1}} \Hom_{\Ba m}(C^m,D^n\{s_2\}\<s_3\>)\{-s_2\}\<s_3\>
\]
is a $\Z \times \Z/2\Z$-graded $\R$-vector space and the differentials are given by 
\[
d(f) = \partial_D \circ f - (-1)^{s_1}f \circ \partial_C
\]
for each $f \in \HOM^{s_1}_{\cK_B}(C,D)$.
It follows that each $H^{n}\left( \HOM^*_{\cK_B}(C,D) \right)$ is a ($\Z \times \Z/2\Z$)-graded $\R$-vector space.
We define the \emph{Poincar\'e polynomial} $\mathfrak{P}(C,D) \in \Z[q_1, q_1^{-1}, q_2, q_2^{-1}, q_3]/\<q_3^2 -1\>$ of $\HOM^*_{\cK_B}(C,D)$ as 
\[ 
\mathfrak{P}(C,D) := \sum_{s_1 \in \Z} q_1^{s_1}\bigrdim_\R\left( H^{s_1}\left( \HOM^*_{\cK_B}(C,D) \right) \right).
\]

\begin{lemma}\label{poly c g}
For any trigraded admissible curve $\check{c},$ the following internal Hom complexes are quasi-isomorphic:
$$ (\HOM^*_{\cK_B}(P^B_j,L_B(\check{c})), d_C^*) \cong  \bigoplus_{\check{g} \in st(\check{c}, j)} (\HOM^*_{\cK_B}(P^B_j,L_B(\check{g})), d_G^*) $$
for all $1 \leq j \leq n$ and  $(s_1,s_2,s_3) \in \Z \times \Z \times \Z/2.$  
\end{lemma}

\begin{proof}
To simplify notation, denote
$
(C^*, \partial_C^*) := L_B(\check{c})
$
and
$
(G^*, \partial_G^*) := \bigoplus_{\check{g} \in st(\check{c}, j)} L_B(\check{g}).
$
Note that $G^*$ can be obtained from $C^*$ by discarding the modules $P^B_k$ in $L_B(\check{c})$ for $| k - j| > 1$.
In particular, for all $m \in \Z$, 
$
C^m = G^m \oplus U^m
$
where $U^m$ consists of all indecomposable $P^B_k$ in $C^m$ with $|k - j|>1$.
Using the decomposition above, let us write $\partial^m_C: C^m = G^m \oplus U^m \ra G^{m+1} \oplus U^{m+1} = C^{m+1}$ as:
\[
\partial_C^m = \begin{bmatrix}
	\tau^m & * \\
	*      & *
	\end{bmatrix}
\]
so that $\tau^m : G^m \ra G^{m+1}$.
Also note that the differential $\partial_G^m: G^m \ra G^{m+1}$ can be obtained from $\tau^m$ by modifying the differentials
$
P^B(x) \xra{\partial_{yx}} P^B(y)
$
to $0$ whenever $x$ and $y$ are crossings of two different $j$-strings of $\check{c}$.
Since 

\begin{equation} \label{Hom zero}
\bigoplus_{(s_1, s_2, s_3) \in \Z \times \Z \times \Z/2\Z} \Hom_{\Ba m}(P^B_j, P^B_k[s_1]\{s_2\}\<s_3\>) = 0
\end{equation}

\noindent for all $k$ such that $|j - k|>1$, it follows that

\begin{equation}\label{Hom zero U}
\bigoplus_{(s_1, s_2, s_3) \in \Z \times \Z \times \Z/2\Z} \Hom_{\Ba m}(P^B_j, U^{s_1}\{s_2\}\<s_3\>) = 0
\end{equation}

\noindent and thus

\begin{equation} \label{Hom zero conc}
\HOM^m_{\cK_B}(P_j^B, C^*)= \HOM^m_{\cK_B}(P_j^B, G^*)
\end{equation}

\noindent for each $m \in \Z$ as underlying graded vector space.
Moreover, we know that $d^m_G = (\partial^m_G\circ -)$ and $d^m_C = (\partial_C^m \circ -)$ by definition of the hom complex.
But \cref{Hom zero U} allows us to conclude that $d_C^m = (\tau^m\circ -)$.
Therefore to prove the proposition, it is sufficient to show that $d_C^m = (\tau^m\circ -)$ and $d^m_G = (\partial^m_G\circ -)$ have isomorphic kernels and isomorphic images for each $m \in \Z$.
For the rest of the proof let $m \in \Z$ be arbitrary.

Let us first consider the simple case when $j \neq 2$.
We claim that $d_C^m = d_G^m$.
Note that when $j\neq 2$, $\partial_{yx}$ in $\tau^m$ that are modifed to $0$ in $\partial_G^m$ are always right multiplication by loops $X_{j-1}$ or $X_{j+1}$.
But for such maps, the corresponding induced maps on the hom complex $(\partial_{yx} \circ -)$ are always 0, so $(\tau^m\circ -) = (\partial^m_G \circ -)$ as required.

\begin{comment}
Moreover, every non-zero differential $\Hom_{\Ba m}(P^B_k, P(x)\{s_2\}\<s_3\>) \ra \Hom_{\Ba m}(P^B_k, P(y)\{s_2\}\<s_3\> )$ in $G$ are agrees with $C$, so every \rb{Gaussian elimination} applied in $G$ can also be applied in $C$.
On the other hand, the only possible non-zero differential components in $C$ that are zero in $G$ are the differentials $ \Hom_{\cK_B}(P^B_k, P(x)\{s_2\}\<s_3\>) \ra \Hom_{\cK_B}(P^B_k, P(y)\{s_2\}\<s_3\> )$ that are induced by the differentials $\d_{yx}$ which we modified to $0$ in $\bigoplus_{\check{g}\in st(\check{c},j)} L_B(\check{g})$ above.
Thus to show that $C$ and  $G$ are quasi-isomorphic, it is sufficient to show that for each non-zero differential component $ \Hom_{\cK_B}(P^B_k, P(x)\{s_2\}\<s_3\>) \ra \Hom_{\cK_B}(P^B_k, P(y)\{s_2\}\<s_3\> )$ in $C$ that is zero in $G$, we also have one of the following:
\begin{itemize}
\item
a non-zero differential $ \Hom_{\Ba m}(P^B_k, P(u)\{s_2\}\<s_3\>) \ra \Hom_{\Ba m}(P^B_k, P(y)\{s_2\}\<s_3\> )$ induced by $\d_{yu}: P^B(u) \ra P^B(y)$ which is also non-zero in $L_B(\check{g})$; or
\item
a non-zero differential $ \Hom_{\Ba m}(P^B_k, P(x)\{s_2\}\<s_3\>) \ra \Hom_{\Ba m}(P^B_k, P(w)\{s_2\}\<s_3\> )$ induced by $\d_{wx}: P^B(x) \ra P^B(w)$ which is also non-zero in $L_B(\check{g})$.
\end{itemize}
\end{comment}

Now let us consider the case when $j = 2$.
The types of maps $\partial_{yx} : P(x) \ra P(y)$ in $\tau^m$ that are modifed to $0$ in $\partial_G^m$ are of the following types:
\begin{enumerate}[(i)]
\item $\partial_{yx}=X_1$ or $X_3$;
\item $\partial_{yx}=(1|2)i$ or $\partial_{yx}=-i(2|1)$.
\end{enumerate}
Moreover, $\partial_{yx}$ of type (ii) does not exist in $\partial_G^m$ by definition of $L_B$.
By the same argument in the case $j \neq 2$, the induced differential in the hom complex by $\partial_{yx}$ of type (i) is 0.
So can relate $d_C^m$ and $d_G^m$ as follows:
\begin{equation} \label{d_C in d_G}
d_C^m = (\tau^m \circ -) = (\partial^m_G\circ -) + (\d\circ -) = d_G^m + (\d\circ -)
\end{equation}
where $\d := \sum \partial_{yx}$, summing over all $\partial_{yx}$ in $\tau^m$ that are of type (ii).

Before we analyse the kernel and image of both $d^m_C$ and $d^m_G$, we shall consider a decomposition of $G^m$ and $G^{m+1}$ using $\tau^m$.
Denote $\mathfrak{G}^m$ and $\mathfrak{G}^{m+1}$ as the subset of all crossings of $\check{c}$ such that $G^m = \bigoplus_{z \in \mathfrak{G}^m} P(z)$ and $G^{m+1} = \bigoplus_{z \in \mathfrak{G}^{m+1}} P(z)$.
We shall reorganise the direct summands of $G^m$ and $G^{m+1}$ in the following way:
\begin{enumerate}
\item Set 
\begin{itemize}
\item $\alpha = 1$
\item $\epsilon := \d$,
\item $X := \mathfrak{G}^m$,
\item $H^m := G^m$, 
\item $Y = \mathfrak{G}^{m+1}$ and 
\item $H^{m+1} = G^{m+1}$.
\end{itemize}
\item If $\epsilon = 0$, then skip to step (3);
otherwise let $\partial_{yx}$ be one of the summands in $\d$. 
Consider the smallest subset $X' \subseteq X$ and $Y' \subseteq Y$ such that 
\begin{itemize}
\item $x \in X'$, 
\item $y \in Y'$ and 
\item $\partial_{zw} = 0, X_1$ or $X_3$ whenever $w \in (X')^c, z\in Y'$ or $w\in X', z\in (Y')^c$.
\end{itemize}
We organise the direct summands of $H^m$ in the following way:
\[
H^m = Q^m_{\alpha} \oplus \left( \bigoplus_{x\in (X')^c} P(x) \right),
\text{ and }
H^{m+1} = Q^{m+1}_{\alpha} \oplus \left( \bigoplus_{y \in (Y')^c} P(y) \right).
\]
where $Q^m_{\alpha} := \bigoplus_{x \in X'} P(x)$ and $Q^{m+1}_{\alpha} :=  \bigoplus_{y \in Y'} P(y)$.
Let $e = \sum \partial_{yx}$, summing over all $\partial_{yx} = (1|2)i, -i(2|1)$  with $x \in X'$ and $y \in Y'$. \\
Redefine 
\begin{itemize}
\item $\alpha := \alpha +1$,
\item $\epsilon:= \epsilon - \gamma$, 
\item $H^m := \bigoplus_{x\in (X')^c} P(x)$ and 
\item $H^{m+1} := \bigoplus_{y \in (Y')^c} P(y)$.
\end{itemize} 
Repeat step (2).
\item If $H^m \neq 0$, then set $Q^m_{\alpha} := H^m$, else if $H^{m+1} \neq 0$, then set $Q^{m+1}_{\alpha} := H^{m+1}$.
\item Output $G^m = \bigoplus_{s \in S} Q^m_s$ and $G^{m+1} = \bigoplus_{s' \in S'} Q^m_{s'}$ with the appropriate index sets $S = \{1, ..., M\}$ and $S' = \{1, ..., M'\}$.
\end{enumerate}
Now consider $\tau^m$ and $\partial_G^m$ as block matrices corresponding to the decomposition obtained above:
\[
\tau^m = [(\tau^m)_{s',s}]_{(s',s) \in S'\times S}, \quad
\partial_G^m = [(\partial_G^m)_{s',s}]_{(s',s) \in S'\times S}.
\]
Note that by the construction of the decomposition we have that the block $(\tau^m)_{s',s}$ only have entries $X_1, X_3$ or 0 for all $s \neq s'$.
On the hom complexes, the decompositions also give us
\[
\HOM^m_{\cK_B}(P_j^B, C^*) = \HOM^m_{\cK_B}(P_j^B, G^*) = \bigoplus_{s \in S} \Hom_{\Ba m} (P_j^B, Q^m_s), \text{ and }
\]
\[\HOM^{m+1}_{\cK_B}(P_j^B, C^*) = \HOM^{m+1}_{\cK_B}(P_j^B, G^*) = \bigoplus_{s' \in S'} \Hom_{\Ba m} (P_j^B, Q^{m+1}_{s'}).
\]
Similarly consider the two differentials $d^m_C$ and $d^m_G$ written as block matrices corresponding to the decompositions:
$
d^m_C = [(d^m_{C})_{s',s}]_{(s',s) \in S' \times S}, \quad
d^m_G = [(d^m_{G})_{s',s}]_{(s', s) \in S'\times S}.
$
The construction of the decomposition guarantees the property that $(d^m_{C})_{s',s} = (d^m_{G})_{s',s}= 0$ whenever $s \neq s'$ (recall that the induced maps $(X_1\circ -)$ and $(X_3\circ -)$ are 0).
So to show that $d^m_C = (\tau^m \circ -)$ and $d^m_G = (\partial_G^m \circ -)$ have isomorphic images and isomorphic kernels, it is now sufficient to show them for each block $(d^m_{C})_{s',s}= ((\tau^m)_{s',s}\circ -)$ and $(d^m_{G})_{s',s} = ((\partial_G^m)_{s',s} \circ -)$  where $s = s'$.

To simplify notation, for the rest of this proof we shall drop the subscript $s$ and denote: 
$$
Q^m := Q^m_s; \quad
Q^{m+1} := Q^{m+1}_s;\quad
d_C := (d^m_{C})_{s,s};\quad
d_G := (d^m_{G})_{s,s};\quad
\tau := \tau^m_{s,s} ; \text{ and }
\partial_G := (\partial^m_G)_{s,s}.
$$

\noindent We shall now look at the possible types of maps $\tau: Q^m \ra Q^{m+1}$ which gives us all possible $d_{C} = (\tau \circ -)$, where $d_{G} = (\partial_G \circ -)$ can be obtained by $d_G = (\tau\circ -) - (\d \circ -)$ (following from \cref{d_C in d_G}).

If $d_C = d_G$, i.e. $\tau$ has no entry of type (ii) so that $\d = 0$, then there is nothing left to show.
Otherwise, $\tau$ contains at least one entry $\partial_{yx}$ of type (ii).
The two possibilities of $\partial_{yx}$ of type (ii) are $(1|2)i$ and $-i(2|1)$.
We will only explicitly show the classification method used to obtain all possible types of $Q^m \xra{\tau} Q^{m+1}$ when $\partial_{yx} = (1|2)i$, where the same method can be applied to the case when $\partial_{yx} = -i(2|1)$.

So let us consider the case when $\tau$ has an entry with $\partial_{yx} = (1|2)i$.
Recall that by the definition of $L_B$, for any $\partial_{yx}$ of type (ii), there must be a corresponding 1-crossing $x'$ of $\check{c}$ such that 
\begin{itemize}
\item $x'$ and $x$ are connected by an essential segment in $D_0$; and
\item $x'$ and $y$ are endpoints of an essential segment of $\check{c}$ in $D_1$.  
\end{itemize}
So in the case $\partial_{yx} = (1|2)i$, $x'$ and $y$ are connected through an essential segment in $D_1$ of type 1 (refer to \cref{sixtype}) and the map $\partial_{yx'}: P(x') \ra P(y)$ is the right multiplication by $(1|2)$.
By the construction of $Q^m$ and $Q^{m+1}$, $Q^m$ must then at least contain the direct summands $P(x)$ and $P(x')$ and $Q^{m+1}$ must at least contain the direct summand $P(y)$, so the differential $Q^m \xra{\tau} Q^{m+1}$ must contain at least two entries $\partial_{yx} = (1|2)i$ and $\partial_{yx'} = (1|2)$.
Thus the crossings $x,x'$ and $y$ must be contained in the corresponding partial curve of $\check{c}$ below:
\begin{figure}[H]
\centering
\begin{tikzpicture} [scale=.9]

\draw[thick]  (13.475,8)--(13.7,8);
\draw[thick]  (13.475,10)--(13.7,10); 
\draw[thick, blue]  (13.475,8)--(13.475,10);
\draw[thick,red]  (11.25,9.65)--(13.475,9.65);
\draw[thick]  (11.25,10)--(13.7,10);  
\draw[thick]  (11.25,8)--(13.7,8);
\draw[thick,green, dashed]  (11.25,10)--(11.25,9.5);
\draw[thick, green, dashed]   (11.25,8)--(11.25,8.5);
\draw[thick, green, dashed, ->]   (11.25,9)--(11.25,9.6);
\draw[thick,  green, dashed, ->]   (11.25,9)--(11.25,8.4);
\filldraw[color=black!, fill=yellow!, very thick]  (11.25,9) circle [radius=0.1]  ;
\draw[fill] (12.75,9) circle [radius=0.1]  ;

\draw[blue] (12,8) -- (12,10);
\draw[thick, red] (11.25,8.35) -- (12,8.35) ;

\node[below right] at (13.475,9.65) {y};
\node[below right] at (12,9.65) {x'};
\node[right] at (12,8.35) {x};
%\draw[thick,red] plot[smooth,tension=1] coordinates { (11.25,8.35) (11.7,8.5) (12.5,9.2) (13.475,9.35)};
\end{tikzpicture}
\caption{{\small The Crossings $x,x'$ and $y$.}} \label{must have form}
\end{figure}
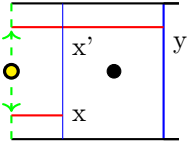
As seen from \cref{must have form} above, $x$ and $y$ are the only crossings that can be in a another distinct essential segment of $\check{c}$.

Let us now first consider the subcase when $x$ is not in another distinct essential segment.
If $y$ is also not in another distinct essential segment, then we have that $Q^m \xra{\tau} Q^{m+1}$ is of the form:
\begin{figure}[H]
\centering
\begin{tikzcd}
P(x') \ar[r, "(1|2)"  ] & P(y) \\
P(x ) \ar[ru,"(1|2)i",swap] \ar[u, phantom, "\oplus" ]
\end{tikzcd}.
\end{figure}
If instead $y$ is part of another essential segment of $\check{c}$ with its other endpoint some crossing $w$, then the essential segment must be in $D_2$.
Since $P(y)$ is a direct summand of $Q^{m+1}$, $w$ must have the property $w_1 = y_1 - 1$ so that $P(w)$ is an entry of $Q^m$ and $\partial_{yw}$ is an entry of $\tau$.
The only two such possibilities are:
\begin{figure}[H]
\centering
\begin{tikzpicture} [scale=.9]

\draw[thick]  (13.475,8)--(15,8);
\draw[thick]  (13.475,10)--(15,10); 
\draw[thick, blue]  (13.475,8)--(13.475,10);
\draw[thick,red]  (11.25,9.65)--(13.475,9.65);
\draw[thick]  (11.25,10)--(13.7,10);  
\draw[thick]  (11.25,8)--(13.7,8);
\draw[thick,green, dashed]  (11.25,10)--(11.25,9.5);
\draw[thick, green, dashed]   (11.25,8)--(11.25,8.5);
\draw[thick, green, dashed, ->]   (11.25,9)--(11.25,9.6);
\draw[thick,  green, dashed, ->]   (11.25,9)--(11.25,8.4);
\filldraw[color=black!, fill=yellow!, very thick]  (11.25,9) circle [radius=0.1]  ;
\draw[fill] (12.75,9) circle [radius=0.1]  ;
\draw[fill] (14.25,9) circle [radius=0.1]  ;

\draw[blue] (12,8) -- (12,10);
\draw[thick, red] (11.25,8.35) -- (12,8.35) ;

\node[below left] at (13.475,9.65) {y};
\node[below right] at (12,9.65) {x'};
\node[below right] at (12,8.45) {x};
\node[left] at (13.475,9.85)  {w};
\draw[thick,red] plot[smooth,tension=.8] coordinates { (13.475,9.65)  (13.9,9.45)  (14.25, 8.7)  ( 14.625, 9 )  (14.35,9.5) (13.475,9.85) };
\draw[thick,red] plot[smooth,tension=.8] coordinates { (12,8.35) (12.4, 8.5)   (12.75,9)  };

\draw[thick]  (18.475,8)--(20.2,8);
\draw[thick]  (18.475,10)--(20.2,10); 
\draw[thick, blue]  (18.475,8)--(18.475,10);
\draw[thick,red]  (16.25,9.65)--(18.475,9.65);
\draw[thick]  (16.25,10)--(18.7,10);  
\draw[thick]  (16.25,8)--(18.7,8);
\draw[thick,green, dashed]  (16.25,10)--(16.25,9.5);
\draw[thick, green, dashed]   (16.25,8)--(16.25,8.5);
\draw[thick, green, dashed, ->]   (16.25,9)--(16.25,9.6);
\draw[thick,  green, dashed, ->]   (16.25,9)--(16.25,8.4);
\filldraw[color=black!, fill=yellow!, very thick]  (16.25,9) circle [radius=0.1]  ;
\draw[fill] (17.75,9) circle [radius=0.1]  ;
\draw[fill] (19.25,9) circle [radius=0.1]  ;

\draw[blue] (17,8) -- (17,10);
\draw[blue] (20,8) -- (20,10);
\draw[thick, red] (16.25,8.35) -- (17,8.35) ;

\node[above right] at (18.475,9.5) {{ y}};
\node[below right] at (17,9.65) {x'};
\node[below right] at (17,8.45) {x};
\node[right] at (20,8.35) {w};
\draw[thick,red] plot[smooth,tension=.5] coordinates { (18.475,9.65)  (18.75, 9.5) (19.25,8.4) (20,8.35)   };
\draw[thick,red] plot[smooth,tension=1] coordinates { (17,8.35) (17.4, 8.5) (17.75,9)  };
\end{tikzpicture}
\caption{{\small The two possible essential segments from $y$.}} \label{two ess y}
\end{figure}
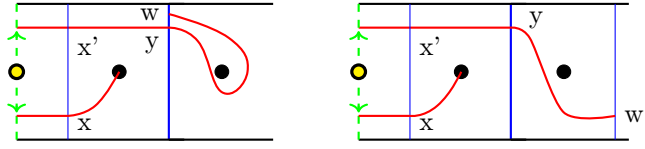

%
%\begin{center}
%\begin{tikzcd}
%P(w) \ar[rd,"X_2 \text{ or } (3|2)"] \\
%	& P(y) 
%\end{tikzcd} 
%\end{center}
%
Now note that if $w$ is a 2-crossing (left picture in \cref{two ess y}, then one sees that $w$ can not be connected to any other crossing $z$ through another distinct essential segment in $\check{c}$ with $P(z)$ a direct summand of $Q^{m+1}$;
if instead $w$ a 3-crossing (right picture in \cref{two ess y}, then the only possibility for $P(z)$ to be a direct summand of $Q^{m+1}$ is when $w$ is also a 3-crossing, with $w$ and $z$ endpoints of an essential segment of $\check{c}$ in $D_3$ of type 2, giving us $\partial_{wz} = X_3$.
Recall the chosen decomposition of $G^m$ and $G^{m+1}$, where $Q^{m+1} \subseteq G^{m+1}$ corresponds to the smallest subset of crossings in $\mathfrak{G}^{m+1}$ which contains $y$, with the property that maps between the direct summands of the decompositions of $G^m$ and $G^{m+1}$ are either 0 or $X_1$ or $X_3$.
Thus $P(z)$ must be excluded from $Q^{m+1}$.
We can therefore conclude that for the subcase when $x$ is not connected to any other distinct essential segments, we have 3 possible forms for $Q^m \xra{\tau} Q^{m+1}$:
\begin{center}
\begin{tikzcd}
P(x') \ar[r, "(1|2)"  ] & P(y) \\
P(x ) \ar[ru,"(1|2)i",swap] \ar[u, phantom, "\oplus" ]
\end{tikzcd} 
, or
\begin{tikzcd}
P(w) \ar[rd, "X_2 \text{ or } (3|2)"] \ar[d, phantom, "\oplus"]\\
P(x') \ar[r, "(1|2)"  ] & P(y) \\
P(x ) \ar[ru,"(1|2)i",swap] \ar[u, phantom, "\oplus" ]
\end{tikzcd}
\end{center}
To analyse the maps $d_C$ and $d_G$, let us identify the morphism spaces as

\begin{align*}
\HOM^{m}_{\cK_B}(P_2^B, G^*) = & 
\bigoplus_{(s, t) \in \Z \times \Z/\Z}\Hom_{\Ba m}(P_2^B, (P(x')\oplus P(x) \oplus P(w) )\{s\}\<t\>)\{s\}\<t\> \\ 
\cong & \left(\R\{(2|1)\}\<x_3\> \oplus \R\{i(2|1)\}\<x_3+1\>\right)\{x_2+1\} 
	\\ &\oplus \left(\R\{(2|1)\}\<x'_3\> \oplus \R\{i(2|1)\}\<x'_3+1\>\right)\{x'_2\}
	\\ &\oplus Z	
\end{align*}

where $Z = 0$ for the first type of $Q^m \xra{\tau} Q^{m+1}$, and

\begin{align*}
\HOM^{m+1}_{\cK_B}(P_2^B, G^*) = & 
\bigoplus_{(s, t) \in \Z \times \Z/\Z}\Hom_{\Ba m}(P_2^B, P(y)\{s\}\<t\>)\{s\}\<t\> \\ 
\cong & \left(\R\{X_2\}\<y_3\> \oplus \R\{X_2 i\}\<y_3+1\>\right)\{y_2+1\} 
	\\ &\oplus \left(\R\{\id\}\<y_3\> \oplus \R\{i\}\<y_3+1\>\right)\{y_2\}.
\end{align*}

Using this identification, we can write $d_C$ and $d_G$  as the corresponding matrices:

$$
d_C =
\begin{bmatrix}
1 &  0 & 0 & -1 & e \\
0 &  1 & 1 & 0  & f \\
0 &  0 & 0 & 0  & 0 \\
0 &  0 & 0 & 0  & 0 \\
\end{bmatrix}
\text{ and }
d_G =
\begin{bmatrix}
1 &  0 & 0 & 0  & e \\
0 &  1 & 0 & 0  & f \\
0 &  0 & 0 & 0  & 0 \\
0 &  0 & 0 & 0  & 0 \\
\end{bmatrix},
$$

\noindent where $d_G$ is obtained from $d_C$ by removing maps that were induced by $(1|2)i$.
It follows that $d_C$ and $d_G$ have the same image and have isomorphic kernels.

Now consider the other subcase where $x$ is in another essential segment of $\check{c}$ with its other endpoint some crossing $y'$.
Note that since $x$ is already part of an essential segment in $D_0$, the essential segment connecting $x$ and $y'$ can only be in $D_1$.
As before, we must have $y'_1 = x_1 - 1$ so that $P(y')$ is a direct summand of $Q^{m+1}$ and that $\partial_{y'x}$ is an entry of $\tau$.
Furthermore, if $x$ and $y'$ is connected by the essential segment of Type 2 in \cref{sixtype}, then $\partial_{y'x} = X_1$.
Therefore such $P(y')$ is excluded from $Q^{m+1}$.
Collecting the results, the only possible essential segment connecting $x$ and $y'$ with $y'_1 = x_1 - 1$ and $\partial_{y'x} \neq X_1$ is the essential segment of Type 1.
Thus the crossings $x,x',y$ and $y'$ must be contained in the 
\begin{figure}[H]
\centering
\begin{tikzpicture} [scale=.9]

\draw[thick]  (13.475,8)--(13.7,8);
\draw[thick]  (13.475,10)--(13.7,10); 
\draw[thick, blue]  (13.475,8)--(13.475,10);
\draw[thick,red]  (11.25,9.65)--(13.475,9.65);
\draw[thick]  (11.25,10)--(13.7,10);  
\draw[thick]  (11.25,8)--(13.7,8);
\draw[thick,green, dashed]  (11.25,10)--(11.25,9.5);
\draw[thick, green, dashed]   (11.25,8)--(11.25,8.5);
\draw[thick, green, dashed, ->]   (11.25,9)--(11.25,9.6);
\draw[thick,  green, dashed, ->]   (11.25,9)--(11.25,8.4);
\filldraw[color=black!, fill=yellow!, very thick]  (11.25,9) circle [radius=0.1]  ;
\draw[fill] (12.75,9) circle [radius=0.1]  ;

\draw[blue] (12,8) -- (12,10);
\draw[thick, red] (11.25,8.35) -- (12,8.35) ;

\node[right] at (13.475,9.65) {y};
\node[below right] at (12,9.65) {x'};
\node[below right] at (12,8.45) {x};
\node[right] at (13.475,9.45) {y'};

\draw[thick,red] plot[smooth,tension=.8] coordinates { (12,8.35) (12.2,8.5)  (12.45, 9.1) (12.9, 9.35) (13.475,9.45)};
\end{tikzpicture}
\caption{The crossings $x,x'$ and $y$ when $x$ is in another essential segment.} 
\end{figure}
The same analysis in the previous subcase on the possible essential segments connected to $y$ can be applied similarly to the crossings $y$ and $y'$ here.
Thus we conclude that for this subcase, $Q^m \xra{\partial^m_C} Q^{m+1}$ is equal to one of the following 6 types (there are 3 possible combinations of $X_2$ and $(3|2)$ maps in the rightmost diagram):
\begin{center}
\begin{tikzcd}
P(x') \ar[rd," " description] \ar[r, "(1|2)"  ] & P(y) \\
P(x ) \ar[ru,"(1|2)i" description] \ar[r,"(1|2)", swap]
	\ar[u, phantom, "\oplus" ] 
	& P(y') \ar[u, phantom, "\oplus"] 
\end{tikzcd}
, \quad
\begin{tikzcd}
P(z) \ar[rd, "X_2 \text{ or } (3|2)"] \ar[d, phantom, "\oplus"]\\
P(x') \ar[rd," " description] \ar[r, "(1|2)"  ] & P(y) \\
P(x ) \ar[ru,"(1|2)i" description] \ar[r,"(1|2)", swap]
	\ar[u, phantom, "\oplus" ] 
	& P(y') \ar[u, phantom, "\oplus"] 
\end{tikzcd}
, or
\begin{tikzcd}
P(z) \ar[rd, "X_2 \text{ or } (3|2)"] \ar[d, phantom, "\oplus"]\\
P(x') \ar[rd," " description] \ar[r, "(1|2)"  ] & P(y) \\
P(x ) \ar[ru,"(1|2)i" description] \ar[r,"(1|2)", swap]
	\ar[u, phantom, "\oplus" ] 
	& P(y') \ar[u, phantom, "\oplus"] \\
P(z') \ar[ru, "X_2 \text{ or } (3|2)", swap] \ar[u, phantom, "\oplus"]
\end{tikzcd}
\end{center}
swapping $x$ with $x'$ (and correspondingly $y$ with $y'$) if necessary.
Let us again identify the morphism spaces as:
\begin{align*}
\HOM^{m}_{\cK_B}(P_2^B, G^*) = & 
\bigoplus_{(s, t) \in \Z \times \Z/\Z}\Hom_{\Ba m}(P_2^B, (P(x')\oplus P(x) \oplus P(z) )\{s\}\<t\>)\{s\}\<t\> \\ 
\cong & \left(\R\{(2|1)\}\<x_3\> \oplus \R\{i(2|1)\}\<x_3+1\>\right)\{x_2+1\} 
	\\ &\oplus \left(\R\{(2|1)\}\<x'_3\> \oplus \R\{i(2|1)\}\<x'_3+1\>\right)\{x'_2\}
	\\ &\oplus Z	
\end{align*}

\noindent where $Z = 0$ for the first type of $Q^m \xra{\tau} Q^{m+1}$, and
\begin{align*}
\HOM^{m+1}_{\cK_B}(P_2^B, G^*) = & 
\bigoplus_{(s, t) \in \Z \times \Z/\Z}\Hom_{\Ba m}(P_2^B, P(y)\{s\}\<t\>)\{s\}\<t\> \\ 
\cong & \left(\R\{X_2\}\<y_3\> \oplus \R\{X_2 i\}\<y_3+1\>\right)\{y_2+1\}
	\\ &\oplus \left(\R\{X_2\}\<y'_3\> \oplus \R\{X_2 i\}\<y'_3+1\>\right)\{y'_2+1\}
	\\ &\oplus \left(\R\{\id\}\<y_3\> \oplus \R\{i\}\<y_3+1\>\right)\{y_2\}
	\\ &\oplus \left(\R\{\id\}\<y'_3\> \oplus \R\{i\}\<y'_3+1\>\right)\{y'_2\}
	\\
	= & \left(\R\{X_2\}\<y_3\> \oplus \R\{X_2 i\}\<y_3+1\>\right)\{y_2+1\}
	\\ &\oplus \left(\R\{X_2\}\<y'_3\> \oplus \R\{X_2 i\}\<y'_3+1\>\right)\{y'_2+1\}
	\\ &\oplus V.
\end{align*}

\noindent Writing $d_C$ and $d_G$  as the corresponding matrix, we get

$$
d_C =
\begin{bmatrix}
1 &  0  & 0 & -1 & e' \\
0 &  1  & 1 & 0  & f' \\
0 &  -1 & 1 & 0  & g' \\
1 &  0  & 0 & 1  & h' \\
0 &  0  & 0 & 0  & 0 
\end{bmatrix}
\text{ and }
d_G =
\begin{bmatrix}
1 &  0 & 0 & 0  & e' \\
0 &  1 & 0 & 0  & f' \\
0 &  0 & 1 & 0  & g' \\
0 &  0 & 0 & 1  & h' \\
0 &  0 & 0 & 0  & 0 
\end{bmatrix},
$$

\noindent which also have the same image and same kernel.
Thus for $\partial_{yx} = (1|2)i$, all possible cases of $d_C$ and $d_G$ have isomorphic images and isomorphic kernels as required.

Applying the same classification method to the case when $\partial_{yx} = -i(2|1)$, the posssible types of $Q^m \xra{\tau} Q^{m+1}$ is given by:
\begin{center}
\begin{tikzcd}
 & P(y) \\
P(x ) \ar[ru,"-i(2|1)"] \ar[r, "(2|1)"]& P(y') \ar[u, phantom, "\oplus" ]
\end{tikzcd} 
, or
\begin{tikzcd}
 	& P(y) \\
P(x ) \ar[ru,"-i(2|1)"] \ar[r, "(2|1)"] \ar[rd, "X_2 \text{ or } (2|3)", swap]
	& P(y') \ar[u, phantom, "\oplus" ] \\
	& P(z) \ar[u, phantom, "\oplus"]
\end{tikzcd} 
\end{center}
when $y$ is not part of another distinct essential segment of $\check{c}$, and
\begin{center}
\begin{tikzcd}
P(x')\ar[r, "(2|1)"] \ar[dr, " " description]
	& P(y) \\
P(x ) \ar[ru,"-i(2|1)" description] \ar[r, "(2|1)"] \ar[u, phantom, "\oplus"]
	& P(y') \ar[u, phantom, "\oplus" ]
\end{tikzcd} 
,
\begin{tikzcd}
P(x')\ar[r, "(2|1)"] \ar[dr, " " description]
 	& P(y) \\
P(x ) \ar[ru,"-i(2|1)" description] \ar[r, "(2|1)"] \ar[rd, "X_2 \text{ or } (2|3)", swap] \ar[u, phantom, "\oplus"]
	& P(y') \ar[u, phantom, "\oplus" ] \\
	& P(z) \ar[u, phantom, "\oplus"]
\end{tikzcd} 
, or
\begin{tikzcd}
	& P(z') \ar[d, "\oplus"] \\
P(x')\ar[r, "(2|1)"] \ar[dr, " " description] \ar[ru, "X_2 \text{ or } (2|3)"]
 	& P(y) \\
P(x ) \ar[ru,"-i(2|1)" description] \ar[r, "(2|1)"] \ar[rd, "X_2 \text{ or } (2|3)", swap] \ar[u, phantom, "\oplus"]
	& P(y') \ar[u, phantom, "\oplus" ] \\
	& P(z) \ar[u, phantom, "\oplus"]
\end{tikzcd} 
\end{center}
when $y$ is part of another distinct essential segment of $\check{c}$, where as before we swap $y$ with $y'$ (correspondingly $x$ with $x'$) if neccessary.
By identifying the morphism spaces and comparing the corresponding matrices of $d_C$ and $d_G$ as before, it follows that $d_C$ and $d_G$ have isomorphic images and isomorphic kernels.
This covers all cases of $\d_C$ and $\d_G$, completing the proof.
\end{proof}

\begin{lemma}\label{poly tri int b g}
The Poincar\'e polynomial $\mathfrak{P}(P^B_j, L_B(\check{g}))$  of
$\HOM^*_{\cK_B}(P^B_j,L_B(\check{g}))$ is equal to the trigraded intersection number $I^{trigr}(\check{b}_j, \check{g})$ for any trigraded $j$-string $\check{g}$.
\end{lemma}
\begin{proof}
This follows exactly as in \cite{KhoSei}.
\end{proof}

\begin{proposition} \label{poin poly equals tri int}
For any $\sigma$ and  $\tau$ in $\cA(B_n),$ and any $1 \leq j,k \leq n,$ the Poincar\'e polynomial of 
$$HOM^*_{\cK_B}( \sigma(P^B_j), \tau(P^B_k)) $$
is equal to the trigraded intersection number $I^{trigr}(\check{\sigma}(\check{b}_j), \check{\tau}(\check{b}_k)).$
\end{proposition}
\begin{proof}
By \cref{compute tri int}, we get that
$
I^{trigr}(\check{b}_j, \check{c}) = \sum_{\check{g} \in st(\check{c},j)} I^{trigr}(\check{b}_j, \check{g}).
$
Using \cref{poly c g}, we instead get that
$
\mathfrak{P}(P^B_j, L_B(\check{c})) = \sum_{\check{g} \in st(\check{c},j)}\mathfrak{P}(P^B_j, L_B(\check{g})).
$
By \cref{poly tri int b g}, each $\mathfrak{P}(P^B_j, L_B(\check{g})) = I^{tri}(\check{b}_j, \check{g})$, thus we can conclude that
$
I^{trigr}(\check{b}_j, \check{c}) =  \mathfrak{P}(P^B_j, L_B(\check{c})).
$
The proposition now follows from the fact that the categorical action of $\cA(B_n)$ respects morphism spaces and similarly the topological action of $\cA(B_n)$ respects trigraded intersection number.
\end{proof}
\begin{remark} \label{faithtop}
Note that we can also use \cref{poin poly equals tri int} to prove the faithfulness of the $\cA(B_n)$ categorical action.
The proof is similar to \cite{KhoSei} modulo the center of $\cA(B_n)$, which is an easy check that elements of the centre act by shifting degrees and therefore are not isomorphic to the identity functor.
\end{remark}

\section{What is actually Type $B$ Zigzag Algebra?} \label{InvCat}

\subsection{Type $B$ zigzag algebra as a subalgebra of type $A$ zigzag algebra}

One might wonder if the the Type $B$ zigzag algebra has any geometric incarnation. 
  Is Type $B$ zigzag algebra ``Fukaya algebra'' in some Fukaya category?
  Here is one answer to it on the cohomology level.
  As in \cite{Seibook}, we will consider the Floer cohomology of type $A$ Milnor fibre in complex coefficient.
  Then, the type $A$ zigzag algebra we considering would be over $\C :$
  \begin{equation} \label{ISO}
A_{n} \cong \bigoplus^n_{i,j =1} HF^* (L_i, L_j; \C).
\end{equation} 
by associating \begin{align*}
 e_i & \mapsto HF^0 (L_i, L_i) \\
 X_i := (i | i+1 | i) = (i | i-1 |i) & \mapsto HF^n (L_i, L_i) \\
 (i | i-1) & \mapsto HF^0 ( L_i, L_{i-1} ),  i \leq n \\
 (i | i-1) & \mapsto HF^n(L_i, L_ {i-1}),  i>n \\
  (i | i+1) & \mapsto HF^0 ( L_i, L_{i+1} ),  i \geq n \\
 (i | i+1) & \mapsto HF^n(L_i, L_ {i+1}),  i < n. \\
\end{align*}
Some words of caution here: This isomorphism assume that the $A_\infty$-algebra of Floer cochain complexes underlying the right hand side is formal (see \cite{SeiTho,Seibook} for more treatment),   type $A$ zigzag algebra in the left hand side \cref{ISO} has to be regarded as a differential graded algebra, and although the associations on the right hand side are vector spaces, they are of dimension one.

     Consider the type $A_{2n-1}$ Milnor fibre
\begin{center}
$ M^d_{A_{2n-1}} = {\left\lbrace(x_0, \ldots , x_{d}) \in \C^{d+1}  \mid x_0^2+  \ldots + x_{d-1}^2+ x_{d}^{2n} = 0 \right\rbrace} $
\end{center}
     The $\Z / 2\Z$-action on the last coordinate $x_d \mapsto - x_d$ of type $A$ Milnor fibre induces a $\Z / 2\Z$-action on the Lagrangian sphere 
     \begin{align*}
     & L_n  \mapsto L_n;  \\
     \text{For $2 \leq j \leq n$,} \\
     & L_{n-j+1}  \mapsto  L_{n+j-1}, \\
     & L_{n+j-1}  \mapsto  L_{n-j+1}.
     \end{align*}
Please refer to \cite{KhoSei} for the construction of Lagrangian spheres from the bifurcation diagram of Milnor fibre which is proven to form a disc as the topological model for the mapping class group to act on. 
     Note that if you consider the symplectic Dehn twist generated by the disjoint union of Lagrangian spheres which are in the same orbit under the action above, they forms the type $B$ symplectomorphism on the Type $A$ Milnor fibre. 
     This is the \cite{BirHil} embedding on $\cA(B_n)$ in $\cA{(A_{2n-1})}.$
     Based on \cite{SeiSym}, the symplectic Dehn twist defined is dependent on the Lagrangian spheres, but not its orientation.

     On the cohomology level, we could also consider a conjugation action $i \mapsto -i$ on the coefficient field, that is given by the non-trivial element $\sigma \in$ Gal$(\C / \R) \cong \Z / 2 \Z.$
     
\begin{theorem} \label{invariant algebra}
The type $B$ zigzag algebra $\Ba_n$ is the $\R$-algebra of the invariant algebra  
$$ \Aa^{inv}_{2n-1} \cong HF^{\text{inv}} \left( \bigoplus^{2n-1}_{i=1} L_i, \bigoplus^{2n-1}_{i=1} L_i \right)$$ under the $\Z / 2\Z$-diagonal action of the product group of $\Z / 2\Z \times \Z / 2\Z$ where the first group factor is coming from the symplectic consideration and the second group factor is coming from the algebraic consideration as explained in the preceding paragraph; 
$$   \Ba_n  \cong_{\R}  \Aa^{inv}_{2n-1}. $$
\end{theorem}      
\begin{proof}
We split 
$$ \Aa_{2n-1} \cong HF \left( \bigoplus^{2n-1}_{i=1} L_i, \bigoplus^{2n-1}_{i=1} L_i; \C \right) = HF \left( L_n, L_n \right) \oplus \bigoplus_{j=2}^n HF( L_{n-j+1}, L_{n+j-1} ) .$$
 We then consider the invariant part of the $\Z / 2 \Z$-diagonal action from  $\Z /2 \Z$-symplectic action and for the $\Z / 2 \Z$-algebraic action.
        
     Then, the $\R$-linear morphism $\Phi_{\R} : \Ba_n \ra \Aa_{2n-1} $ defined as follow:
    \begin{align*}
     e_1 & \mapsto e_n  \\
     X_1 & \mapsto X_n   \\
     \\
     \text{For $2 \leq j \leq n$,} \\
     e_j & \mapsto e_{n-j+1} + e_{n+j-1}  \\
     X_j & \mapsto X_{n-j+1} + X_{n+j-1} \\
     {ie}_j & \mapsto ie_{n-j+1} - ie_{n+j-1}  \\
      iX_j & \mapsto iX_{n-j+1} - iX_{n+j-1} \\
      (j-1|j) & \mapsto    ((n-j+2) | (n-j+1)) + ((n+j-2)|(n+j-1))           \\  
       (j-1|j)i & \mapsto   i((n-j+2) | (n-j+1)) - i((n+j-2)|(n+j-1))           \\   
      (j|j-1) & \mapsto     ((n-j+1) | (n-j+2)) + ((n+j-1)|(n+j-2))        \\  
           i(j|j-1) & \mapsto     i((n-j+1) | (n-j+2)) - i ((n+j-1)|(n+j-2))        \\   
    \end{align*}
  It is not hard to verify that $\Phi_{\R}$ is a grading preserving $\R$-algebra morphism (It is actually implicit in the proof of \cref{isomorphic algebras}). 
  
 This map is particularly injective as the images of the map are linearly independent.
  The diagonal $\Z / 2\Z$-action of the product group on the Type A Fukaya algebra is 
  \begin{align*}
     e_n & \mapsto e_n  \\
     X_n & \mapsto X_n   \\
     \\
     \text{For $2 \leq j \leq n$,} \\
    e_{n-j+1} + e_{n+j-1}  & \mapsto e_{n+j-1} + e_{n-j+1}  \\
    X_{n-j+1} + X_{n+j-1} & \mapsto   X_{n+j-1} + X_{n-j+1}  \\
    ie_{n-j+1} - ie_{n+j-1} & \mapsto ie_{n+j-1} - ie_{n-j+1} = -( ie_{n-j+1} - ie_{n+j-1})  \\
     iX_{n-j+1} - iX_{n+j-1} & \mapsto  iX_{n+j-1} - iX_{n-j+1} = - (  iX_{n-j+1} - iX_{n+j-1} )  \\
    ((n-j+2) | (n-j+1)) + ((n+j-2)|(n+j-1))   & \mapsto    ((n+j-2)|(n+j-1)) + ((n-j+2) | (n-j+1))           \\  
     i((n-j+2) | (n-j+1)) - i((n+j-2)|(n+j-1))  & \mapsto   i((n+j-2)|(n+j-1)) - i((n-j+2) | (n-j+1))            \\   
     ((n-j+1) | (n-j+2)) + ((n+j-1)|(n+j-2))  & \mapsto    ((n+j-1)|(n+j-2))    + ((n-j+1) | (n-j+2))    \\  
          i((n-j+1) | (n-j+2)) - i ((n+j-1)|(n+j-2))  & \mapsto   i ((n+j-1)|(n+j-2))  -  i((n-j+1) | (n-j+2))      \\   
    \end{align*}  
which is, notably invariant under the diagonal action.    
  
  Hence, $\Ba_n$ sits as an invariant $\R$-algebra inside $  \Aa_{2n-1} \cong HF^{\text{inv}} \left( \bigoplus^{2n-1}_{i=1} L_i, \bigoplus^{2n-1}_{i=1} L_i \right).$
\end{proof}

\subsection{Categorifying intersection form and Cartan matrix}

 In \cite{ArnB}, Arnold studied the degenerate critical points of functions on a manifold with boundary. 
 Here, a manifold with boundary is a smooth manifold with a fixed smooth hypersurface.
 More accounts on the subject can be found in \cite{Arn1} and \cite{Arn2}.
 One theorem Arnold achieved in \cite{ArnB} is that the simple crititcal points of functions on the boundary of a manifold with boundary is classified by the simple Lie algebras $B_n,$ $C_n,$ and $F_4.$
  The Type $B_n$ singularities is the function
   \begin{equation} \label{typeB}
   g(x_0,x_1, \hdots,  x_{d}^n ) = x_0^2 + x_1^2 + \hdots + x_{d-1}^2 + x_{d}^n 
   \end{equation}
   
\noindent   on manifold with boundary $x_d = 0.$
   The diffeomorphisms of a manifold with boundary are the diffeomorphisms of the manifold that fix the boundary.
   Let $\widehat{\C}^{d+1} \cong \C^{d+1}$ be the double branched covering of the space $\C^{d+1}$ along the hyperplane $\{ x_d=0 \} \cong \C^d$ defined by, on coordinate, 
$$ x_0 = \hat{x}_0, x_1 = \hat{x}_1, \hdots   ,x_d = \hat{x}_d^2  $$ 
 Under the natural involution $(\hat{x}_0, \hat{x}_1, \hdots, \hat{x}_d) \mapsto (\hat{x}_0, \hat{x}_1, \hdots, -\hat{x}_d),$
the function $ g(x_0,x_1, \hdots,  x_{d}^n )$ induces a function $ \hat{g}(\hat{x}_0, \hat{x}_1, \hdots, \hat{x}_d) = g(\hat{x}_0, \hat{x}_1, \hdots, \hat{x}_d^2 )$  on the space $\widehat{\C}^{d+1} $ which is invariant under the action of the group $\Z / 2 \Z.$
 A locally analytic automorphisms of the space $\C^{d+1}$ preserving the subspace $\C^d$ induces a locally analytic automorphisms of the space $\widehat{\C}^{d+1}$ which commutes with the action of the group $\Z / 2 \Z.$
   As a result, a singularity of a function on a manifold with boundary is equivalent to a singularity of a function on its covering which is invariant under the invariant.
   So, studying \cref{typeB} is equivalent to study 
   \begin{equation}
   \hat{g}(\hat{x}_0, \hat{x}_1, \hdots, \hat{x}_d) = \hat{x}_0 + \hat{x}_1, \hdots + \hat{x}_d^{2n}
   \end{equation}
on $\widehat{\C}^{d+1}.$
  Then, the $(d-1)$-th homology of the non-singular level manifold 
  \begin{equation}
  \widehat{V}_\epsilon = \{ x \in \C^n \mid \hat{g}= \epsilon, \|x\| \leq \delta \}
  \end{equation}
contain a subspace $H^-$ which is anti-invariant under the involution.
     The basis of $H^-$ which are the long vanishing cycles corresponding to the preimage of the vanishing cycles constructed by the  paths joining the critical values of the perturbed $g$ and the short vanishing cycles corresponding to the preimage of the vanishing hemicycles constructed by paths joining the critical values of the perturbed $g|_{x_n=0}$
   The Type $B$ intersection form is defined using the type $B$ Dynkin diagram.
     Following \cite[Chapter VI]{Bourbaki}, the type $B$ root system has the length of the long root is $\sqrt{2}$ of the length of the short root with the associated Dynkin diagram
     \begin{figure}[H]
     \centering
\begin{tikzpicture}
\draw[thick] (0,0.1) -- (1,0.1) ;
\draw[thick] (0,-0.1) -- (1,-0.1) ;
\draw[thick] (1,0) -- (2,0) ;
\draw[thick] (2,0) -- (3.2,0) ;
\draw[thick,dashed] (3.2,0) -- (4,0) ;
\draw[thick,dashed] (4,0) -- (4.8,0) ;
\draw[thick] (4.8,0) -- (6,0) ;
\draw[thick] (6,0) -- (7,0) ;
\filldraw[color=black!, fill=white!]  (0,0) circle [radius=0.1];
\filldraw[color=black!, fill=white!]  (1,0) circle [radius=0.1];
\filldraw[color=black!, fill=white!]  (2,0) circle [radius=0.1];
\filldraw[color=black!, fill=white!]  (3,0) circle [radius=0.1];
\filldraw[color=black!, fill=white!]  (5,0) circle [radius=0.1];
\filldraw[color=black!, fill=white!]  (6,0) circle [radius=0.1];
\filldraw[color=black!, fill=white!]  (7,0) circle [radius=0.1];

\node[below] at (0,-0.1) {1};
\node[below] at (1,-0.1) {2};
\node[below] at (2,-0.1) {3};
\node[below] at (3,-0.1) {4};
\node at (0.5, 0) {{\Huge $<$}};
\node[below] at (5,-0.1) {$n$-2};
\node[below] at (6,-0.1) {$n$-1};
\node[below] at (7,-0.1) {$n$};

\end{tikzpicture}
\end{figure}
\noindent where the inequality sign is oriented to the shorter root. 
 Now, the the long vanishing cycles corresponding to the long roots have self-intersection number $-4$ and the short vanishing cycle corresponding to the short root has self intersection number $-2.$
 The intersection number of distinct vanishing cycles corresponding to the edges joining two roots is two times the multiplicity of the edge which is $2$ if both roots are long and equals to the multiplicity of the edge which $2$ if one root is short and the other is long.
  In short, the type $B$ intersection form is 
  
  $$
  \begin{pmatrix}
-2 &  2 & 0 & & \cdots & 0 \\ 
2  & -4 & 2 & 0 & \cdots &  0 \\
0  &2 &  -4 & 2 & \ddots & \vdots \\
\vdots  & \ddots & \ddots  & \ddots & \ddots  & 0 \\ 
\vdots  &\cdots  & 0 & 2 & -4  & 2  \\ 
0  & 0 &\cdots  &  0 & 2 & -4  \\
\end{pmatrix}.
$$

\noindent  On the other hand, the type $B$ Cartan matrix (\cite[Chapter IV.1.3]{Bourbaki}, \cite[Chapter 4.3]{ComCox}) is 
 
  $$
  \begin{pmatrix}
2 &  -1 & 0 & & \cdots & 0 \\ 
-2  & 2 & -1 & 0 & \cdots &  0 \\
0  & -1 &  2 & -1 & \ddots & \vdots \\
\vdots  & \ddots & \ddots  & \ddots & \ddots  & 0 \\ 
\vdots  &\cdots  & 0 & -1 & 2  & -1  \\ 
0  & 0 &\cdots  &  0 & -1 & 2  \\
\end{pmatrix}.
$$
 
\begin{corollary} \label{catint}
The type $B$ zigzag algebra categorify the intersection form of the type $B$ boundary singularities.
\end{corollary}
\begin{proof}
By \cref{Bcat}, the path-length-graded dimension matrix associated to the Type $B$ zigzag algebra when $t=-1$ is 
\begin{align*}
& \begin{pmatrix}
g_rdim_{\R} \ _1P_1  & g_rdim_{\R} \ _2P_1 & g_rdim_{\R} \ _3P_1   &\cdots & g_rdim_{\R} \ _nP_1 \\ 
g_rdim_{\R} \ _1P_2  & g_rdim_{\R} \ _2P_2 & g_rdim_{\R} \ _3P_2  &\cdots & g_rdim_{\R} \ _nP_2 \\
g_rdim_{\R} \ _1P_3  & g_rdim_{\R} \ _2P_3 & g_rdim_{\R} \ _3P_3   &\cdots & g_rdim_{\R} \ _nP_3  \\
\vdots  &  && \ddots & \vdots  \\ 
g_rdim_{\R} \ _1P_n  & g_rdim_{\R} \ _2P_n & g_rdim_{\R} \ _3P_n  &\cdots & g_rdim_{\R} \ _nP_n  \\
\end{pmatrix} |_{t=-1}
\\ =&
\begin{pmatrix}
1+t^2  &  2t & 0 &  & \cdots & 0 \\ 
2t  & 2(1+t^2) & 2t & 0 & \cdots &  0 \\
0  &2t &  2(1+t^2) & 2t & \ddots & \vdots \\
\vdots  & \ddots & \ddots  & \ddots & \ddots  & 0 \\ 
\vdots  &\cdots  & 0 & 2t & 2(1+t^2)  & 2t  \\ 
0  & 0 &\cdots  &  0 & 2t & 2(1+t^2)  \\
\end{pmatrix} |_{t=-1}
\\ =&
\begin{pmatrix}
2 &  -2 & 0 &  & \cdots & 0 \\ 
-2  & 4 & -2 & 0 & \cdots &  0 \\
0  &-2 &  4 & -2 & \ddots & \vdots \\
\vdots  & \ddots & \ddots  & \ddots & \ddots  & 0 \\ 
\vdots  &\cdots  & 0 & -2 & 4  & -2  \\ 
0  & 0 &\cdots  &  0 & -2 & 4  \\
\end{pmatrix}
\end{align*}
which is the minus of the Type $B$ intersection form.
\end{proof}

\begin{remark}
This is in accordance with the idea the singularities corresponding to the Lie algebras $B_n$ arises as an invariant germ of a function relative to the linear action of a finite group $G$ on the space $\C^n.$
  It would be best if the algebraic action can come from geometry as well, perhaps one could naturally consider the Pin structure on the Lagrangian submanifold. 
  In particular, Seidel considered, in his book,  a $\Z / 2 \Z$-action on Fukaya category which induces a non-trivial action on the Pin structure of the Lagrangian submanifold. 
  There we have a pair of isomorphisms bearing the identities $f \circ f' = -e$ where $f$ and $f'$ are isomorphisms transforming the Lagrangian submanifold with its Pin structure and $-e$ is the nontrivial element in the Ker(Pin $\ra O_n$).
   The conjugation action on the cohomology seems more natural on K-theory. 
   Since there is a Chern character map from the K-theory to cohomology theory, are there any connection between them?
   Note that the category of vector bundle and category of coherent sheaves are equivalent.
    Otherwise, we need to make sense of a negative of a generator in Floer cohomology, in order to give a geometric interpretation. 
\end{remark}

\begin{corollary} \label{catCar}
The type $B$ zigzag algebra categorify the type $B$ Cartan matrix.
\end{corollary}
\begin{proof}
By \cref{Bcat}, the path-length-graded left module dimension matrix associated to the Type $B$ zigzag algebra when $t=-1$ is 
\begin{align*}
& \begin{pmatrix}
g_rdim_{\R} \ _1P_1  & g_rdim_{\C} \ _2P_1 & g_rdim_{\C} \ _3P_1   &\cdots & g_rdim_{\C} \ _nP_1 \\ 
g_rdim_{\R} \ _1P_2  & g_rdim_{\C} \ _2P_2 & g_rdim_{\C} \ _3P_2  &\cdots & g_rdim_{\C} \ _nP_2 \\
g_rdim_{\R} \ _1P_3  & g_rdim_{\C} \ _2P_3 & g_rdim_{\C} \ _3P_3   &\cdots & g_rdim_{\C} \ _nP_3  \\
\vdots  &  && \ddots & \vdots  \\ 
g_rdim_{\R} \ _1P_n  & g_rdim_{\C} \ _2P_n & g_rdim_{\C} \ _3P_n  &\cdots & g_rdim_{\C} \ _nP_n  \\
\end{pmatrix} |_{t=-1}
\\ =&
\begin{pmatrix}
1+t^2  &  t & 0 & & \cdots & 0 \\ 
2t  & 1+t^2 & t & 0 & \cdots &  0 \\
0  &t &  1+t^2 & t & \ddots & \vdots \\
\vdots  & \ddots & \ddots  & \ddots & \ddots  & 0 \\ 
\vdots  &\cdots  & 0 & t & 1+t^2  & t  \\ 
0  & 0 &\cdots  &  0 & t & 1+t^2  \\
\end{pmatrix} |_{t=-1}
\\ =&
\begin{pmatrix}
2 &  -1 & 0 & & \cdots & 0 \\ 
-2  & 2 & -1 & 0 & \cdots &  0 \\
0  &-1 &  2 & -1 & \ddots & \vdots \\
\vdots  & \ddots & \ddots  & \ddots & \ddots  & 0 \\ 
\vdots  &\cdots  & 0 & -1 & 2  & -1  \\ 
0  & 0 &\cdots  &  0 & -1 & 2  \\
\end{pmatrix}
\end{align*}
which is the minus of the type $B$ Cartan matrix.
\end{proof}

\begin{remark}
There is a subtle difference between the choice of Type $B$ Cartan matrix above and in the next chapter (see \cref{maintheorem}), but they "play" the same number game in \cite{ComCox}.

\end{remark}

\chapter{The Faithfulness of the Type B 2-Braid Group.} \label{Chap2}

\section{Background and Outline of the Chapter}
From the first chapter, we know that a categorical braid group action realised by Khovanov-Seidel \cite{KhoSei} for type $A$ braid group uses complexes of modules over a zigzag algebra.
Following their work, Licata-Queffelec \cite{LicQuef} similarly defined the simply-laced type $ADE$ zigzag algebras, which provides categorical actions for the corresponding generalised braid groups.
In our previous chapter, we suggested a suitable candidate for the type $B$ zigzag algebra, which similarly allows a categorical action of the type $B$ generalised braid group.
Note that in all of the above cases, the faithfulness of the braid group actions is known.

Another way to construct the categorical actions of Artin groups is through the theory of Soergel bimodules.
This was introduced by Rouquier in \cite{Rouq} as the \emph{2-braid groups}, where the braid generators act on the category of Soergel bimodules via tensoring with the Rouquier complexes.
 Rouquier also conjectured that this categorification is faithful, which was later proven by Brav-Thomas \cite{BravTho} for type $ADE$, and then generalised to all finite types by Jensen \cite{jensen_master,jensen_2016}.
In type $A$, the categorical actions  via zigzag algebras can be described as a ``quotient'' of the ones via Soergel bimodules, as shown in the work of Jensen in his Master thesis \cite{jensen_master}.
We shall briefly describe this relationship here.
Firstly, note that the zigzag algebra for type $A$ leads to a categorification of (a further quotient of) the type $A$ Temperley-Lieb algebra, using certain (monoidal, additive) subcategory of $\Aa$-bimodules \cite{KhoSei}.
Just as Temperley-Lieb algebras are quotients of Hecke algebras (categorified by Soergel bimodules), this relation can be understood categorically through a full and essentially surjective (quotient) functor from the category of Soergel bimodules to this subcategory of $\Aa$-bimodules, as described in \cite{jensen_master}.
Our construction of the type $B$ zigzag algebra $\Ba$ also fits into this whole construction, where a certain subcategory of $\Ba$-bimodules categorifies (a further quotient of) the type $B$ Temperley-Lieb algebra and is also a quotient category of Soergel bimodules.
    This provides an alternate proof to Rouquier's conjecture for the type $B$ case, in addition to serving as a piece of evidence for the type $B$ zigzag algebra constructed being a correct candidate.

\subsection*{Outline of the chapter}

In \cref{Soergel Bimodules}, we recall the definition of Soergel bimodules. 
 We first define Coxeter systems by generators and relations as well as the corresponding generalised braid groups.
 A combinatorial tool known as the two-coloured quantum number is introduced before defining a realisation of the Coxeter system. 
 Demazure surjectivity is mentioned for general theory to hold and for the purpose of computation later.

 \cref{TL2C} is where we introduce the Temperley-Lieb category, and subsequently, its souped-up version, the two-coloured Temperley-Lieb $2$-category. 
  We also provide a recursive formula for the Jones-Wenzl projectors and the two-coloured Jones-Wenzl projectors.

In \cref{DiaCatSB}, we recall the whole definition of diagrammatic category for Soergel bimodules from \cite{EliGeoSoe} where they proved the equivalence between the diagrammatic category with the original algebraic combinatorial Soergel bimodule category introduced by Sorgel himself. 

\cref{2BraidGrp} is where we recall the construction  of $2$-braid groups introduced by Rouquier \cite{Rouq}.

To make this chapter self-contained, in \cref{CatTypeB}, we extract the algebraic construction of the type $B$ $2$-braid group from \cref{B zigzag} in \cref{Chap1}.

 \cref{Faith2BraidGrp} is where we present the faithfulness of the type $B$ $2$-braid group by constructing a monoidal functor from the diagrammatic Soergel homotopy category to type $B$ bimodule homotopy category.

\newpage

	\section{Soergel Bimodules} \label{Soergel Bimodules}

\subsection{Coxeter systems by generators and relations and its corresponding generalised braid groups}	
	
	  Recalling that \cite{BraidGroups} a \textit{Coxeter matrix} $M= (m_{s,t})_{s,t \in S}$ is a symmetric matrix where $S$ is a finite set, $m_{s,s} = 1$ for every $s$ in $S$ and $m_{s,t} \in \{ 2, 3, \cdots, \infty \}$. 
	  Equivalently, $m$ can be represented by a \textit{Coxeter graph}, or \textit{Coxeter diagram} with its node set as $S$ and its edges determined by the unordered pairs ${s,t}$ such that $m_{s,t} \geq 3.$
	For edges with $m_{s,t} \geq 4,$ we will label them by that number where the label will be dropped if $m_{s,t} = 3.$
	When $s$ and $t$ commute,  we have $m_{s,t}= 2,$  and they are not connected in the Coxeter graph.
	  
	 To every Coxeter matrix $M,$ we can associate a group $W_M$ with presentation:  the generators being the elements of $S$ and the relations being $(st)^{m_{s,t}} = e,$ for all $s,t$ in $S$ such that $m_{s,t} \neq \infty$ where $ e $ is the identity element  in $W_M.$
	 By convention, $m_{s,s} = 1$
 Subsequently,  we will call the pair $(W_M,S)$,  a Coxeter system where $W_M$ is a group, or more specifically a Coxeter group, with  a set $S$ of its generators subject to the following relations: 
\begin{align}
&& s^2 &= e && \text{for all $s \in S,$}  \label{sqid}\\
\text{(Braid relation)} && \underbrace{sts \cdots}_{m_{st}} &= \underbrace{tst \cdots} _{m_{st}}  && \text{for all $s,t \in S,$} \label{braidrel}
\end{align}
where $ 2 \leq m_{s,t} < \infty $ and $m_{s,t}= m_{t,s}$.
If $m_{s,t} = \infty,$ then the braid relation is omitted.
 The rank of $(W,S)$ is defined to be the cardinality of $S$ .
We might use the words \textit{index} and \textit{colour} to refer to an element $s \in S.$ 
 An \textit{expression} for an element $w \in W$ is a finite sequence of indices, denoted by $\underline{w}.$
 
 If we now remove the relation \cref{sqid}, then the group $\cA(W_M)$ generated by $S$ and the braid relations \cref{braidrel} is called the \textit{generalised braid group}\footnote{Also known as \textit{Artin group} or \textit{Artin-Tits group}.} associated to $W_M.$
 In which case, there is a surjective map $\pi: \cA(W_M) \ra W_M.$

\begin{example}
For $n \geq 2,$ the type $B_n$ braid group $\cA({B_n})$ associated to the type $B_n$ Coxeter graph 

\begin{figure}[H]
\centering
\begin{tikzpicture}
\draw[thick] (0,0) -- (1,0) ;
\draw[thick] (1,0) -- (2,0) ;
\draw[thick] (2,0) -- (3.2,0) ;
\draw[thick,dashed] (3.2,0) -- (4,0) ;
\draw[thick,dashed] (4,0) -- (4.8,0) ;
\draw[thick] (4.8,0) -- (6,0) ;
\draw[thick] (6,0) -- (7,0) ;
\filldraw[color=black!, fill=white!]  (0,0) circle [radius=0.1];
\filldraw[color=black!, fill=white!]  (1,0) circle [radius=0.1];
\filldraw[color=black!, fill=white!]  (2,0) circle [radius=0.1];
\filldraw[color=black!, fill=white!]  (3,0) circle [radius=0.1];
\filldraw[color=black!, fill=white!]  (5,0) circle [radius=0.1];
\filldraw[color=black!, fill=white!]  (6,0) circle [radius=0.1];
\filldraw[color=black!, fill=white!]  (7,0) circle [radius=0.1];

\node[below] at (0,-0.1) {1};
\node[below] at (1,-0.1) {2};
\node[below] at (2,-0.1) {3};
\node[below] at (3,-0.1) {4};
\node[below] at (5,-0.1) {$n$-2};
\node[below] at (6,-0.1) {$n$-1};
\node[below] at (7,-0.1) {$n$};
\node[above] at (.5,0.1) {4};
\end{tikzpicture}
\end{figure}

\noindent or the type $ B_n$ Coxeter matrix

$$
\begin{bmatrix}
1 & 3 &  2 & 2 & \cdots & 2 \\
3 & 1 & 3  &  2 & &  \vdots           \\     
2 & 3 & 1  &  3 &  \ddots &   2       \\
2 & 2 & 3 & 1 &  \ddots &  2 \\
\vdots  & & \ddots &   \ddots   & \ddots  &  3 \\
2 &  \cdots  & 2 & 2 &  3 &1  
\end{bmatrix}
$$
\noindent is generated by
$$ \sigma_1^B,\sigma_2^B,\ldots, \sigma_{n}^B $$

\noindent subject to the relations

\begin{align}
\sigma_1^B \sigma_2^B \sigma_1^B \sigma_2^B & = \sigma_2^B \sigma_1^B \sigma_2^B \sigma_1^B;  \\
      \sigma_j^B  \sigma_k^B  &= \sigma_k^B  \sigma_j^B   & \text{for} \ |j-k|> 1;\\    
      \sigma_j^B  \sigma_{j+1}^B  \sigma_j^B  &=  \sigma_{j+1}^B  \sigma_{j}^B  \sigma_{j+1}^B & \text{for} \ j= 2,3, \ldots, n.
\end{align} 
\end{example}

\subsection{Two-coloured quantum numbers}  \label{sec2}

Remember that the quantum numbers, as defined by $[m] = [m]_q := \frac{q^m -q^{-m}}{q-q^{-1}}$ for $m \in \Z_{\geq 0}.$
 In particular, $[2]_q = q + q^{-1}, [1]_q = 1$, and $[0]_q = 0.$
 The subscript $q$ will be suppressed in general, but one point to bear in mind is that if $q = e^{i \theta}$ is a root of unity, we have a familiar number $[2]^2_q = 4$cos$^2\theta$ from the representation theory of Coxeter graphs.
 
 Let $\delta$ be an indeterminate, such that the ring $\Z[\delta]$ is a subring of $\Z[q,q^{-1}]$ under the specialization $\delta \mapsto [2].$
 Next, we introduced the subring $\Z[x,y]$ of $\Z[\delta]$ where the corresponding subring $\Z[x,y]$ can be identified with the subring $\Z[\delta^2]$ via $xy = \delta^2$.
 Subsequently, the \textit{two-coloured quantum numbers} is defined to be the elements in $\Z[x,y]$ inductively as follows: 
 \begin{enumerate} [(i)]
   \item $[0]_x = [0]_y = 0$ and $[1]_x = [1]_y = 1;$ 
 \item $[2]_x [m]_y = [m+1]_x + [m-1]_x;$
 %$\text{for $m$ even, } [m]_x = x \frac{[m]}{[2]}, [m]_y = y \frac{[m]}{[2]} ;$
 \item $[2]_y [m]_x = [m+1]_y + [m-1]_y.$
 %$\text{for $m$ odd, } [m]_x=[m]_y = [m].$
 \end{enumerate}
 For $m$ odd, since $[m]_x = [m]_y,$ we shorten the notation to $[m].$

We also have the following two-coloured analogue of quantum number equations: for $m \geq a,$
\begin{equation} \label{Desurj}
[m-a]_y + [m+1]_x[a]_y = [m]_y[a+1]_x  \text{ and }
[m-a]_x + [m+1]_y[a]_x = [m]_x[a+1]_y.
\end{equation}

In particular, we can take 
$[2]_x=x$ and $[2]_y = y;$ $[3]_x = xy-1=[3]_y;$ $[4]_x =x^2y-2x$ and $[4]_y = xy^2 - 2y;$ $[5]_x = x^2 y^2 - 3xy + 1= [5]_y;$ $[6]_x = x^3 y^2 - 4x^2 y +2x$ and $[6]_x = x^2 y^3 -4xy^2 +2y.$

\subsection{Definition of Soergel bimodules}

\subsubsection{Realisation of Coxeter System}

Let $\Bbbk$ be a commutative ring and $(W,S)$ be a Coxeter system. 
A \textit{realisation} of $(W,S)$ over $\Bbbk$ is a free, finite rank $\Bbbk$-module $\h,$ together with a choice of simple co-roots $\{\alpha_s^\vee \mid s \in S \} \subset \h$ and roots $\{ \alpha_s \mid s \in S \} \subset \h^* = \Hom_\Bbbk(\h, \Bbbk)$ having a Cartan matrix $(a_{s,t})_{{s,t} \in S}$ where $ a_{s,t} := \< \alpha_s^\vee, \alpha_t \> $ satisfying:
\begin{enumerate}[(i)]
\item $a_{s,s} =2$ for all $s \in S;$
\item for any $s,t \in S$ with $m_{st} < \infty,$ %if $\Bbbk$ is given a $\Z[x,y]$-algebra structure where%
let $x = a_{s,t}$ and $y=a_{t,s},$ then $[m_{st}]_x = [m_{st}]_y = 0;$
\item the assignment $s(v) := v - \< v, \alpha_s \> \alpha_s^\vee$ for all $v \in \h$ yields a representation of $W.$
\end{enumerate}

We call a realization \textit{symmetric} if $a_{s,t} = a_{t,s}$ for all $s,t \in S.$

A realization is \textit{balanced} if for each $s,t \in S$ with $m_{st} < \infty,$ one has $[m_{st}-1]_x = [m_{st}-1]_y = -1$ for $m_{st}$ odd or $[m_{st}-1] = 1$ for $m_{st}$ even.

\begin{remark} Note that we use a slightly different definition of a balanced realization from \cite{BenTwoColor, BenGeo} because our circle in the Temperley-Lieb algebra (refer \cref{TLd}) evaluates to $\delta$ instead of -$\delta.$

\end{remark}

%\begin{remark}
Given a realization $(\h, \{\alpha^\vee_s\},\{ \alpha_s \})$ of $(W,S)$ over $\Bbbk,$ we can look at its \textit{Cartan matrix} $(\<\alpha^\vee_s, \alpha_t\>)_{s,t \in S}$ which is symmetric if and only if the realization is.
  Conversely, suppose we start with a matrix $(a_{s,t})_{s,t \in S}$ such that $a_{s,s} = 2.$ 
  We can construct the free $\Bbbk$-module $\h = \bigoplus_{s \in S} \Bbbk \alpha^\vee_s,$ and define the roots $\alpha_s \in \h^*$ by the formula $\< \alpha^\vee_s, \alpha_t \> = a_{s,t}.$
  If this yields a realization of $(W,S),$ then the matrix $(a_{s,t})_{s,t \in S}$ is the \textit{Cartan matrix} for $(W,S)$ over $\Bbbk.$
   We call a realization minimal if the basis $\{ \alpha^\vee_s\}$ of $\h$ can be recontructed from its Cartan matrix from the above way.
%\end{remark}

 With such a realization $(\h, \{ \alpha_s^\vee \}, \{ \alpha_s \})$ of $(W,S),$ we define $R := \bigoplus_{k \geq 0} S^k(\h^*)$ to be the $\Z$-graded $\Bbbk$-symmetric algebra on $\h^*$ with deg $\h^*$ = 2.
 % This would imply that 	$S^k(\h^*) = 0$ for $k$ odd.
	We consider the action of $W$ on $\h^*$ via its geometric representation, that is, for all,$ \sigma_s (\gamma) = \gamma - \< \alpha_s^\vee, \gamma \> \alpha_s,$ and subsequently, extend it to an action on the whole $R$ by graded automorphisms. 
	In one way, we can see $R$ as the algebra of regular functions on $\h.$

\begin{example}
Let  $(W,S)$ be a Coxeter system of finite rank. 
	Let $\Bbbk= \R$ and $\h = \bigoplus_{s \in S} \R \alpha^\vee_s.$
	 Consider the following realization by defining the elements $\{ \alpha_s \} \subset \h^*$ by
	
	\begin{equation}
	\< \alpha^\vee_s, \alpha_t \> = -2 cos \left( \frac{\pi}{m_{s,t}} \right)
\end{equation}	 

\noindent with the convention that $m_{s,s} = 1$ and $\pi / \infty = 0.$
We get the \textit{geometric representation} in \cite{Hump} %5.3%
 Moreover, this is a balanced symmetric realization of $(W_M,S)$.
\end{example}
	
\begin{example} \label{Type B Geo Rep}
Consider the geometric representation of the type B Coxeter graph, as defined in \cite{ComCox}. 
Let $\Bbbk = \R.$ 
 For $m_{s,t} = 4,$ we define $a_{s,t} = -2$ and $a_{t,s} = -1$ whereas for $m_{s,t} = 3,$ we define $a_{s,t} = -1 = a_{t,s} .$ 
 This is a balanced non-symmetric realization of the type $B_n$ Coxeter system.
  
\end{example}

\subsubsection{Demazure Surjectivity and Soergel bimodules}

For $n \in \Z,$ we denote the grading shift by $(n)$ such that given a $\Z$-graded (bi)module $M= \bigoplus_i M^i,$ the shifted object is defined by the formula $M(n)^i = M^{i+n}.$
	We wish to view the category of $(R,R)$-bimodules, denoted by \textit{$(R,R)$-bimod}, as the graded, abelian, monoidal category of $\Z$-graded $(R,R)$-bimodules that are finitely generated as left and as right $R$-modules with grading-preserving bimodule homomorphisms as morphism.
    For any two graded $(R,R)$-bimodules $M$ and $N,$ we denote the graded Hom space by $\Hom^*(M,N) = \bigoplus_{k \in \Z} \Hom_{(R,R)-bimod} (M,N(k)).$
  
  Fix a $s \in S,$ Let $R^s \subseteq R$ be the invariant subring under the action of $s,$ and we will extend the evaluation map $\< \alpha_s^\vee, \cdot \> : \h^* \ra \Bbbk$ to the \textit{Demazure operator} $ \partial_s:R \ra R^s(-2),$ by the formula 

\begin{equation}
\partial_s(f) = \frac{f-\sigma_s(f)}{\alpha_s}.
\end{equation}
  
\noindent Observe that both the numerator and denominator is $s$-antiinvariant, so that the fraction will lie in the subring $R^s$ of $s$-invariants.
  But, in order for the fraction to be well-defined, we still need Demazure Surjectivity on our realization, or rather on the base ring $\Bbbk$, so that $\Bbbk$ contains certain algebraic integers. 
  Since the base rings we considered in this paper are $\R$ and $\C,$ we have the following assumption:
  
  \begin{assumption}[Demazure Surjectivity]
  For all $s \in S,$ the map $\alpha_s : \h \ra \Bbbk$ and the  the map $\partial_s : \h^* \ra \Bbbk$ are surjective. 
  \end{assumption}    
    
   As long as Demazure Surjectivity  holds, $\alpha_s \neq 0,$ 
   so that there exist some $\delta \in \h^*$ such that $\< \alpha^\vee_s, \delta \> = 1$ and $\sigma_s(\delta) = \delta - \alpha_s \neq \delta.$
   The well-definedness of Demazure Operator will be answered by following lemma: 
   
   \begin{lemma} \label{uniquedelta}
     Any element $f \in R$  can be written uniquely as $f = g \delta + h$ for $g,h \in R^s.$
   \end{lemma}
 
\begin{proof} Please refer to \cite[Claim 3.9]{BenGeo}.
\end{proof}
% (Uniqueness) If $f$ is of the form, then $f - \sigma_s(f) = g(\delta- \sigma_s(\delta)) = g \alpha_s.$
 %Suppose now $g \delta + h = f = g' \delta + h' .$
 %Then, by before, $g \alpha_s = g' \alpha_s.$
 %Since $\Bbbk$ is a domain, so is $R.$ 
 %Therefore, $g = g'$ and, so $h = h'.$
 %(Existence)  Note that $a_{s,t} \delta - \alpha_t \in \h^*$ is s-invariant and it is in the kernel of $\< \alpha_s, \cdot \>.$
% In particular, this gives that any polynomial in R can be written as a polynomial in $\delta$ with coefficients in $\R^s.$
 % Also, we have that $\delta^2 = \delta (\delta + \sigma_s(\delta)) - \delta \sigma_s(\delta),$ where both $\delta + \sigma_s(\delta)$ and $\delta \sigma(\delta)$ are $s$-invariant. 
 % Hence, every polynomial in $\delta$ can be expressed in the form of $g \delta + h$ for $g, h \in R^s.$ 

  Using this lemma, we could could have just defined $\partial_s(f) := g.$
    One can also see that the above definition is well-defined irrespective of the choice of $\delta$ by a similar argument.
    
    What is interesting about \cref{uniquedelta} is that it also shows $R$ is free of rank $2$ over $R^s$ with basis $\{1, \delta\}.$

\begin{comment} 
\begin{theorem}
There is a $R^s$-module isomorphism $R \cong R^s \oplus R^s(-2).$
\end{theorem} 
\begin{proof}
It follows from \cref{uniquedelta}.
\end{proof}
\end{comment}
%

    For all $s,$ we define the $(R,R)$-bimodule $B_s:= R \otimes_{R^s} R(1)$
   Given an expression $\underline{w} = s_1 s_2 \cdots s_k,$ 
   the corresponding \textit{Bott-Samelson bimodule} is the tensor product 
   
   $$B_{\underline{w}} := B_{s_1} \otimes_R B_{s_2} \otimes_R  \cdots \otimes B_{s_k} \cong R \otimes_{R^{s_1}} R \otimes_{R^{s_2}} R \otimes_{R^{s_3}} \cdots \otimes_{R^{s_k}} R(k).$$
   
\noindent    The category  of Bott-Samelson bimodules, denoted by $\B\bS$-bimod, is the full monoidal subcategory of $(R,R)$-bimod generated by $B_s$ for $s \in S.$
   Take the additive closure of $\B\bS$-bimod and then its Karoubi envelope, we form the category of Soergel bimodules, denoted by $\bS$-bimod.

\section{ Temperley-Lieb 2-Category}	 \label{TL2C}
	
\subsection{The (uncoloured) Temperley-Lieb category}	 \label{TLd}
	
%\begin{definition}	
	We define \textit{(uncoloured) Temperley-Lieb category $\cT \cL$} as having coefficients which lie in $\Z[\delta].$ 
	It is a monoidal category with objects $n \in \N$ which can be seen as $n$ points on a line.
	Its morphisms from $n$ to $m$ are the $\Z[\delta]$-module spanned by crossingless matchings with $n$ points on bottom to $m$ points on top. 
	The multiplication is given by vertical concatenation of diagrams, and resolving any circles with the scalar $\delta.$
%\end{definition}	
%	
%	
%\begin{remark}
Note that End$(n)= TL_n$ is called the uncolored Temperley-Lieb algebra on $n$ strands with the identity element $\mathds{1}_n.$
%\end{remark}

	\subsubsection{Jones-Wenzyl Projectors}
Consider the representation of the quantum group $U_q(\fs \fl  _2)$ of $\fs \l_2$ with irreducible representations $V_k$ having highest weight $q^k.$ 
 In \cite{FrenKho}, the category $\cT \cL$ governs the representations of quantum group $U_q(\fs \fl  _2)$ by serving as the algebra of intertwining operators of tensor products of the fundamental representation $V_1$. 
 As a result, after doing the base change to $\Q(q)$ under the map $\delta \mapsto q+ q^{-1},$ the category $\cT \cL$ is equivalent to the full subcategory of $U_q(\fs \fl  _2)$-representations,  such that $\Hom_{\cT \cL}(n,m) \otimes_{\Z} \Q(q) = \Hom(V_1^{\otimes n},V_1^{\otimes m}).$
  Since every $V_1^{\otimes n}$ contains an indecomposable representation $V_n$ with multiplicity one as one of its summands, there exists a canonical idempotent $JW_n \in \End_{\cT \cL}(n) \otimes \Q(q) = \End_{U_q(\fs \fl_2)}  (V_1 ^{\otimes n})$  called the \textit{Jones-Wenzyl projectors} projecting from $V_1^{\otimes n}$ to $V_n.$
  Performing another base change $\delta \mapsto [2]_q$ again, we can equally use quantum numbers to express them, provided certain $n$-th quantum numbers appearing as quantum binomial coefficients are invertible.

\begin{proposition}
The Temperley-Lieb $TL_n,$ after extension of scalers, contains canonical idempotents which project $V_1^{\otimes n}$ to each isotypic component. 
	Given a choice of primitive idempotents, $TL_n$ contains maps which realize the isomorphisms between the different irreducible summands of the same isotypic component.
	These maps can be defined in any extension of $\Z[\delta]$ where the quantum numbers $[2],[3], \hdots, [n]$ are invertible. 
\end{proposition}

\begin{proof}
Please refer to \cite[Prop 4.2]{BenTwoColor}.
\end{proof}

For more properties of Jones-Wenzyl projectors, please refer to \cite{EliTwo,Scott,EliDiGr}.
 But,  we will write down one recursive formula which allows us to get our desired Jones-Wenzyl projectors which sums over the possible positions of cups:

\begin{equation}
\includegraphics[scale=1]{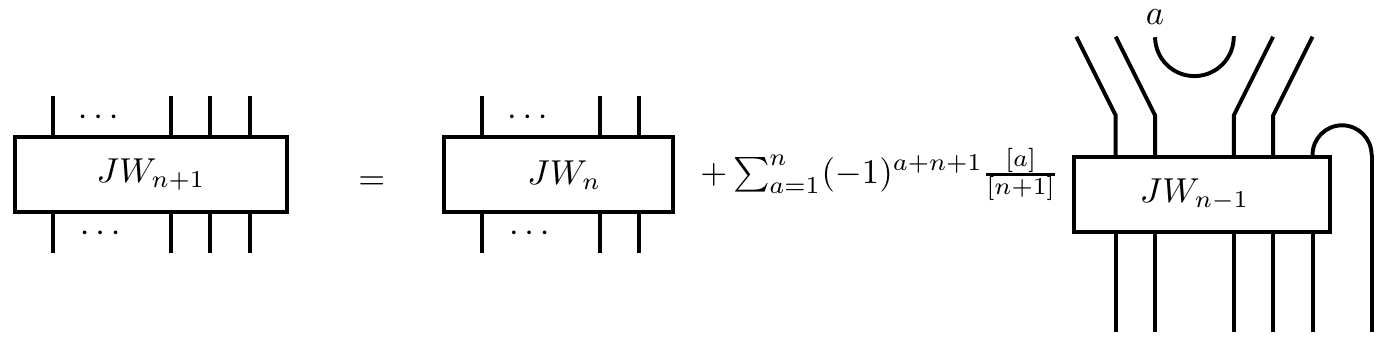}
\end{equation}

\begin{example} 
\begin{align*} 
JW_1 &= \ \ \ 
\begin{tikzcd}  
\draw[line width=0.8mm] (2,-.5) -- (2,.5);
\end{tikzcd}
 && JW_2= \ \ \ 
 \begin{tikzcd}  
\draw[line width=0.8mm] (2,-.5) -- (2,.5);
\draw[line width=0.8mm] (2.5,-.5) -- (2.5,.5);
\end{tikzcd}
\ \ - \ \frac{1}{[2]} \ \ 
 \begin{tikzcd}  
\draw[line width=0.8mm] (0,0.5) arc(0 : -180: 0.3);
\draw[line width=0.8mm] (-.6,-0.5) arc(180 : 0: 0.3);
\end{tikzcd}
\end{align*}
\begin{align*} 
 JW_3= \ \ \ 
 \begin{tikzcd}  
\draw[line width=0.8mm] (2,-.5) -- (2,.5);
\draw[line width=0.8mm] (2.5,-.5) -- (2.5,.5);
\draw[line width=0.8mm] (3,-.5) -- (3,.5);
\end{tikzcd}
\ \ - \ \frac{[2]}{[3]} \ \ 
 \begin{tikzcd}  
 \draw[line width=0.8mm] (-1,-.5) -- (-1,.5);
\draw[line width=0.8mm] (0,0.5) arc(0 : -180: 0.3);
\draw[line width=0.8mm] (-.6,-0.5) arc(180 : 0: 0.3);
\end{tikzcd}
\ \ - \ \frac{[2]}{[3]} \ \ 
 \begin{tikzcd} 
 \draw[line width=0.8mm] (.5,-.5) -- (.5,.5);
\draw[line width=0.8mm] (0,0.5) arc(0 : -180: 0.3);
\draw[line width=0.8mm] (-.6,-0.5) arc(180 : 0: 0.3);
\end{tikzcd}
\ \ + \ \frac{1}{[3]} \ \ 
 \begin{tikzcd} 
 \draw[line width=0.8mm] (.5,-.5) -- (-.5,.5);
\draw[line width=0.8mm] (.6,0.5) arc(0 : -180: 0.3);
\draw[line width=0.8mm] (-.6,-0.5) arc(180 : 0: 0.3);
\end{tikzcd}
\ \ + \ \frac{1}{[3]} \ \ 
 \begin{tikzcd} 
 \draw[line width=0.8mm] (.5,.5) -- (-.5,-.5);
\draw[line width=0.8mm] (0,0.5) arc(0 : -180: 0.3);
\draw[line width=0.8mm] (0,-0.5) arc(180 : 0: 0.3);
\end{tikzcd}
\end{align*}

\end{example}
	
\subsection{The two-coloured Temperley-Lieb 2-category}	 \label{twocolorTL}
	
%\begin{definition}
We define \textit{two-coloured Temperley-Lieb 2-category $2\cT \cL$} as having coefficients which lie in $\Z[x,y].$ 
   This category $2\cT \cL$ has two objects: red and blue.
   It has two generating 1-morphisms: a map from red to blue and  a map from blue to red.
   The 2-morphisms are the $Z[x,y]$-module spanned by appropriately-coloured crossingless matchings.
   The multiplication is given by vertical concatenation of diagrams, and replacing any closed component (i.e. circle) with a scalar indeterminate;  a circle with red (resp. blue) interior evaluates to $x$ (resp. $y$). 
   %\end{definition}
   
   We can obtain the usual uncoloured Temperley-Lieb category  $\cT \cL$ by letting $x = y = \delta$ and removing the colours.
  On the other hand, any crossingless matchings in $JW_{m-1}$ will divide the planar strip into $m+1$ regions which can be coloured alternatingly across strands with two different colours, say red and blue.

\subsubsection{Two-coloured Jones-Wenzyl Projectors}

%\begin{remark}
  The corresponding two-coloured Temperley-Lieb algebra in $2 \cT \cL$ also contains Jones-Wenzl projectors, but with each coming in two flavours with different coefficients depending on whether it is left-blue-aligned or left-red-aligned. 
  The other flavour can be acquired by switching the colours as well as switching $x$ and $y.$
  
%\end{remark}

\begin{example}
\begin{align*} 
JW_1 &= \ \ \ 
\begin{tikzcd}  
\draw[fill=red!30] (2,-.5) -- (2,.5) -- (2.25, .5) -- (2.25,-.5) -- (2,-.5);
\draw[fill=blue!30] (2.25,-.5) -- (2.25,.5) -- (2.5, .5) -- (2.5,-.5) -- (2.25,-.5);
\end{tikzcd} \\
 JW_2 &= \ \ \ 
 \begin{tikzcd}  
 \draw[fill=red!30] (2,-.5) -- (2,.5) -- (2.25, .5) -- (2.25,-.5) -- (2,-.5);
\draw[fill=red!30] (2.75,-.5) -- (2.75,.5) -- (3, .5) -- (3,-.5) -- (2.75,-.5);
\draw[fill=blue!30] (2.25,-.5) -- (2.25,.5) -- (2.75, .5) -- (2.75,-.5) -- (2.25,-.5);
\end{tikzcd}
\ \  - \frac{1}{y} \ \ 
 \begin{tikzcd}  
 \draw [fill=red!30]  (-.85,.5) -- (0,0.5) arc(0 : -180: 0.3) -- (.25,.5) -- (.25,-.5) -- (-.6,-0.5) arc(180 : 0: 0.3) -- (-.85,-.5) --(-.85,.5);
\draw[fill=blue!30] (0,0.5) arc(0 : -180: 0.3) -- (-.85,.5) -- (.25,.5);
\draw[fill=blue!30] (-.6,-0.5) arc(180 : 0: 0.3)-- (-.85,-.5) -- (.25,-.5);
\end{tikzcd} \\
 JW_3 &= \ 
  \begin{tikzcd}  
\draw[fill=red!30] (2,-.5) -- (2,.5) -- (2.25, .5) -- (2.25,-.5) -- (2,-.5);
\draw[fill=blue!30] (2.25,-.5) -- (2.25,.5) -- (2.75, .5) -- (2.75,-.5) -- (2.25,-.5);
\draw[fill=red!30] (2.75,-.5) -- (2.75, .5)-- (3.25,.5) --  (3.25,-.5)  -- (2.75,-.5);
\draw[fill=blue!30] (3.25,-.5) -- (3.25,.5) -- (3.5, .5) -- (3.5,-.5) -- (3.25,-.5);
\end{tikzcd}
\  - \frac{y}{xy-1} \ \ 
 \begin{tikzcd}  
 \draw[fill=red!30] (-1.25,-.5) -- (-1.25,.5) -- (-1,.5) -- (-1,-.5) -- (-1.25,-.5) ;
 \draw[fill=blue!30] (-1,-.5) -- (-1,.5) -- (.25,.5) -- (.25, -.5) -- (-1,-.5);
\draw[fill=red!30] (0,0.5) arc(0 : -180: 0.3) -- (0,.5) -- (.25,.5);
\draw[fill=red!30] (-.6,-0.5) arc(180 : 0: 0.3) -- (-0.6,-.5) -- (0,-.5);
\end{tikzcd}
\ - \frac{x}{xy-1} \ \ 
 \begin{tikzcd}  
 \draw[fill=blue!30] (.4,-.5) -- (.4,.5) -- (.65,.5) -- (.65,-.5) -- (.45,-.5);
  \draw[fill=red!30] (.4,-.5) -- (.4,.5) -- (-.85, .5) -- (-.85, -.5)-- (.5,-.5);
\draw[fill=blue!30] (0,0.5) arc(0 : -180: 0.3) -- (0,.5) -- (-.6,.5);
\draw[fill=blue!30] (-.6,-0.5) arc(180 : 0: 0.3) -- (-.6,-.5) -- (0,-.5);
\end{tikzcd} 
\  + \frac{1}{xy-1} \ \ 
 \begin{tikzcd} 
 \draw[fill=red!30] (.5,-.5) -- (-.5,.5) -- (-.75,.5) -- (-.75,-.5) -- (.5,-.5);
  \draw[fill=blue!30] (.5,-.5) -- (-.5,.5) -- (.75,.5) -- (.75,-.5) -- (.5,-.5);
\draw[fill=red!30] (.6,0.5) arc(0 : -180: 0.3) -- (0,.5) -- (.6,.5);
\draw[fill=blue!30] (-.6,-0.5) arc(180 : 0: 0.3) -- (-.6,-.5) -- (0,-.5);
\end{tikzcd}
\   + \frac{1}{xy-1} \ \ 
 \begin{tikzcd} 
 \draw[fill=red!30] (.5,.5) -- (-.5,-.5) -- (-.75, -.5) -- (-.75, .5)-- (.5,.5);
  \draw[fill=blue!30] (.5,.5) -- (-.5,-.5) -- (.75, -.5) -- (.75, .5)-- (.5,.5);
\draw[fill=blue!30] (0,0.5) arc(0 : -180: 0.3) -- (-.6,.5) -- (0,.5);
\draw[fill=red!30] (0,-0.5) arc(180 : 0: 0.3) -- (0.6,-.5) -- (0,-.5);
\end{tikzcd}
\end{align*}
\end{example}

As dicussed before, the fact that $[m]_x = [m]_y = 0$ guaranteed that $JW_{m-1}$ is well-defined, that is, the denominators  $[m-1]_x,[m-1]_y $ are invertible.
 
 \subsubsection{Rotation of Jones-Wenzyl Projectors}
 A \textit{rotation} of $JW_{m-1}$ by a single strand is obtained by cupping on the leftmost bottom and capping on the rightmost top.
 It is known that a rotation of $JW_{m-1}$ by two strands will produce $JW_{m-1}.$
 If a rotation of an uncolored morphism $JW_{m-1}$	by a single strand returns $JW_{n-1},$ then $JW_{m-1}$	is said to be \textit{rotation-invariant}.
 Note that the following property is provided by a balanced realization.
 
 \begin{proposition}
 Rotating the left-blue-aligned Jones-Wenzl projectors $JW_{m-1}$ by one strand yields the left-red-aligned Jones-Wenzl projectors
if and only if $[m-1]=1$ for $m$ even or $[m-1]_x = [m-1]_y = -1$ for $m$ odd.

 \end{proposition}
 
 \begin{proof}
 See \cite{BenTwoColor} for more information.
  \end{proof}
  
 This property of rotational invariant is to ensure the well-definedness of certain Soergel graphs whose definition will be made explicit in the next section.

\section{Diagrammatic Category for Soergel Bimodules} \label{DiaCatSB}

\subsection{Soergel graphs}

%\begin{definition}
A \textit{planar graph in the strip} is a finite graph $(\Gamma, \partial \Gamma)$ embedded in $(\R \times [0,1], \R \times \{0,1\}).$ 
	In other words, all vertices of the graph $\Gamma$ lie in the interior $\R \times (0,1),$ and removing the vertices we have a $1$-manifold with boundary whose intersection with $\R \times \{0,1\}$ is precisely its boundary.
	In particular, this allows edges that connect two vertices, edges that connect a vertex to the boundary, edges that connect two points on the boundary, and edges that form circles (closed 1-manifolds embedded in the plane).
	
	The \textit{boundary} $\R \times \{0,1\}$ consists of two components, the \textit{top boundary} $\R \times \{1\}$ and the \textit{bottom boundary}  $\R \times \{0\}.$
	A \textit{boundary edge} is a local segment of an edge which hits the boundary; there is one such boundary edge for every point on the boundary of the graph.
	A connected component of a graph with boundary might just be referred as \textit{component}.

%\end{definition}

%\begin{definition}
A \textit{Soergel graph} is an isotopy class of a planar graph in the strip, though the isotopy we considered must preserve the top and bottom boundary. 
	The edges in this graph are labelled/coloured by an element $s$ of $S.$
	The edge labels meeting the boundary provide two sequences of colours labelling the top and bottom boundary of the Soergel graph.
	 The vertices in this graph are of three types:
	 \begin{enumerate} [(i)]
	 \item Univalent vertices (dots);
	 \item Trivalent vertices connecting three edges of the same colour;
	 \item $2 m_{st}$-valent vertices connecting edges which alternate in colour between two elements $s,t$ of $S$ with $m_{st} < \infty.$
	 \end{enumerate}

\begin{figure} [H]
\centering
\begin{tikzpicture}[scale=0.7]
\draw[ purple, line width=0.8mm] (0,.7) -- (0,0);
\filldraw[purple] (0,.7) circle (3pt) ;
\node at (0,-.7) {$(i)$};
\draw[ purple, line width=0.8mm] (2,1.15) -- (2,.5);
\draw[ purple, line width=0.8mm] (2,.5) -- (1.5,0);
\draw[ purple, line width=0.8mm] (2,.5) -- (2.5,0);
\node at (2,-.7) {$(ii)$};
\draw[color=blue, line width=0.8mm] (4,1.15) -- (5.3,-.15);
\draw[color=blue, line width=0.8mm] (5.3,1.15) -- (4,-.15);
\draw[color=red, line width=0.8mm] (3.75,0.5) -- (5.5,0.5);
\draw[color=red, line width=0.8mm] (4.65,1.3) -- (4.65,-.3);
\node at (5.3,-.7) {$(iii)$ {\small For $m_{\textcolor{red}{s}\textcolor{blue}{t}} = 3$}};
\end{tikzpicture}
\begin{caption} {The vertices in a Soergel graph.}
\end{caption}
\end{figure}

	 A \textit{region} of the graph is a connected component component of the complement of the Soergel graph in $\R \times [0,1].$
	 We may place a decoration of boxes each labelled with a homogeneous polynomial $f$ in $R$ inside every region of the Soergel graph;  they can be regarded as $0$-valent vertices with labels.
	 A Soergel graph has a \textit{degree}, which accumulates $+1$ for each dots $(i)$, $-1$ for each trivalent vertices $(ii),$ $0$ for each $2m_{st}$-valent vertices $(iii)$, and the degree of each polynomials.

%\end{definition}

\subsection{Jones-Wenzyl projectors as Soergel graphs} Continuing the spirit of \cref{twocolorTL}, we do the following construction.
 Fix two indices $s,t \in S$ and a $2$-coloured crossingless matching.
 We produce a Soergel graph  on the planar disc by first deformation retracting each region into a tree composed out of trivalent and univalent vertices and colouring these trees appropriately.
 With this process, we will always obtain a Soergel graph of degree $+2.$  
 We also specialise $\Z[x,y]$ to $\Bbbk$ under the map sending $x \mapsto a_{s,t}$ and $y \mapsto a_{t,s}.$
  We then call the linear combination of Soergel graphs associated to a Jones-Wenzl projector, the \textit{Jones-Wenzl morphism}.
  As usual  it comes in two colour-flavours.

  \begin{example}

\begin{align*} 
JW_1 &= 
\begin{tikzcd}  
\draw[dashed,color=black!60] (0,0) circle (0.7);
\draw[red, line width=0.8mm] (-.15,0) -- (-.7,0);
\draw[blue, line width=0.8mm] (.15,0) -- (.7,0);
\filldraw[red] (-.25,0) circle (3pt) ;
\filldraw[blue] (.25,0) circle (3pt) ;
\end{tikzcd}
 \\
  JW_2 &= 
\begin{tikzcd}
\draw[dashed,color=black!60] (0,0) circle (0.7);
\draw[blue, line width=0.8mm] (0,.7) -- (0,-.7);
\draw[red, line width=0.8mm] (-.2,0) -- (-.7,0);
\draw[red, line width=0.8mm] (.2,0) -- (.7,0);
\filldraw[red] (-.3,0) circle (3pt) ;
\filldraw[red] (.3,0) circle (3pt) ;
\end{tikzcd} 
- \frac{1}{a_{t,s}}
\begin{tikzcd}
\draw[dashed,color=black!60] (0,0) circle (0.7);
\draw[red, line width=0.8mm] (.7,0) -- (-.7,0);
\draw[blue, line width=0.8mm] (0,-.2) -- (0,-.7);
\draw[blue, line width=0.8mm] (0,.2) -- (0,.7);
\filldraw[blue] (0,-.3) circle (3pt) ;
\filldraw[blue] (0,.3) circle (3pt) ;
\end{tikzcd} 
\\
JW_3 &= 
\begin{tikzcd}
\draw[dashed,color=black!60] (0,0) circle (0.7);
\draw[red, line width=0.8mm] (.125,.68) -- (.125,-.68);
\draw[blue, line width=0.8mm] (-.125,.68) -- (-.125,-.68);
\draw[red, line width=0.8mm] (-.5,0) -- (-.7,0);
\draw[blue, line width=0.8mm] (.45,0) -- (.7,0);
\filldraw[red] (-.425,0) circle (3pt) ;
\filldraw[blue] (.425,0) circle (3pt) ;
\end{tikzcd} 
- \frac{a_{t,s}}{a_{s,t}a_{t,s}-1} 
\begin{tikzcd}
\draw[dashed,color=black!60] (0,0) circle (0.7);
\draw[red, line width=0.8mm] (-.5,0) -- (-.7,0);
\draw[red, line width=0.8mm] (.23,0.4) -- (.23,0.66);
\draw[red, line width=0.8mm] (.23,-0.4) -- (.23,-0.66);
\draw[blue, line width=0.8mm] (0,0) -- (-.45,.58);
\draw[blue, line width=0.8mm] (0,0) -- (-.45,-.58);
\draw[blue, line width=0.8mm] (0,0) -- (.7,0);
\filldraw[red] (-.4,0) circle (3pt) ;
\filldraw[red] (.23,0.3) circle (3pt) ;
\filldraw[red] (.23,-0.3) circle (3pt) ;
\end{tikzcd} 
- \frac{a_{s,t}}{a_{s,t}a_{t,s}-1} 
\begin{tikzcd}
\draw[dashed,color=black!60] (0,0) circle (0.7);
\draw[blue, line width=0.8mm] (.5,0) -- (.7,0);
\draw[blue, line width=0.8mm] (-.23,0.4) -- (-.23,0.66);
\draw[blue, line width=0.8mm] (-.23,-0.4) -- (-.23,-0.66);
\draw[red, line width=0.8mm] (0,0) -- (.45,.58);
\draw[red, line width=0.8mm] (0,0) -- (.45,-.58);
\draw[red, line width=0.8mm] (0,0) -- (-.7,0);
\filldraw[blue] (.4,0) circle (3pt) ;
\filldraw[blue] (-.23,0.3) circle (3pt) ;
\filldraw[blue] (-.23,-0.3) circle (3pt) ;
\end{tikzcd} 
+ \frac{1}{a_{s,t}a_{t,s}-1} 
\begin{tikzcd}
\draw[dashed,color=black!60] (0,0) circle (0.7);
\draw[blue, line width=0.8mm] (.13,0) -- (.7,0);
\draw[blue, line width=0.8mm] (.13,0) -- (-.45,.58);
\draw[red, line width=0.8mm] (-.15,0) -- (.45,-.58);
\draw[red, line width=0.8mm] (-.15,0) -- (-.7,0);
\draw[red, line width=0.8mm] (.23,0.4) -- (.25,0.66);
\draw[blue, line width=0.8mm] (-.23,-0.4) -- (-.25,-0.66);
\filldraw[red] (.23,0.3) circle (3pt) ;
\filldraw[blue] (-.23,-0.3) circle (3pt) ;
\end{tikzcd} + \frac{1}{a_{s,t}a_{t,s}-1}
\begin{tikzcd}
\draw[dashed,color=black!60] (0,0) circle (0.7);
\draw[blue, line width=0.8mm] (.13,0) -- (.7,0);
\draw[red, line width=0.8mm] (.23,-0.4) -- (.23,-0.66);
\draw[blue, line width=0.8mm] (-.23,0.4) -- (-.23,0.66);
\draw[red, line width=0.8mm] (-.13,0) -- (.45,.58);
\draw[red, line width=0.8mm] (-.13,0) -- (-.7,0);
\draw[blue, line width=0.8mm] (.13,0) -- (-.45,-.58);
\filldraw[red] (.23,-0.3) circle (3pt) ;
\filldraw[blue] (-.23,0.3) circle (3pt) ;
\end{tikzcd} \\	
\end{align*}

  \end{example}

\subsection{Soergel bimodules as a diagrammatic category}

Let $\cD_S,$ or simply $\cD$ denote the $\Bbbk$-linear monoidal category defined as follows;
\begin{enumerate} [i.]
\item Objects: finite sequences $\underline{w} = s_{i_1}s_{i_2}\hdots s_{i_k}$, sometimes denoted $B_{\underline{w}},$ with a monoidal structure given by concatenation.
\item  Homomorphisms: the Hom space $Hom_{\cD}(\uw, \uv)$ is the free $\Bbbk$-module generated by Soergel graphs with bottom boundary $\uw$ and top boundary $\uv$ modulo the relations listed below. 
 Hom spaces will be graded by the degree of the Soergel graphs, and all the relations below are homogeneous. 

\begin{enumerate} [(a)]
\item The polynomial relations: 

\begin{itemize}  
\item {the barbell relation} 
\begin{equation}
 \label{barbell} 
 \centering
\begin{tikzpicture}[scale=0.7]
\draw[dashed,color=black!60] (0,0) circle (1.0);
\draw[dashed,color=black!60] (3,0) circle (1.0);
\draw[ purple, line width=0.8mm] (0,.55) -- (0,-.55);
\filldraw[purple] (0,.55) circle (3pt) ;
\filldraw[purple] (0,-.55) circle (3pt) ;
\node at (1.5,0) {=};
\node at (3,0) {{\LARGE $\alpha_{\textcolor{purple}{s}}$}};
\end{tikzpicture}, 
\end{equation}
\item {the polynomial forcing relation}
\begin{equation}   \label{polyforce}
\centering
\begin{tikzpicture} [scale = 0.8]
\draw[dashed,color=black!60] (0,0) circle (1.0);
\draw[dashed,color=black!60] (3,0) circle (1.0);
\draw[dashed,color=black!60] (6,0) circle (1.0);
\draw[color=purple, line width=0.8mm] (0,-1) -- (0,1);
\draw[purple, line width=0.8mm] (3,-1) -- (3,1);
\draw[purple, line width=0.8mm] (6,1) -- (6,.55);
\draw[purple, line width=0.8mm] (6,-1) -- (6,-.55);
\filldraw[purple] (6,.55) circle (3pt) ;
\filldraw[purple] (6,-.55) circle (3pt) ;
\node at (1.5,0) {=};
\node at (4.5,0) {+};
\node at (-.5,0) {{\LARGE f}};
\node at (3.5,0) {{\large $\textcolor{purple}{s} (f)$}};
\node at (6,0) {{\Large $\partial_{\textcolor{purple}{s}} f$}};
\end{tikzpicture}, 
\end{equation}
\end{itemize}

\item The one colour relations:

\begin{itemize}
\item the needle relation
\begin{equation} \label{needle}
\centering
\begin{tikzpicture} [scale=0.7]
\draw[dashed,color=black!60] (0,0) circle (1.0);
\draw[purple, line width=0.8mm] (0,0) circle (.35);
\draw[purple, line width=0.8mm] (0,-.35) -- (0,-1);
\node at (1.5,0) {=};
\node at (2,0) {{\Large $0$}};
\end{tikzpicture}
\end{equation}
\item the Frobenius relations
\begin{equation} \label{Frobenius}
\centering
\begin{tikzpicture} [scale = .7]
\draw[ purple, line width=0.8mm] (0,.25) -- (0,-.25);
\draw[ purple, line width=0.8mm] (0,.25) -- (-.5,.75);
\draw[ purple, line width=0.8mm] (0,.25) -- (.5,.75);
\draw[ purple, line width=0.8mm] (0,-.25) -- (-.5,-.75);
\draw[ purple, line width=0.8mm] (0,-.25) -- (.5,-.75);

\draw[ purple, line width=0.8mm] (2.75,0) -- (3.25,0);
\draw[ purple, line width=0.8mm] (2.75,0) -- (2.35,0.3);
\draw[ purple, line width=0.8mm] (2.35,0.3) -- (2.35,0.75);
\draw[ purple, line width=0.8mm] (3.25,0) -- (3.65,0.3);
\draw[ purple, line width=0.8mm] (3.65,0.3) -- (3.65,0.75);
\draw[ purple, line width=0.8mm] (2.75,0) -- (2.35,-0.3);
\draw[ purple, line width=0.8mm] (2.35,-0.3) -- (2.35,-0.75);
\draw[ purple, line width=0.8mm] (3.25,0) -- (3.65,-0.3);
\draw[ purple, line width=0.8mm] (3.65,-0.3) -- (3.65,-0.75);
\node at (1.5,0) {=};
\node at (8,0) {\text{general associativity}};

\end{tikzpicture}
\end{equation}
\begin{equation} \label{wall}
\centering
\begin{tikzpicture} [scale=.7]
\draw[ purple, line width=0.8mm] (0,.75) -- (0,-.75);
\draw[ purple, line width=0.8mm] (0,0) -- (0.45,0);
\filldraw[purple] (0.45,0) circle (3pt) ;
\node at (1,0) {=};
\draw[ purple, line width=0.8mm] (1.75,.75) -- (1.75,-.75);
\node at (6,0) {\text{general unit}};
\end{tikzpicture}
\end{equation}
\end{itemize}

\item The two colour relations:
\begin{itemize}
\item the two colour associativity
\begin{equation}  \label{assoc3}
\centering
\begin{tikzpicture} [scale = 0.7]
\draw[color=blue, line width=0.8mm] (-1.45,0) -- (0,0);

\draw[color=red, line width=0.8mm] (.55,0) -- (1.25,0.75);
\draw[color=red, line width=0.8mm] (.55,0) -- (1.25,-0.75);

\draw[color=blue, line width=0.8mm] (.75,.75) -- (0,0);

\draw[color=blue, line width=0.8mm] (0,0) -- (.75,-0.75);

\draw[color=red, line width=0.8mm] (0,0) -- (-.75,-0.75);

\draw[color=red, line width=0.8mm] (-.75,.75) -- (0,0);

\draw[color=red, line width=0.8mm] (0,0) -- (.55,0);

\node at (2,0) {=};

\draw[color=red, line width=0.8mm] (3.85,0.35) -- (4.35,0.35);

\draw[color=red, line width=0.8mm] (4.35,0.35) -- (4.7,0.7);

\draw[color=red, line width=0.8mm] (3.85,-0.35) -- (4.35, -0.35);
\draw[color=red, line width=0.8mm] (4.35,-0.35) -- (4.7, -0.7);
\draw[color=red, line width=0.8mm] (3.35,1) -- (3.85,.35);
\draw[color=red, line width=0.8mm] (3.85,.35) -- (3.55,-.02);
\draw[color=red, line width=0.8mm] (3.35,-1) -- (3.85,-.35);
\draw[color=red, line width=0.8mm] (3.85,-.35) -- (3.55,0.01);

\draw[color=blue, line width=0.8mm] (2.5,0) -- (3,0);
\draw[color=blue, line width=0.8mm] (3,0) -- (3.35,0.35);
\draw[color=blue, line width=0.8mm] (3.35,0.35) -- (3.85,0.35);
\draw[color=blue, line width=0.8mm] (3,0) -- (3.35, -0.35);
\draw[color=blue, line width=0.8mm] (3.35,-0.35) -- (3.85, -0.35);
\draw[color=blue, line width=0.8mm] (3.85,.35) -- (4.15,-.02);
\draw[color=blue, line width=0.8mm] (4.35,1) -- (3.85,.35);
\draw[color=blue, line width=0.8mm] (3.85,-.35) -- (4.15,0.01);
\draw[color=blue, line width=0.8mm] (4.35,-1) -- (3.85,-.35);

\node at (7,0) {\text{if $m_{\textcolor{red}{s} \textcolor{blue}{t}} = 3 $ }};

\end{tikzpicture}
\end{equation}
\begin{equation}  \label{assoc4}
\centering
\begin{tikzpicture} [scale = 0.7]
\draw[color=red, line width=0.8mm] (-1.45,0) -- (.55,0);
\draw[color=red, line width=0.8mm] (.55,0) -- (1.25,0.75);
\draw[color=red, line width=0.8mm] (0,0.75) -- (0,-0.75);
\draw[color=red, line width=0.8mm] (.55,0) -- (1.25,-0.75);
\draw[color=blue, line width=0.8mm] (.75,.75) -- (-.75,-0.75);
\draw[color=blue, line width=0.8mm] (-.75,.75) -- (.75,-0.75);
\node at (2,0) {=};

\draw[color=red, line width=0.8mm] (2.5,0) -- (3,0);
\draw[color=red, line width=0.8mm] (3,0) -- (3.35,0.35);
\draw[color=red, line width=0.8mm] (3.35,0.35) -- (4.35,0.35);
\draw[color=red, line width=0.8mm] (4.35,0.35) -- (4.7,0.7);

\draw[color=red, line width=0.8mm] (3,0) -- (3.35, -0.35);
\draw[color=red, line width=0.8mm] (3.35,-0.35) -- (4.35, -0.35);
\draw[color=red, line width=0.8mm] (4.35,-0.35) -- (4.7, -0.7);

\draw[color=red, line width=0.8mm] (3.85,1) -- (3.85,-1);
\draw[color=blue, line width=0.8mm] (3.35,1) -- (4.15,-.02);
\draw[color=blue, line width=0.8mm] (4.35,1) -- (3.55,-.02);
\draw[color=blue, line width=0.8mm] (3.35,-1) -- (4.15,0.01);
\draw[color=blue, line width=0.8mm] (4.35,-1) -- (3.55,0.01);

\node at (7,0) {\text{if $m_{\textcolor{red}{s} \textcolor{blue}{t}} = 4 $ }};

\end{tikzpicture}
\end{equation}

\item the dot-crossing relations

\end{itemize}
\begin{equation} \label{mst2}
\begin{tikzpicture} 
\draw[color=red, line width=0.8mm] (-1.45,0) -- (.55,0);

\draw[color=blue, line width=0.8mm] (0,0.75) -- (0,-0.75);

\filldraw[red]  (0.55,0) circle (3pt) ;

\node at (2,0) {=};

\draw[color=red, line width=0.8mm] (2.5,0) -- (3.5,0);
\filldraw[red]  (3.5,0) circle (3pt) ;

\filldraw[blue]  (4.3,0) circle (3pt) ;
\draw[color=blue, line width=0.8mm] (4.3,0) -- (4.85,0);
\draw[color=blue, line width=0.8mm] (4.85,-.75) -- (4.85,.75);

\draw[color=red, line width=0.8mm] (2.5,0) -- (3,0);

\node at (7,0) {\text{if $m_{\textcolor{red}{s} \textcolor{blue}{t}} = 2 $ }};

\end{tikzpicture}
\end{equation}
\begin{equation} \label{mst3}
\centering
\begin{tikzpicture} 
\draw[color=blue, line width=0.8mm] (-1.45,0) -- (0,0);

\filldraw[red]  (0.55,0) circle (3pt) ;

\draw[color=blue, line width=0.8mm] (.75,.75) -- (0,0);

\draw[color=blue, line width=0.8mm] (0,0) -- (.75,-0.75);

\draw[color=red, line width=0.8mm] (0,0) -- (-.75,-0.75);

\draw[color=red, line width=0.8mm] (-.75,.75) -- (0,0);

\draw[color=red, line width=0.8mm] (0,0) -- (.55,0);

\node at (2,0) {=};

\draw[color=red, line width=0.8mm] (3.825,.45) -- (3.825,.75);
\draw[color=red, line width=0.8mm] (3.825,-.45) -- (3.825,-.75);

\draw[color=blue, line width=0.8mm] (4.65,0) -- (4.85,0);
\draw[color=blue, line width=0.8mm] (4.85,-.75) -- (4.85,.75);

\draw[color=blue, line width=0.8mm] (2.5,0) -- (3,0);
\filldraw[fill=white, draw=black,rounded corners] (3,-.45) rectangle (4.65,.45);

\node at (3.85,0) {$JW_{m_{\textcolor{red}{s} \textcolor{blue}{t}}-1}$}; 

\node at (7,0) {\text{if $m_{\textcolor{red}{s} \textcolor{blue}{t}} = 3 $ }};

\end{tikzpicture}
\end{equation}
\begin{equation} \label{mst4}
\begin{tikzpicture} 
\draw[color=red, line width=0.8mm] (-1.45,0) -- (.55,0);

\draw[color=red, line width=0.8mm] (0,0.75) -- (0,-0.75);

\filldraw[red]  (0.55,0) circle (3pt) ;
\draw[color=blue, line width=0.8mm] (.75,.75) -- (-.75,-0.75);
\draw[color=blue, line width=0.8mm] (-.75,.75) -- (.75,-0.75);
\node at (2,0) {=};

\draw[color=red, line width=0.8mm] (2.5,0) -- (3,0);

\draw[color=blue, line width=0.8mm] (3.55,.45) -- (3.55,.75);
\draw[color=blue, line width=0.8mm] (3.55,-.45) -- (3.55,-.75);

\draw[color=red, line width=0.8mm] (4.1,.45) -- (4.1,.75);
\draw[color=red, line width=0.8mm] (4.1,-.45) -- (4.1,-.75);

\draw[color=blue, line width=0.8mm] (4.6,0) -- (4.85,0);
\draw[color=blue, line width=0.8mm] (4.85,-.75) -- (4.85,.75);

\draw[color=red, line width=0.8mm] (2.5,0) -- (3,0);
\filldraw[fill=white, draw=black,rounded corners] (3,-.45) rectangle (4.65,.45);

\node at (3.85,0) {{ $JW_{m_{\textcolor{red}{s} \textcolor{blue}{t}}-1}$}};

\node at (7,0) {\text{if $m_{\textcolor{red}{s} \textcolor{blue}{t}} = 4 $ }};

\end{tikzpicture}
\end{equation}
\item The three colour relations or  ``Zamolodzhiko'' relations:
\begin{itemize}
\item A sub-Coxeter system of type $A_1 \times I_2(m),$ that is, the Coxeter graph of the parabolic subgroup genetated by $( $\textcolor{blue}{s}$,$\textcolor{red}{t}$,$\textcolor{green}{u}$)$ is 
\begin{tikzcd}
\filldraw[black]  (0,0) circle (2pt) ;
\filldraw[black]  (0.7,0) circle (2pt) ;
\filldraw[black]  (1.4,0) circle (2pt) ;
\draw [line width = 0.2mm] (0,0) -- (0.7,0);
\draw  (0.35,.2) node {{m}};
\draw (0,-.35) node {{\textcolor{blue}{s}}};
\draw  (0.7,-.35) node {{\textcolor{red}{t}}};
\draw  (1.4,-.35) node {{\textcolor{green}{u}}};
\end{tikzcd}
\begin{equation}    \label{Zamo4}
\centering
\begin{tikzpicture} [scale = 0.7]

\draw[color=red, line width=0.8mm] (0,0) -- (.3,0.9);
\draw[color=red, line width=0.8mm] (0,0) -- (-.55,0.75);
\draw[color=red, line width=0.8mm] (0,0) -- (.65,-0.75);
\draw[color=red, line width=0.8mm] (-.2,-.9) -- (0,0);

\draw[color=blue, line width=0.8mm] (0,0) -- (-.55,-0.75);
\draw[color=blue, line width=0.8mm] (0,0) -- (.3,-0.9);
\draw[color=blue, line width=0.8mm] (0,0) -- (-.2,.9);
\draw[color=blue, line width=0.8mm] (.65,.75) -- (0,0);
\draw[color=green!200, line width=0.8mm]  (.9, .3) .. controls (-.1, .7) and (.1,.7) .. (-.8, .3) ;

\node at (2,0) {=};

\draw[color=red, line width=0.8mm] (4,0) -- (4.3,0.9);
\draw[color=red, line width=0.8mm] (4,0) -- (3.45,0.75);
\draw[color=red, line width=0.8mm] (4,0) -- (4.65,-0.75);
\draw[color=red, line width=0.8mm] (3.8,-.9) -- (4,0);

\draw[color=blue, line width=0.8mm] (4,0) -- (3.45,-0.75);
\draw[color=blue, line width=0.8mm] (4,0) -- (4.3,-0.9);
\draw[color=blue, line width=0.8mm] (4,0) -- (3.8,.9);
\draw[color=blue, line width=0.8mm] (4.65,.75) -- (4,0);
\draw[color=green!200, line width=0.8mm]  (4.9, -.3) .. controls (3.9, -.7) and (4.1,-.7) .. (3.2, -.3) ;

\node at (8,0) {\text{if $m_{\textcolor{red}{s} \textcolor{blue}{t}} = 4 $ }};
\node at (-4,0) {$(B_2 \times A_1)$};
\end{tikzpicture}
\end{equation}
\begin{equation}  \label{Zamo3}
\centering
\begin{tikzpicture} [scale = 0.7]

\draw[color=red, line width=0.8mm] (0,0) -- (0,.85);
\draw[color=red, line width=0.8mm] (0,0) -- (-.5,-0.75);
\draw[color=red, line width=0.8mm] (0,0) -- (.5,-0.75);

\draw[color=blue, line width=0.8mm] (0,0) -- (0,-.85);
\draw[color=blue, line width=0.8mm] (0,0) -- (-.5,0.75);
\draw[color=blue, line width=0.8mm] (0,0) -- (.5,0.75);

\draw[color=green!200, line width=0.8mm]  (.9, .3) .. controls (-.1, .7) and (.1,.7) .. (-.8, .3) ;

\node at (2,0) {=};

\draw[color=red, line width=0.8mm] (4,0) -- (4,.85);
\draw[color=red, line width=0.8mm] (4,0) -- (3.5,-0.75);
\draw[color=red, line width=0.8mm] (4,0) -- (4.5,-0.75);

\draw[color=blue, line width=0.8mm] (4,0) -- (4,-.85);
\draw[color=blue, line width=0.8mm] (4,0) -- (3.5,0.75);
\draw[color=blue, line width=0.8mm] (4,0) -- (4.5,0.75);
\draw[color=green!200, line width=0.8mm]  (4.9, -.3) .. controls (3.9, -.7) and (4.1,-.7) .. (3.2, -.3) ;

\node at (8,0) {\text{if $m_{\textcolor{red}{s} \textcolor{blue}{t}} = 3 $ }};
\node at (-4,0) {$(A_2 \times A_1)$};
\end{tikzpicture}
\end{equation}
\begin{equation}  \label{Zamo2}
\centering
\begin{tikzpicture} [scale = 0.7]

\draw[color=red, line width=0.8mm] (-.5,-.75) -- (0.5,.75);
\draw[color=blue, line width=0.8mm] (.5,-.75) -- (-0.5,.75);

\draw[color=green!200, line width=0.8mm]  (.9, .3) .. controls (-.1, .7) and (.1,.7) .. (-.8, .3) ;

\node at (2,0) {=};

\draw[color=red, line width=0.8mm] (3.5,-.75) -- (4.5,.75);
\draw[color=blue, line width=0.8mm] (4.5,-.75) -- (3.5,.75);

\draw[color=green!200, line width=0.8mm]  (4.9, -.3) .. controls (3.9, -.7) and (4.1,-.7) .. (3.2, -.3) ;

\node at (8,0) {\text{if $m_{\textcolor{red}{s} \textcolor{blue}{t}} =  2 $ }};

\node at (-3.5,0) {$(A_1 \times A_1 \times A_1)$};
\end{tikzpicture}
\end{equation}
\begin{equation} \label{A3}
\centering
\includegraphics[scale=1]{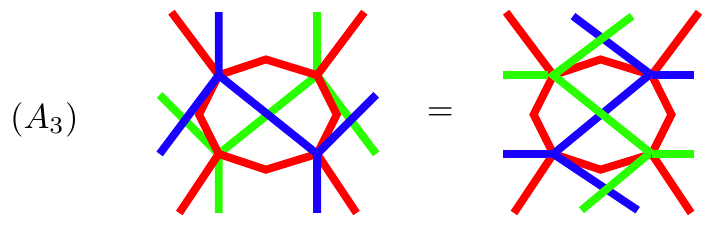}
\end{equation} 
\begin{equation} \label{B3}
\centering
\includegraphics[scale=1]{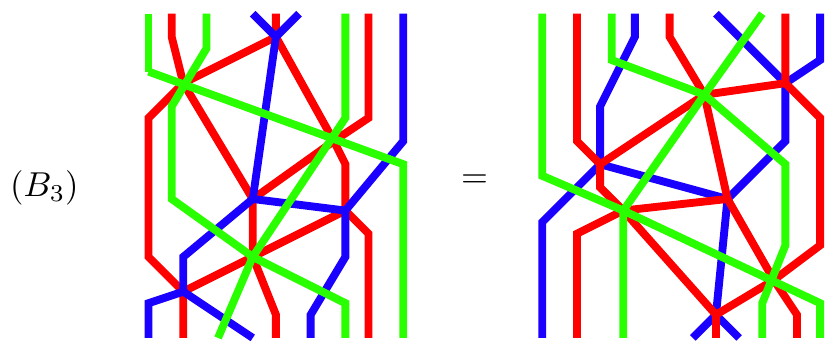}
\end{equation}
\end{itemize}
\end{enumerate} 
 \end{enumerate} 
 
This concludes the definition of $\cD.$ 
 
 Note that the one colour relations from \cref{Frobenius} to \cref{wall} encode the data of $B_s$ being a Frobenius object.
 For practical reason, we will provide those one colour relations required for a Frobenius object.
 These relations are expressed in term of cups and caps defined as follows:
 \begin{center}
 \begin{tikzcd}
 \draw[purple, line width=0.8mm] (-.6,-0.5) arc(180 : 0: 0.3);
 \draw[purple, line width=0.8mm] (-.6,-0.5) -- (-.6, -.7);
  \draw[purple, line width=0.8mm] (0,-0.5) -- (0, -.7);

\node at (.5, -0.5) {:=};
\draw[ purple, line width=0.8mm] (1.5,0) -- (1.5,-.45);
\draw[ purple, line width=0.8mm] (1.5,-.4) .. controls (1.3,-.5) and (1.25, -.55)..  (1.2,-0.8);
\draw[ purple, line width=0.8mm] (1.5,-.4) .. controls (1.7,-.5) and (1.75, -.55)..  (1.8,-0.8);
\filldraw[purple] (1.5,0) circle (3pt) ;

 \node at (-3,-.5) {\text{The definition of cap}};
  \end{tikzcd} 
  \end{center}
  \begin{center}
   \begin{tikzcd}
 \draw[purple, line width=0.8mm] (-.6,0.5) arc(-180 : -0: 0.3);
 \draw[purple, line width=0.8mm] (-.6,0.5) -- (-.6, .8);
  \draw[purple, line width=0.8mm] (0,0.5) -- (0, .8);

\node at (.5, 0.5) {:=};
\draw[ purple, line width=0.8mm] (1.5,0) -- (1.5,.45);
\draw[ purple, line width=0.8mm] (1.5,.4) .. controls (1.3,.5) and (1.25, .55)..  (1.2,0.8);
\draw[ purple, line width=0.8mm] (1.5,.4) .. controls (1.7,.5) and (1.75, .55)..  (1.8,0.8);
\filldraw[purple] (1.5,0) circle (3pt) ;

 \node at (-3,.5) {\text{The definition of cup}};
 
  \end{tikzcd}
  \end{center}
 \noindent The trivalent vertices give multiplication or comultiplication while the dots provide unit or counit.
   To obtain \cref{Frobenius}, we can rotate \cref{assocmult} (resp. \cref{coassoccomult}) using a cup (resp. a cap) and apply \cref{comult rot mult} (resp. \cref{mult rot comult}).
   On the other hand, \cref{mult rot comult} to \cref{unit rot counit} records the cyclicity of all morphisms. 
   Therefore, \cref{wall} can be replaced by \cref{enddot counit comult} and \cref{startdot unit multi}.
Here is the list of them:
 \begin{equation} \label{assocmult} \begin{tikzcd}
\draw[ purple, line width=0.8mm] (1.5,1.2) -- (1.5,.65);
\draw[ purple, line width=0.8mm] (1.5,.7) .. controls (1.2,.6) and (1, .3)..  (.9,-.4);
\draw[ purple, line width=0.8mm] (1.5,.7) .. controls (1.7,.55) and (1.75, .45)..  (1.8,.2);

\draw[ purple, line width=0.8mm] (1.8,.2) .. controls (1.45,0) and (1.5, -.55)..  (1.45,-.4);
\draw[ purple, line width=0.8mm] (1.8,.2) .. controls (2.15,0) and (2.1, -.55)..  (2.15,-.4);

\node at (2.55, 0.55) {=};

\draw[ purple, line width=0.8mm] (3.6,1.2) -- (3.6,.65);
\draw[ purple, line width=0.8mm] (3.6,.7) .. controls (3.9,.6) and (4.1, .3)..  (4.2,-.4);
\draw[ purple, line width=0.8mm] (3.6,.7) .. controls (3.4,.55) and (3.35, .45)..  (3.3,.2);

\draw[ purple, line width=0.8mm] (3.3,.2) .. controls (3.65,0) and (3.6, -.55)..  (3.65,-.4);
\draw[ purple, line width=0.8mm] (3.3,.2) .. controls (2.95,0) and (2.9, -.55)..  (2.95,-.4);

 \node at (8,.55) {\text{ the associativity of the multiplication}.};
 
  \end{tikzcd}
\end{equation}

 \begin{equation}  \label{coassoccomult}
 \begin{tikzcd}
\draw[ purple, line width=0.8mm] (1.5,-0.1) -- (1.5,.45);
\draw[ purple, line width=0.8mm] (1.5,.4) .. controls (1.2,.5) and (1, .8)..  (.9,1.5);
\draw[ purple, line width=0.8mm] (1.5,.4) .. controls (1.7,.55) and (1.75, .65)..  (1.8,.9);

\draw[ purple, line width=0.8mm] (1.8,.9) .. controls (1.45,1.1) and (1.5, 1.45)..  (1.45,1.5);
\draw[ purple, line width=0.8mm] (1.8,.9) .. controls (2.15,1.1) and (2.1, 1.45)..  (2.15,1.5);

\node at (2.55, 0.55) {=};

\draw[ purple, line width=0.8mm] (3.6,-0.1) -- (3.6,.45);
\draw[ purple, line width=0.8mm] (3.6,.4) .. controls (3.9,.5) and (4.1, .8)..  (4.2,1.5);
\draw[ purple, line width=0.8mm] (3.6,.4) .. controls (3.4,.55) and (3.35, .65)..  (3.3,.9);

\draw[ purple, line width=0.8mm] (3.3,.9) .. controls (3.65,1.1) and (3.6, 1.45)..  (3.65,1.5);
\draw[ purple, line width=0.8mm] (3.3,.9) .. controls (2.95,1.1) and (2.9, 1.45)..  (2.95,1.5);

 \node at (8,.55) {\text{ the coassociativity of the comultiplication.}};
 
  \end{tikzcd}
\end{equation} 

 \begin{equation} \label{enddot counit comult}\begin{tikzcd}
\draw[ purple, line width=0.8mm] (0,0) -- (0,.65);
\draw[ purple, line width=0.8mm] (0,.6) .. controls (-.2,.7) and (-.25, .75)..  (-.3,1.1);
\draw[ purple, line width=0.8mm] (0,.6) .. controls (.2,.7) and (.25, .75)..  (.25,.75);

\node at (.9, 0.4) {=};
\draw[ purple, line width=0.8mm] (2.8,0) -- (2.8,.65);
\draw[ purple, line width=0.8mm] (2.8,.6) .. controls (2.6,.7) and (2.55, .75)..  (2.55,0.75);
\draw[ purple, line width=0.8mm] (2.8,.6) .. controls (3.0,.7) and (3.05, .75)..  (3.1,1.1);

\draw[ purple, line width=0.8mm] (1.4, 0) -- (1.4, 1);

\filldraw[purple] (.25,.75) circle (3pt) ;

\filldraw[purple] (2.55,0.75) circle (3pt) ;

\node at (2, 0.4) {=};

 \node at (7.5,.5) {\text{the enddot as counit for the comultiplication.}};
 
  \end{tikzcd}
\end{equation} 

\begin{equation} \label{startdot unit multi}    
\begin{tikzcd}
\draw[ purple, line width=0.8mm] (0,1) -- (0,.35);
\draw[ purple, line width=0.8mm] (0,.4) .. controls (-.2,.3) and (-.25, .25)..  (-.3,-0.1);
\draw[ purple, line width=0.8mm] (0,.4) .. controls (.2,.3) and (.25, .25)..  (.25,.25);

\node at (.9, 0.4) {=};
\draw[ purple, line width=0.8mm] (2.8,1) -- (2.8,.35);
\draw[ purple, line width=0.8mm] (2.8,.4) .. controls (2.6,.3) and (2.55, .25)..  (2.55,0.25);
\draw[ purple, line width=0.8mm] (2.8,.4) .. controls (3.0,.3) and (3.05, .25)..  (3.1,-0.1);

\draw[ purple, line width=0.8mm] (1.4, 0) -- (1.4, 1);

\filldraw[purple] (.25,.25) circle (3pt) ;

\filldraw[purple] (2.55,0.25) circle (3pt) ;

\node at (2, 0.4) {=};

 \node at (7.5,.5) {\text{the startdot as unit for the multiplication.}};
 
  \end{tikzcd}
\end{equation} 
 \begin{equation} \label{selfadjoint}    \begin{tikzcd}
\draw[purple, line width=0.8mm] (-.3,0.5) arc(180 : 0: 0.25);
\draw[ purple, line width=0.8mm] (.2,.5) .. controls (0.1,-.1) and (-.2, 0)..  (-.3,-.5);
\draw[purple, line width=0.8mm] (-.3,0.5) arc(0 : -180: 0.25);
\draw[ purple, line width=0.8mm] (-.8,.5) .. controls (-.7,1.1) and (-.4, 1)..  (-.3,1.5);

\node at (.75, 0.4) {=};
\draw[ purple, line width=0.8mm] (1.5,1.5) -- (1.5,-.5);

\node at (2.1, 0.4) {=};

\draw[purple, line width=0.8mm] (3.2,0.5) arc(-180 : 0: 0.25);
\draw[ purple, line width=0.8mm] (3.7,.5) .. controls (3.6,1.1) and (3.3, 1)..  (3.2,1.5);
\draw[purple, line width=0.8mm] (3.2,0.5) arc(0 : 180: 0.25);
\draw[ purple, line width=0.8mm] (2.7,.5) .. controls (2.8,-.1) and (3.1, 0)..  (3.2,-.5);

 \node at (7,.5) {\text{the self-biadjointness of $B_s$.}};
 
  \end{tikzcd}
\end{equation} 
 \begin{equation} \label{mult rot comult} 
   \begin{tikzcd}
\draw[ purple, line width=0.8mm] (-.3,0) -- (-.3,0.45);
\draw[ purple, line width=0.8mm] (-.3,.4) .. controls (-.5,.5) and (-.55, .55)..  (-.6,0.8);
\draw[ purple, line width=0.8mm] (-0.3,.4) .. controls (-.1,.7) and (0.17, .7)..  (0.2,0);

\node at (.6, 0.4) {=};
\draw[ purple, line width=0.8mm] (1.5,0.8) -- (1.5,.35);
\draw[ purple, line width=0.8mm] (1.5,.4) .. controls (1.3,.3) and (1.25, .25)..  (1.2,0);
\draw[ purple, line width=0.8mm] (1.5,.4) .. controls (1.7,.3) and (1.75, .25)..  (1.8,0);

\node at (2.3, 0.4) {=};

\draw[ purple, line width=0.8mm] (3.2,0) -- (3.2,0.45);
\draw[ purple, line width=0.8mm] (3.2,.4) .. controls (3.4,.5) and (3.45, .55)..  (3.5,0.8);
\draw[ purple, line width=0.8mm] (3.2,.4) .. controls (3,.7) and (2.8, .7)..  (2.7,0);

 \node at (8,.5) {\text{the mult. is a  ``half rotation'' of the comult.}};
 
  \end{tikzcd}
\end{equation} 
\begin{equation}  \label{comult rot mult}
   \begin{tikzcd}
\draw[ purple, line width=0.8mm] (-.3,0.8) -- (-.3,0.35);
\draw[ purple, line width=0.8mm] (-.3,.4) .. controls (-.5,.3) and (-.55, .25)..  (-.6,0);
\draw[ purple, line width=0.8mm] (-0.3,.4) .. controls (-.1,.1) and (0.17, .1)..  (0.2,0.8);

\node at (.6, 0.4) {=};
\draw[ purple, line width=0.8mm] (1.5,0) -- (1.5,.45);
\draw[ purple, line width=0.8mm] (1.5,.4) .. controls (1.3,.5) and (1.25, .55)..  (1.2,0.8);
\draw[ purple, line width=0.8mm] (1.5,.4) .. controls (1.7,.5) and (1.75, .55)..  (1.8,0.8);

\node at (2.3, 0.4) {=};

\draw[ purple, line width=0.8mm] (3.2,0.8) -- (3.2,0.35);
\draw[ purple, line width=0.8mm] (3.2,.4) .. controls (3.4,.3) and (3.45, .25)..  (3.5,0);
\draw[ purple, line width=0.8mm] (3.2,.4) .. controls (3,.1) and (2.8, .1)..  (2.7,0.8);

 \node at (8,.5) {\text{the comult. is a half ``rotation'' of the mult.}};
 
  \end{tikzcd}
\end{equation} 
 \begin{equation} \label{counit rot unit}
\begin{tikzcd}
 \draw[purple, line width=0.8mm] (-.6,0.5) arc(180 : 0: 0.3);
  \draw[purple, line width=0.8mm] (0,0.5) -- (0, 0);
\filldraw[purple] (-.6,0.45) circle (3pt) ;

 \draw[purple, line width=0.8mm] (2,0.5) arc(180 : 0: 0.3);
 \draw[purple, line width=0.8mm] (2,0.5) -- (2, 0);
\filldraw[purple] (2.6,0.45) circle (3pt) ;

\node at (.5, 0.45) {=};
\draw[ purple, line width=0.8mm] (1,0) -- (1,.9);
\filldraw[purple] (1,0.9) circle (3pt) ;

\node at (1.5, 0.5) {=};

 \node at (7,.5) {\text{the counit is a ``rotation'' of the unit}};
 \end{tikzcd}
\end{equation}  
\begin{equation} \label{unit rot counit}
\begin{tikzcd}
 \draw[purple, line width=0.8mm] (-.6,0.5) arc(-180 : -0: 0.3);
  \draw[purple, line width=0.8mm] (0,0.5) -- (0, 1);
  \draw[purple, line width=0.8mm] (-.6,0.5) -- (-.6, .6);
\filldraw[purple] (-.6,0.55) circle (3pt) ;
 \draw[purple, line width=0.8mm] (2,0.5) arc(-180 : -0: 0.3);
 \draw[purple, line width=0.8mm] (2,0.5) -- (2, 1);
\filldraw[purple] (2.6,0.55) circle (3pt) ;
\node at (.5, 0.5) {=};
\draw[ purple, line width=0.8mm] (1,0.2) -- (1,1);
\filldraw[purple] (1,0.2) circle (3pt) ;
\node at (1.5, 0.5) {=};
 \node at (7,.5) {\text{the unit is a ``rotation'' of the counit}};
 \end{tikzcd}
\end{equation} 

\noindent This marks the end of the supplementary one colour relations.

Consider a monoidal functor $\cF: \cD \ra \B \bS$-bimod defined by sending the sequences $\uw$ in $\cD$ to the Bott-Samelson bimodules $B_{\uw}$ in $\B \bS$-bimod on generators. 
On morphism, it is given by the table belows:

\begin{table} [H]
\centering
\begin{tabular}{c c c c l}
\begin{tikzcd}
\draw[dashed,color=black!60] (0,0) circle (0.5);
\draw[ purple, line width=0.8mm] (0,0) -- (0,-.5);
\filldraw[purple] (0,0) circle (3pt) ;
\end{tikzcd} 
& deg 1
& $\longmapsto$ 
& $B_{\textcolor{purple}{s}} \ra R$ 
& $f \otimes g \mapsto fg$
\\
\\
\begin{tikzcd}
\draw[dashed,color=black!60] (0,0) circle (0.5);
\draw[ purple, line width=0.8mm] (0,0) -- (0,.5);
\filldraw[purple] (0,0) circle (3pt) ;
\end{tikzcd}
& deg 1
& $\longmapsto$
& $ R \ra B_{\textcolor{purple}{s}} $
& $1 \mapsto \frac{1}{2} (\alpha_s \otimes 1 + 1 \otimes \alpha_s)$
\\
\\
\begin{tikzcd}
\draw[dashed,color=black!60] (0,0) circle (0.5);
\draw[ purple, line width=0.8mm] (0,0) -- (0,-.5);
\draw[ purple, line width=0.8mm] (0,0) -- (-.35,.35);
\draw[ purple, line width=0.8mm] (0,0) -- (.35,.35);
\end{tikzcd}
& deg -1
&$\longmapsto$
&  $ B_{\textcolor{purple}{s}} \ra B_{\textcolor{purple}{s}}B_{\textcolor{purple}{s}} $
& $1 \otimes 1 \mapsto 1 \otimes 1 \otimes 1$
\\
\\
\begin{tikzcd}
\draw[dashed,color=black!60] (0,0) circle (0.5);
\draw[ purple, line width=0.8mm] (0,0) -- (0,.5);
\draw[ purple, line width=0.8mm] (0,0) -- (-.35,-.35);
\draw[ purple, line width=0.8mm] (0,0) -- (.35,-.35);
\end{tikzcd}
& deg -1
&$\longmapsto$
&$ B_{\textcolor{purple}{s}}B_{\textcolor{purple}{s}} \ra B_{\textcolor{purple}{s}} $
& $1 \otimes g \otimes 1 \mapsto \partial_s g \otimes 1$
\\
\\
\begin{tikzcd}
\draw[dashed,color=black!60] (0,0) circle (0.5);
\draw (0,-.1) node {{\Huge f}} ;
\end{tikzcd}
& deg $f$
&$\longmapsto$
&$R \ra R$
& $1 \mapsto f$ 
\\
\\
\begin{tikzcd}
\draw[dashed,color=black!60] (0,0) circle (0.5);

\draw[ blue] (-.2,.45) -- (.2,-.45);
\draw[ blue] (.2,.45) -- (-.2,-.45);
\draw[ blue] (.45,.2) -- (-.45,-.2);
\draw[ blue] (-.45,.2) -- (.45,-.2);

\draw[ red ] (0,0) -- (0,0.5);
\draw[ red] (0,0) -- (0.5,0);
\draw[ red] (0,0) -- (0,-.5);
\draw[ red] (0,0) -- (-0.5,0);
\draw[ red] (.35,.35) -- (-.35,-.35);
\draw[ red] (-.35,.35) -- (.35,-.35);

\end{tikzcd}
& deg 0
&$\longmapsto$
& $ \underbrace{B_{\textcolor{red}{s}}B_{\textcolor{blue}{t}} \cdots}_{m_{\textcolor{red}{s}\textcolor{blue}{t}}} \xra{i \circ p (B_{\textcolor{red}{s} \textcolor{blue}{t}})} \underbrace{B_{\textcolor{blue}{t}}B_{\textcolor{red}{s}}  \cdots}_{m_{\textcolor{red}{s}\textcolor{blue}{t}}} $
& 

\end{tabular}
\caption{The correspondence between generating morphisms in $\cD$ and $\B \bS$-bimod. } \label{generating morphism}
\end{table}

Note that the bimodule homomorphism for $2 m_{st}$-valent vertex is not presented because its formula is difficult and unenlightening to write down in general.
  Suppose $B_{s,t}$ is the indecomposable Soergel bimodule indexed by the longest element of the finite rank two parabolic subgroup generated by $s$ and $t.$ 
  Then, $B_{s,t}$ appears as a direct summand of both $B_{\textcolor{red}{s}}B_{\textcolor{blue}{t}} \cdots$ and $B_{\textcolor{blue}{t}}B_{\textcolor{red}{s}} \cdots$ with multiplicity one.
   In principal, the $2 m_{st}$-valent vertex is just the projection and inclusion of this indecomposable summand $B_{s,t}.$

By construction, this functor is essentialy surjective, so it suffices to show that it is fully-faithful. 
	By a Light Leaves Basis theorem from Libedinsky \cite{Lightleaves}, the collection of the morphisms in the  right hand side of \cref{generating morphism} forms a basis of Bott-Samelson bimodule homomorphisms.
	As a result, $\cF$ is full.
	It can be argued that the graded dimensions of the homomorphism spaces in $\cD$ and $\B \bS$-bimod coincide, and therefore $\cF$ induces an isomorphism on Hom spaces as a consequence of being a surjection betweeen graded vector spaces of the same finite dimension in each graded component.
	Finally, $\cF$ is fully-faithful, as desired. 
	Since idempotent completion preserves equivalences, we get the following theorem:

\begin{theorem}
Consider the Karoubi envelope $\cK ar(\overline{\cD})$ of the graded additive closure $\overline{\cD}$ of $\cD.$
  Then, the two idempotent completions  $\cK ar(\overline{\cD})$ and $\bS$-bimod are equivalent as monoidal categories.
\end{theorem}

\begin{proof} \cite[Theorem 6.28]{BenGeo}
\end{proof}

\section{The 2-Braid Group} \label{2BraidGrp}

%\begin{definition}
Let $(M,e)$ be a monoid with unit $e$ and $\cC$ be a category. A \textit{weak action} of $M$ on $\cC$ consists of a family of functors $\{F_f: \cC \ra \cC\ \mid f \in M \}$ and a natural isomorphism $c_e: F_e \xra{\cong} id $ from $F_e$ to the identity functor such that $F_{gh} \cong F_g \circ F_h$ for all $g,h \in M.$

An \textit{action} of $M$ on $\cC$ is a weak action of $M$ on $\cC$ together with natural isomorphisms of functors $c_{f,g}: F_f F_g \xra{\cong} F_{fg}$ for all $f,g \in M$ such that $c_{f,e}$ and $c_{e,f}$ are induced by $c_e$ and such that the following diagram commutes:
$$
\begin{tikzcd}
F_{ghk}  \arrow[r, "\cong"] \arrow[d, "\cong"] & F_{gh} \circ F_k  \arrow[d, "\cong"] \\
F_g \circ F_{hk}  \arrow[r, "\cong"] & F_g \circ F_h \circ F_k
\end{tikzcd}
$$

An action of $M$ on $\cC$ is called \textit{faithful}, if the functors $F_g$ and $F_h$ are not isomorphic for $g \neq h \in G.$
An action of a group $G$ on $\cC$ is an action of the underlying unital monoid on $\cC.$

%\end{definition}

We define the elementary Rouquier complexes corresponding to a simple reflection $s \in S$ as in \cite{Rouq}:
\begin{align*}
F_s &:= 0 \ra B_s \xra{\begin{tikzcd}
\draw[ purple, line width=0.7mm] (0,0) -- (0,-.35);
\filldraw[purple] (0,0) circle (3pt) ;
\end{tikzcd} } R(1) \ra 0 \\
E_s = F_{s^{-1}} &:= 0 \ra R(-1) \xra{\begin{tikzcd}
\draw[ purple, line width=0.7mm] (0,0) -- (0,-.35);
\filldraw[purple] (0,-.35) circle (3pt) ;
\end{tikzcd} } B_s \ra 0
\end{align*}
with both $B_s$ in cohomological degree $0.$

Rouquier proves the following result:

\begin{proposition}
Suppose $s \neq t \in S$ with $m_{st} \leq \infty.$
Then, we obtain the following morphisms:
\begin{align*}
F_s \otimes E_s \cong &R \cong E_s \otimes F_s; \\
\underbrace{F_s \otimes F_t \otimes F_s \otimes  \cdots}_{m_{st} \text{terms}} &\cong \underbrace{F_t \otimes F_s \otimes F_t \otimes \cdots}_{m_{st} \text{terms}}
\end{align*}
in the bounded homotopy category $K^b(\bS$-bimod) of Soergel bimodules. 
Moreover, the map $s \mapsto F_s$ extends to a morphism from $\cA(W)$ to the group of isomorphism classes of invertible objects of $K^b(\bS$-bimod). 
 The composition with the group homomorphism into the group of isomophism classes of autoequivalences on $K^b(\bS$-bimod) defines a weak action of $B_{(W,S)}$ on $K^b(\bS$-bimod).
\end{proposition}

Moreover, the weak action of $B_{(W,S)}$ can be upgraded to an action on $K^b({\bS}$-bimod) as explained in  \cite[Thm 9.5]{Rouq}.

   This implies that we have a monoidal functor from the strict monoidal category with set of objects $B_{(W,S)},$ with only arrows the identity mapds and with tensor product given by multiplication to the strict monoidal category of endofunctors of $K^b({\bS}$-bimod).
We then define the \textit{2-braid group}, denoted by $2$-$\cB r$ as the full monoidal subcategory of  $K^b({\bS}$-bimod) generated by $E_s$ and $F_s$ for $s \in S.$
     Observe that the set of isomorphism classes of objects in $2$-$\cB r,$ denoted by Pic$(2$-$\cB r)$, forms a group called the \textit{Picard group} of the monoidal category $2$-$\cB r$   under tensor product composition.
     Finally, Rouquier conjectures the following:

  \begin{conjecture}
  Pic$(2$-$\cB r)$ is isomorphic to $B_{(W,S)}.$
  \end{conjecture}
  
  Yensen proved this conjecture for type A Coxeter group in his master's thesis following the work of \cite{KhoSei} and for all arbitrary finite type Coxeter groups in \cite{jensen_2016} inspired by the work of \cite{BravTho}.

     \section{Categorified Action of Type $B$} \label{CatTypeB}
    
For completeness of this chapter, we recall some essential construction and result from chapter 1.

\subsection{Type $B_n$ zigzag algebra $\Ba_n$} 
Consider the following quiver $Q_n$:
\begin{figure}[H]
\centering
\begin{tikzcd}[column sep = 1.5cm]
1				\arrow[r,bend left,"1|2"]  														 &		 
2 			\arrow[l,bend left,"2|1"] \arrow[r,bend left, "2|3"] 
				\arrow[color=blue,out=70,in=110,loop,swap,"ie_2"] &
3 			\arrow[l,bend left,"3|2"] \arrow[r,bend left, "3|4"] 
				\arrow[color=blue,out=70,in=110,loop,swap,"ie_3"] &
\cdots \arrow[l,bend left,"4|3"] \arrow[r,bend left, "n-1|n"] &
n 		.	\arrow[l,bend left,"n|n-1"]
				\arrow[color=blue,out=70,in=110,loop,swap,"ie_n"]
\end{tikzcd}
\caption{{\small The quiver $Q_n$.}}
\label{B quiver}
\end{figure}
%We will use a slightly different path length grading.
%All (oriented) edges $\xi$ in black of \cref{B quiver} have length 1, whereas edges {\color{blue}$\xi'$} in blue have length 0.
%and each $e_j$ also has length 0.
%The length of an arbitrary path $\xi_1...\xi_m$ is then the sum of the length of each $\xi_k$.
%The algebra $\R Q_n$ is then a graded algebra under this path length grading.
%
Take its path algebra $\R Q_n$ over $\R$ and consider the path gradings on $\R Q_n$ denoted by $(-)$ where the "imaginary" path $ie_j$ has grading $0.$
 For the purpose of this paper, $(1)$ is a grading shift down by $1.$

\begin{comment}
	\begin{enumerate} [(i)]
	\item the first grading is defined following the convention in  \cite{KhoSei} where we set \begin{itemize}
	\item 	the degree of $(j+1|j)$ to be $1$ for all $j$, and
	\item   the degree of $e_j$ and of $(j|j+1)$ to be $0$ for all $j$.
\end{itemize}	 
	\item the second grading is a $\Z/2\Z$-grading defined by setting 
	\begin{itemize}
	\item  the degree of $ie_j$ (all blue paths in \cref{B quiver}) as 1 for all $j$, and
	\item the degree of all other paths in \cref{B quiver} and the constant paths as zero. 
	\end{itemize}
	\end{enumerate}
\noindent We denote  a shift in the first grading by  $\{-\}$ and a shift in the second grading by $\< - \>.$ 
\end{comment}

We are now ready to define the zigzag algebra of type $B_n$:
\begin{definition}
The zigzag path algebra of $B_n$, denoted by $\Ba_n$, is the quotient algebra of the path algebra $\R Q_n$ modulo the usual zigzag relations given by
\begin{align}
(j|j-1)(j-1|j) &= (j|j+1)(j+1|j) \qquad (=: X_j);\\
(j-1|j)(j|j+1) = & 0 = (j+1|j)(j|j-1);
\end{align}

\noindent for $2\leq j \leq n-1$, in addition to the relations
\begin{align}
(ie_j)(ie_j) &= -e_j, \qquad \text{for } j \geq 2 \label{imaginary};\\
(ie_{j-1})(j-1|j) &= (j-1|j)(ie_j), \qquad \text{for } j\geq 3; \label{complex symmetry 1}\\
(ie_{j})(j|j-1) &= (j|j-1)(ie_{j-1}), \qquad \text{for } j\geq 3; \label{complex symmetry 2}\\
(1|2)(ie_2)(2|1) &= 0, \\
(ie_2) X_2 &= X_2 (ie_2).
\end{align}
\end{definition}
Since the relations are all homogeneous with respect to the given gradings, $\Ba_n$ is also a bigraded algebra.
As a $\R$-vector space, $\Ba_n$ has dimension $8  n-6$, with basis $\{ e_1 , \ldots, e_{n}, ie_2 , \ldots, ie_{n},  
\newline (1|2), \ldots, (n-1 | n), (2|1), \ldots (n | n-1), (ie_2)(2|1), (1|2)(ie_2), (ie_2)(2|3), \ldots, (ie_{n-1})(n-1|n), (3|2)(ie_2),
\newline \ldots, (n|n-1)(ie_{n-1}), (1|2|1), \ldots, (2n-1|2n-2|2n-1), (ie_2)(2|1|2), \ldots, (ie_n)(n|n-1|n) \}$.

The indecomposable (left) projective $\Ba_n$-modules are given by $P^B_j := \Ba_n e_j$.
For $j=1$, $P^B_j$ is naturally a $(\Ba_n, \R)$-bimodule; there is a natural left $\Ba_n$-action given by multiplication of the algebra and the right $\R$-action induced by the natural left $\R$-action.
But for $j\geq 2$, we shall endow $P^B_j$ with a right $\C$-action.
Note that \cref{imaginary} is the same relation satisfied by the complex imaginary number $i$.
We define a right $\C$-action on $P^B_j$ by $p * (a+ib) = ap + bp(ie_j)$ for $p \in P^B_j, a+ib\in \C$.
Further note that this right action restricted to $\R$ agrees with the natural right (and left) $\R$-action.
This makes $P^B_j$ into a $(\Ba_n,\C)$-bimodule for $j\geq 2$.
Dually, we shall define ${}_jP^B := e_j\Ba_n$, where we similarly consider it as a ($\R,\Ba_n)$-bimodule for $j=1$ and as a $(\C,\Ba_n)$-bimodule for $j\geq 2$.

It is easy to check that we have the following isomorphisms of $\Z$-graded bimodules:
\begin{proposition}
\[
  {}_jP^B_k := {}_jP^B\otimes_{\Ba_n}P^B_k \cong 
  \begin{cases}
  		\vspace{1mm}
  		\ _{\C}\C_{\C} & \text{as } (\C,\C)\text{-$g_r$bimod, for } j,k \in \{2,\hdots, n\} \text{ and }  k-j=1;\\
  		\vspace{1mm}
  		\ _{\C}\C_{\C}(-1) & \text{as } (\C,\C)\text{-$g_r$bimod, for } j,k \in \{2,\hdots, n\} \text{ and }  j-k=1;\\
  		\vspace{1mm}
  		\ _{\C}\C_{\C} \oplus \ _{\C}\C_{\C}(-2) & \text{as } (\C,\C)\text{-$g_r$bimod, for } j=k=2,3,\hdots,n; \\
  		\vspace{1mm}
        \ _{\R}\C_{\C} & \text{as } (\R,\C)\text{-$g_r$bimod, for } j=1 \text{ and } k=2; \\
        \vspace{1mm}
        \ _{\C}\C_{\R}(-1) & \text{as } (\C,\R)\text{-$g_r$bimod, for } j=2 \text{ and } k=1; \\ 
        \vspace{1mm}
        \ _{\R}\R_{\R} \oplus \ _{\R}\R_{\R}(-2) & \text{as } (\R,\R)\text{-$g_r$bimod, for } j=k=1. \\    
  \end{cases}
\]
\label{bimodule isomorphism}
\end{proposition}
%\begin{remark}
Note that all the graded bimodules in \cref{bimodule isomorphism} can be restricted to a $(\R, \R)$-bimodule.
By identifying ${}_\R \C_\R \cong \R\oplus \R$ as $\Z/2\Z$-graded $(\R, \R)$-bimodules, the bimodules in \cref{bimodule isomorphism} restricted to just the $\R$ actions are also isomorphic as bigraded ($\Z$ and $\Z/2\Z$) $(\R, \R)$-bimodule.
For example, ${}_1 P_2^B$ as a $(\R, \R)$-bimodule is generated by $(1|2)$ and $(1|2)i$, so it is isomorphic to $\R \oplus \R \cong {}_\R \C_\R$.
%\end{remark}

\begin{lemma} \label{bimodulemaps}
Denote $k_j := \R$ when $j = 1$ and $k_j := \C$ when $j \geq 2$.
%\begin{cases} 
%\R, & \text{ for } j = 1; \\
%\C, & \text{ for }j \geq 2.
%\end{cases}
The maps 

$$\beta_j: P^B_j \otimes_{k_j} {}_jP^B  \to \Ba_n \text{ and } \gamma_j: \Ba_n  \to  P^B_j \otimes_{k_j} {}_jP^B  (2)$$ 

\noindent defined by:
\begin{align*} 
\beta_j(x\otimes y) &:= xy, \\
\gamma_j(1) &:= 
\begin{cases}
X_j \otimes e_j + e_j \otimes X_j + (j+1|j) \otimes (j|j+1) \\
 \hspace{8mm} + (-ie_{j+1})(j+1|j) \otimes (j|j+1)(ie_{j+1}), &\text{for } j=1;\\
X_j \otimes e_j + e_j \otimes X_j + (j-1|j) \otimes (j|j-1) + (j+1|j) \otimes (j|j+1), &\text{for } 1<j < n; \\
X_j \otimes e_j + e_j \otimes X_j + (j-1|j) \otimes (j|j-1), &\text{for } j = n,
\end{cases}
\end{align*}
are $(\Ba_n,\Ba_n)$-bimodule maps.
\end{lemma}

\begin{proof}
It is obvious that $\beta_j$ are for all $j$. 
The fact that $\gamma_j$ is a $(\Ba_n,\Ba_n)$-bimodule map also follows from a tedious check on each basis elements, which we shall omit and leave it to the reader.
%in particular, the condition that $\gamma_i(xy) = x\gamma_j(y) = \gamma_j(x)y$ for all $x,y \in \Ba_n$ can be checked on a basis.
 \end{proof}

\begin{definition} \label{defbraid}
Define the following complexes of graded $(\Ba_n,\Ba_n)$-bimodules:
\begin{align*}
R_j &:= (0 \to P^B_j \otimes_{\F_j} {}_jP^B \xra{\beta_j} \Ba_n \to 0), \text{and} \\
R_j' &:= (0 \to \Ba_n  \xra{\gamma_j}  P^B_j \otimes_{\F_j} {}_jP^B(2) \to 0).
\end{align*}
for each $j \in \{1,2, \cdots, n\},$ with both $\Ba_n$ in cohomological degree 0, $\F_1 = \R$ and $\F_j = \C$ for $j \geq 2$.
\end{definition}

\begin{proposition} 
There are isomorphisms in the homotopy category, $\Kom^b ((\Ba_n, \Ba_n )$-$\text{bimod})$, of complexes of projective graded $(\Ba_n,\Ba_n)$-bimodules:
\begin{align}
R_j \otimes R_j^{'} \cong & \Ba_n \cong R_j^{'} \otimes R_j; \\
R_j \otimes R_k & \cong R_k \otimes R_j, \quad \text{for } |k-j| > 1;\\
R_j \otimes R_{j+1} \otimes R_j &\cong R_{j+1} \otimes R_j \otimes R_{j+1}, \quad \text{for } j \geq 2;\\
R_2 \otimes_{\Ba_n} R_1 \otimes_{\Ba_n} R_2 \otimes_{\Ba_n} R_1 &\cong 
R_1 \otimes_{\Ba_n} R_2 \otimes_{\Ba_n} R_1 \otimes_{\Ba_n} R_2.
\end{align}
\end{proposition}

\begin{proof} 
Please refer to \cref{standard braid relations} and \cref{type B relation}.
\end{proof}

\begin{theorem} (=\cref{Cat B action})
We have a (weak) $\mathcal{A}(B_n)$-action on $\Kom^b(\Ba_n$-$p_r g_r mod)$, where each standard generator $\sigma^B_j$ for $j \geq 2$ of $\mathcal{A}(B_n)$ acts on a complex $M \in \Kom^b(\Ba_n$-$p_rg_rmod)$ via $R_j$, and $\sigma^B_1$ acts via $R_1 \<1\>$:
$$\sigma^B_j(M):= R_j \otimes_{\Ba_n} M, \text{ and } {\sigma^B_j}^{-1}(M):= R_j' \otimes_{\Ba_n} M,$$
\[
\sigma^B_1(M):= R_1 \<1\> \otimes_{\Ba_n} M, \text{ and } {\sigma^B_1}^{-1}(M):= R_1' \<1\> \otimes_{\Ba_n} M.
\]
\end{theorem}
\begin{proof}
This follows directly from \cref{standard braid relations} and \cref{type B relation}, where the required relations still hold with the extra third grading shift $\<1\>$ on $R_1$ and $R_1'$.
\end{proof}
\subsection{Functor realisation of the type $B$ Temperley-Lieb algebra} \label{Functorial TL}
In \cite[section 2b]{KhoSei}, Khovanov-Seidel showed that the $(\Aa_m, \Aa_m)$-bimodules $\mathcal{U}_j := P_j^A \otimes_\C {}_j P^A$ provides a functor realisation of the type $A_m$ Temperley-Lieb algebra.
We will see that we have a similar type $B_n$ analogue of this.

Recall that the type $B_n$ Temperley-Lieb algebra $TL_v(B_n)$ over $\Z[v,v^{-1}]$ can be described explictly as (see \cite[Proposition 1.3]{Green}) the algebra generated by $E_1, ..., E_n$ with relation
\begin{align*}
E_j^2 &= vE_j + v^{-1}E_j; \\
E_j E_k &= E_k E_j, \quad \text{if } |j-k| > 1; \\
E_j E_k E_j &= E_j, \quad \text{if } |j-k| = 1 \text{ and } j,k > 1; \\
E_j E_k E_j E_k &= 2E_j E_k, \quad \text{if } \{j, k\} = \{1,2\}.
\end{align*}

To match with the above, for this subsection only, we shall adopt the path length grading on our algebra $\Ba_n$ instead, where we insist that the blue paths $i e_j$ in \ref{B quiver} have length 0. We shall denote this path length grading shift by $(-)$ to avoid confusion.
Define $\cU_j := P_j^B \otimes_{\mathbb{F}_j} {}_j P^B (1)$, where $\mathbb{F}_1 = \R$ and $\mathbb{F}_j = \C$ when $j \geq 2$.
It is easy to check that:
\begin{proposition} \label{TBTemp}
The following are isomorphic as (path length graded) $(\Ba_n, \Ba_n)$-bimodules:
\begin{align*}
\cU_j^2 &\cong U_j(1) \oplus \cU_j(-1) \\
\cU_j \cU_k &\cong 0 \quad \text{if } |j-k| > 1, \\
\cU_j \cU_k \cU_j &\cong \cU_k \quad \text{if } |j-k| = 1 \text{ and } j,k > 1, \\
\cU_j \cU_k \cU_j \cU_k &\cong \cU_j \cU_k \oplus \cU_j \cU_k \quad \text{if } \{j, k\} = \{1,2\}.
\end{align*}
\end{proposition}
Comparing this with the relations of $TL_v(B_n)$ above, it follows that $U_j$ provides a (degenerate) functor realisation of the type $B$ Temperley-Lieb algebra.

\begin{theorem} (=\cref{faithful action})
The (weak) action of $\cA(B_n)$ on the category $\Kom(\Ba_n$-$\text{p$_{r}$g$_{r}$mod})$ given in \ref{Cat B action} is faithful.
\end{theorem}
        
\begin{proof}
Please refer to \cref{faithful action} and \cref{faithtop}.
\end{proof}

\section{Faithfulness of the $2$-Braid Group in Type $B$} \label{Faith2BraidGrp}

The aim of this section is to construct a faithful 2-representation of the type $B$ 2-braid group using the action of $\cA(B_n)$ on $K^b(\Ba_n$-$\text{$g_r$mod}).$

   Fix $n \geq 2.$ 
   Denote by $\Ba$ the additive monoidal subcategory of $(\Ba_n,\Ba_n)$-bimod generated by the $U_i$ for $1 \leq i \leq n$ with morphisms being their grading shifts with bimodule homomorphisms of all degrees.
   Let $\ol{\Ba}$ be the graded version of $\Ba.$
   In this way,  $\ol{\Ba}$  is a full additive graded monoidal subcategory of $(\Ba_n,\Ba_n)$-bimod generated by the $U_i$ for $1 \leq i \leq n$ where objects are direct sums of tenor porducts of copies of the $U_i$'s and morphisms are grading-preserving bimodule maps.
   
   Fix the realization in \cref{Type B Geo Rep} and consider its diagrammatic category $\cD.$
   We want to construct a functor $G: \cD \ra \Ba$ which induces a grading-preserving functor $\ol{G}: K^b(Kar(\ol{\cD})) \ra K^b(Kar(\ol{\Ba}))$ macthing the Rouquier and the Khovanov Type B complexes associated to a braid word up to a perverse shift.
   Since $\cD$ is a strict monoidal category with generators and relations, it suffices to define the functor $G$ on the generating objects and morphisms as well as to verify that all relations in  $\cD$ hold in $\Ba.$

    In the following section, we use \textcolor{red}{red} or \textcolor{violet}{violet} for a simple reflection $s_j$, \textcolor{blue}{blue} for an adjacent simple reflection $s_{j \pm 1},$ and \textcolor{green!100}{green} for a distant simple reflection $s_k$ with $|j-k|>1.$
    We are going to define $G$ on generating objects by sending $s_j$ in $S$ to $\cU_j$ and on generating morphisms as in the proof of \cref{maintheorem} where the following $(\Ba_n, \Ba_n)$-bimodule homomorphisms are defined as follows:
\begin{align*}
\alpha_j:  & \ \ \ \ \  \cU_j && \ra \ \cU_j(-1) \oplus \cU_j(1) &&\xra{\cong} \ \ \ \ \ \   \cU_j \otimes_{\Ba_n} \cU_j    &&  \text{ of degree -1, } \\
 & \ e_j \otimes e_j && \mapsto  \ \ \  \ \ \ \ (0, e_j \otimes e_j) &&\mapsto \ e_j \otimes e_j \otimes e_j \otimes e_j \\
 \\
 \delta_j: &  \ \ \ \  \ \ \ \ \cU_j \otimes_{\Ba_n} \cU_j && \xra{\cong} \ \cU_j(-1) \oplus \cU_j(1) && \ra \ \  \ \ \cU_j  &&  \text{ of degree -1, } \\
& \ e_j \otimes X_j \otimes e_j \otimes e_j && \mapsto \ \  (e_j \otimes e_j,0 ) && \mapsto  e_j \otimes e_j \\
& \ e_j \otimes e_j \otimes e_j \otimes e_j && \mapsto  \ \ \ \ \ \ \ \ \ \ \  (0, e_j \otimes e_j)  && \mapsto \ \ \  \ \  0  
\end{align*}
\begin{align*}
 \epsilon_j:  \Ba_n & \ra \Ba_n  && \text{ of degree 2, } \\
  1 & \mapsto \begin{cases}
(-1)^{i+1} (2X_j + 2X_{j+1}), &\text{for } j=1;\\
(-1)^{i+1} (2X_j + X_{j-1} + X_{j+1}), &\text{for } 1<j < n; \\
(-1)^{i+1} ( 2X_j + X_{j-1}), &\text{for } j = n,
\end{cases}
\end{align*}

\noindent and, recall from \cref{bimodulemaps},
\begin{align*} 
\beta_j : \cU_j &\ra \Ba_n  \text{ of degree 1, }\\
 e_j \otimes e_j &\mapsto e_j
 \end{align*}
\begin{align*}
\gamma_j  :  \Ba_n &\ra  \cU_j  \text{ of degree 1, } \\
(1) &\mapsto
\begin{cases}
X_j \otimes e_j + e_j \otimes X_j + (j+1|j) \otimes (j|j+1) \\
 \hspace{8mm} + (-ie_{j+1})(j+1|j) \otimes (j|j+1)(ie_{j+1}), &\text{for } j=1;\\
X_j \otimes e_j + e_j \otimes X_j + (j-1|j) \otimes (j|j-1) + (j+1|j) \otimes (j|j+1), &\text{for } 1<j < n; \\
X_j \otimes e_j + e_j \otimes X_j + (j-1|j) \otimes (j|j-1), &\text{for } j = n,
\end{cases}
\end{align*}

\begin{lemma} Denote $k_j := \R$ when $j = 1$ and $k_j := \C$ when $j \geq 2$.
The space of $(\Ba_n,\Ba_n)$-bimodule homomorphism:
\begin{enumerate}
\item $\cU_j \ra \Ba_n$ of degree $1$ is isomorphic to $k_j \beta_j;$
\item $\Ba_n \ra \cU_j$ of degree $1$ is isomorphic to $k_j \gamma_j;$
\item $\cU_j \ra \cU_j \otimes_{\Ba_n} \cU_j $ of degree $-1$ is isomorphic to $k_j \alpha_j;$
\item $\cU_j \otimes_{\Ba_n} \cU_j \ra \cU_j $ of degree $-1$ is isomorphic to $k_j \delta_j;$
\item $\Ba_n \ra \Ba_n$ of degree 2 is isomorphic to $\oplus_{1 \leq j \leq n} k_jX_j,$ where $X_j$ is interpreted with left multiplication or right multiplication with $X_j$ giving an $\Ba_n$-bimodule homomorphism as all paths of lengths $\geq$ are killed in $\Ba_n.$
\end{enumerate}
\end{lemma}
\begin{proof}
This is an easy verification.
\end{proof}
Recall that, in previous section  (\cref{defbraid}), we define the following complexes of graded $(\Ba_n,\Ba_n)$-bimodules: for each $j \in \{1,2, \cdots, n\},$
\begin{align*}
R_j &:= (0 \to \cU_j(-1) \xra{\beta_j} \Ba_n \to 0), \text{and} \\
R_j' &:= (0 \to \Ba_n  \xra{\gamma_j}  \cU_j(1) \to 0).
\end{align*}

\noindent with both $\Ba_n$ in cohomological degree 0, $\F_1 = \R$ and $\F_j = \C$ for $j \geq 2$.

\begin{theorem} \label{maintheorem}
$G : \cD \ra \sB$ is a well-defined monoidal functor sending 
\begin{enumerate}[(i)]
\item the empty sequence in $S$ to  $\sB_n,$ the monoidal identity in $\sB,$ $s_i$ in $\sB;$
\item $s_i$ in $S$ to $U_i;$
\item a sequence $s_{i_1}s_{i_2} \cdots s_{i_k}$ in $S$ of length $k \geq 2$ to $G_0(s_{i_1}s_{i_2} \cdots s_{i_{k-1}}) \otimes_{\Ba_n} G_0(s_{i_k}).$
\end{enumerate}

For two sequences $\uw$ and $\uw'$ in S, the coherence morphism $G_0(\uw) \otimes_{A_n} G_0(w') \xra{\cong} G_0(\underline{ww'})$ is given by either a chosen sequence of associator morphisms moving all parentheses of the right block corresponding to the left to match the configuration of parenthese on the right hand side or a left or right unitor depending on whether $\uw$ or $\uw'$ is the empty sequence in $S.$
\end{theorem}

\begin{proof}
The assignment on objects is done in the theorem statement whereas on generating morphisms is carried out as follows:

\begin{center}
\begin{tabular}{c c l}
\begin{tikzcd}
\draw[dashed,color=black!60] (0,0) circle (0.5);
\draw[ violet, line width=0.8mm] (0,0) -- (0,-.5);
\filldraw[violet] (0,0) circle (3pt) ;
\end{tikzcd} 
& $\longmapsto$ 
& $b_j \beta_j:\cU_j \ra \Ba_n   \text{ of degree 1}$ \\
\\
\begin{tikzcd}
\draw[dashed,color=black!60] (0,0) circle (0.5);
\draw[ violet, line width=0.8mm] (0,0) -- (0,.5);
\filldraw[violet] (0,0) circle (3pt) ;
\end{tikzcd}
& $\longmapsto$
& $ c_j \gamma_j:\Ba_n \ra \cU_j   \text{ of degree 1 }$
\\
\\
\begin{tikzcd}
\draw[dashed,color=black!60] (0,0) circle (0.5);
\draw[ violet, line width=0.8mm] (0,0) -- (0,-.5);
\draw[ violet, line width=0.8mm] (0,0) -- (-.35,.35);
\draw[ violet, line width=0.8mm] (0,0) -- (.35,.35);
\end{tikzcd}
&$\longmapsto$
& $a_j \alpha_j:\cU_j \ra \cU_j \otimes_{\Ba_n} \cU_j  \text{ of degree -1 }$
\\
\\
\begin{tikzcd}
\draw[dashed,color=black!60] (0,0) circle (0.5);
\draw[ violet, line width=0.8mm] (0,0) -- (0,.5);
\draw[ violet, line width=0.8mm] (0,0) -- (-.35,-.35);
\draw[ violet, line width=0.8mm] (0,0) -- (.35,-.35);
\end{tikzcd}
&$\longmapsto$
&$d_j \delta_j: \cU_j \otimes_{\Ba_n} \cU_j \ra \cU_j  \text{ of degree -1 }$ 
\\
\\
\begin{tikzcd}
\draw[dashed,color=black!60] (0,0) node{\textcolor{black}{\alpha}_{\textcolor{violet}{s_i}}} circle (0.5);
\end{tikzcd}
&$\longmapsto$
&$\epsilon_j:\Ba_n \ra \Ba_n, 1 \mapsto \sum\limits_{k=1}^n f^j_k X_k  \text{ of degree 2 }$ 
\\
\\
\begin{tikzcd}
\draw[dashed,color=black!60] (0,0) circle (0.5);

\draw[ violet, line width=0.8mm] (.35,.35) -- (-.35,-.35);
\draw[ green, line width=0.8mm] (-.35,.35) -- (.35,-.35);
\end{tikzcd}
&$\longmapsto$
&$0$ \text{  as $\cU_j \otimes_{\Ba_n} \cU_k = 0$, for $|j-k| > 1$} 
\\
\\
\begin{tikzcd}
\draw[dashed,color=black!60] (0,0) circle (0.5);
\draw[ red, line width=0.8mm] (0,0) -- (0,.5);
\draw[ red, line width=0.8mm] (0,0) -- (-.35,-.35);
\draw[ red, line width=0.8mm] (0,0) -- (.35,-.35);
\draw[ blue, line width=0.8mm] (0,0) -- (0,-.5);
\draw[ blue, line width=0.8mm] (0,0) -- (-.35,.35);
\draw[ blue, line width=0.8mm] (0,0) -- (.35,.35);
\end{tikzcd}
&$\longmapsto$
& $0$
\\
\\
\begin{tikzcd}
\draw[dashed,color=black!60] (0,0) circle (0.5);
\draw[ red, line width=0.8mm] (0,0) -- (0,0.5);
\draw[ red, line width=0.8mm] (0,0) -- (0.5,0);
\draw[ red, line width=0.8mm] (0,0) -- (0,-.5);
\draw[ red, line width=0.8mm] (0,0) -- (-0.5,0);

\draw[ blue, line width=0.8mm] (.35,.35) -- (-.35,-.35);
\draw[ blue, line width=0.8mm] (-.35,.35) -- (.35,-.35);
\end{tikzcd}
&$\longmapsto$
& $0$
\end{tabular}
\end{center}

%\end{tabular}
%\end{center}
%\begin{center}
%\begin{tabular}{c c l}

\noindent where $a_j,$ $b_j,$ $c_j,$ $d_j,$ $f^j_k$ $\in \R$ for $j = 1$ and $\in \C$ for $j > 1.$ 
	Next, we want to find a set of scalars such that the restrictions imposed by the relations in $\cD$ are satisfied.

  To illustrate, let us consider the barbell relation (\cref{barbell}) for $s_1$:

\begin{equation*}
 \centering
\begin{tikzpicture}[scale=0.7]
\draw[dashed,color=black!60] (0,0) circle (1.0);
\draw[dashed,color=black!60] (3,0) circle (1.0);
\draw[ violet, line width=0.8mm] (0,.55) -- (0,-.55);
\filldraw[violet] (0,.55) circle (3pt) ;
\filldraw[violet] (0,-.55) circle (3pt) ;
\node at (1.5,0) {=};
\node at (3,0) {{\LARGE $\alpha_{\textcolor{violet}{s_1}}$}};
\end{tikzpicture}. 
\end{equation*}
By definition,
\begin{align*}
1 & \mapsto c_1[X_1 \otimes e_1 + e_1 \otimes X_1 + (2|1) \otimes (1|2) + (-ie_2)(2|1) \otimes (1|2)(ie_2)]\\
& \mapsto c_1b_1 [ X_1 + X_1 + X_2 + X_2 ] \\
& = c_1b_1[2X_1 + 2X_2] 
\end{align*}
\noindent Equating with the right hand side, we get $c_1b_1[2X_1 + 2X_2] = \sum^n_{k=1} f^1_k X_k$ which, in turn, implies $2c_1 b_1 = f^1_1 $ and $2c_1 b_1 = f^1_2.$

  By the same token, let us look at the type B 8-valences relation \cref{mst4} again: suppose \textcolor{red}{red} encodes \textcolor{red}{$s_1$} and \textcolor{blue}{blue} encodes \textcolor{blue}{$s_2$},
\begin{align*}
\begin{tikzcd}
\draw[color=red, line width=0.8mm] (-1.45,0) -- (.55,0);
\draw[color=red, line width=0.8mm] (0,0.75) -- (0,-0.75);
\filldraw[red]  (0.55,0) circle (3pt) ;
\draw[color=blue, line width=0.8mm] (.75,.75) -- (-.75,-0.75);
\draw[color=blue, line width=0.8mm] (-.75,.75) -- (.75,-0.75);
\end{tikzcd} 
 &= 
\begin{tikzcd}
\draw[red, line width=0.8mm] (.125,.68) -- (.125,-.68);
\draw[blue, line width=0.8mm] (-.125,.68) -- (-.125,-.68);
\draw[red, line width=0.8mm] (-.5,0) -- (-1,0);
\draw[blue, line width=0.8mm] (.45,0) -- (.7,0);
\draw[blue, line width=0.8mm] (.7,.68) -- (.7,-.68);
\filldraw[red] (-.425,0) circle (3pt) ;
\filldraw[blue] (.425,0) circle (3pt) ;
\end{tikzcd} 
- \frac{a_{s_1,s_2}}{a_{s_2,s_1}a_{s_1,s_2}-1} 
\begin{tikzcd}
\draw[red, line width=0.8mm] (-.5,0) -- (-1,0);
\draw[blue, line width=0.8mm] (.7,.68) -- (.7,-.68);
\draw[red, line width=0.8mm] (.23,0.4) -- (.23,0.66);
\draw[red, line width=0.8mm] (.23,-0.4) -- (.23,-0.66);
\draw[blue, line width=0.8mm] (0,0) -- (-.45,.58);
\draw[blue, line width=0.8mm] (0,0) -- (-.45,-.58);
\draw[blue, line width=0.8mm] (0,0) -- (.7,0);
\filldraw[red] (-.4,0) circle (3pt) ;
\filldraw[red] (.23,0.3) circle (3pt) ;
\filldraw[red] (.23,-0.3) circle (3pt) ;
\end{tikzcd} 
- \frac{a_{s_2,s_1}}{a_{s_2,s_1}a_{s_1,s_2}-1} 
\begin{tikzcd}
\draw[blue, line width=0.8mm] (.5,0) -- (.7,0);
\draw[blue, line width=0.8mm] (.7,.6) -- (.7,-.6);
\draw[blue, line width=0.8mm] (-.23,0.4) -- (-.23,0.66);
\draw[blue, line width=0.8mm] (-.23,-0.4) -- (-.23,-0.66);
\draw[red, line width=0.8mm] (0,0) -- (.45,.58);
\draw[red, line width=0.8mm] (0,0) -- (.45,-.58);
\draw[red, line width=0.8mm] (0,0) -- (-1,0);
\filldraw[blue] (.4,0) circle (3pt) ;
\filldraw[blue] (-.23,0.3) circle (3pt) ;
\filldraw[blue] (-.23,-0.3) circle (3pt) ;
\end{tikzcd}  \\
&+ \frac{1}{a_{s_2,s_1}a_{s_1,s_2}-1} 
\begin{tikzcd}
\draw[blue, line width=0.8mm] (.13,0) -- (.7,0);
\draw[blue, line width=0.8mm] (.7,.6) -- (.7,-.6);
\draw[blue, line width=0.8mm] (.13,0) -- (-.45,.58);
\draw[red, line width=0.8mm] (-.15,0) -- (.45,-.58);
\draw[red, line width=0.8mm] (-.15,0) -- (-1,0);
\draw[red, line width=0.8mm] (.23,0.4) -- (.25,0.66);
\draw[blue, line width=0.8mm] (-.23,-0.4) -- (-.25,-0.66);
\filldraw[red] (.23,0.3) circle (3pt) ;
\filldraw[blue] (-.23,-0.3) circle (3pt) ;
\end{tikzcd} 
+ \frac{1}{a_{s_2,s_1}a_{s_1,s_2}-1}
\begin{tikzcd}
\draw[blue, line width=0.8mm] (.13,0) -- (.7,0);
\draw[blue, line width=0.8mm] (.7,.68) -- (.7,-.68);
\draw[red, line width=0.8mm] (.23,-0.4) -- (.23,-0.66);
\draw[blue, line width=0.8mm] (-.23,0.4) -- (-.23,0.66);
\draw[red, line width=0.8mm] (-.13,0) -- (.45,.58);
\draw[red, line width=0.8mm] (-.13,0) -- (-1,0);
\draw[blue, line width=0.8mm] (.13,0) -- (-.45,-.58);
\filldraw[red] (.23,-0.3) circle (3pt) ;
\filldraw[blue] (-.23,0.3) circle (3pt) ;
\end{tikzcd} \\	
\end{align*}  
which, by construction, is equivalent to
\begin{align*}
0 
\  &= \ 
\begin{tikzcd}
\draw[red, line width=0.8mm] (.125,.68) -- (.125,-.68);
\draw[blue, line width=0.8mm] (-.125,.68) -- (-.125,-.68);
\draw[red, line width=0.8mm] (-.425,0) -- (-.425,-0.68);
\draw[blue, line width=0.8mm] (.45,.68) -- (.45,-.68);
\filldraw[red] (-.425,0) circle (3pt) ;
\end{tikzcd} 
 \ - \frac{a_{s_1,s_2}}{a_{s_2,s_1}a_{s_1,s_2}-1}  \  \
\begin{tikzcd}
\draw[red, line width=0.8mm] (-.4,0) -- (-.4,-.68);
\draw[blue, line width=0.8mm] (.7,.68) -- (.7,-.68);
\draw[red, line width=0.8mm] (.3,0.4) -- (.3,0.68);
\draw[red, line width=0.8mm] (.3,-0.4) -- (.3,-0.68);
\draw[blue, line width=0.8mm] (-0.05,0) -- (-.05,.68);
\draw[blue, line width=0.8mm] (-0.05,0) -- (-.05,-.68);
\draw[blue, line width=0.8mm] (0,0) -- (.7,0);
\filldraw[red] (-.4,0) circle (3pt) ;
\filldraw[red] (.3,0.3) circle (3pt) ;
\filldraw[red] (.3,-0.3) circle (3pt) ;
\end{tikzcd} 
 \ - \frac{a_{s_2,s_1}}{a_{s_2,s_1}a_{s_1,s_2}-1} \
\begin{tikzcd}
\draw[blue, line width=0.8mm] (.5,.68) -- (.5,-.68);
\draw[blue, line width=0.8mm] (-.23,0.4) -- (-.23,0.68);
\draw[blue, line width=0.8mm] (-.23,-0.4) -- (-.23,-0.68);
\draw[red, line width=0.8mm] (0.15,0) -- (0.15,.68);
\draw[red, line width=0.8mm] (0.15,0) -- (0.15,-.68);
\draw[red, line width=0.8mm] (0.15,0) .. controls (-.5,-.05) and (-.58,-.1) .. (-.6,-0.68);
\filldraw[blue] (-.23,0.3) circle (3pt) ;
\filldraw[blue] (-.23,-0.3) circle (3pt) ;
\end{tikzcd}  \\
 & \ + \frac{1}{a_{s_2,s_1}a_{s_1,s_2}-1} \ 
\begin{tikzcd}
\draw[blue, line width=0.8mm] (.7,0) .. controls (.05, .05) and (-.07,.1) .. (-.05,.68);
\draw[blue, line width=0.8mm] (.7,.68) -- (.7,-.68);
\draw[red, line width=0.8mm] (-.05,0) .. controls (0,-.009) and (.3,-.08) .. (.35,-.68);
\draw[red, line width=0.8mm] (-.05,0) .. controls (-0.1,-.009) and (-.35,-.08) .. (-.4,-.68);
\draw[red, line width=0.8mm] (.35,0.4) -- (.35,0.66);
\draw[blue, line width=0.8mm] (-.05,-0.4) -- (-.05,-0.66);
\filldraw[red] (.35,0.3) circle (3pt) ;
\filldraw[blue] (-.05,-0.3) circle (3pt) ;
\end{tikzcd} 
\ + \frac{1}{a_{s_2,s_1}a_{s_1,s_2}-1} \ 
\begin{tikzcd}
\draw[blue, line width=0.8mm] (.7,0) .. controls (.05, -.05) and (-.07,-.1) .. (-.05,-.68);
\draw[red, line width=0.8mm] (-.05,0) .. controls (0,.009) and (.3,.08) .. (.35,.68);
\draw[red, line width=0.8mm] (-.05,0) .. controls (-0.1,-.009) and (-.35,-.08) .. (-.4,-.68);
\draw[blue, line width=0.8mm] (.7,.68) -- (.7,-.68);
\draw[red, line width=0.8mm] (.35,-0.4) -- (.35,-0.66);
\draw[blue, line width=0.8mm] (-.05,0.4) -- (-.05,0.66);
\filldraw[red] (.35,-0.3) circle (3pt) ;
\filldraw[blue] (-.05,0.3) circle (3pt) ;
\end{tikzcd} \\	
\end{align*}  
\noindent Observe that there are five terms of Soergel graphs on the right hand side. One thing to note here is that $\cU_1 \otimes \cU_2 \otimes \cU_1 \otimes \cU_2$ is spanned by  $e_1 \otimes_\R (1|2) \otimes_{\Ba_n} e_2 \otimes_\C e_2 \otimes_{\Ba_n} (2|1) \otimes_\R (1|2) \otimes_{\Ba_n} e_2 \otimes_\C e_2  $ and $e_1 \otimes_\R (1|2) \otimes_{\Ba_n} e_2 \otimes_\C e_2 \otimes_{\Ba_n} (-ie_2)(2|1) \otimes_\R (1|2) \otimes_{\Ba_n} e_2 \otimes_\C e_2.$ 
 Without the coefficients,  looking at the second term in the above equation and applying appropriate,  we get
\begin{align*}
e_1 \otimes (1|2) \otimes e_2 \otimes e_2 \otimes (2|1) \otimes (1|2) \otimes e_2 \otimes e_2  
 & \xmapsto{\beta_1}  b_1  (1|2) \otimes e_2 \otimes (2|1) \otimes (1|2) \otimes e_2 \otimes e_2  \\
& \xmapsto{\beta_1} b_1^2  (1|2) \otimes X_2 \otimes e_2 \otimes e_2 \\
& \xmapsto{\delta_2} d_2 b_1^2  (1|2) \otimes e_2 \\
& \xmapsto{\alpha_2} a_2 d_2 b_1^2  (1|2) \otimes e_2 \otimes e_2 \otimes e_2 \\
& \xmapsto{\gamma_1} c_1 a_2 d_2 b_1^2  (1|2) \otimes e_2 \otimes (-ie_2)(2|1) \otimes (1|2)(ie_2) \otimes e_2 \otimes e_2  \\
&  \ \ \ \ \ + c_1 a_2 d_2 b_1^2  (1|2) \otimes e_2 \otimes (2|1) \otimes (1|2) \otimes e_2 \otimes e_2
\end{align*}
Similarly, 
\begin{align*}
e_1 & \otimes (1|2) \otimes e_2 \otimes e_2 \otimes (-ie_2)(2|1) \otimes (1|2) \otimes e_2 \otimes e_2  \\
& \mapsto c_1 a_2 d_2 b_1^2 \left[ (1|2) \otimes e_2 \otimes (2|1) \otimes (1|2)(-ie_2) \otimes e_2 \otimes e_2  
 + (1|2) \otimes e_2 \otimes (-ie_2)(2|1) \otimes (1|2) \otimes e_2 \otimes e_2 \right]
\end{align*}
Once the calculations for all the five Soergel graphs, comparing coefficients coming from four basis elements in the codomain will yield four defining equations.

Before giving you all the relations, we will deal with the unnecessary or overlapped relations. For the \cref{needle}, it says $b_j d_j a_j \beta_j \delta_j \alpha_j = 0$ which is true as $\delta_j \alpha_j = 0.$
 On the other hand, the \cref{Frobenius} (equiv. \cref{assocmult} and \cref{coassoccomult}) do not impose any restrictions on the coefficients whereas the \cref{wall} is replaced by \cref{enddot counit comult} and \cref{startdot unit multi}.
    In addition, the \cref{assoc3}, \cref{assoc4}, \cref{Zamo4} \cref{Zamo3}, \cref{Zamo2}, \cref{A3}, and \cref{B3} are all trivially satisfied as the $2m_{st}$-valents vertices are killed for every $s,t \in S.$
    Moreover, the \cref{mst2} has both side equal as a zero object in $\B_a.$

Now, we summarise all the relations: (Mostly, the same from the type A case, except the two colour relation.)

\cref{barbell} \  $\Ra$ \ \ $f^1_1 = 2b_1 c_1, f^1_2 =2 b_1 c_1, f^j_j = 2 b_j c_j, f^j_{j \pm 1} = b_j c_j, f^j_k=0$ for $j,k \geq 2$ and $|j-k| > 1,$

\cref{polyforce} \ $\Ra$  \  $f^1_1 = 2b_1 c_1, f^1_2 = - 2 b_2 c_2, f^j_j = 2 b_j c_j, f^{j \pm 1}_j = -b_j c_j, f^j_k=0$ for $j,k \geq 2$ and $|j-k| > 1,$

\cref{mst3} (Type A 6-valences relation) \ $\Ra$ \ $d_{j \pm 1} b_j c_j = -b_{j \pm 1}$ for suitable $j,$  

\cref{mst4} (Type B 8-valences relation) \ $\Ra$  \
Since $a_{s_2,s_1} a_{s_1,s_2}-1 = (-2)(-1) -1 =1,$ we will simplify the denominator first. 

For $U_1$ left-aligned, we get
$$b_1 - {a_{s_1,s_2}}b_1^2 d_2 a_2 c_1 - {a_{s_2,s_1}}b_2 d_1 c_2 + b_2 d_1 b_1 a_2 c_1 + b_1 d_2 c_2 = 0,$$
$$b_1 - a_{s_1,s_2} b_1^2 d_2 a_2 c_1 = 0, \ \
-a_{s_1,s_2} b_1^2 d_2 a_2 c_1 + b_2 d_1 b_1 a_2 c_1 = 0,  \ \
-a_{s_1,s_2} b_1^2 d_2 a_2 c_1 + b_1 d_2 c_2 = 0,$$
while for $U_2$ left-aligned, we get
$$b_2 - a_{s_2,s_1} b_2^2 d_1 a_1 c_2 - a_{s_1,s_2} b_1 d_2 c_1 + b_1 d_2 a_1 b_2 c_2 + b_2 d_1 c_1 = 0,$$
$$b_2 -  a_{s_1,s_2} b_1 d_2 c_1 = 0, \ \
 - a_{s_1,s_2} b_1 d_2 c_1 + b_1 d_2 a_1 b_2 c_2 = 0, \ \
-a_{s_1,s_2} b_1 d_2 c_1 + b_2 d_1 c_1 = 0, $$
\cref{enddot counit comult} \ \  $\Ra$ \ \ $a_jb_j = 1$  for all $1 \leq j \leq n,$

\cref{startdot unit multi} \ \ $\Ra$ \ \ $c_jd_j = 1$  for all $1 \leq j \leq n,$

\cref{selfadjoint} \ \ $\Ra$ \ \ $a_jb_jc_jd_j = 1$  for all $1 \leq j \leq n,$

\cref{mult rot comult} \ \ $\Ra$ \ \ $a_jb_jd_j = d_j$  for all $1 \leq j \leq n,$

\cref{comult rot mult} \ \ $\Ra$ \ \ $a_jc_jd_j = a_j$  for all $1 \leq j \leq n,$

\cref{counit rot unit} \ \ $\Ra$ \ \ $b_jc_jd_j = b_j$  for all $1 \leq j \leq n,$

\cref{unit rot counit} \ \ $\Ra$ \ \ $a_jb_jc_j = c_j$  for all $1 \leq j \leq n.$

The solution $a_j = b_j = (-1)^{j+1}, c_j = d_j = 1, f^1_2 = 2, f^j_j = (-1)^{j+1} 2, f^j_{j \pm 1} = (-1)^{j+1}$ and $f^j_k = 0$ for $|j-k|>1$ gives our desired functor.
\end{proof}

\begin{corollary}
There is a well-defined additive, monoidal functor $\ol{G}: K^b(Kar(\ol{\cD})) \ra K^b(Kar(\ol{\Ba}))$ which sends $F_{s_j}$ to $R_j[-1](1)$ and $E_{s_1}$ to  $R'_j[1](-1)$  and thus matches Rouquier and the Khovanov Type B complex associated to a braid word in $Br_{n+1}$ up to inner grading shift and cohomological shift.
\end{corollary}

\begin{corollary}[Faithfullness of the 2-braid group in type B] \label{Faith2}
Let $(W,S)$ be a Coxeter group of type $B_n,$ for $n \geq 2$. 
 For any two distinct braids $\sigma \neq \beta \in Br_{(W,S)}$ the corresponding Rouquier complexes $F_\sigma$ and $F_\beta$ in the $2$-braid group are non-isomorphic.

\end{corollary}

             \chapter{A Graded Fock Vector for a Crossingless Matching.} \label{Chap3}

\section{Background and Outline of the Chapter}
In \cite{Jones}, Vaughan Jones introduced a polynomial invariant for knots, thanks to the uniqueness of the trace on a class of von Neumann algebras called the type $II_1$ factors. 
  This is the well-known Jones polynomial.
  He introduced the Jones polynomial via a representation of braid groups through the representation of Temperley-Lieb algebras. 
  The goal of this chapter is to seek a new presentation of Jones polynomial via the Heisenberg algebra.
  This comes from a conjectural braid group action on the Fock space, which is a representation of the Heisenberg algebra.
  The braid group action has its origin in Lusztig's braid group action constructed from a quantum group \cite{Lusztig} in the form of vertex operators \cite{cautis2014vertex}, which is categorical in flavour.
  Though it still remains a basic conjectural statement that the newly-defined braid operators in this chapter indeed braid, we  start off the program by identifying a vector subspace of Fock space which is conjectured to be a Temperley-Lieb representation.
  Recall that the Temperley-Lieb algebra $\cT \cL_{n+1}$ has a graphical interpretation with the following assignment to its generator $u_i$:
     \begin{figure}[H]  
     \centering
   \begin{tikzpicture}
   \draw (-1,-.5) -- (-1,.5);
   \draw (-.5,-.5) -- (-.5,.5);
    \draw (0.25,-.5) -- (0.25,.5);
         \node at (-.1,0) {$\cdots$};
          \draw (1.75,-.5) -- (1.75,.5);
            \node at (2.15,0) {$\cdots$};
            \draw (2.5,-.5) -- (2.5,.5);
             \draw (3,-.5) -- (3,.5);
    \draw (1.25,0.5) -- (1.25,0.5) arc(0:-180:0.25);
      \draw (1.25,-0.5) -- (1.25,-0.5) arc(0:180:0.25);
     \node [below] at (-1,-0.5) {{\footnotesize 1}};
      \node [below] at (-.5,-0.5) {{\footnotesize 2}};
      \node at (-1.5,0) {$u_i =$};
        \node [below] at (0.25,-.5) {{\scriptsize i-1}};
        \node [below] at (.75,-.5) {{\footnotesize i}};
        \node [below] at (1.25,-.5) {{\scriptsize i+1}}; 
        \node [below] at (1.75,-.5) {{\scriptsize i+2}}; 
        \node [below] at (2.5,-.5) {{\scriptsize n}}; 
          \node [below] at (3,-.5) {{\scriptsize n+1}}; 
   \end{tikzpicture}
   \end{figure} 
\noindent and this generator acts on the set of crossingless matchings.
    We devise a way to assign each crossingless matchings a graded vector in the Fock space.

    On the other hand, the $(n+1)$-strand braid group or the type $A_n$ braid group, $\cA(A_n)$ has a graphical interpretation with the following assignment to its generators $s_i:$
         \begin{figure}[H]  
         \centering
   \begin{tikzpicture}
   \draw (-1,-.5) -- (-1,.5);
   \draw (-.5,-.5) -- (-.5,.5);
    \draw (0.25,-.5) -- (0.25,.5);
         \node at (-.1,0) {$\cdots$};
          \draw (1.75,-.5) -- (1.75,.5);
            \node at (2.15,0) {$\cdots$};
            \draw (2.5,-.5) -- (2.5,.5);
             \draw (3,-.5) -- (3,.5);
    \draw (0.75,0.5) -- (1.25,-.5);
    \draw (0.75, -0.5) -- (.95,-.1);
      \draw (1.05, 0.1) -- (1.25,.5);
     \node [below] at (-1,-0.5) {{\footnotesize 1}};
      \node [below] at (-.5,-0.5) {{\footnotesize 2}};
      \node at (-1.5,0) {$\sigma_i =$};
        \node [below] at (0.25,-.5) {{\scriptsize i-1}};
        \node [below] at (.75,-.5) {{\footnotesize i}};
        \node [below] at (1.25,-.5) {{\scriptsize i+1}}; 
        \node [below] at (1.75,-.5) {{\scriptsize i+2}}; 
        \node [below] at (2.5,-.5) {{\scriptsize n}}; 
          \node [below] at (3,-.5) {{\scriptsize n+1}}; 
   \end{tikzpicture}
   \end{figure}
  \noindent  Following \cite{Kauff}, note that there is a representation $\pi$ of the $(n+1)$-strand braid group into the Temperley-Lieb algebra which is given by 
  \begin{figure}[H]  
  \centering
   \begin{tikzpicture}
   \draw (-7,-.5) -- (-7,.5);
   \draw (-6.5,-.5) -- (-6.5,.5);
    \draw (-5.75,-.5) -- (-5.75,.5);
         \node at (-6.1,0) {$\cdots$};
          \draw (-4.25,-.5) -- (-4.25,.5);
            \node at (-3.85,0) {$\cdots$};
            \draw (-3.5,-.5) -- (-3.5,.5);
             \draw (-3,-.5) -- (-3,.5);
    \draw (-5.25,0.5) -- (-4.75,-.5);
    \draw (-5.25, -0.5) -- (-5.05,-.1);
      \draw (-4.95, 0.1) -- (-4.75,.5);
     \node [below] at (-7,-0.5) {{\footnotesize 1}};
      \node [below] at (-6.5,-0.5) {{\footnotesize 2}};
        \node [below] at (-5.75,-.5) {{\scriptsize i-1}};
        \node [below] at (-5.25,-.5) {{\footnotesize i}};
        \node [below] at (-4.75,-.5) {{\scriptsize i+1}}; 
        \node [below] at (-4.25,-.5) {{\scriptsize i+2}}; 
        \node [below] at (-3.5,-.5) {{\scriptsize n}}; 
          \node [below] at (-3,-.5) {{\scriptsize n+1}};

   \node at (-2.25,0) {$\mapsto$};
    \node at (-1.5,0) {$A$};

   \draw (-1,-.5) -- (-1,.5);
   \draw (-.5,-.5) -- (-.5,.5);
    \draw (0.25,-.5) -- (0.25,.5);
         \node at (-.1,0) {$\cdots$};
          \draw (1.75,-.5) -- (1.75,.5);
            \node at (2.15,0) {$\cdots$};
            \draw (2.5,-.5) -- (2.5,.5);
             \draw (3,-.5) -- (3,.5);
    \draw (1.25,0.5) -- (1.25,0.5) arc(0:-180:0.25);
      \draw (1.25,-0.5) -- (1.25,-0.5) arc(0:180:0.25);
     \node [below] at (-1,-0.5) {{\footnotesize 1}};
      \node [below] at (-.5,-0.5) {{\footnotesize 2}};
        \node [below] at (0.25,-.5) {{\scriptsize i-1}};
        \node [below] at (.75,-.5) {{\footnotesize i}};
        \node [below] at (1.25,-.5) {{\scriptsize i+1}}; 
        \node [below] at (1.75,-.5) {{\scriptsize i+2}}; 
        \node [below] at (2.5,-.5) {{\scriptsize n}}; 
          \node [below] at (3,-.5) {{\scriptsize n+1}}; 
          
\node at (3.5,0) {$+$};          
          
           \node at (4.5,0) {$A^{-1}$};
          
   \draw (5,-.5) -- (5,.5);
   \draw (5.5,-.5) -- (5.5,.5);
    \draw (6.25,-.5) -- (6.25,.5);
         \node at (5.9,0) {$\cdots$};
          \draw (7.75,-.5) -- (7.75,.5);
            \node at (8.15,0) {$\cdots$};
            \draw (8.5,-.5) -- (8.5,.5);
             \draw (9,-.5) -- (9,.5);
    \draw (6.75,0.5) -- (6.75,-0.5) ;
      \draw (7.25,-0.5) -- (7.25,0.5) ;
     \node [below] at (5,-0.5) {{\footnotesize 1}};
      \node [below] at (5.5,-0.5) {{\footnotesize 2}};
        \node [below] at (6.25,-.5) {{\scriptsize i-1}};
        \node [below] at (6.75,-.5) {{\footnotesize i}};
        \node [below] at (7.25,-.5) {{\scriptsize i+1}}; 
        \node [below] at (7.75,-.5) {{\scriptsize i+2}}; 
        \node [below] at (8.5,-.5) {{\scriptsize n}}; 
          \node [below] at (9,-.5) {{\scriptsize n+1}};

   \end{tikzpicture}
   \end{figure} 
\noindent  The choice of scalar here is due to the requirement of the invariant under type II Reidemeister move.

Hence, the vector subspace spanned by the crossingless matchings is conjectured to be preserved by the braid operators via the Temperley-Lieb representation.

 \subsection*{Outline of the chapter}

 In \cref{CombPar}, we introduce various definitions of combinatorics related to tableaux which are essential to the theory in the chapter.
 
\cref{QuaLatHei} is where we recall and reconcile the definition of type $A$ quantum lattice Heisenberg algebra from various literature.

 In \cref{Fockrep}, we present the main battlefield of the theory --  the Fock space representation of the Heisenberg algebra.
 
 \cref{braidop} is where we define the braid operator in the Fock space representation coming from private communication with my supervisor, Licata \cite{LicPrivate}.

 \cref{Fockinv} is the key section where we introduce a graded Fock vector for a crossingless matching. 
 We also make two conjectures on these graded Fock vectors.

 \newpage
  
    \section{Combinatorics of Partitions of a Number} \label{CombPar}

  For the general study of Young tableaux, the readers are encouraged to read \cite{Fulton, Sagan, Stan1}.
 In this section, we will also recall a specific combinatoric tool from \cite{ErikLi}. 
  
  Let $\lambda = (\lambda_1,  \lambda_2, \hdots,  \lambda_r, \hdots)$ with an  infinite number of non-zero terms satisfying $\lambda_1 \geq \lambda_2 \geq \cdots \geq  \lambda_r \geq \cdots$ be a partition of a natural number $n \in \N_{> 0}.$
  We identify a finite sequence  $(\lambda_1,  \lambda_2, \hdots,  \lambda_r) $ with the infinite sequence obtained by setting $\lambda_i =0$ for $i > r.$
   These $\lambda_i$ are called the \textit{parts} of $\lambda.$
   We then have the sum $|\lambda| := \lambda_1 + \lambda_2 + \cdots + \lambda_r = n.$  We also denote $\lambda \vdash n.$ 
   We could also write $\lambda = (1^{m_1} 2^{m_2} \cdots )$ where $m_i$ is the number of parts of $\lambda$ equals to $i.$
   For example, $\lambda = (3,2,2,1) = (1 \cdot 2^2 \cdot 3).$

  To each $\lambda,$ we can associate a \textit{Young diagram} (in English notation) which is a collection of boxes arranged in left-justified rows with a weakly decreasing number of boxes in each row.  
     We define the \textit{conjugate partition} or the \textit{transpose} of a partition $\lambda^t$ to be the Young diagram obtained by reflection in the main diagonal, in other words, interchanging rows and columns. 
     It follows that if $\lambda = (\lambda_1,  \lambda_2, \hdots,  \lambda_r),$ then the number of parts of $\lambda^t$ that equals to $i$  is $\lambda_{i+1} - \lambda_i.$
   For example, if $\lambda = (3,2,2,1) ,$ then $\lambda^t = (4,3,1).$
   We write $\mu \subset \lambda$ if the Young diagram of $\mu$ is contained in that of $\lambda$, or equivalently $\mu_i \leq \lambda_i$ for all $i.$
   If $\mu \subset \lambda,$ we can then form a skew diagram $\lambda - \mu$ by removing $\mu$ from $\lambda$ in such a way that the first $\mu_i$ boxes are removed from the part $\lambda_i $ for all $i$.
   
   Next, we want to introduce the lattice permutation.
   A \textit{Young tableau} or a \textit{semistandard Young tableau} is a filling of a Young diagram by  positive integers that is weakly increasing across each row and strictly increasing down each column and denote a Young tableau by $SSYT(\lambda)$.
   If the Young diagram is filled with the natural numbers $1,2, \hdots, k,$ then it is called the standard Young tableau and denoted by $SYT(\lambda).$  
   The \textit{word} $w(\lambda_T)$ of a tableau $T$ is the set of numbers in the diagram read from right to left and top to bottom.
   On the other hand, the \textit{content} $w(\lambda)$ of a partition is the word of the tableau of shape $\lambda$ with $i$-th row filled with the number $i.$ 
   The word $a_1a_2 \cdots a_r$ is a \textit{lattice permutation}\footnote{Also known as \textit{Yamanouchi word} or \textit{ballot sequence} in different literature.}   if for any $i$ and $k \leq r,$ the number of $i$'s in $a_1a_2 \cdots a_r$ is at least as great as the number of $i+1$'s.
     For example, 
   
      \begin{figure}[H]  
    \centering
  \begin{tikzpicture}
   \draw (-1,0) -- (.5,0);
   \draw (-1,-2) -- (-1,0);
   \draw (-.5,0) -- (-.5,-2);
   \draw ( -1,  -2 ) --(-.5, -2 );
     \draw (.5,-.5) -- (.5, 0);
     \draw (0, -1.5) -- (0,  0);
     \draw (-1,-1.5) -- (0,  -1.5);
     \draw (-1, -1) -- (0, -1);
     \draw (-1,-.5) -- (.5, -.5);
     \node [below right] at (-1,0) {$1$};
        \node [below right] at (-.5,0) {$1$};
        \node [below right] at (0,  0) {$2$};
        \node [below right] at (-1, -.5) {$2$};
      \node [below right] at  (-.5, -.5) {$3$};
       \node [below right] at  (-1, -1)  {$4$};
     \node [below right] at  (-.5, -1) {$4$};
        \node [below right] at  (-1, -1.5) {$5$};
         \node  at  (-1.5, -1) {$\lambda_T=$};
  \node  at  (-1, -3) {$w(\lambda_T)=21132445$};
 \node  at  (-6, -4) {It is a lattice permutation.};

   \draw (-6,0) -- (-4.5,0);
   \draw (-6,-2) -- (-6,0);
   \draw (-5.5,0) -- (-5.5,-2);
   \draw ( -6,  -2 ) --(-5.5, -2 );
     \draw (-4.5,-.5) -- (-4.5, 0);
     \draw (-5, -1.5) -- (-5,  0);
     \draw (-6,-1.5) -- (-5,  -1.5);
     \draw (-6, -1) -- (-5, -1);
     \draw (-6,-.5) -- (-4.5, -.5);
      \node  at  (-6.5, -1) {$\lambda=$};
        \node  at  (-6, -3) {$w(\lambda)=11122334$};
         \node  at  (-1, -4) {It is not a lattice permutation}; 
                \node  at  (-1, -4.5) {because of the first $2$.}; 
     
   \end{tikzpicture}
   \end{figure}     
   
   We denote the set of skew Young tableaux $(\lambda - \nu)_T$ with content $w( \mu - \kappa)$ whose concatenation of words $w(\kappa)w((\lambda - \nu)_T)$ are lattice permutation as $\fl \fp(\lambda- \nu, \mu- \kappa).$
   We also define \textit{$t$-Young tableau} $\lambda_{\fT}$ as a filling of monomials of the form $at^b$ for $a \geq 1$ and $b \in \{0,1\}$ with a total order given by 
   \begin{align}
   at^b \leq ct^d \iff b < d \text{ or } b=d \text{ and } a \leq c.
   \end{align}
   The word of $t$-Young tableau reads the coefficients of the monomials from right to left and top to bottom with the monomials of the form $at^0$ read before the monomials of the form $at^1$.  
   We denote the set of skew $t$-Young tableaux$(\lambda - \nu)_\fT$ with monomials having coefficient in content $w( \mu - \kappa)$ and whose concatenation of words $w(\kappa)w((\lambda - \nu)_\fT)$ are lattice permutation as $t$-$\fl \fp(\lambda- \nu, \mu- \kappa)$.
   We also define a statistic on $t$-Young tableau by 
      \begin{align}
      c(\lambda_\fT) = \prod_{at^b \in \lambda_\fT} t^b.
      \end{align}
 For example, let $\lambda = (3,2,2,1), \nu = (3,1), \mu = (3,1,1), \kappa = (1).$
 Then, $w(\mu- \kappa) = 1123.$
 Below is the list of elements in $t$-$\fl \fp(\lambda- \nu, \mu- \kappa)$ with the its first element in $\fl \fp(\lambda- \nu, \mu- \kappa).$
 
         \begin{figure}[H]   
         \begin{subfigure}{0.2 \textwidth}
           \centering
  \begin{tikzpicture}
     \draw[fill=gray!30]    (-1,0) -- (.5,0) -- (.5,-.5) -- (-1,-.5);
       \draw[fill=gray!30]    (-1,-.5) -- (-.5,-.5) -- (-.5,-1) -- (-1,-1);
   \draw (-1,0) -- (.5,0);
   \draw (-1,-2) -- (-1,0);
   \draw (-.5,0) -- (-.5,-2);
   \draw ( -1,  -2 ) --(-.5, -2 );
     \draw (.5,-.5) -- (.5, 0);
     \draw (0, -1.5) -- (0,  0);
     \draw (-1,-1.5) -- (0,  -1.5);
     \draw (-1, -1) -- (0, -1);
     \draw (-1,-.5) -- (.5, -.5);
      \node  at  (-.25, -.75) {$1$};     
          \node  at  (-.75, -1.25) {$1$};   
          \node  at  (-.25, -1.25) {$2$};   
               \node  at  (-.75, -1.75) {$3$};   
              \node  at  (-.25, -2.35) {$w = 11213$};      
               \node  at  (-.55, -2.65) {$c = 1$}; 
       \end{tikzpicture}
       \end{subfigure}
      \hfill
           \begin{subfigure}{0.2 \textwidth}
             \centering
  \begin{tikzpicture}
      \draw[fill=gray!30]    (-1,0) -- (.5,0) -- (.5,-.5) -- (-1,-.5);
       \draw[fill=gray!30]    (-1,-.5) -- (-.5,-.5) -- (-.5,-1) -- (-1,-1);
   \draw (-1,0) -- (.5,0);
   \draw (-1,-2) -- (-1,0);
   \draw (-.5,0) -- (-.5,-2);
   \draw ( -1,  -2 ) --(-.5, -2 );
     \draw (.5,-.5) -- (.5, 0);
     \draw (0, -1.5) -- (0,  0);
     \draw (-1,-1.5) -- (0,  -1.5);
     \draw (-1, -1) -- (0, -1);
     \draw (-1,-.5) -- (.5, -.5);
             \node  at  (-.25, -.75) {$1$};  
             \node  at  (-.75, -1.25) {$2$};   
              \node  at  (-.25, -1.25) {$t$};   
              \node  at  (-.75, -1.75) {$3$};   
              \node  at  (-.25, -2.35) {$w = 11231$};      
               \node  at  (-.55, -2.65) {$c = t$}; 
       \end{tikzpicture}
       \end{subfigure}
       \hfill
           \begin{subfigure}{0.2 \textwidth}
             \centering
  \begin{tikzpicture}
     \draw[fill=gray!30]    (-1,0) -- (.5,0) -- (.5,-.5) -- (-1,-.5);
       \draw[fill=gray!30]    (-1,-.5) -- (-.5,-.5) -- (-.5,-1) -- (-1,-1);
   \draw (-1,0) -- (.5,0);
   \draw (-1,-2) -- (-1,0);
   \draw (-.5,0) -- (-.5,-2);
   \draw ( -1,  -2 ) --(-.5, -2 );
     \draw (.5,-.5) -- (.5, 0);
     \draw (0, -1.5) -- (0,  0);
     \draw (-1,-1.5) -- (0,  -1.5);
     \draw (-1, -1) -- (0, -1);
     \draw (-1,-.5) -- (.5, -.5);
              \node  at  (-.25, -.75) {$1$};  
             \node  at  (-.75, -1.25) {$1$};   
              \node  at  (-.25, -1.25) {$2$};   
              \node  at  (-.75, -1.75) {$3t$};   
              \node  at  (-.25, -2.35) {$w = 11213$};      
               \node  at  (-.55, -2.65) {$c = t$}; 
       \end{tikzpicture}
       \end{subfigure}
        \hfill
           \begin{subfigure}{0.2 \textwidth}
             \centering
  \begin{tikzpicture}
     \draw[fill=gray!30]    (-1,0) -- (.5,0) -- (.5,-.5) -- (-1,-.5);
       \draw[fill=gray!30]    (-1,-.5) -- (-.5,-.5) -- (-.5,-1) -- (-1,-1);
   \draw (-1,0) -- (.5,0);
   \draw (-1,-2) -- (-1,0);
   \draw (-.5,0) -- (-.5,-2);
   \draw ( -1,  -2 ) --(-.5, -2 );
     \draw (.5,-.5) -- (.5, 0);
     \draw (0, -1.5) -- (0,  0);
     \draw (-1,-1.5) -- (0,  -1.5);
     \draw (-1, -1) -- (0, -1);
     \draw (-1,-.5) -- (.5, -.5);
          \node  at  (-.25, -.75) {$1$};  
             \node  at  (-.75, -1.25) {$1$};   
              \node  at  (-.25, -1.25) {$3t$};   
              \node  at  (-.75, -1.75) {$2$};   
              \node  at  (-.25, -2.35) {$w = 11123$};      
               \node  at  (-.55, -2.65) {$c = t$}; 
       \end{tikzpicture}
       \end{subfigure}
   \end{figure}

        \begin{figure}[H]   
         \begin{subfigure}{0.2 \textwidth}
           \centering
  \begin{tikzpicture}
     \draw[fill=gray!30]    (-1,0) -- (.5,0) -- (.5,-.5) -- (-1,-.5);
       \draw[fill=gray!30]    (-1,-.5) -- (-.5,-.5) -- (-.5,-1) -- (-1,-1);
   \draw (-1,0) -- (.5,0);
   \draw (-1,-2) -- (-1,0);
   \draw (-.5,0) -- (-.5,-2);
   \draw ( -1,  -2 ) --(-.5, -2 );
     \draw (.5,-.5) -- (.5, 0);
     \draw (0, -1.5) -- (0,  0);
     \draw (-1,-1.5) -- (0,  -1.5);
     \draw (-1, -1) -- (0, -1);
     \draw (-1,-.5) -- (.5, -.5);
      \node  at  (-.25, -.75) {$2$};     
          \node  at  (-.75, -1.25) {$3$};   
          \node  at  (-.25, -1.25) {$t$};   
               \node  at  (-.75, -1.75) {$t$};   
              \node  at  (-.25, -2.35) {$w = 12311$};      
               \node  at  (-.55, -2.65) {$c = t^2$}; 
       \end{tikzpicture}
       \end{subfigure}
      \hfill
           \begin{subfigure}{0.2 \textwidth}
             \centering
  \begin{tikzpicture}
      \draw[fill=gray!30]    (-1,0) -- (.5,0) -- (.5,-.5) -- (-1,-.5);
       \draw[fill=gray!30]    (-1,-.5) -- (-.5,-.5) -- (-.5,-1) -- (-1,-1);
   \draw (-1,0) -- (.5,0);
   \draw (-1,-2) -- (-1,0);
   \draw (-.5,0) -- (-.5,-2);
   \draw ( -1,  -2 ) --(-.5, -2 );
     \draw (.5,-.5) -- (.5, 0);
     \draw (0, -1.5) -- (0,  0);
     \draw (-1,-1.5) -- (0,  -1.5);
     \draw (-1, -1) -- (0, -1);
     \draw (-1,-.5) -- (.5, -.5);
             \node  at  (-.25, -.75) {$1$};  
             \node  at  (-.75, -1.25) {$2$};   
              \node  at  (-.25, -1.25) {$t$};   
              \node  at  (-.75, -1.75) {$3t$};   
              \node  at  (-.25, -2.35) {$w = 11213$};      
               \node  at  (-.55, -2.65) {$c = t^2$}; 
       \end{tikzpicture}
       \end{subfigure}
       \hfill
           \begin{subfigure}{0.2 \textwidth}
             \centering
  \begin{tikzpicture}
     \draw[fill=gray!30]    (-1,0) -- (.5,0) -- (.5,-.5) -- (-1,-.5);
       \draw[fill=gray!30]    (-1,-.5) -- (-.5,-.5) -- (-.5,-1) -- (-1,-1);
   \draw (-1,0) -- (.5,0);
   \draw (-1,-2) -- (-1,0);
   \draw (-.5,0) -- (-.5,-2);
   \draw ( -1,  -2 ) --(-.5, -2 );
     \draw (.5,-.5) -- (.5, 0);
     \draw (0, -1.5) -- (0,  0);
     \draw (-1,-1.5) -- (0,  -1.5);
     \draw (-1, -1) -- (0, -1);
     \draw (-1,-.5) -- (.5, -.5);
              \node  at  (-.25, -.75) {$1$};  
             \node  at  (-.75, -1.25) {$2$};   
              \node  at  (-.25, -1.25) {$3t$};   
              \node  at  (-.75, -1.75) {$t$};   
              \node  at  (-.25, -2.35) {$w = 11231$};      
               \node  at  (-.55, -2.65) {$c = t^2$}; 
       \end{tikzpicture}
       \end{subfigure}
        \hfill
           \begin{subfigure}{0.2 \textwidth}
             \centering
  \begin{tikzpicture}
     \draw[fill=gray!30]    (-1,0) -- (.5,0) -- (.5,-.5) -- (-1,-.5);
       \draw[fill=gray!30]    (-1,-.5) -- (-.5,-.5) -- (-.5,-1) -- (-1,-1);
   \draw (-1,0) -- (.5,0);
   \draw (-1,-2) -- (-1,0);
   \draw (-.5,0) -- (-.5,-2);
   \draw ( -1,  -2 ) --(-.5, -2 );
     \draw (.5,-.5) -- (.5, 0);
     \draw (0, -1.5) -- (0,  0);
     \draw (-1,-1.5) -- (0,  -1.5);
     \draw (-1, -1) -- (0, -1);
     \draw (-1,-.5) -- (.5, -.5);
          \node  at  (-.25, -.75) {$1$};  
             \node  at  (-.75, -1.25) {$1$};   
              \node  at  (-.25, -1.25) {$2t$};   
              \node  at  (-.75, -1.75) {$3t$};   
              \node  at  (-.25, -2.35) {$w = 11123$};      
               \node  at  (-.55, -2.65) {$c = t^2$}; 
       \end{tikzpicture}
       \end{subfigure}
   \end{figure}

     \begin{figure}[H]   
         \begin{subfigure}{0.4 \textwidth}
           \centering
  \begin{tikzpicture}
     \draw[fill=gray!30]    (-1,0) -- (.5,0) -- (.5,-.5) -- (-1,-.5);
       \draw[fill=gray!30]    (-1,-.5) -- (-.5,-.5) -- (-.5,-1) -- (-1,-1);
   \draw (-1,0) -- (.5,0);
   \draw (-1,-2) -- (-1,0);
   \draw (-.5,0) -- (-.5,-2);
   \draw ( -1,  -2 ) --(-.5, -2 );
     \draw (.5,-.5) -- (.5, 0);
     \draw (0, -1.5) -- (0,  0);
     \draw (-1,-1.5) -- (0,  -1.5);
     \draw (-1, -1) -- (0, -1);
     \draw (-1,-.5) -- (.5, -.5);
      \node  at  (-.25, -.75) {$2$};     
          \node  at  (-.75, -1.25) {$1$};   
          \node  at  (-.25, -1.25) {$t$};   
               \node  at  (-.75, -1.75) {$3t$};   
              \node  at  (-.25, -2.35) {$w = 12113$};      
               \node  at  (-.55, -2.65) {$c = t^2$}; 
       \end{tikzpicture}
       \end{subfigure}
      \hfill
           \begin{subfigure}{0.4 \textwidth}
             \centering
  \begin{tikzpicture}
      \draw[fill=gray!30]    (-1,0) -- (.5,0) -- (.5,-.5) -- (-1,-.5);
       \draw[fill=gray!30]    (-1,-.5) -- (-.5,-.5) -- (-.5,-1) -- (-1,-1);
   \draw (-1,0) -- (.5,0);
   \draw (-1,-2) -- (-1,0);
   \draw (-.5,0) -- (-.5,-2);
   \draw ( -1,  -2 ) --(-.5, -2 );
     \draw (.5,-.5) -- (.5, 0);
     \draw (0, -1.5) -- (0,  0);
     \draw (-1,-1.5) -- (0,  -1.5);
     \draw (-1, -1) -- (0, -1);
     \draw (-1,-.5) -- (.5, -.5);
             \node  at  (-.25, -.75) {$2$};  
             \node  at  (-.75, -1.25) {$1$};   
              \node  at  (-.25, -1.25) {$3t$};   
              \node  at  (-.75, -1.75) {$t$};   
              \node  at  (-.25, -2.35) {$w = 12131$};      
               \node  at  (-.55, -2.65) {$c = t^2$}; 
       \end{tikzpicture}
       \end{subfigure}
   \end{figure}

           \begin{figure}[H]   
         \begin{subfigure}{0.2 \textwidth}
           \centering
  \begin{tikzpicture}
     \draw[fill=gray!30]    (-1,0) -- (.5,0) -- (.5,-.5) -- (-1,-.5);
       \draw[fill=gray!30]    (-1,-.5) -- (-.5,-.5) -- (-.5,-1) -- (-1,-1);
   \draw (-1,0) -- (.5,0);
   \draw (-1,-2) -- (-1,0);
   \draw (-.5,0) -- (-.5,-2);
   \draw ( -1,  -2 ) --(-.5, -2 );
     \draw (.5,-.5) -- (.5, 0);
     \draw (0, -1.5) -- (0,  0);
     \draw (-1,-1.5) -- (0,  -1.5);
     \draw (-1, -1) -- (0, -1);
     \draw (-1,-.5) -- (.5, -.5);
      \node  at  (-.25, -.75) {$1$};     
          \node  at  (-.75, -1.25) {$t$};   
          \node  at  (-.25, -1.25) {$2t$};   
               \node  at  (-.75, -1.75) {$3t$};   
              \node  at  (-.25, -2.35) {$w = 11213$};      
               \node  at  (-.55, -2.65) {$c = t^3$}; 
       \end{tikzpicture}
       \end{subfigure}
      \hfill
           \begin{subfigure}{0.2 \textwidth}
             \centering
  \begin{tikzpicture}
      \draw[fill=gray!30]    (-1,0) -- (.5,0) -- (.5,-.5) -- (-1,-.5);
       \draw[fill=gray!30]    (-1,-.5) -- (-.5,-.5) -- (-.5,-1) -- (-1,-1);
   \draw (-1,0) -- (.5,0);
   \draw (-1,-2) -- (-1,0);
   \draw (-.5,0) -- (-.5,-2);
   \draw ( -1,  -2 ) --(-.5, -2 );
     \draw (.5,-.5) -- (.5, 0);
     \draw (0, -1.5) -- (0,  0);
     \draw (-1,-1.5) -- (0,  -1.5);
     \draw (-1, -1) -- (0, -1);
     \draw (-1,-.5) -- (.5, -.5);
             \node  at  (-.25, -.75) {$t$};  
             \node  at  (-.75, -1.25) {$1$};   
              \node  at  (-.25, -1.25) {$2t$};   
              \node  at  (-.75, -1.75) {$3t$};   
              \node  at  (-.25, -2.35) {$w = 11123$};      
               \node  at  (-.55, -2.65) {$c = t^3$}; 
       \end{tikzpicture}
       \end{subfigure}
       \hfill
           \begin{subfigure}{0.2 \textwidth}
             \centering
  \begin{tikzpicture}
     \draw[fill=gray!30]    (-1,0) -- (.5,0) -- (.5,-.5) -- (-1,-.5);
       \draw[fill=gray!30]    (-1,-.5) -- (-.5,-.5) -- (-.5,-1) -- (-1,-1);
   \draw (-1,0) -- (.5,0);
   \draw (-1,-2) -- (-1,0);
   \draw (-.5,0) -- (-.5,-2);
   \draw ( -1,  -2 ) --(-.5, -2 );
     \draw (.5,-.5) -- (.5, 0);
     \draw (0, -1.5) -- (0,  0);
     \draw (-1,-1.5) -- (0,  -1.5);
     \draw (-1, -1) -- (0, -1);
     \draw (-1,-.5) -- (.5, -.5);
              \node  at  (-.25, -.75) {$2$};  
             \node  at  (-.75, -1.25) {$t$};   
              \node  at  (-.25, -1.25) {$t$};   
              \node  at  (-.75, -1.75) {$3t$};   
              \node  at  (-.25, -2.35) {$w = 12113$};      
               \node  at  (-.55, -2.65) {$c = t^3$}; 
       \end{tikzpicture}
       \end{subfigure}
        \hfill
           \begin{subfigure}{0.2 \textwidth}
             \centering
  \begin{tikzpicture}
     \draw[fill=gray!30]    (-1,0) -- (.5,0) -- (.5,-.5) -- (-1,-.5);
       \draw[fill=gray!30]    (-1,-.5) -- (-.5,-.5) -- (-.5,-1) -- (-1,-1);
   \draw (-1,0) -- (.5,0);
   \draw (-1,-2) -- (-1,0);
   \draw (-.5,0) -- (-.5,-2);
   \draw ( -1,  -2 ) --(-.5, -2 );
     \draw (.5,-.5) -- (.5, 0);
     \draw (0, -1.5) -- (0,  0);
     \draw (-1,-1.5) -- (0,  -1.5);
     \draw (-1, -1) -- (0, -1);
     \draw (-1,-.5) -- (.5, -.5);
          \node  at  (-.25, -.75) {$t$};  
             \node  at  (-.75, -1.25) {$t$};   
              \node  at  (-.25, -1.25) {$2t$};   
              \node  at  (-.75, -1.75) {$3t$};   
              \node  at  (-.25, -2.35) {$w = 11213$};      
               \node  at  (-.55, -2.65) {$c = t^4$}; 
       \end{tikzpicture}
       \end{subfigure}
   \end{figure}

\section{Type $A$ Quantum Lattice Heisenberg Algebra, $\h_A$} \label{QuaLatHei}

We now give an explicit description of $\h_A.$ 
 Fix an orientation $\epsilon$ of the type $A$ Dynkin diagram.
 Denote $I_A$ to be the set of vertices in  the type $A$ Dynkin diagram.
 Given two vertices $i,j,$ we set
\[ \epsilon_{ij} =
\begin{cases}
                                   0 & \text{if $i,j$ are not connected by an edge;} \\
   1 & \text{if $i,j$ are connected by an edge with tail $i$ and head $j$ agree with the orientation $\epsilon$; } \\
   -1 & \text{if $i,j$ are connected by an edge with tail $i$ and head $j$ disagree with the orientation $\epsilon$. }
  \end{cases} 
\]
It is thus apparent that $\epsilon_{ij} = - \epsilon_{ji}.$

\begin{comment}
Let $\Gamma$ be a finite subgroup in $ SL_2(\F).$ 
   Suppose $V$ is the standard two dimensional representation of $\Gamma.$
   The McKay correspondence says that the finite subgroup $\Gamma$ corresponds to a Dynkin diagram with vertex set $I_\Gamma$ and edge set $E_\gamma.$
   Each vertex $i \in I_\Gamma$ is labelled by an non-trivial irreducible representation $V_i$ of $\Gamma.$
   If $V_j$ appears $r$ times as a direct summand in the decomposition of the tensor product $V \otimes V_i,$ then the two vertices $i,j$ are joined by $r$ edges.

\end{comment}

        Let      $[r+1] = t^{-r} + t^{-r+2} + \cdots + t^{r-2} + t^r = \frac{t^{-r}-t^{r}}{t^{-1}-t}$ be the shifted quantum integer. 
We define the \textit{type $A$ quantum Heisenberg algebra} $\h_A$ to be the unital $\C[t, t^{-1}]$-algebra with generators $p_i^{(n)} ,q_i^{(n)} $ for $i \in I_\Gamma,  n \in \N_{\geq 0}$ and relations
\[ \begin{cases} 
      q_i^{(n)} p_i^{(m)} = \sum_{r \geq 0} [r+1] p_i^{(m-r)} q_i^{(n-r)}, \\
      q_i^{(n)} p_j^{(m)} = p_j^{(m)} q_i^{(n)} + p_j^{(m-1)} q_i^{(n-1)} , & \text{if $j = i \pm 1$,} \\
      q_i^{(n)} p_j^{(m)} = p_j^{(m)} q_i^{(n)} , & \text{if $|i - j| > 1$,} \\
      p_i^{(n)} p_j^{(m)} = p_j^{(m)} p_i^{(n)}, \\
      q_i^{(n)} q_j^{(m)} = q_j^{(m)} q_i^{(n)}.
   \end{cases}
\]
  By convention, $p_i^{(0)} = q_i^{(0)} = 1$ is the identity element in  $\h_A$.

Alternatively,  $\h_\Gamma$ can be generated by $a_i(r),$ for $i \in I_\Gamma$ and $r,s \in \Z-0,$ satisfying

$$[a_i(r), a_j(s)] = \delta_{r,-s}[s\<i,j\>] \frac{[s]}{s}.$$

\noindent For any partition $\lambda \vdash n,$ we can define $p_i^\lambda$ using generating functions involving $a_i(r).$
  For example,
  
 $$ exp \left( \sum_{r \geq 1} \frac{a_i(-r)}{[r]} z^r \right) = \sum_{s \geq 0}  p_i^{(s)} z^s \text{  and  } exp \left( \sum_{r \geq 1} \frac{a_i(r)}{[r]} z^r \right) = \sum_{s \geq 0}  q_i^{(s)} z^s,$$ 
 
  $$ exp \left( - \sum_{r \geq 1} \frac{a_i(-r)}{[r]} z^r \right) = \sum_{s \geq 0}  (-1)^s p_i^{(1^s)} z^s \text{  and  } exp \left( - \sum_{r \geq 1} \frac{a_i(r)}{[r]} z^r \right) = \sum_{s \geq 0}  (-1)^s q_i^{(1^s)} z^s.$$ 
 
\noindent In general, for partition $\lambda = (\lambda_1,  \lambda_2, \hdots,  \lambda_t) = (1^{m_1}2^{m_2} \cdots)$, we first define $a_{i,+}^\lambda = a_i(\lambda_1) a_i(\lambda_2) \cdots a_i(\lambda_s), $ and similarly $ a_{i,-}^\lambda = a_i(-\lambda_1) a_i(-\lambda_2) \cdots a_i(-\lambda_s).$  
  Let $ [z]_\lambda = [1]^{m_1} m_1! [2]^{m_2} m_2! \cdots, $ and $h_{i, \pm}^n = \sum_{\lambda \vdash n}  [z]_\lambda^{-1} a_{i,\pm}^\lambda.$
We define 
$$ p^\lambda_i = det(h^{\lambda_j-k+l}_{i,-})^t_{k,l=1}, \text{ and } q^\lambda_i = det(h^{\lambda_j-k+l}_{i,+})^t_{k,l=1}.$$

\noindent Note that this resembles the Jacobi-Trudi identity for expressing Schur functions in terms of complete homogeneous symmetric function. 
 In fact, as pointed out in \cite{ErikLi}, there is an isometric isomorphism between the Heisenberg algebra generators $a_i(r)$ and the power sum symmetric functions.
  Combining with the transition matrix between complete homogeneous symmetric function, we obtain the formula for the Schur functions. 
  We could have defined $p^\lambda_i,q^\lambda_i$ using the character table of the symmetric group which serves as the transition matrix between the Schur functions and the power symmetric functions as explained in \cite{Mac}.
  However, we found that the definitions given above made the required coefficients apparent and therefore, are practical for computations. 

   Consequently, deriving from \cite{ErikLi}, we will then have the commuting relation between $q_i^{\lambda}$ and $p_j^{\mu}$ :
$$ \begin{cases} 
      q_i^{(\mu)} p_i^{(\lambda)} = \sum_{\nu \subseteq \lambda, \kappa \subseteq \mu, |\lambda|- |\nu| = |\mu| - |\kappa|}  \fC^{\mu \lambda}_{\nu \kappa}(t) p_i^{(\nu)} q_i^{(\kappa)}, \\
      q_i^{(\mu)} p_j^{(\lambda)} = \sum_{\nu \subseteq \lambda^t, \kappa \subseteq \mu, |\lambda|- |\nu| = |\mu| - |\kappa|}  C^{\mu \lambda^t}_{\nu \kappa} p_i^{(\nu)} q_i^{(\kappa)}, & \text{if $j = i \pm 1$,} \\
      q_i^{(\mu)} p_j^{(\lambda)} = p_j^{(\lambda)} q_i^{(\mu)} , & \text{if $|i - j| > 1$,} \\
      p_i^{(\mu)} p_j^{(\lambda)} = p_j^{(\lambda)} p_i^{(\mu)}, \\
      q_i^{(\mu)} q_j^{(\lambda)} = q_j^{(\lambda)} q_i^{(\mu)}, 
   \end{cases}
$$

\noindent where 
\begin{align}
C^{\mu \lambda}_{\nu \kappa} &= \sum_{\lambda_T \in \fl \fp(\lambda- \nu, \mu- \kappa)} 1 \text{    , and    } \\
\fC^{\mu \lambda}_{\nu \kappa}(t) &= \sum_{ \lambda_\fT \in t \text{-}\fl \fp(\lambda- \nu, \mu- \kappa)} t^{-(|\lambda|-|\nu|)}c(\lambda_\fp)(t^2).
\end{align}

\begin{comment}
   In fact,  we'll be working on the idempotent modification of $\h_\Gamma$ where the unit $1$ is replaced by a collection of orthogonal idempotents $\{1_k\}_{k \in \Z}.$
   Let $\lambda$ be a partition of a natural number $k \in \N_{\geq 0},$ so that the size $|d|$ of the partition is $k.$
   We impose the following relation 
   $$ 1_{r+n} p_i^{\lambda} = 1 _{r+n} p_i^{\lambda} 1_r =  p_i^{\lambda} 1_r $$
 and
  $$ 1_{r-n} q_i^{\lambda} = 1 _{r-n} q_i^{\lambda} 1_r =  q_i^{\lambda} 1_r.  $$
This extends to the remaining defining relations in $\h_\Gamma.$
     Any $\F$-algebra with a collection of idempotents is a $\F$-linear category.  
     This implies that $\h_\Gamma$ is a $\F[t,t^{-1}]$-linear category with its object being the integers,  and the morphisms from $r$ to $s$ being the $\F[t,t^{-1}]$-module $1_s \h_\Gamma 1_r.$
     
\end{comment}

         \section{Fock Space Representation of $\h_A$} \label{Fockrep}
     
In this section, we recall the main construction of Fock space representation in \cite{AliLi}.

 Consider the subalgebra $\h_A^{-} \subset \h_A$ generated by the $q_i^{(n)}$ for all $i \in I_A, n \geq 0.$
 If $\h_A^{+}$ is the subalgebra of $\h_A$ generated by the $p_i^{n}, n \geq 0,$ then $\h_A \cong \h_A^{+} \otimes \h_A^{-}$ as vector spaces.
   Let $triv_0$ be the trivial representation of $\h^{-}_A,$ where $1$ acts as the identity and $q_i^{(n)}$ acts by $0$ for $n > 0.$
   Then,  the $\h_A$-module
   
    $$\cF_A := \text{Ind}^{\h_A}_{\h^-_A}(triv_0) = \h_A \otimes_{\h^-_A} triv_0 $$
    
 \noindent   is the \textit{Fock space representation} of $\h_A.$
    Let $v_0$ be the vacuum vector in the Fock space where $q_i^{(n)} \cdot v_0 = 0$ for $n >0.$
    This will help to distinguish between the elements of $\h_A$ and its representation $\cF_A$.
    
    Note that the only relation that holds in $\h_A^{-}$ and $\h_A^{+}$ respectively is the abelian multiplication of generators.
	One can readily see that the Fock space $\cF_\Gamma$ is naturally isomorphic to the space of polynomials in the commuting variables $\{ p^{(n)}_i v_0 \}_{i \in I_\Gamma}.$

   \section{Braid Operator in the Fock Space Representation} \label{braidop}

   We define an integer statistic $d_\lambda$ associated to a partition $\lambda$ as follows: if $\lambda' \subset \lambda$ is the largest $(r \times r)$ square, also known as the \textit{Durfee square}, whose Young diagram fits properly inside of $\lambda,$ then the complement $\lambda \setminus \lambda'$ is a skew partition with the top horizontal component $\alpha$ and the bottom vertical component $\beta$ and we set 
   $$ d_\lambda := |\alpha| - |\beta| + r .$$
   
   For example,  let $\lambda = (5,5,3,3,1,1)$ be the following Young diagram. 
   
   \begin{figure}[H]  
   \centering
   \begin{tikzpicture}
   \draw (-1,0) -- (1.5,0);
   \draw (-1,-3) -- (-1,0);
   \draw (-.5,0) -- (-.5,-3);
    \draw (-1,-3) -- (-.5,-3);
     \draw (-1,-2.5) -- (-.5,-2.5);
     \draw ( -1,  -2 ) --(.5, -2 );
     \draw (.5,-2) -- (.5, 0);
     \draw (0, -2) -- (0,  0);
     \draw (-1,-1.5) -- (.5,  -1.5);
     \draw (-1, -1) -- (1.5, -1);
     \draw (-1,-.5) -- (1.5, -.5);
     \draw (1,0) -- (1,-1);
     \draw (1.5,0) -- (1.5,-1);
     \fill[gray!100, nearly transparent] (-1,0) -- (0.5,0) -- (.5,-1.5) -- (-1,-1.5) -- cycle;
     \fill[yellow!100, nearly transparent] (.5,0) -- (1.5,0) -- (1.5,-1) -- (.5,-1) -- cycle;
     \fill[blue!100, nearly transparent] (-1,-1.5) -- (.5,-1.5) -- (.5,-2) -- (-.5,-2) -- (-.5,-3) -- (-1,-3) -- cycle;
     \node at (-.25,0.25) {$3$};
       \node at (-1.25,-0.75) {$3$};
       \node at (-1.25,  -2.25) {$\beta$};
       \node at (1.75, -.5) {$\alpha$};
       
      \node at  (5, -1) {$d_{(5,5,3,3,1,1)}$};
       \node at  (5.6, -1.5) {$= |(2,2)| - |(3,1,1)| + 3$};
      \node at  (4.6, -2) {$= 4- 5 + 3$};
        \node at  (4, -2.5) {$= 2$};
   \end{tikzpicture}
   \end{figure}  
   
   We will concentrate on the Heisernberg algebra $\h_{A_{2n-1}}$ associated to the type $A_{2n-1}$ Dynkin diagram.  The type $A_{2n-1}$ Dynkin diagram has its vertices labelled by $\{ 1,  2, \hdots, 2n-1 \}$ and only $i$ and $i+1$ vertices are connected by one edge. 
    Then, we consider the Fock space representation $F_{A_{2n-1}}$ of $\h_{A_{2n-1}}.$
     It is known that the Fock space $F_{A_{2n-1}}$ is a graded vector space which is isomorphic to $\oplus_{m \in \N_{\geq 0}} F_{A_{2n-1}}^m,$ with 
    $$F_{A_{2n-1}}^m = \text{ span }  \left\{p_{1}^{\lambda_1} p_{2}^{\lambda_2} \cdots p_{{2n-1}}^{\lambda_{2n-1}}v_0 \middle\vert \  \sum_i |\lambda_i| = m \right\}.$$
 
 For each $m \in \N_{\geq 0}$ $i \in I_\Gamma,$ we define the distinguished element   $\varsigma_{i,m}  \in \h_{A_{2n-1}}$ to be
 
  $$\varsigma_{i,m} := \sum_{ \{ \lambda :  0 \leq |\lambda| \leq m \} } (-1)^{d_\lambda}t^{-d_\lambda}p_i^{\lambda}q_i^{\lambda^T}$$
  
\noindent where $\lambda^T$ denotes the transpose of $\lambda$ and the other distinguished element $(\varsigma_i^m)' \in \h_{A_{2n-1}}$ to be
  
$$\varsigma_{i,m}':= \sum_{ \{ \lambda :  0 \leq |\lambda| \leq m,  \lambda = \lambda^T \} } (-1)^{d_\lambda}t^{d_\lambda}p_i^{\lambda}q_i^{\lambda^T} + \sum_{ \{ \lambda :  0 \leq | \lambda | \leq m,  \lambda \neq \lambda^T \} } (-1)^{d_\lambda +2}t^{d_\lambda+2}p_i^{\lambda}q_i^{\lambda^T} $$

\noindent Note that $\varsigma_{i,m} \in \h_{A_{2n-1}}$  can be expressed as the truncation of the generating function 

$$P(x,y,t) = 1+ \sum^\infty_{a=1} \frac{(xyt)^a}{\prod^a_{b=1}(1-(xyt)^b)(1-(xyt^{-1})^b)} $$

\noindent where the $P$ is a variant of Euler partition function enumerated based on Durfee square of partitions. 
     
    Next, we introduce the braid operators $\Sigma_{i,m},$ $\Sigma_{i,m}'$  $: F^m_{A_{2n-1}} \ra F^m_{A_{2n-1}} $   of multiplying with $\varsigma_{i,m},  \varsigma_{i,m}'$ respectively: 
    $$ \Sigma_{i,m}(x) = \varsigma_{i,m} \cdot x,  \hspace{5mm}  \Sigma_{i,m}'(x) = \varsigma_{i,m}' \cdot x,  \hspace{5mm} x \in F^m_{A_{2n-1}}$$

\begin{conjecture} 
These braid operators define a braid group action representation on the Fock space.

\end{conjecture}

\begin{proof}[Strategy of Proof]\renewcommand{\qedsymbol}{}

Prove the following identity 

$$ \left( \sum_{ \{ \lambda :  0 \geq |\lambda| \leq n \} } (-1)^{d_\lambda}t^{-d_\lambda}p_i^{\lambda}q_i^{\lambda^T} \right) \left(  \sum_{ \{ \lambda :  0 \leq |\lambda| \leq n,  \lambda = \lambda^T \} } (-1)^{d_\lambda}t^{d_\lambda}p_i^{\lambda}q_i^{\lambda^T} + \sum_{ \{ \lambda :  0 \leq | \lambda | \leq n,  \lambda \neq \lambda^T \} } (-1)^{d_\lambda +2}t^{d_\lambda+2}p_i^{\lambda}q_i^{\lambda^T} \right)= id$$

$$ \left( \sum_{ \{ \lambda :  0 \geq |\lambda| \leq n \} } (-1)^{d_\lambda}t^{-d_\lambda}p_i^{\lambda}q_i^{\lambda^T} \right)  \left( \sum_{ \{ \lambda :  0 \geq |\lambda| \leq n \} } (-1)^{d_\lambda}t^{-d_\lambda}p_{i+1}^{\lambda}q_{i}^{\lambda^T} \right) \left( \sum_{ \{ \lambda :  0 \geq |\lambda| \leq n \} } (-1)^{d_\lambda}t^{-d_\lambda}p_i^{\lambda}q_i^{\lambda^T} \right) \\$$
$$ =  \left( \sum_{ \{ \lambda :  0 \geq |\lambda| \leq n \} } (-1)^{d_\lambda}t^{-d_\lambda}p_{i+1}^{\lambda}q_{i}^{\lambda^T} \right)  \left( \sum_{ \{ \lambda :  0 \geq |\lambda| \leq n \} } (-1)^{d_\lambda}t^{-d_\lambda}p_i^{\lambda}q_i^{\lambda^T} \right)  \left( \sum_{ \{ \lambda :  0 \geq |\lambda| \leq n \} } (-1)^{d_\lambda}t^{-d_\lambda}p_{i+1}^{\lambda}q_{i}^{\lambda^T} \right) $$

We first realize the braid operators as generating functions.
 Then, we require generating functions for operation on (generalised) Littlewood-Richardson rule.
 Next, we probably need appropriate basic hypergeometric series identity, partition identity, or $q$-series identity (might be from Ramanujan) to help us do cancellation. 
  Do the maths and find out the unnecessary terms cancelled nicely. 
  The author has checked the braid relations to $n=3$ case by brute force.

\end{proof}

\section{Crossingless Matching to Algebra} \label{Fockinv}
     
Recall that a \textit{crossingless matching} $c$ on a three-sphere $\mathbb{S}_3$ consist of multiple curves $c_r$ with  endpoints up to ambient isotopy. Given a crossingless matching $c$ of $2n$ points,  we construct a graded Fock vector $v_c \in F^n_{A_{2n-1}}$ as follows.

  \subsection{Crossingless Matching to Cross Diagram} 
  
  In this section, we introduce some essential terminology and describe steps coverting crossingless matchings to cross diagrams.

   A crossingless matching is called \textit{decomposable} if there exists a vertical line $v_i$ that doesn't intersect any curves.
   Subsequently, a decomposable crossingless matching is partitioned into disjoint parts. 
   A \textit{subcrossingless matching} is a crossingless matching obtained by successively removing the curves with the largest difference of endpoints (it is possible to have more than one curve of such property) or curves in the other disjoint parts of a decomposable crossingless matching. 
     A crossingless matching is called \textit{nested} if it contains a decomposable subcrossingless matching, but it is not decomposable.

\begin{enumerate}
\item Enumerate the crossingless matching $c$ starting with curve possessing the least difference of endpoints from left to right by \textcolor{orange}{$ c_1,  \hdots,  c_n $}. 
\item Draw $2n - 1$ vertical lines \textcolor{blue}{$v_i$} that bisect perpendicularly  the horizontal line connecting $i$-th and $(i$+$1)$-th vertices from left to right,  for $i \in \{ 1, \hdots, 2n-1 \}$. 
 Note that every curve in crossingless matching intersects the vertical lines in odd number $2m+1.$
 Then, we define $n-m$ to be the \textit{level} of a curve. 
\item We shall label each intersection points  \textcolor{orange}{$c_r$}$\cap$\textcolor{blue}{$v_i$}  with a triple $($\textcolor{orange}{$r$},\textcolor{blue}{$i$},\textcolor{red}{$\ell_{r,i}$}$)$ where \textcolor{red}{$\ell_{r,i}$} $\in \N_{\geq 0}$ enumerates intersections on each \textcolor{blue}{$v_i$} from top to bottom.
   This intersection points are called the \textit{distinguished crosses}.
 \item  A \textit{cross diagram }is a set of distinguished crosses sitting in $\N \times \N$ lattice.
 Replace  the multicurves diagram (e.g. \cref{c}) with a cross diagram (e.g. \cref{d}) labelled by the triples $($\textcolor{orange}{$r$},\textcolor{blue}{$i$},\textcolor{red}{$\ell_{r,i}$}$)$;
 a distinguished cross $(x,y)$ is picked for each intersection point in a way that $x$ is the level of the curve where it lies and $y$ is subscript $i$ of the vertical line it intersects. 
 Note that the crosses are arranged in a floating pyramid-like arrangement reflecting their relative positions in the original crossingless matching. 
 
 \end{enumerate}

%
 %       We think of the cross diagram as living in a $\N \times \N$ lattice.
 %    Then, the process of replacing the multicurves diagram with a cross diagram picks up distinguished points on the lattice.
  %   These distinguished points are then referred as the crosses.
 %    Curves are then translated to \textit{lattice segment} in the lattice.
  %   See \cref{c} and \cref{d} for an example.

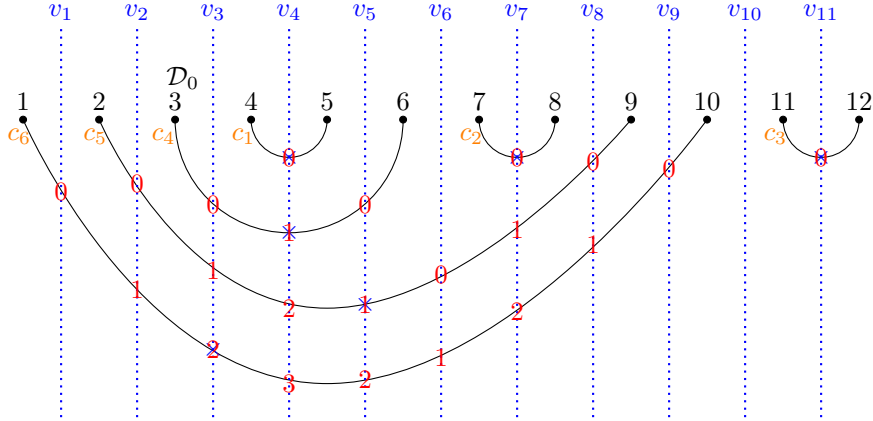
\begin{figure}[H]  
\centering
\begin{tikzpicture} 

\draw[fill] (-5.5,0) circle [radius=0.045];
\draw[fill] (-4.5,0) circle [radius=0.045];
\draw[fill] (-3.5,0) circle [radius=0.045];
\draw[fill] (-2.5,0) circle [radius=0.045];
\draw[fill] (-1.5,0) circle [radius=0.045];
\draw[fill] (-0.5,0) circle [radius=0.045];
\draw[fill] (0.5,0) circle [radius=0.045];
\draw[fill] (1.5,0) circle [radius=0.045];
\draw[fill] (2.5,0) circle [radius=0.045];
\draw[fill] (3.5,0) circle [radius=0.045];
\draw[fill] (4.5,0) circle [radius=0.045];
\draw[fill] (5.5,0) circle [radius=0.045];

\draw (-1.5,0) arc(0:-180:.5) ;
\draw (-0.5,0) arc(0:-180:1.5) ;
\draw (1.5,0) arc(0:-180:.5) ;
\draw plot[smooth, tension=1]coordinates {(2.5,0) (-1.5,-2.5)(-4.5,0)};
\draw plot[smooth, tension=1]coordinates {(3.5,0) (-1.5,-3.5)(-5.5,0)};
\draw (5.5,0) arc(0:-180:.5) ;

%vertical line
\draw [thick,blue,dotted] (-5, 1.2) -- (-5, -4);
\draw [thick,blue,dotted] (-4, 1.2) -- (-4, -4);
\draw [thick,blue,dotted](-3, 1.2) -- (-3, -4);
\draw [thick, blue,dotted] (-2, 1.2) -- (-2, -4);
\draw [thick,blue,dotted](-1, 1.2) -- (-1, -4);
\draw [thick,blue,dotted](0, 1.2) -- (0, -4);
\draw [thick,blue,dotted] (5, 1.2) -- (5, -4);
\draw [thick,blue,dotted] (4, 1.2) -- (4, -4);
\draw [thick,blue,dotted](3, 1.2) -- (3, -4);
\draw [thick, blue,dotted] (2, 1.2) -- (2, -4);
\draw [thick,blue,dotted](1, 1.2) -- (1, -4);

\node at (-3.4,0.55) {$\cD_0$};

\node [above,blue] at (-5, 1.2) {$v_1$};
\node  [above,blue] at (-4, 1.2) {$v_2$} ;
\node [above, blue] at (-3, 1.2) {$v_3$};
\node [above, blue] at (-2, 1.2) {$v_4$};
\node [above,  blue] at (-1, 1.2) {$v_5$};
\node [above,  blue] at (0, 1.2) {$v_6$};
\node  [above,  blue] at (1,1.2) {$v_7$} ;
\node [above,  blue] at (2,1.2) {$v_8$};
\node [above,  blue] at (3,1.2) {$v_9$};
\node  [above,  blue] at (4, 1.2) {$v_{10}$} ;
\node [above,  blue] at (5, 1.2) {$v_{11}$};

\node [above] at (-5.5, 0) {$1$};
\node  [above] at (-4.5, 0) {$2$} ;
\node [above] at (-3.5,0) {$3$};
\node [above] at (-2.5, 0) {$4$};
\node [above,] at (-1.5, 0) {$5$};
\node [above,] at (-0.5, 0) {$6$};
\node  [above] at (.5,0) {$7$} ;
\node [above] at (1.5,0) {$8$};
\node [above] at (2.5,0) {$9$};
\node  [above] at (3.5, 0) {$10$} ;
\node [above] at (4.5, 0) {$11$};
\node [above] at (5.5, 0) {$12$};

\node[blue] at (-2,-.5) {$\times$}	;
\node[blue] at (1,-.5) {$\times$};
\node[blue] at (5,-.5) {$\times$};
\node[blue] at (-2,-1.5) {$\times$};
\node[blue] at (-1,-2.45) {$\times$};
\node[blue] at (-3,-3.05) {$\times$};

\node[red] at (-2,-.5) {$0$};
\node[red] at (-2,-1.5) {$1$};
\node[red] at (-2,-2.5) {$2$};
\node[red] at (-2,-3.5) {$3$};
\node[red] at (-3,-1.125) {$0$};
\node[red] at (-3,-2.) {$1$};
\node[red] at (-3,-3.05) {$2$};
\node[red] at (-1,-1.125) {$0$};
\node[red] at (-1,-2.45) {$1$};
\node[red] at (-1,-3.45) {$2$};
\node[red] at (-4,-.85) {$0$};
\node[red] at (-4,-2.25) {$1$};
\node[red] at (-5,-.95) {$0$};
\node[red] at (0,-2.05) {$0$};
\node[red] at (0,-3.15) {$1$};
\node[red] at (1,-1.45) {$1$};
\node[red] at (1,-2.55) {$2$};
\node[red] at (1,-.5) {$0$};
\node[red] at (2,-.55) {$0$};
\node[red] at (2,-1.65) {$1$};
\node[red] at (3,-.65) {$0$};
\node[red] at (5,-.5) {$0$};

\node [right, below, orange] at (-5.55,0) {$c_{6}$};
\node [right, below, orange] at (-4.55,0) {$c_{5}$};
\node [right, below, orange] at (-3.65,0) {$c_{4}$};
\node [right, below, orange] at (-2.6,0) {$c_{1}$};
\node [right, below, orange] at (0.4,0) {$c_{2}$};
\node [right, below, orange] at (4.4,0) {$c_{3}$};

%Horizontal line

\end{tikzpicture}
\caption{An example of crossingless matching c in a $2n$-punctured $3$-sphere.} \label{c}
\end{figure}
\begin{figure}[H]  \label{crossdiagram}
\centering
\begin{tikzpicture}  
\node [blue]  at (-4, -.5) {X};
\node [blue] at (-4, -1.5) {X};
\node  at (-4, -2.5) {X};
\node  at (-4, -3.5) {X};
\node  at (-3, -1.5) {X};
\node  [blue] at (-3, -2.5) {X};
\node  at (-3, -3.5) {X};
\node  at (-2, -2.5) {X};
\node  at (-2, -3.5) {X};
\node  at (-5, -1.5) {X};
\node  at (-5, -2.5) {X};
\node [blue] at (-5, -3.5) {X};
\node  at (-6, -2.5) {X};
\node  at (-6, -3.5) {X};
\node  at (-7, -3.5) {X};
\node [blue] at (-1, -.5) {X};
\node  at (-1, -2.5) {X};
\node  at (-1, -3.5) {X};
\node  at (0, -2.5) {X};
\node  at (0, -3.5) {X};
\node  at (1, -3.5) {X};
\node  [blue] at (3, -.5) {X};

\node [below]  at (-4, -.5) {(\textcolor{orange}{$1$},\textcolor{blue}{$4$},\textcolor{red}{$0$})} ;
\node [below] at (-4, -1.5) {(\textcolor{orange}{$4$},\textcolor{blue}{$4$},\textcolor{red}{$1$})};
\node [below] at (-4, -2.5) {(\textcolor{orange}{$5$},\textcolor{blue}{$4$},\textcolor{red}{$2$})};
\node  [below] at (-4, -3.5) {(\textcolor{orange}{$6$},\textcolor{blue}{$4$},\textcolor{red}{$3$})};
\node [below] at (-3, -1.5) {(\textcolor{orange}{$4$},\textcolor{blue}{$5$},\textcolor{red}{$0$})};
\node [below] at (-3, -2.5) {(\textcolor{orange}{$5$},\textcolor{blue}{$5$},\textcolor{red}{$1$})};
\node  [below] at (-3, -3.5) {(\textcolor{orange}{$6$},\textcolor{blue}{$5$},\textcolor{red}{$2$})};
\node [below] at (-2, -2.5) {(\textcolor{orange}{$5$},\textcolor{blue}{$6$},\textcolor{red}{$0$})};
\node [below] at (-2, -3.5) {(\textcolor{orange}{$6$},\textcolor{blue}{$6$},\textcolor{red}{$1$})};
\node [below]  at (-5, -1.5) {(\textcolor{orange}{$4$},\textcolor{blue}{$3$},\textcolor{red}{$0$})};
\node [below] at (-5, -2.5) {(\textcolor{orange}{$5$},\textcolor{blue}{$3$},\textcolor{red}{$1$})};
\node [below] at (-5, -3.5) {(\textcolor{orange}{$6$},\textcolor{blue}{$3$},\textcolor{red}{$2$})};
\node [below] at (-6, -2.5) {(\textcolor{orange}{$5$},\textcolor{blue}{$2$},\textcolor{red}{$0$})};
\node [below] at (-6, -3.5) {(\textcolor{orange}{$6$},\textcolor{blue}{$2$},\textcolor{red}{$1$})};
\node [below] at (-7, -3.5) {(\textcolor{orange}{$6$},\textcolor{blue}{$1$},\textcolor{red}{$0$})};
\node [below] at (-1, -.5) {(\textcolor{orange}{$2$},\textcolor{blue}{$7$},\textcolor{red}{$0$})};
\node [below] at (-1, -2.5) {(\textcolor{orange}{$5$},\textcolor{blue}{$7$},\textcolor{red}{$1$})};
\node [below] at (-1, -3.5) {(\textcolor{orange}{$6$},\textcolor{blue}{$7$},\textcolor{red}{$2$})};
\node [below] at (0, -2.5) {(\textcolor{orange}{$5$},\textcolor{blue}{$8$},\textcolor{red}{$0$})};
\node [below] at (0, -3.5) {(\textcolor{orange}{$6$},\textcolor{blue}{$8$},\textcolor{red}{$1$})};
\node [below]  at (1, -3.5) {(\textcolor{orange}{$6$},\textcolor{blue}{$9$},\textcolor{red}{$0$})};
\node [below] at (3, -.5) {(\textcolor{orange}{$3$},\textcolor{blue}{$11$},\textcolor{red}{$0$})};

%bottom big square
\draw [dashed, purple] (-5.5, -4.25) -- (-5.5, -1.25);
\draw [dashed, purple] (-2.5, -4.25) -- (-2.5, -1.25);
\draw [dashed, purple] (-5.5, -4.25) -- (-2.5, -4.25);
\draw [dashed, purple] (-5.5, -1.25) -- (-2.5, -1.25);

%three small squares
\draw [dashed, purple] (-4.5, -2.15) -- (-4.5, 0);
\draw [dashed, purple] (-1.5, -1.15) -- (-1.5, 0);
\draw [dashed, purple] (2.35, -1.15) -- (2.35, 0);

\draw [dashed, purple] (-3.5, -2.15) -- (-3.5, 0);
\draw [dashed, purple] (-.5, -1.15) -- (-.5, 0);
\draw [dashed, purple] (3.65, -1.15) -- (3.65, 0);

\draw [dashed, purple] (-4.5, -2.15) -- (-3.5, -2.15);
\draw [dashed, purple] (-1.5, -1.15) -- (-.5, -1.15);
\draw [dashed, purple] (2.35, -1.15) -- (3.65, -1.15);

\draw [dashed, purple] (-4.5, 0) -- (-3.5, 0);
\draw [dashed, purple] (-1.5, 0) -- (-.5, 0);
\draw [dashed, purple] (2.35, 0) -- (3.65, 0);
\end{tikzpicture}
\caption{A cross diagram for the example with counting fence in purple dashed line.} \label{d}
\end{figure}
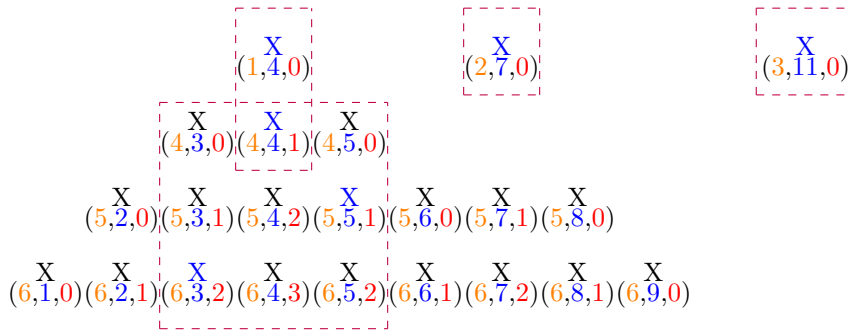

     A \textit{counting rectangle} is a rectangle containing either complete rows of distinguished crosses or empty rows of crosses.
     The rows of distinguished crosses contained in a counting rectangle are called the \textit{lattice segments}.
     Given a cross diagram and a choice of $b$ marked distinguished crosses,  we can form the (possibly empty) smallest counting rectangle containing the $b$ marked distinguished crosses, or marked crosses, for simplicity. 
     A lattice segment is \textit{effective} if it contains marked crosses and \textit{potential} if it doesn't.  
     Observe that a counting rectangle could have more potential lattice segments than effective lattice segments.
     When $b=1,$ the smallest counting rectangle is one-by-one grid by definition.
     Now, for $1 \leq a \leq b,$ we consider the set of all counting rectangles coming from possible $a$-wise marked crosses but remove those rectangles that are contained in a larger one. 
    This particular set of counting rectangles is called a counting fence.

In particular, we are interested in the set of sets of marked crosses where each of them comes from different curves in a crossingless matching $c$, denoted by $ \mathcal{X}_c $. 
   To give an overview of the remaining chapter, we provide our end goal here without going into the details.
   Details will be laid out in the future sections. 
  Every elements in $\cX_c$ will be realised as a vector in the Fock space via a map $\psi.$
  Finally, we will associate $c$ to a vector $v_c$ defined as
  
 \begin{equation}
  v_c = \sum_{   p_{i_{1}}^{\lambda_1}p_{i_{2}}^{\lambda_2} \hdots p_{i_{2n-1}}^{\lambda_{2n-1}} v_0 \in  \psi( \cX_c) } \fs \fp(p_{i_{1}}^{\lambda_1}p_{i_{2}}^{\lambda_2} \hdots p_{i_{2n-1}}^{\lambda_{2n-1}} v_0) \cdot p_{i_{1}}^{\lambda_1}p_{i_{2}}^{\lambda_2} \hdots p_{i_{2n-1}}^{\lambda_{2n-1}} v_0
 \end{equation}

\noindent  where $\fs \fp$ is called the \textit{spreading polynomial} of the vector in a given crossingless matching.
   The spreading polynomial  $\fs \fp$ lives in the Laurent ring $\Z[t,t^{-1}]$ and will be defined based on several combinatoric invariants.

\subsection{Cross Diagram to Fock Vector}

In this section, we will lay out the details of constructing the the spreading polynomial  $\fs \fp$, and subsequently, introducing a graded Fock vector for a crossingless matching. 
  At the end, we present some conjectures relating to the vectors we established.
 
 \begin{enumerate} \addtocounter{enumi}{4}
 
\item Given a crossingless matching $c$ with curves $c_1, c_2, \cdots, c_n$, pick exactly one intersection point from each curves \textcolor{orange}{$c_r$} in $c$. 
These are $n$ marked crosses each lying on distinct curves  with the label  $($\textcolor{orange}{$r$},\textcolor{blue}{$i$},\textcolor{red}{$\ell_{r,i}$}$)$.
 Consider the set $\cI_{i_1, \hdots,  i_{n}}$ containing  them:

\begin{center}
$  \cI_{i_1, i_2 \hdots,  i_{n}} := \{$ $($\textcolor{orange}{$1$},\textcolor{blue}{$i_1$},\textcolor{red}{$\ell_{1,i}$}$)$
$($\textcolor{orange}{$2$},\textcolor{blue}{$i_2$},\textcolor{red}{$\ell_{2,i}$}$)$, $\cdots$
$($\textcolor{orange}{$n$},\textcolor{blue}{$i_n$},\textcolor{red}{$\ell_{n,i}$}$)$
 $\mid \text{ $1 \leq i_r \leq 2n-1$, $ 0 \leq \ell \leq n-1$} \},$
\end{center}  

\noindent which we call a marked cross configuration.
     We associate a  Fock vector  ${p_{i_{1}}} p_{i_{2}} \cdots p_{i_{n}} v_0$ to each set $\cI_{i_1, i_2, \cdots,  i_{n}} .$ 
 
 For example,  the two set 
 $I_{4,7,11,4,5,3} = \{($\textcolor{orange}{$1$},\textcolor{blue}{$4$},\textcolor{red}{$0$}$)$,$($\textcolor{orange}{$2$},\textcolor{blue}{$7$},\textcolor{red}{$0$}$)$,$($\textcolor{orange}{$3$},\textcolor{blue}{$11$},\textcolor{red}{$0$}$)$,$($\textcolor{orange}{$4$},\textcolor{blue}{$4$},\textcolor{red}{$1$}$)$, $($\textcolor{orange}{$5$},\textcolor{blue}{$5$},\textcolor{red}{$1$}$)$, $($\textcolor{orange}{$6$},\textcolor{blue}{$3$},\textcolor{red}{$2$})$\}  \neq \{($\textcolor{orange}{$1$},\textcolor{blue}{$4$},\textcolor{red}{$0$}$)$, $($\textcolor{orange}{$2$},\textcolor{blue}{$7$},\textcolor{red}{$0$}$)$,$($\textcolor{orange}{$3$},\textcolor{blue}{$11$},\textcolor{red}{$0$}$)$,$($\textcolor{orange}{$4$},\textcolor{blue}{$3$},\textcolor{red}{$0$}$)$,$($\textcolor{orange}{$5$},\textcolor{blue}{$4$},\textcolor{red}{$2$}$)$,$($\textcolor{orange}{$6$},\textcolor{blue}{$5$},\textcolor{red}{$2$}$)$ $\} = I_{4,7,11,3,4,5}$
 give $p_4 p_7 p_{11} p_4 p_5 p_3 v_0 $ and $ p_4 p_7 p_{11} p_3 p_4 p_5 v_0$ respectively. 
 But, we want to remark that the two vectors are equal in the Fock space by the Heisenberg algebra relations.
 
 \item Next, we define the sets
 $$\cX_{i_1, i_2, \hdots,  i_{n}} = \bigsqcup_{s \in \text{Perm}(i_1, i_2, \hdots,   i_{n})} \cI_{s}, \hspace{3mm} \text{ and }  \hspace{3mm}  \cX_c = \bigsqcup_{p \in \{(i_1, i_2, \hdots,  i_{n})\}} \cX_{p}$$
 where the first set is taking the disjoint union over all permutations of  $(i_1, i_2, \hdots,  i_{n})$ and the second set is taking the disjoint union of all combinations of $(i_1, i_2, \hdots,  i_{n})$ in $\cX_{(i_1, i_2, \hdots,  i_{n})}.$
 Then, $\cX_c $ is the set of all marked cross configurations of $c.$
 By previous step,  there exists a a map $\psi : \cX_c \ra F^n_{A_{2n-1}}$ sending $\cI_{i_1, i_2, \cdots,  i_{n}} \mapsto   {p_{i_{1}}} p_{i_{2}} \cdots p_{i_{n}} v_0$. 
 By abuse of notation, we set $\cX_{p_{i_1}p_{i_2} \cdots p_{i_n}}:= \cX_{i_1, i_2, \hdots,  i_{n}}.$
%  If you further extend the the Note that the image of map never span the codomain.
\item After that,  we will associate a spreading polynomial $\fs \fp$ to each Fock vector $ p_{1}^{k_1}p_{2}^{k_2} \hdots p_{2n-1}^{k_{2n-1}} v_0$ which are defined  based on its preimage
$$\psi^{-1} \left( p_{1}^{k_1}p_{2}^{k_2} \hdots p_{2n-1}^{k_{2n-1}} v_0 \right)$$

following the steps below:
 \begin{enumerate}
 \item \label{parta}
 First,  we would like to compute the cardinality $\left|\psi^{-1} \left( p_{1}^{k_1}p_{2}^{k_2} \hdots p_{2n-1}^{k_{2n-1}} v_0 \right)  \right|$ of the preimage of a Fock vector $p_{1}^{k_1}p_{2}^{k_2} \hdots p_{2n-1}^{k_{2n-1}} $ under the map $\psi$ .  
  We outline a series of steps to express the cardinality as a \textbf{specific product} of binomial coefficients, that is the total number of a Fock vector $p_{1}^{k_1}p_{2}^{k_2} \hdots p_{2n-1}^{k_{2n-1}} v_0 $ appearing in $c$ will be in the form of $\prod_j {y_j \choose z_j}$ for some $y_j$ and $z_j.$ 
  It amounts to compute the number of ways getting each $p_j$ or the corresponding $j$-crosses.
  We start by choosing one element in the preimage   
 $\psi^{-1} \left( p_{1}^{k_1}p_{2}^{k_2} \hdots p_{2n-1}^{k_{2n-1}} v_0 \right).$

       Starting with a cross diagram of $c,$ for curves that has only one intersection with $v_{i},$ it is clear that the sole choice we can make is $p_i,$ in other words,  they contribute $1 \choose 1$. 
       We can then delete those marked crosses/rows of crosses with the same \textcolor{orange}{$r$}-label from the diagram.
       
       After that,  we form a counting fence for the  remaining marked crosses.  
     If the counting fence consists of one and only one counting rectangle, it is obvious that the number of effective lattice segments, in this case, the number of rows of the unique rectangle counts the possible combination of marked cross configuration in the form of product of binomial coefficients.

     When $p_j$ is present in an intersection  of more than one counting rectangles, this implies that there exists other marked crosses not in the lattice segment of the same counting rectangle containing $p_j$.
     The counting of possible $p_j$ combination will then be dependant on multiple rectangles. 
     In fact, each $p_j$ can only be contained in one and only one counting triangle or in the intersection of two counting rectangles.
     Since if it is in the intersection of more than two counting rectangles, there must exist a marked cross configuration such that it is only contained in the intersection of two counting rectangles by the nature of crossingless matching.
       
       Hence, we can always find a $i$ and at least one $p_i$  of $p_i^k$ such that $p_i$ lies in curve of the highest or lowest level among the other $p_j$'s lie  and  stays in one and only one counting rectangle.
       We would call it a \textit{single} $p_i.$
       In particular, $p_i$'s coming from curves that has only one intersection with $v_{i}$ are single.
       We should start by counting the $p_i^{k_i}$ which contain single $p_i.$
       Let $y_i$ be the number of effective lattice segments in the unique counting rectangle.
       Then,  we have $y_i \choose k_i$ way to pick $p_i^{k_i}$ from that unique counting rectangle containing $p_i.$
       We will then delete all the single $p_i$'s and the rows containing it.
       We continue to do the same for the other single $p_j$'s in the unique counting rectangle containing $p_i.$
       Collectively, the remaining $p_j$'s including $p_i$ will all contain in this unique counting rectangle containing $p_i$ and we will regard them as coming from the same column when forming the next counting fence.
       In principle, by counting and deleting the single $p_j$'s, we have fixed a configuration that they could appear in another counting rectangle that intersects with the original unique rectangle and subsequently, the column $p_j$ present become irrelevant. 
       
       We will keep picking out the single $p_j$'s until we reach a counting fence consisting of one and only one counting rectangle.
       Taking the product of the possible $ {y_j \choose z_j}$ for some $y_j$ and $z_j,$ we had computed the desired cardinality.

       To justify this, this is the least upper bound for the cardinality since to turn any potential lattice segments into effective lattice segments, it must be the case that there are other marked crosses in the lattice segments below or above the potential segments. 
        These are counted in the maximality of a counting fence.
       On the other hand, it is also the greatest lower bound because it realises the number of distinct combinations of set of marked crosses.

 For example, in \cref{c},  there are ${3 \choose 1} {2\choose 1}$ of way realising $p_3 p^2_4 p_5 p_7 p_{11} v_0 .$ 
 Depending on whether we count $p_3$ or $p_4$ first, we have $ {3 \choose 1} = {3 \choose 2}$ of way realising $p_3 p^3_4 p_7 p_{11}  v_0$ in $c.$

 \item Then, we associate a Fock vector $ {p_{1}^{k_1}p_{2}^{k_2} \hdots p_{2n-1}^{k_{2n-1}}} v_0$ a polynomial $\phi$ in $\Z[t,t^{-1}]$ based on an expression of the cardinality $\left| \cX_ {p_{1}^{k_1}p_{2}^{k_2} \hdots p_{2n-1}^{k_{2n-1}}} \right|$ = $\prod^k_{j=1} {y_j \choose z_j}$ of its preimage obtained from \cref{parta}.  
  To that expression, we can associate the product $\phi$ of the graded dimension of the cohomology of Grassmannian $Gr(z_j,y_j),$ or the Gaussian polynomial, that is, 
 
  \[ \phi(p_{1}^{k_1}p_{2}^{k_2} \hdots p_{2n-1}^{k_{2n-1}} v_0) = \prod^k_{j=1} g_r\text{dim} Gr(z_j, y_j) = \prod^k_{j=1} \begin{bmatrix}
  y_j  \\ z_j
  \end{bmatrix} 
  \]

\noindent where $[n] := \frac{1-t^n}{1-t} = 1 + t + t^2 + \cdots + t^{n-1}, [n]! := \prod_{j=1}^{n}[j] = \sum_{w \in S_n} t^{\text{inv}(w)}$ $\begin{bmatrix}  y  \\ z  \end{bmatrix} := \frac{[y]!}{[y]![y-z]!} = \frac{(1-t^{y})(1-t^{y-1}) \cdots (1-t^{y-z+1})}{(1-t)(1-t^2) \cdots (1-t^z)}.$
  Note that the quantum integer defined here is not the shifted one as in \cref{sec2}. 

  At first sight,  it might be lured that different expressions could give different $\phi$. 
    For example,  $\phi (p_3 p_4^3  p_7 p_{11}  v_0) = \begin{bmatrix}  3  \\ 2  \end{bmatrix} = 1 + t + t^2 = \begin{bmatrix}  3  \\ 1  \end{bmatrix} .$ 
To explain it, we need this lemma.

\begin{lemma} \label{perm}
Suppose $\omega \in S_n.$ Then,

$$\begin{bmatrix}  j_1 + j_2 + \cdots + j_n  \\ j_n  \end{bmatrix} \cdots  \begin{bmatrix}  j_1 + j_2   \\ j_2  \end{bmatrix}\begin{bmatrix}  j_1   \\ j_1  \end{bmatrix} 
= \begin{bmatrix}  j_{\omega(1)} + j_{\omega(2)} + \cdots + j_{\omega(n)}  \\ j_{\omega(n)}  \end{bmatrix} \cdots  \begin{bmatrix}  j_{\omega(1)} + j_{\omega(2)}   \\ j_{\omega(2)}  \end{bmatrix}\begin{bmatrix}  j_{\omega(1)}   \\ j_{\omega(1)}  \end{bmatrix}.$$

\end{lemma}

\begin{proof}

Note that the a Gaussian polynomial $\begin{bmatrix}  r+s   \\ s  \end{bmatrix}$  can be computed by counting the $q$-graded partition of $m \leq s \times r$ that fits in $(s \times r)$-rectangular partition. 
  By the transpose symmetry of the counting method, it is not hard to see that we have
  
  \begin{equation} \label{transpose symmetry}
   \begin{bmatrix}  r + s  \\ s  \end{bmatrix} = \begin{bmatrix}  r + s  \\ r  \end{bmatrix}.
  \end{equation}

Also, we have 
\begin{equation} \label{chain rule symmetry}
     \begin{bmatrix}  r  \\ s  \end{bmatrix} \begin{bmatrix}  s  \\ t  \end{bmatrix} =  \begin{bmatrix}  r  \\ t  \end{bmatrix} \begin{bmatrix}  r -t \\ r-s  \end{bmatrix} 
\end{equation} 
 
since 
  $$\begin{bmatrix}  r  \\ s  \end{bmatrix} \begin{bmatrix}  s  \\ t  \end{bmatrix} 
  = \frac{[r]!}{[s]![r-s]!} \frac{[s]!}{[t]![s-t]!} 
  =  \frac{[r]!}{[t]!} \frac{1}{[r-s]! [s-t]!}
   =  \frac{[r]!}{[t]![r-t]!} \frac{[r-t]!}{[r-s]! [s-t]!} 
   = \begin{bmatrix}  r  \\ t  \end{bmatrix} \begin{bmatrix}  r -t \\ r-s  \end{bmatrix}.$$

We will prove \cref{perm} by inducting on $n.$
  When $n=1,$ the statement is trivial. 
  When $n=2,$ the statement is the manifest of \cref{transpose symmetry} and the last factor can easily be  transformed  $$\begin{bmatrix}  j_1  \\ j_1  \end{bmatrix} = \begin{bmatrix}  j_2  \\ j_2  \end{bmatrix}.$$  
  Just to demonstrate how \cref{chain rule symmetry} comes into proving the general statement, we start with $n=3$.
  Suppose $n=3,$  if $w$ fixes $3,$ then it is really the case when $n=2.$
  Then, we are left with $\omega \in \{ (23), (13), (123), (132) \},$ or in other words, $\omega$ sends $3$ to $2$ or $1.$
  Let us deal with $\omega (3) = 1$ since the other case require less manipulation.
  \begin{align*}
  \begin{bmatrix}  j_1 + j_2 + j_3  \\ j_3  \end{bmatrix}   \begin{bmatrix}  j_1 + j_2   \\ j_2  \end{bmatrix}\begin{bmatrix}  j_1   \\ j_1  \end{bmatrix}  
  &=    \begin{bmatrix}  j_1 + j_2 + j_3  \\ j_1 + j_2  \end{bmatrix}   \begin{bmatrix}  j_1 + j_2   \\ j_1  \end{bmatrix}\begin{bmatrix}  j_1   \\ j_1  \end{bmatrix}         (\cref{transpose symmetry})  \\
&=  \begin{bmatrix}  j_1 + j_2 + j_3  \\ j_1  \end{bmatrix} \begin{bmatrix}  j_2 + j_3   \\ j_3  \end{bmatrix} \begin{bmatrix}  j_1   \\ j_1   \end{bmatrix}    (\cref{chain rule symmetry})\\
&= \begin{bmatrix}  j_1 + j_2 + j_3  \\ j_1  \end{bmatrix} \begin{bmatrix}  j_2 + j_3   \\ j_3  \end{bmatrix} \begin{bmatrix}  j_2   \\ j_2   \end{bmatrix} \\
&= \begin{bmatrix}  j_1 + j_2 + j_3  \\ j_1  \end{bmatrix} \begin{bmatrix}  j_2 + j_3   \\ j_2  \end{bmatrix} \begin{bmatrix}  j_3   \\ j_3   \end{bmatrix}   (\cref{transpose symmetry})   \\
&= \begin{bmatrix}  j_{\omega(1)} + j_{\omega(2)}  + j_{\omega(3)}  \\ j_{\omega(3)}  \end{bmatrix} \begin{bmatrix}  j_{\omega(1)} + j_{\omega(2)}   \\ j_{\omega(2)}  \end{bmatrix}\begin{bmatrix}  j_{\omega(1)}   \\ j_{\omega(1)}  \end{bmatrix}
  \end{align*}
  
  for $\omega \in \{ (13), (123)  \}.$
  
  Suppose it is true for $n \leq m-1.$ We want prove for $m.$
  For $\omega \in S_m$ fixes $m,$ it is done by induction hypothesis.
  Now, let $\omega$ sends $m$ to $\omega(m) \neq m.$
\begin{align*} & \begin{bmatrix}  j_1 + j_2 + \cdots + j_n  \\ j_n  \end{bmatrix} \begin{bmatrix}  j_1 + j_2 + \cdots + j_{n-1}  \\ j_{n-1}  \end{bmatrix} \cdots  \begin{bmatrix}  j_1 + j_2   \\ j_2  \end{bmatrix}\begin{bmatrix}  j_1   \\ j_1  \end{bmatrix}   \\
&= \begin{bmatrix}  j_1 + j_2 + \cdots + j_n  \\ j_n  \end{bmatrix} \begin{bmatrix}  j_{\omega'({1})} + j_{\omega'({2})} + \cdots + j_{\omega'({n-1})}  \\ j_{\omega'({n-1})}  \end{bmatrix} \cdots  \begin{bmatrix}  j_{\omega'({1})} + j_{\omega'({2})}   \\ j_{\omega'({2})}  \end{bmatrix}\begin{bmatrix}  j_{\omega'({1})}   \\ j_{\omega'({1})}  \end{bmatrix}     \\
&= \begin{bmatrix}  j_1 + j_2 + \cdots + j_n  \\  j_{\omega'({1})} + j_{\omega'({2})} + \cdots + j_{\omega'({n-1})}  \end{bmatrix} \begin{bmatrix}  j_{\omega'({1})} + j_{\omega'({2})} + \cdots + j_{\omega'({n-1})}  \\ j_{\omega'({n-1})}  \end{bmatrix} \cdots  \begin{bmatrix}  j_{\omega'({1})} + j_{\omega'({2})}   \\ j_{\omega'({2})}  \end{bmatrix}\begin{bmatrix}  j_{\omega'({1})}   \\ j_{\omega'({1})}  \end{bmatrix}     \\
&= \begin{bmatrix}  j_1 + j_2 + \cdots + j_n  \\  j_{\omega'({n-1})}   \end{bmatrix} 
\begin{bmatrix}  j_1 + j_2 + \cdots + j_n - j_{\omega'({n-1})}   \\ j_n  \end{bmatrix} \cdots  \begin{bmatrix}  j_{\omega'({1})} + j_{\omega'({2})}   \\ j_{\omega'({2})}  \end{bmatrix}\begin{bmatrix}  j_{\omega'({1})}   \\ j_{\omega'({1})}  \end{bmatrix}    \\
&= \begin{bmatrix}  j_{\omega(1)} + j_{\omega(2)} + \cdots + j_{\omega(n)}  \\ j_{\omega(n)}  \end{bmatrix} \cdots  \begin{bmatrix}  j_{\omega(1)} + j_{\omega(2)}   \\ j_{\omega(2)}  \end{bmatrix}\begin{bmatrix}  j_{\omega(1)}   \\ j_{\omega(1)}  \end{bmatrix}
\end{align*}
 
 where the first equality follows from IH that there exists $ \omega' \in S_{m-1}$ such that $ \omega'(n-1) = \omega(n),$
 the second equality follows from \cref{transpose symmetry}, and the third equality follows from \cref{chain rule symmetry}, and the fourth equality follows from induction hypothesis again. 
\end{proof}

By the lemma above, the Gaussian polynomial of a Fock vector associated to a crossingless matching is an invariant of the  cardinality  obtained in \cref{parta}.
  We need to argue why the order of ``$t$-counting" doesn't matter. 
  If there are only two choices to make in the counting rectangle, then it is equivalent to  $\cref{transpose symmetry}.$
  If there are more than two choices exist in the counting rectangle, then it is covered by $\cref{perm}.$
  It is not hard to extend the argument to $p_j$'s containing in the intersection of two rectangles.

\begin{lemma} \label{ell}
The spreading polynomial  $\fs \fp({p_{i_{1}}^{\lambda_1}p_{i_{2}}^{\lambda_2} \hdots p_{i_{2n-1}}^{\lambda_{2n-1}} v_0})$ of a Fock vector doesn't depend on the order you taking the $p_{i_j},$ or in another word,  it does not depend on the preimage  $\psi^{-1}({p_{i_{1}}^{k_1}p_{i_{2}}^{k_2} \hdots p_{i_{2n-1}}^{k_{2n-1}} v_0}).$ 
\end{lemma}

 \item As mentioned before,  a single Fock vector can be further decomposed into the basis of Fock space in terms of multipartitioned Fock vector $p_{1}^{\lambda_1}p_{2}^{\lambda_2} \hdots p_{2n-1}^{\lambda_{2n-1}} v_0$,  $\lambda_i \vdash k_i.$
 
 Each basis element will also be associated to a polynomial $\theta \in \Z[t, t^{-1}]$ based on the major statistic with standard Young tableaux fillings on multipartition.
     
     For each partition $\lambda \vdash n$,  the total number of standard Young tableaux  associated can be interpreted as the dimension of irreducible representation of $S_n$ or more directly the number of Gelfand-Tsetlin basis.
     We now define a major statistic to these standard Young tableaux.
     Given a standard Young tableau $\lambda_S$ of $\lambda,$ a descent of $T$ is defined to be a value of $i,$ $1 \leq i \leq |\lambda|$ for which $i + 1$ occurs in a row below $i$ in $\lambda_S.$
     A \textit{major statistic} $m_{{\lambda_S}}$ associated to a standard Young tableau $\lambda_S$ is the sum 
     $$m_{{\lambda_S}} := \sum_{i \in \text{descent of $\lambda_S$}} i.$$
     Let $q$ be the dimension of irreducible representation of $\lambda.$
     Then 
     $$\theta_\lambda := \sum_{i =1}^q {t^{m_{{\lambda_{S_i}}}}}$$
   and, for multipartitioned case,
 $$ \theta (p_{1}^{\lambda_1}p_{2}^{\lambda_2} \hdots p_{2n-1}^{\lambda_{2n-1}} v_0) := \theta_{\lambda_1} \theta_{\lambda_2} \cdots \theta_{\lambda_{2n-1}}. $$

 For example,  in the Fock space representation, $p_3 p_4^2  p_5 p_7 p_{11} v_0= p_3  p_4^{(2)} p_5 p_7 p_{11} v_0 + p_3 p_4^{(1^2)} p_5  p_7 p_{11} v_0.$
 Then,  $\theta (p_3  p_4^{(2)} p_5  p_7 p_{11} v_0) = 1 $ and $  \theta ( p_3 p_4^{(1^2)} p_5  p_7 p_{11} v_0)= t.$ In the case of $p_4^{(2,1)} p_7 p_{11} p_9 v_0,$ we get $\theta(p_4^{(2,1)} p_7 p_{11} p_9 v_0 ) = t + t^2.$ Another example would be  $p_4^{(2)}p_5^{(1^2)}p_7 p_{11} v_0,$ we get $\theta (p_4^{(2)}p_5^{(1^2)}p_7 p_{11} v_0) = 1 \cdot t = t.$

\item  It can be shown(\cref{ell + 1 term}) that,  all the elements in a single set of intersection points has the same total sum $\ell$ of \textcolor{red}{$\ell_{r,s}$}. Then,  we will assign the spreading polynomial $\fs \fp(p_{i_{1}}^{\lambda_1}p_{i_{2}}^{\lambda_1} \hdots p_{i_{2n-1}}^{\lambda_{2n-1}} v_0)  \in \Z[t,t^{-1}] $ of an existing basis element $p_{s_{1}}^{\lambda_1}p_{i{2}}^{\lambda_1} \hdots p_{i_{2n-1}}^{\lambda_{2n-1}} v_0$ in $c$ by 
   
$$\fs \fp(p_{i_{1}}^{\lambda_1}p_{i_{2}}^{\lambda_2} \hdots p_{i_{2n-1}}^{\lambda_{2n-1}} v_0) = \theta \left( p_{i_{1}}^{\lambda_1}p_{i_{2}}^{\lambda_1} \hdots p_{i_{2n-1}}^{\lambda_{2n-1}} v_0 \right) \phi \left( p_{i_{1}}^{|\lambda_1|}p_{i_{2}}^{|\lambda_2|} \hdots p_{i_{2n-1}}^{|\lambda_{2n-1}|} v_0 \right)  \Biggr\rvert_{t \mapsto t^2} \cdot t^{- \ell}.$$

Note that the definition of $\fs \fp$ can be extended linearly, so that $\fs \fp(p_{i_{1}}^{k_1}p_{i_{2}}^{k_2} \hdots p_{i_{2n-1}}^{k_{2n-1}} v_0),$ for $k_i \in \Z$ is defined to be the sum of the spreading polynomial of the basis elements that it decomposes into.

For example,  $\{ ($\textcolor{orange}{$1$},\textcolor{blue}{$4$},\textcolor{red}{$0$}$)$,$($\textcolor{orange}{$2$},\textcolor{blue}{$7$},\textcolor{red}{$0$}$)$,$($\textcolor{orange}{$3$},\textcolor{blue}{$11$},\textcolor{red}{$0$}$)$,$($\textcolor{orange}{$4$},\textcolor{blue}{$4$},\textcolor{red}{$1$}$)$,$($\textcolor{orange}{$5$},\textcolor{blue}{$5$},\textcolor{red}{$1$}$)$,$($\textcolor{orange}{$6$},\textcolor{blue}{$3$},\textcolor{red}{$2$}$) \} \\ \neq \{ ($\textcolor{orange}{$1$},\textcolor{blue}{$4$},\textcolor{red}{$0$}$)$,$($\textcolor{orange}{$2$},\textcolor{blue}{$7$},\textcolor{red}{$0$}$)$,$($\textcolor{orange}{$3$},\textcolor{blue}{$11$},\textcolor{red}{$0$}$)$,$($\textcolor{orange}{$4$},\textcolor{blue}{$3$},\textcolor{red}{$0$}$)$,$($\textcolor{orange}{$5$},\textcolor{blue}{$4$},\textcolor{red}{$2$}$)$, $($\textcolor{orange}{$6$},\textcolor{blue}{$5$},\textcolor{red}{$2$}$) \}$, but $0 + 0 +0 +1 + 1 +2 = 4 = 0+ 0 + 0 +0 +2 +2.$ 
  So, $\ell_{p_3 p_4^{2} p_5 p_7 p_{11} v_0} = 4.$
  Then,  
  $$\fs \fp \left( p_3 p_4^{(2)} p_5 p_7 p_{11}  v_0 \right) = 1 \cdot (1 + 2 t + 2 t^2  + t^3)   \Biggr\rvert_{t \mapsto t^2} \cdot t^{- 4} = t^{-4} + 2t^{-2} + 2 + t^2,$$ and
  $$\fs \fp \left(p_3 p_4^{(1^2)} p_5 p_7 p_{11} v_0  \right) = t \cdot (1 + 2 t + 2 t^2  + t^3)   \Biggr\rvert_{t \mapsto t^2} \cdot t^{- 4} = t^{-2} + 2 + 2 t^2 + t^4.$$
We can extend the definition as a homomorphism to a general Fock vector $p_{1}^{k_1}p_{2}^{k_2} \hdots p_{2n-1}^{k_{2n-1}}.$
For example,  
\begin{align*}
\fs \fp(p_3 p_4^2 p_5 p_7 p_{11} v_0) 
&= \fs \fp \left(p_3 p_4^{(2)} p_5 p_7 p_{11}  v_0 \right) + \fs \fp \left(p_3 p_4^{(1^2)} p_5 p_7 p_{11}   v_0 \right) \\
&=  t^{-4} + 3t^{-2} + 4 + 3t^2  + t^{4}
\end{align*}
which has $5 = \ell + 1$ terms. This is a fact that we will prove in \cref{ell + 1 term}.

 \end{enumerate}  
  
  \item Consider all such sets of intersection points.  We will associate a crossingless matching $c$ a finite sum $F(c)$ of Fock vectors based on the set
\begin{center}
$\underset{r,i} {\bigcup}\{$(\textcolor{orange}{$c_r$},\textcolor{blue}{$v_i$},\textcolor{red}{$\ell_{r,i}$})$\mid \text{no repeated $r$} \}.$
\end{center}  
The set above will give us a list of Fock vectors contained in $c$.
We then have

$$v_c = \sum_{   p_{i_{1}}^{\lambda_1}p_{i_{2}}^{\lambda_2} \hdots p_{i_{2n-1}}^{\lambda_{2n-1}} v_0 \in  \psi( X_c) } \fs \fp(p_{i_{1}}^{\lambda_1}p_{i_{2}}^{\lambda_2} \hdots p_{i_{2n-1}}^{\lambda_{2n-1}} v_0) \cdot p_{i_{1}}^{\lambda_1}p_{i_{2}}^{\lambda_2} \hdots p_{i_{2n-1}}^{\lambda_{2n-1}} v_0.$$

\end{enumerate}

\begin{comment}
\begin{lemma} \label{ell}
The spreading polynomial  $\fs \fp({p_{i_{1}}^{\lambda_1}p_{i_{2}}^{\lambda_2} \hdots p_{i_{2n-1}}^{\lambda_{2n-1}} v_0})$ of a Fock vector doesn't depend on the order you taking the $p_{i_j},$ or in another word,  it does not depend on the preimage  $\psi^{-1}({p_{i_{1}}^{k_1}p_{i_{2}}^{k_2} \hdots p_{i_{2n-1}}^{k_{2n-1}} v_0}).$ 
\end{lemma}

\begin{proof}
The one that needs explaining is about $\ell.$ And this is because the level-labelling of multicurves reflects how many choices you can make for a given $p_{i_m}.$ Hence, all the element in preimage  $\psi^{-1}({p_{i_{1}}^{k_1}p_{i_{2}}^{k_2} \hdots p_{i_{2n-1}}^{k_{2n-1}} v_0})$ would have the same $\ell.$
\end{proof}

\end{comment}

\begin{comment}
\begin{lemma}
For each Fock vector $p_{i_{1}}^{k_1}p_{i_{2}}^{k_2} \hdots p_{i_{2n-1}}^{k_{2n-1}} $,  the number of its isomorphic terms can be computed as a product of combinations.
\end{lemma}

\begin{proof}
It is best to work with the cross diagram associate to $c.$
   For curves that has only one intersection with $v_i,$ they only contribute to one possibilities that is $ 1 \choose 1.$
   For $p_{s_i}$ containing in curves that has more than one intersection with $v_i,$ it is possible to found $v_i$ in a rectangle of height  $y_{i}$ in the crossing diagram.  Then, there is $y_{i} \choose k_i$ of $p_{i_k}.$
\end{proof}
\end{comment}

\begin{lemma} \label{ell + 1 term}
The spreading polynomial $\sum_{ \lambda_1 \vdash k_1,  \lambda_2 \vdash k_2, \cdots \lambda_{2n-1} \vdash k_{2n-1} } \fs \fp(p_{i_{1}}^{\lambda_1}p_{i_{2}}^{\lambda_2} \hdots p_{i_{2n-1}}^{\lambda_{2n-1}} v_0)$  has $ \ell + 1$ terms.
Moreover,  the spreading polynomial $\sum_{ \lambda_1 \vdash k_1,  \lambda_2 \vdash k_2, \cdots \lambda_{2n-1} \vdash k_{2n-1} } \fs \fp(p_{i_{1}}^{\lambda_1}p_{i_{2}}^{\lambda_2} \hdots p_{i_{2n-1}}^{\lambda_{2n-1}} v_0) \big\rvert_{t=1} $ counts the total number of basis elements required to express $p_{i_{1}}^{| \lambda_1 |}p_{i_{2}}^{| \lambda_2 |} \hdots p_{i_{2n-1}}^{|\lambda_{2n-1}| } v_0$.
\end{lemma}

%The cochain complex associated to a crossingless matching $c$ spans $2 d_{\textbf{P}} + 1$ cohomological degree from $-d_{\textbf{P}}$ to $d_{\textbf{P}}$ where $\textbf{P}$ is the Heisenberg module with the highest degree.
 %Moreover,  its indecomposable components span the alternate degrees

\begin{proof}  Suppose $\ell = 0.$ 
That means there is only one possible choice of $p_{i_{1}}^{\lambda_1}p_{i_{2}}^{\lambda_2} \hdots p_{i_{2n-1}}^{\lambda_{2n-1}} v_0.$ 
Note that in this case $\lambda_i \in \{0, 1\}.$
   Then, by construction, it  is a single graded vector over $\Z.$
   Suppose now,  we only look at a Fock vector $p_i^{k_i}$ consisting of a repeated indices $i.$
   Then, we have $p_i^{k_i} = \sum_{\lambda_i \vdash k_i} p_i^{\lambda_i}.$
   Then, $\ell + 1 = 0 + 1 + 2 + \cdots + k_i-2 + k_i-1 + 1, $ but $\fs \fp (p_i^{k_i}) = \sum_{\lambda_i \vdash k_i} \fs \fp(p_i^{\lambda_i}) = [k_i]! $ which has $k_i -1 + k_{i} -2 + k_{i} -3 + \cdots + 1,$ as desired. 
   
   For the rest of the case, note that we can reduce the case of a single non-decomposable crossingless matching since the function $\fs \fp$ is multiplicative on the decomposable crossingless matchings. 
   First, we discuss the case when the crossingless matching is non-nested.
   Consider its associated lattice diagram where it is noticed as a pyramid with the tip, say $p_k.$ 
   By discussion before, we can pick a specific representative of the vector by marking the crosses from top to the bottom in the following sequence:    $p_k^{\lambda_k}, p_{k-1}^{\lambda_{k-1}}, p_{k+1}^{\lambda_{k+1}}, p_{k-2}^{\lambda_{k-2}}, p_{k+2}^{\lambda_{k+2}}, \hdots.$
   Forming the counting fence, we see that the counting rectangles are those with top vertices having $\ell$-label as $0.$ 
   Note that, by precious paragraph, $\ell$-label on $p_{i_j}^{\lambda_{i_j}}$ carries the contribution of $ \lambda_{i_j} - 1 +  \lambda_{i_j} - 2 + \cdots + 1.$
   Then, from the desingated sequence, it is easy to see that the $\ell$-label on $p_{i_j}^{\lambda_{i_j}}$ is the total contribution from the way to choose  $p_{i_j}$ in the lattice diagram and the spreading polynomial function.
   To see it, we could start counting the way to choose the last term $p_{i_j}$ in the sequence since it is not in the intersection with other counting rectangles or, in other words, choosing it is independent of choosing the other.
   The $\ell$-label of the first appearing $p_{i_j}$ indicates there are $\begin{pmatrix}  \ell +  \lambda_{i_j}  \\ \lambda_{i_j} \end{pmatrix}$ choices of it which its Gaussian polynomial $$\begin{bmatrix}  \ell +  \lambda_{i_j}  \\ \lambda_{i_j} \end{bmatrix}$$ has degree $\ell \cdot \lambda_{i_j}$. Here, we recall that a degree of a polynomial with one indeterminate is the highest power attaining by the indeterminate appearing in the summand.
   Similarly, we apply the method to all the other $p_{i_j}.$  
   We then obtained the desired result.

 For the case of nested crossingless matching, it might be the case that all the counting rectangles has nontrivial intersections on marked crosses. 
 However, it is not hard to extend the argument to get the desired result.

Statement 2 is obvious.
\end{proof}

\begin{conjecture} \label{graTemp}

The span of the graded vectors $v_c$ in $F^{n}_{A_{2n-1}}$ form a Temperley-Lieb representation.

\end{conjecture}

\begin{conjecture} \label{preJones}
By realising a link in $3$-sphere as a plat closure of a braid \cite{BirBraid}, through Temperley-Lieb representation \cite{Kauff} , it consists of a combination of gluing crossingless matchings $c$ and upside down crossingless matchings $c^t$.
  This topological gluing can be interpreted as multiplying the vector $v_{c^t}$ (replacing the generator $p$ with $q$) with the vector $v_c.$ 
  Then, one can recover the Jones polynomial of the link by counting the graded vacuum vectors in the Fock space.
\end{conjecture}

 % \include{Type B Floer Cohomology}
  %etc

% APPENDICES

  % change chapter name and counters (eg Chapter 1 -> Appendix A)
  \appendix

  % assuming there are files appendix1.tex etc...
 % \include{appendix1}
  %\include{appendix2}
  %\include{appendix3}
  %etc

% BIBLIOGRAPHY

  % add Bibliography to table of contents
  \addcontentsline{toc}{chapter}{Bibliography}

  % list BibTeX (.bib) files and choose bibliography style
 %\bibliography{references1,references2}
 % \bibliographystyle{abbrv}

  % OR... do it manually in the file bibliography.tex
 \include{bibB2}
         \printbibliography
\end{document}